\documentclass[11 pt,letterpaper]{amsart} 
\usepackage{amssymb}
\usepackage{amsmath}
\usepackage{amsthm}
\usepackage[pdfpagelabels,hyperindex,hidelinks]{hyperref} 
\usepackage{amsfonts} 
\usepackage{amsmath, amssymb} 
\usepackage{graphicx} 
\usepackage[font=small,labelfont=bf]{caption} 
\usepackage{epstopdf} 
\usepackage[pdfpagelabels,hyperindex,hidelinks]{hyperref}
\usepackage{xcolor} 
\usepackage{amsthm} 
\usepackage{float}
\usepackage{pgfplots}
\usepackage{listings}
\usepackage{longtable} 
\usepackage{mathrsfs}
\usepackage[T1]{fontenc}
\usepackage{lmodern}
\usepackage{microtype} 
\usepackage[margin=1.3in]{geometry} 
\usepackage{mathtools}
\usepackage{float}
\usepackage{bm}
\makeatletter
\renewcommand{\normalsize}{%
	\@setfontsize\normalsize\@xiipt{14}
	\abovedisplayskip 12pt plus 3pt minus 7pt
	\abovedisplayshortskip \z@ plus 3pt
	\belowdisplayshortskip 6.5pt plus 3.5pt minus 3pt
	\belowdisplayskip \abovedisplayskip
}
\makeatother

\theoremstyle{plain}
\newtheorem{theorem}{Theorem}\numberwithin{theorem}{section}
{}
\newtheorem{main}{Main~Theorem}{}
\newtheorem{lemma}{Lemma}\numberwithin{lemma}{section}
\newtheorem{proposition}{Proposition}\numberwithin{proposition}{section}
\numberwithin{proposition}{section}
\numberwithin{corollary}{section}
\newtheorem{claim}{Claim}\numberwithin{claim}{section}
\theoremstyle{definition}
\newtheorem{definition}{Definition}\numberwithin{definition}{section}
\theoremstyle{remark} \theoremstyle{exam} \theoremstyle{ob}
\newtheorem{remark}{Remark}\numberwithin{remark}{section}
\numberwithin{app}{section}
\newtheorem{question}{Question}\numberwithin{question}{section}
\newtheorem{exam}{Example}\numberwithin{exam}{section}
\newtheorem{ob}{Observation}\numberwithin{ob}{section}
\numberwithin{nt}{section}
\numberwithin{equation}{section}

\usepackage{etoolbox}

\renewcommand{\contentsnamefont}{\bfseries}

\makeatletter

\patchcmd{\@starttoc}
{\centering\contentsnamefont}
{\raggedright\contentsnamefont}
{}{}

\def\@tocline#1#2#3#4#5#6#7{%
	\relax
	\ifnum #1>\c@tocdepth
	\else
	\par
	\addpenalty\@secpenalty
	\addvspace{#2}%
	\begingroup
	\hyphenpenalty\@M
	\@ifempty{#4}
	{%
		\@tempdima
		\csname r@tocindent\number#1\endcsname
		\relax
	}
	{%
		\@tempdima#4\relax
	}%
	\parindent\z@
	\leftskip#3\relax
	\advance\leftskip\@tempdima\relax
	\rightskip\@pnumwidth plus 4em
	\parfillskip-\@pnumwidth
	#5%
	\leavevmode
	\hskip-\@tempdima
	#6\nobreak
	\ifnum #1=2
	\leaders\hbox to .8em{\hss.\hss}\hfill
	\else
	\hfil
	\fi
	\nobreak
	\hbox to\@pnumwidth{\@tocpagenum{#7}}%
	\par
	\nobreak
	\endgroup
	\fi
}

\def\l@section{%
	\@tocline{1}{6pt plus 1pt}{0pt}{2.5pc}{\bfseries}%
}

\def\l@subsection{%
	\@tocline{2}{0pt}{1pc}{5pc}{}%
}

\makeatother

\begin{document}
\title[Multi-Parameter Exponential Sums with Product Hilbert Kernels]
{Multi-Parameter Exponential Sums with Product Hilbert Kernels}

\author
{Joonil Kim,  Hoyoung Song}

\address{Joonil Kim, Department of Mathematics \\
	Yonsei University \\
	Seoul 120-729, Republic of Korea}
\email{jikim7030@yonsei.ac.kr}
\address{Hoyoung Song, Department of Mathematics \\
	Yonsei University \\
	Seoul 120-729, Republic of Korea}
\email{nonspin0070@gmail.com}


 \begin{abstract}
 We establish necessary and sufficient conditions for the uniform boundedness of the multi-parameter exponential sums with product Hilbert kernels
$$
\sum_{1\le|t_1|\le N_1,\cdots,1\le|t_k|\le N_k} \frac{e^{2\pi i P(t_1,\dots,t_k)}}{t_1\cdots t_k},
$$
where $P:\mathbb{Z}^k\to\mathbb{R}$ is a polynomial of the form
$
P(t)=\sum_{\mathfrak{m}\in \Lambda} c_{\mathfrak{m}}\, t^{\mathfrak{m}},
$
with real coefficients. The resulting bound is uniform in both the coefficients
$c_{\mathfrak m}$ and the truncation parameters $N_1,\ldots,N_k$.  To this end, we develop a higher-dimensional version of the multi-parameter circle method.  Under the sufficient condition, we further prove $\ell^p$-boundedness of the associated discrete multiple Hilbert transform.
	\end{abstract}

\date{\today}

\maketitle

\setcounter{tocdepth}{2}
\tableofcontents

  \section{Introduction}
The exponential sum
\[
m_N^{\rm disc}(\xi)
:=
\sum_{1\le |t|\le N}
\frac{e^{2\pi i\xi t}}{t}
\]
is one of the fundamental objects in classical Fourier analysis. It is
the Fourier multiplier associated with the truncated discrete Hilbert
transform and lies at the heart of the classical theory of convergence
and divergence of conjugate Fourier series. A cornerstone of the classical theory is the estimate
\[
\sup_{N\in\mathbb N,\ \xi\in\mathbb R}|m_N^{\rm disc}(\xi)|<\infty.
\]
A key step in its proof is the uniform comparison
\begin{align}\label{cdp1}
	m_N^{\rm disc}(\xi)
	=
	m_N^{\rm cont}(\xi)+O(1),
\end{align}
where
\[
m_N^{\rm cont}(\xi)
:=
\int_{1\le |t|\le N+1}
\frac{e^{2\pi i\xi t}}{t}\,dt
\]
is the Fourier multiplier associated with the truncated continuous
Hilbert transform. The comparison \eqref{cdp1} follows from the
Euler--Maclaurin summation formula, while the boundedness of the
continuous multiplier is a direct consequence of oscillatory
cancellation, whose limiting multiplier is
$-i\pi\, \mathrm{sgn}(\xi)$, the Fourier multiplier of the Hilbert
transform.

The study of multi-parameter analogues began with the development of
multiple Fourier series by Hardy and his collaborators. The
corresponding conjugate Fourier series naturally lead to the multiplier
\[
m_{\vec N}^{\rm disc}(\xi)=
\sum_{1\le |t_1|\le N_1}\cdots
\sum_{1\le |t_k|\le N_k}
\frac{e^{2\pi i\langle\xi,t\rangle}}
{t_1\cdots t_k},
\]
which is associated with the simplest discrete multiple Hilbert
transform.  Owing to its product structure, the corresponding theory
parallels the classical one-parameter theory of the Hilbert
transform. In particular, the analogue of the comparison \eqref{cdp1} remains valid in this product setting.

A natural generalization is to replace the linear phase
$\langle\xi,t\rangle$ by the polynomial phase
\[
\sum_{\mathfrak m\in\Lambda}
\xi_{\mathfrak m}t^{\mathfrak m},
\]
where $\Lambda\subset\mathbb Z_+^k$ is a finite set,
$\xi=(\xi_{\mathfrak m})_{\mathfrak m\in\Lambda}\in\mathbb R^{|\Lambda|}$,
and
$t^{\mathfrak m}=t_1^{m_1}\cdots t_k^{m_k}$.

In the one-parameter setting $(k=1)$, Arkhipov and Oskolkov \cite{AO}
proved that for every finite
$\Lambda\subset\mathbb Z_+$,
there exists a constant $C_\Lambda>0$ such that
\begin{align}\label{11}
	\sup_{N\in\mathbb N,\,
		\xi=(\xi_{\mathfrak m})_{\mathfrak m\in\Lambda}\in\mathbb R^{|\Lambda|}}
	\left|
	\sum_{1\le |t|\le N}
	\frac{
		e^{2\pi i
			\sum_{\mathfrak m\in\Lambda}
			\xi_{\mathfrak m}t^{\mathfrak m}}
	}{t}
	\right|
	\le C_\Lambda.
\end{align}

In the multi-parameter setting $k\ge 2$, the theory of multi-parameter Weyl sums began with the work of
Arkhipov, Chubarikov, and Karatsuba \cite{AC} and was developed
systematically in the monograph \cite{ACK}.  They investigated the
exponential sum 
\[
\sum_{1\le |t_1|\le N_1}\cdots
\sum_{1\le |t_k|\le N_k}
e^{2\pi i \sum_{\mathfrak{m}\in \Lambda} \xi_{\mathfrak{m}} t^{\mathfrak{m}} },
\]
for each finite $\Lambda\subset \mathbb{Z}^k_+$,  establishing fundamental results on the uniform distribution of polynomial phases over arbitrary rectangular domains.   However, their theory concerns exponential sums without the singular
kernel
\[
\frac1{t_1\cdots t_k}.
\]
It is therefore natural to ask whether the one-parameter theorem of
Arkhipov and Oskolkov admits a multi-parameter analogue, namely  for every finite
$\Lambda\subset\mathbb Z_+^k$   whether 
there exists a constant
$C_\Lambda>0$ such that
\begin{align}\label{1tt}
	\sup_{\substack{
			\vec N\in\mathbb N^k\\
			\xi=(\xi_{\mathfrak m})_{\mathfrak m\in\Lambda}\in\mathbb R^{|\Lambda|}
	}}
	\left|
	\sum_{t\in R(\vec N)\cap \mathbb{Z}^k}
	\frac{e^{2\pi i\sum_{\mathfrak m\in\Lambda}
			\xi_{\mathfrak m}t^{\mathfrak m}}}{t_1\cdots t_k}
	\right|
	\le C_\Lambda 
\end{align}
where
\[
R(\vec N)
=
\prod_{\nu=1}^k
\{t_\nu\in \mathbb{R}:1\le |t_\nu|\le N_{\nu}\}\ \text{for}\ \vec N=(N_1,\ldots,N_k)\in\mathbb N^k.
\]

Unlike the linear-phase case, polynomial phases generally destroy the
product structure of the kernel, giving rise to genuinely new
phenomena.  In the two-parameter setting, Garaev \cite{G} proved that, for the monomial phase
$\xi t_1t_2$, the corresponding multiplier diverges for a suitable choice of $\xi$. In particular,
$$
\sup_{\substack{
		\vec N\in\mathbb N^2,,
		\xi\in\mathbb R}}
\left|
\sum_{t\in R(\vec N)\cap\mathbb Z^2}
\frac{e^{2\pi i\xi t_1t_2}}{t_1t_2}
\right|
=\infty.
$$
This implies that (\ref{1tt}) fails if $\Lambda =\{(1,1)\}$.  Thus, unlike the  one parameter case,
oscillation alone is insufficient to guarantee uniform boundedness. Hence, the central question of this paper is to determine those finite sets $\Lambda\subset\mathbb Z_+^k$
for which there exists a constant $C_\Lambda>0$ such that (\ref{1tt}) holds. Those $\Lambda$ are determined by the configuration of monomials appearing in
$\sum_{\mathfrak m\in\Lambda}\xi_{\mathfrak m}t^{\mathfrak m}$.
This configuration interacts with the product singular kernel in a subtle way
and determines whether the oscillation provides sufficient cancellation.

Since the corresponding multiple Hilbert transform is the continuous
analogue of the above discrete operator, we compare the discrete estimate \eqref{1tt} with the following continuous counterpart:
\begin{align}\label{1tth}
	\sup_{\substack{
			\vec N\in\mathbb N^k\\
			\xi=(\xi_{\mathfrak m})_{\mathfrak m\in\Lambda}\in\mathbb R^{|\Lambda|}
	}}
	\left|
	\int_{R(\vec N+\vec{\bf{1}}) }
	\frac{e^{2\pi i\sum_{\mathfrak m\in\Lambda}
			\xi_{\mathfrak m}t^{\mathfrak m}}}{t_1\cdots t_k}\,
	dt
	\right|
	\le C_\Lambda,
\end{align}where $\vec{\bf{1}}=(1,\cdots,1)\in \mathbb{N}^{k}.$
Although the estimates \eqref{1tt} and \eqref{1tth} are closely related,
somewhat surprisingly, the families of index sets $\Lambda$ for which these estimates hold
do not coincide.
Before stating our main theorem, we recall the condition on $\Lambda$
under which the continuous estimate \eqref{1tth} holds.  

\subsection{Evenness Conditions for the Continuous Hilbert Transform}
For $k=1$, Stein and Wainger established $L^p$ bounds under a finite-type
condition \cite{SW1}.
The study of multi-parameter singular Radon transforms associated with
kernels adapted to multi-parameter dilation structures was initiated by
Nagel and Wainger~\cite{NW}. Their work extended the classical
Calderón--Zygmund theory to kernels whose singularities may occur along
subvarieties, including the coordinate hyperplanes, and encompassed a rich
family of examples beyond the multiple Hilbert kernel.

Building on this framework, Ricci and Stein~\cite{RS1} developed a general
$L^p(\mathbb R^d)$ theory for multi-parameter singular Radon transforms
associated with general dilation structures. As a special case, their
theory yields the following characterization for multiple Hilbert
transforms, stated here in a form adapted to the present paper.
\begin{theorem}[Ricci and Stein \cite{RS1}]\label{cho4}
	Let  ${\bf e}_\nu = (0, \cdots, 0, 1, 0, \cdots, 0) \in \mathbb{Z}^k$ be the standard unit vector with $1$ in the $\nu^{\text{th}}$ coordinate.  Suppose  $\{{\bf e}_1,\cdots,{\bf e}_k\}\subset \Lambda$. \footnote{The condition
		$\{{\bf e}_1,\ldots,{\bf e}_k\}\subset\Lambda$
		corresponds to the graph case
		$
		(t,P(t))=(t_1,\ldots,t_k,P(t_1,\ldots,t_k)).
		$
		In this setting, the corresponding estimate holds uniformly with respect to the coefficients of $P$.}
	Then,
	\begin{align*} 
		& \sup_{\vec{N}\in \mathbb{N}^k \   \text{and}\  \xi\in\mathbb{R}^{|\Lambda|}  }\left|\int_{  R(\vec{N}+\vec{\bf{1}})}e^{2\pi i \sum_{\mathfrak{m}\in\Lambda  } \xi_{\mathfrak{m}}t^\mathfrak{m}} \frac{ dt_1\cdots dt_k}{t_1\cdots t_k}\right|<\infty\ \text{if and only if}\nonumber\\
		& \text{ every $\mathfrak{m} \in \Lambda$ has at most one odd component.  }
	\end{align*}
\end{theorem}

To formulate the condition on $\Lambda$ governing the continuous
estimate, we introduce the following notion of parity for subsets of
$\Lambda$.
\begin{definition}[Even or Odd Set]\label{14a}
	Let $\Omega\subset \mathbb{Z}^k$. 
	
	\begin{itemize}
		\item
		$\Omega$ is \emph{odd} if every component of
		$\sum_{\mathfrak m\in\Omega}\mathfrak m$
		is odd;
		
		\item
		$\Omega$ is \emph{even} if at least one component of
		$\sum_{\mathfrak m\in\Omega}\mathfrak m$
		is even.
	\end{itemize}
	
	For $\mathfrak m=(m_1,\dots,m_k)\in\mathbb Z^k$, define
	\[
	[\mathfrak m]_2=([m_1]_2,\dots,[m_k]_2),
	\]
	where $[a]_2$ denotes the remainder of $a\in\mathbb Z$ modulo $2$.
	Then $\Omega$ is odd if and only if
	\[
	\Big[\sum_{\mathfrak m\in\Omega}\mathfrak m\Big]_2=(1,\dots,1).
	\]
\end{definition}

\begin{exam}\label{eexx1}
	Let
	\[
	\Lambda=\{(1,1,0),(0,1,1),(0,1,0)\}.
	\]
	Every singleton in $\Lambda$ has at least one even component $0$.
	Moreover, the sum of any two vectors in $\Lambda$ has at least one even component equal to $2$.
	However,
	\[
	(1,1,0)+(0,1,1)+(0,1,0)=(1,3,1).
	\]
	Thus the only odd subset of $\Lambda$ is $\Lambda$ itself.
\end{exam}

Let $\Lambda\subset\mathbb Z_+^k$ be a finite set. For
$\vec N\in\mathbb N^k$ and
$\xi\in\mathbb R^{|\Lambda|}$, define
\[
\mathcal H_{\vec N}^{\Lambda}(\xi)
:=
\int_{R(\vec N+\vec{\bf{1}})}
e^{2\pi i\sum_{\mathfrak m\in\Lambda}
	\xi_{\mathfrak m}t^{\mathfrak m}}
\frac{dt_1\cdots dt_k}{t_1\cdots t_k}.
\]
We extend Theorem~\ref{cho4} by removing the graph
assumption
$
\{{\bf e}_1,\ldots,{\bf e}_k\}\subset\Lambda.
$

\begin{theorem}\label{thji}
	Let $\Lambda\subset \mathbb{Z}_+^k$ be a finite set.
	
	\medskip
	\noindent
	{\rm (i)} Suppose that $\Lambda$ contains no odd subset. Then
	\[
	\sup_{\vec{N}\in \mathbb{N}^{k},\ \xi\in\mathbb{R}^{|\Lambda|}}
	|\mathcal{H}^{\Lambda}_{\vec{N}}(\xi)| =0.
	\]
	
	\medskip
	\noindent
	{\rm (ii)} Suppose that $\Lambda$ contains an odd subset. Then  a constant $C_{\Lambda}>0 $ exists such that
	\begin{align*}
		\sup_{\vec{N}\in \mathbb{N}^k,\ \xi\in\mathbb{R}^{|\Lambda|}}
		\left| \mathcal{H}^{\Lambda}_{\vec{N}}(\xi)\right|
		\le C_{\Lambda}
	\end{align*}
	if and only if 
	\begin{align}\label{v16}
		\ \text{every odd subset $\Omega\subset \Lambda$ satisfies}\ 
		\mathrm{rank}(\Omega)=k .
	\end{align}
	Here $\operatorname{rank}(\Omega)=\dim(\operatorname{span}(\Omega))$.
	Combining {\rm(i)} and {\rm(ii)},  we  conclude that 
	$\mathcal H^\Lambda_{\vec N}(\xi)$
	is uniformly bounded  if and only if either
	\begin{align*}
		&\text{$\Lambda$ contains no odd subset},
		\ \text{or every odd subset of }\Lambda\text{ has full rank }k.
	\end{align*}

\end{theorem}

Although Theorem \ref{thji} follows from Main Theorem 3.2 of \cite{K}, we prove its sufficient part in Section \ref{sec2}, since
several estimates used in the proof will also be needed later in
establishing the main theorem.

\subsection{Statement of Main Theorems}

We now state our main  theorems. Let   $\Lambda\subset \mathbb{Z}_+^k$ be a finite set. For  any $\vec{N}\in \mathbb{N}^k$ and any $\xi\in \mathbb{R}^{|\Lambda|}$, we define
\begin{align}\label{hm11}
	H^{\Lambda}_{\vec{N}}(\xi):= \sum_{t\in  R(\vec{N})\cap \mathbb{Z}^k}\frac{  e^{2\pi i \sum_{\mathfrak{m}\in\Lambda} \xi_{\mathfrak{m}}t^\mathfrak{m}} }{t_1\cdots t_k}.
\end{align}
\begin{main}\label{mt1}
	Let $\Lambda\subset \mathbb{Z}_+^k$ be a finite set. Then the following dichotomy holds.
	
	\medskip
	\noindent
	{\rm (i)} Suppose that $\Lambda$ contains no odd subset. Then
	\begin{align}
		\sup_{\vec{N}\in \mathbb{N}^{k},\ \xi\in\mathbb{R}^{|\Lambda|}}
		|H^{\Lambda}_{\vec{N}}(\xi)|=0.
		\label{10099}
	\end{align}
	
	\medskip
	\noindent
	{\rm (ii)} Suppose that $\Lambda$ contains an odd subset. Then  a constant $C_{\Lambda}>0$ exists such that
	\begin{align*}
		\sup_{\vec{N}\in \mathbb{N}^{k},\ \xi\in\mathbb{R}^{|\Lambda|}}
		\left| H^{\Lambda}_{\vec{N}}(\xi)\right|
		\le C_{\Lambda}
	\end{align*}
	if and only if
	\begin{align}
		\text{every }\mathfrak{m}=(m_1,\dots,m_k)\in\Lambda
		\text{ has at most one odd component.}
		\label{1009}
	\end{align}
	Combining (i) and (ii),   we conclude that
	$H^{\Lambda}_{\vec N}(\xi)$
	is uniformly bounded if and only if either
	\begin{align*}
		&\text{$\Lambda$ contains no odd subset,}\\
		&\text{or every odd subset of }\Lambda\ \text{is}\ 
		\{{\bf e}_1,\dots,{\bf e}_k\}\ \text{or}\ \{{\bf e}_1,\dots,{\bf e}_k, {\bf 0}\}
		\pmod 2.
	\end{align*}
\end{main}

\begin{proof}[Proof of (\ref{10099})]
	
	Suppose  that $\Lambda$ contains no odd subset. Then by definition, every subset $\Omega\subset \Lambda$ is even, and  there exists at least one $\nu\in \{1,\cdots,k\}$ such that the $\nu^{th}$ component of $  \sum_{\mathfrak{m}\in \Omega} \mathfrak{m}  $ is an even number, which implies that the function
	\begin{align}\label{337}
		t_{\nu}\rightarrow\frac{ \prod_{\mathfrak{m}\in \Lambda\setminus\Omega} \cos(2\pi \xi_{\mathfrak m} t^{\mathfrak{m}}) \prod_{\mathfrak{m}\in \Omega}  i\sin(2\pi \xi_{\mathfrak m} t^{\mathfrak{m}})   }{t_1\cdots t_k}
	\end{align}
	is  odd in variable $t_\nu$.
	Therefore, one can obtain that
	\begin{align*}
		H^{\Lambda}_{\vec{N}}(\xi)
		&=\sum_{(t_1,\cdots,t_k)\in R(\vec{N})\cap   \mathbb{Z}^k} \frac{\prod_{\mathfrak{m}\in \Lambda} \left(\cos (2\pi   \xi_{\mathfrak m}t^{\mathfrak{m}})+i\sin (2\pi \xi_{\mathfrak m}t^{\mathfrak{m}} ) \right) }{t_1\cdots t_k}=0,
	\end{align*}
	since each term in the product expansion is of the form \eqref{337} and is odd in $t_\nu$ for some $\nu$, hence canceling out when summed over a symmetric range.
\end{proof}

To exclude  the vanishing cases in Main Theorem \ref{mt1} and Theorem \ref{thji},  assume that
\begin{align}\label{111}
	\Lambda \text{ contains an odd subset}
\end{align}
and
compare the two conditions (\ref{v16}) and (\ref{1009}) for  odd   $\Omega\subset \Lambda\subset \mathbb{Z}^k$:
\begin{itemize}
	\item[(C)]  Continuous case (\ref{v16}):   $\text{rank}(\Omega)=k$ 
	\item[(D)] Discrete case (\ref{1009}):    $\Omega \equiv
	\{{\bf e}_1,\dots,{\bf e}_k\}\ \text{or}\ \{{\bf e}_1,\dots,{\bf e}_k, {\bf 0}\}
	\pmod 2.$
\end{itemize}
\begin{remark}\label{rk423}
	Assume   (\ref{111}). Then   $  (C)\Leftarrow  (D) $ is obvious.
	If $k=2$, then $ (C)\Rightarrow (D)$  also holds,  since  $(odd,odd)\notin \Lambda$ implies that  every odd $\Omega$ in  (C) after  $(even,even)$ deleted,  must be $\{(odd,even),(even,odd)\}  \equiv
	\{{\bf e}_1, {\bf e}_2\}
	\pmod 2 $ satisfying (D).  
\end{remark}
Therefore, the uniform boundedness of $H^{\Lambda}_{\vec{N}}(\xi)$ implies that of $\mathcal{H}^{\Lambda}_{\vec{N}}(\xi)$. 
For $k=2$, the uniform boundedness of $\mathcal{H}^{\Lambda}_{\vec{N}}(\xi)$ and $H^{\Lambda}_{\vec{N}}(\xi)$ is equivalent. 
However, as the following example shows, this equivalence can fail
when $k\ge 3$.
\begin{exam}
	The equivalence between {\rm(C)} and {\rm(D)} fails in dimensions
	$k\ge3$. Indeed, when $k=3$, let 
	\[
	\Lambda=\{(1,1,0),\ (0,1,1),\ (0,1,0)\}.
	\]
	Then the only odd subset of $\Lambda$ is $\Lambda$ itself, as we see in Example \ref{eexx1}. Observe that $\operatorname{rank}(\Lambda)=3$, so condition~$(C)$
	holds. However,
	\[
	[\Lambda]_2
	\neq
	\{\mathbf e_1,\mathbf e_2,\mathbf e_3\}
	\quad\text{and}\quad
	[\Lambda]_2
	\neq
	\{\mathbf e_1,\mathbf e_2,\mathbf e_3,\mathbf 0\}.
	\]
	Thus, condition~$(D)$ fails.
	For this $\Lambda$, it follows from Theorem \ref{thji} together with  Main Theorem \ref{mt1} that  
	\[
	\sup_{\vec{N},\xi} \bigl| \mathcal{H}^{\Lambda}_{\vec{N}}(\xi) \bigr| \le C < \infty, 
	\qquad\text{whereas}\qquad  \sup_{\vec{N},\xi} \bigl| H^{\Lambda}_{\vec{N}}(\xi) \bigr|=\infty.
	\]
	Consequently,
	\[
	\sup_{\vec N,\xi}
	\left|
	H^\Lambda_{\vec N}(\xi)
	-
	\mathcal H^\Lambda_{\vec N}(\xi)
	\right|
	=\infty.
	\]
	This shows that the classical discrete-to-continuous comparison formula
	\[
	m_N^{\rm disc}(\xi)
	=
	m_N^{\rm cont}(\xi)+O(1)
	\]
can	fail in the genuinely multi-parameter polynomial setting.
\end{exam}

We next state an $\ell^p$-boundedness result for the discrete
multiple Hilbert transforms associated with the multipliers
$H^\Lambda_{\vec N}(\xi)$. Let
\[
c_{00}(\mathbb{Z}^d)
:=
\left\{
f:\mathbb{Z}^d\to\mathbb{C}
:
\operatorname{supp}(f)\ \text{is finite}
\right\},
\]
which  is dense in
$\ell^p(\mathbb{Z}^d)$ for every $1\leq p<\infty$.  For  
$f\in c_{00}(\mathbb{Z}^{|\Lambda|}),$ define
${\bf H}^{\Lambda}_{\vec N}f$  by
\begin{align}\label{18p}
	{\bf H}^{\Lambda}_{\vec N}f(x)
	:=
	\sum_{t\in R(\vec N)\cap\mathbb{Z}^k}
	\frac{
		f\left(
		(x_{\mathfrak m}-t^{\mathfrak m})_{\mathfrak m\in\Lambda}
		\right)
	}{
		t_1\cdots t_k
	}\ \text{where $x=(x_{\mathfrak m})_{\mathfrak m\in\Lambda}
		\in\mathbb{Z}^{|\Lambda|}$ }.
\end{align}
Its Fourier multiplier is $H_{\vec N}^{\Lambda}(\xi)$ defined in
\eqref{hm11}. 
When considering the limit of  \eqref{18p} in $\vec{N}$,
it suffices to assume that
\begin{align*}
	\text{for every $\nu\in[k]$, there exists $\mathfrak m\in\Lambda$
		such that $\mathfrak m\cdot{\bf e}_\nu\neq0$.}
\end{align*}
Indeed, if $\mathfrak m\cdot{\bf e}_\nu=0$ for every
$\mathfrak m\in\Lambda$, then the summand, apart from the factor
$1/t_\nu$, is independent of $t_\nu$, and the corresponding sum in \eqref{18p} 
vanishes by symmetry.

Under this assumption, the finite support of $f$ guarantees that the
pointwise limit of \eqref{18p} exists as
$\min\{N_1,\ldots,N_k\}\to\infty$.
 We therefore define the discrete
multiple Hilbert transform initially on
$c_{00}(\mathbb{Z}^{|\Lambda|})$ by
\begin{align}\label{s2}
	{\bf H}^{\Lambda}f(x)
	:=
	\lim_{\min\{N_1,\ldots,N_k\}\to\infty}
	{\bf H}^{\Lambda}_{(N_1,\ldots,N_k)}f(x),
	\qquad
	x\in\mathbb{Z}^{|\Lambda|}.
\end{align}
If, for every $p\in(1,\infty)$, one can prove that
\[
\|{\bf H}^{\Lambda}f\|_{\ell^p(\mathbb{Z}^{|\Lambda|})}
\leq
C_{\Lambda,p}
\|f\|_{\ell^p(\mathbb{Z}^{|\Lambda|})}
\]
for all $f\in c_{00}(\mathbb{Z}^{|\Lambda|})$, then, by density,
${\bf H}^{\Lambda}$ extends uniquely to a bounded operator on
$\ell^p(\mathbb{Z}^{|\Lambda|})$. The following main theorem concerns the $\ell^p$-boundedness of the discrete multiple Hilbert transform.

\begin{main}\label{mt2} 
Let $\Lambda \subset \mathbb{Z}_+^k$ be a finite set. Then, one of the following two cases occurs:
\begin{enumerate}
	\item[(i)] Suppose that every subset of $\Lambda$ is even. Then
	${\bf H}^{\Lambda} \equiv 0.$

	\item[(ii)] Suppose that $\Lambda$ contains an odd subset. 
	If every $\mathfrak{m} \in \Lambda$ has at most one odd component as in (\ref{1009}), then for any
	$p \in (1,\infty)$ there exists a constant $C_{\Lambda,p} > 0$ such that
	\[
	\| {\bf H}^{\Lambda} \|_{\ell^p(\mathbb{Z}^{|\Lambda|}) \to \ell^p(\mathbb{Z}^{|\Lambda|})}
	\le C_{\Lambda,p}.
	\]
\end{enumerate}
Here, the constant $C_{\Lambda,p}$ depends only on $\Lambda$ and $p$.
\end{main}

\subsection{More General Problem}
To extend our result in its natural context, we  consider a general polynomial mapping
\[
\vec P=(P_1,\ldots,P_d):\mathbb Z^k\to\mathbb Z^d,
\qquad
P_\nu\in\mathbb Z[t_1,\ldots,t_k].
\]
Given $\vec P$ and $\vec N\in\mathbb N^k$, define the truncated discrete
multiple Hilbert transform by
\begin{align}\label{mtp2}
	{\bf H}^{\vec P}_{\vec N}f(x)
	:=
	\sum_{t\in R(\vec N)\cap\mathbb Z^k}
	\frac{f\bigl(x-\vec P(t)\bigr)}{t_1\cdots t_k},
\end{align}for $x\in\mathbb Z^d$ and $f:\mathbb Z^d\to\mathbb C.$
Its Fourier multiplier is
\begin{align*}
	H^{\vec P}_{\vec N}(\xi)
	:=
	\sum_{t\in R(\vec N)\cap\mathbb Z^k}
	\frac{e^{2\pi i\xi\cdot\vec P(t)}}{t_1\cdots t_k}.
\end{align*}

A more general question is the following.

\medskip
\noindent
\begin{question}\label{q11}
	For an appropriate class of functions $f$, for instance rapidly decreasing
	functions, does the limit
	\[
	\lim_{\min\{N_1,\ldots,N_k\}\to\infty}
	{\bf H}^{\vec P}_{\vec N}f(x)
	\]
	exist for every $x\in\mathbb Z^d$? Denoting the resulting operator by
	${\bf H}^{\vec P}$, can one characterize the polynomial mappings $\vec P$
	for which
	\begin{align*}
		\|{\bf H}^{\vec P}f\|_{\ell^p(\mathbb Z^d)}
		\le
		C_{p,\vec P}\|f\|_{\ell^p(\mathbb Z^d)}
	\end{align*}
	holds for some $1<p<\infty$?
\end{question}

\noindent \textbf{Discrete Singular Radon Transforms.}
The case $k=1$ corresponds to discrete singular Radon transforms
along polynomial curves, and Question~1.1 is completely resolved
when $k=1$. 

 From this perspective, we briefly review the historical development of discrete singular Radon transforms.
Let $K$ be a Calder\'on--Zygmund kernel on $\mathbb{R}^k$, and let 
$\vec{P}=(P_1,\dots,P_d):\mathbb{Z}^k\to \mathbb{Z}^d$ be a polynomial mapping with 
$P_\nu \in \mathbb{Z}[t_1,\dots,t_k]$ for $\nu=1,\dots,d$. 
For $N\in\mathbb{N}$, define the truncated discrete singular Radon transform
\[
R_N f(x)
:=
\sum_{\substack{t\in \mathbb{Z}^k\\ 1\le |t|\le N}}
f\bigl(x-\vec P(t)\bigr)\,K(t),
\qquad x\in \mathbb{Z}^d,
\]
initially for finitely supported functions $f:\mathbb{Z}^d\to\mathbb{C}$. 
When the limit exists, we define $Rf(x):=\lim_{N\to\infty}R_N f(x)$.

In the case \(k=1\) with kernel \(K(t)=1/t\), the operator \(R_N\) coincides with \(H_{\vec{N}}^{\vec{P}}\) defined in \eqref{mtp2}.
 The uniform boundedness estimate \eqref{11}, proved by Arkhipov and
 Oskolkov~\cite{AO} in 1987, implies uniform boundedness of the
 associated Fourier multipliers. By Plancherel, this yields the $\ell^2(\mathbb{Z}^d)$ boundedness of the corresponding discrete singular Radon transform $R$.

In the 1990s, Stein and Wainger~\cite{SW2,SW3} established the $\ell^p(\mathbb{Z}^d)$-boundedness of the discrete singular Radon transforms $R$ for the range $3/2< p < 3$ in arbitrary dimensions $d\ge 1$.

In 2006, Ionescu and Wainger~\cite{IW} resolved this problem by developing a general framework for estimating the associated Fourier multipliers via exponential sums, thereby establishing the $\ell^p(\mathbb{Z}^d)$-boundedness of $R$ for the full range $1<p<\infty$.

The development of Ionescu and Wainger~\cite{IW}  was foreshadowed by the method of Magyar, Stein, and Wainger~\cite{MSW} for handling discrete maximal operators, including the discrete spherical maximal operator. For surveys, see~\cite{IMW}, and for results on jump inequalities for discrete Radon transforms, see~\cite{MSZ}.  \\

\noindent\textbf{Double Hilbert Transform.}
In contrast to the case $k=1$, the corresponding boundedness problem
for individual polynomial mappings remains open even in the
two-parameter graph case.  To resolve Question~1.1 in the case of two or more parameters, it appears necessary to understand the $L^p$-boundedness criteria for the corresponding continuous Hilbert transforms. Unlike in the one-parameter setting, however, these criteria are formulated in terms of geometric conditions associated with Newton polyhedra. We therefore begin by describing the relevant continuous Hilbert transforms in  detail and then introduce the definition of the associated Newton polyhedra.
 Let
\[
\Lambda=(\Lambda_1,\ldots,\Lambda_d),
\qquad
\Lambda_\nu\subset\mathbb Z_+^k,
\]  
where we assume that the sets $\Lambda_1,\ldots,\Lambda_d$ are finite and mutually disjoint in this paper. For
each $\nu\in[d]$, define
\[
\mathbb R_{\Lambda_\nu}[t]
:=
\left\{
\sum_{\mathfrak m\in\Lambda_\nu}
c_{\mathfrak m}t^{\mathfrak m}
:
c_{\mathfrak m}\in\mathbb R
\right\},
\]
and set
\[
\vec{\mathbb R}_{\Lambda}[t]
:=
\left\{
\vec P(t)=\bigl(P_1(t),\ldots,P_d(t)\bigr)
:
P_\nu\in\mathbb R_{\Lambda_\nu}[t]
\text{ for every }\nu\in[d]
\right\}.
\]

To treat local and global oscillatory integrals in a unified framework,
let
\[
\vec b=(b_1,\ldots,b_k)\in \{-1,0,1\}^k.
\]
For $\vec P\in\vec{\mathbb R}_{\Lambda}[t]$, define the multiple Hilbert
transform $ \bm{\mathcal{H}}^{\vec P}_{\vec b}$ by the Fourier multiplier
\begin{align}\label{os11}
	\mathcal H^{\vec P}_{\vec b}(\xi)
	:=p.v.
	\int_{\prod_{\nu=1}^k
		\{ |t_\nu|^{b_\nu}\le 1\}}
	e^{2\pi i\langle\xi,\vec P(t)\rangle}
	\frac{dt_1}{t_1}\cdots\frac{dt_k}{t_k}.
\end{align}
Here the case $b_\nu=0$ is understood as
\[
\{|t_\nu|^0\le1\}
=
\{\,t_\nu:0<|t_\nu|<\infty\,\},
\]
since $|t_\nu|^0$ is undefined at $t_\nu=0$.
Consequently,
\[
\{|t_\nu|^{b_\nu}\le1\}
=
\begin{cases}
	\{\,|t_\nu|\ge1\,\}, & b_\nu=-1,\\
	\{\,0<|t_\nu|<\infty\,\}, & b_\nu=0,\\
	\{\,|t_\nu|\le1\,\}, & b_\nu=1.
\end{cases}
\]
Moreover, the principal value is taken by letting
$\epsilon_\nu\rightarrow0$ independently for each
$\nu=1,\ldots,k$, where the integral is truncated to the region
\[
\epsilon_\nu<|t_\nu|<1/\epsilon_\nu,
\qquad
\nu=1,\ldots,k.
\]
In particular, when $k=2$,
$\bm{\mathcal{H}}^{\vec P}_{(1,1)}$
is a local singular integral operator, whereas
$\bm{\mathcal{H}}^{\vec P}_{(0,0)}$
is the corresponding global singular integral operator.
The geometry governing the operator depends on whether each variable is
integrated near the origin, near infinity, or over the whole line.
This motivates the following generalized Newton polyhedron.
For $\vec b\in\{-1,0,1\}^k$ and $\Lambda_{\nu}\subset \mathbb{Z}_+^k$, 
\[
{\bf N}_{\Lambda_\nu}(\vec b)
:=
\operatorname{conv}
\left(
\Lambda_\nu+
\operatorname{cone}
\left( b_1{\bf e}_1,
\ldots,
b_k{\bf e}_k 
\right)
\right),
\]
and  denote the collection of all faces of
${\bf N}_{\Lambda_\nu}(\vec b)$ by
\[
\mathcal F\bigl({\bf N}_{\Lambda_\nu}(\vec b)\bigr).
\]

\begin{theorem}[Double Hilbert transforms]
	\label{DCP}
	The following theorem is due to
	Carbery--Wainger--Wright~\cite{CWW} in the local case and
	Patel~\cite{P2} in the global case, reformulated using the framework
	introduced in~\cite{K}.
	Let
	\[
	\vec b=(1,1)
	\quad\text{or}\quad
	(0,0),
	\]
	and let
	\[
	\Lambda=(\{{\bf e}_1\},\{{\bf e}_2\},\Lambda_3),
	\qquad
	\vec P(t)\in\vec{\mathbb R}_{\Lambda}[t].
	\]
Then, for every $1<p<\infty$,
$\mathcal H_{\vec b}^{\vec P}$ is bounded on $L^p(\mathbb R^3)$
	if and only if
	$\mathbb F\cap\Lambda_3$
	contains no odd subset whenever
	$
	\mathbb F\in
	\mathcal F({\bf N}_{\Lambda_3}(\vec b))
	$
	has     $\text{rank}(\mathbb{F})\le 1$.
\end{theorem}

Recently, the authors \cite{KS}   studied the  discrete
analogue of Theorem~\ref{DCP} in the graph case for $k=2$.
Let
$\Lambda=(\{{\bf e}_1\},\{{\bf e}_2\},\Lambda_3).$
The authors proved that if every vertex of
${\bf N}_{\Lambda_3}(-1,-1)$
is not $(\mathrm{odd},\mathrm{odd})$, 
\begin{align}\label{yt1}
\|\textbf{H}^{\vec P}\|_{\ell^p(\mathbb Z^3)\rightarrow\ell^p(\mathbb Z^3)}
\le C_{p,\vec P}.	
\end{align}

However, the exact necessity condition of \eqref{yt1} remains unknown. Moreover, for general two-parameter polynomial mappings $\vec{P}=\vec{P}(t_1,t_2)=(P_1(t_1,t_2),\cdots,P_{d}(t_1,t_2))$, the problem of determining the necessary and sufficient conditions for   the associated operator to be bounded on $\ell^p$, corresponding to the case $k=2$ of Question~\ref{q11}, remains open.
\\

\noindent\textbf{Multiple Hilbert Transform.}
To consider the three-parameter case, let
\[
\Lambda=(\{{\bf e}_1\},\{{\bf e}_2\},\{{\bf e}_3\},\Lambda_4),
\qquad
\vec P\in\vec{\mathbb R}_{\Lambda}[t].
\]
Carbery, Wainger, and Wright~\cite{CWW2}  observed that the
$L^p$ boundedness of the associated triple Hilbert transform
$\bm{\mathcal{H}}^{\vec P}_{(1,1,1)}$
for an individual polynomial mapping $\vec P$ is not determined solely
by the associated Newton polyhedron, but also depends on the
coefficients of the polynomials. This naturally led to the problem of
finding conditions on $\Lambda$ under which
$\bm{\mathcal{H}}^{\vec P}_{(1,1,1)}$
is bounded for every
$\vec P\in\vec{\mathbb R}_{\Lambda}[t]$.
In~\cite{CWW2}, the authors established both universal and individual
$L^2$ boundedness results under additional assumptions on $\Lambda_4$.
See also~\cite{CHKY} for  the   universal $L^p$ boundedness  for general $\Lambda_4$.

However, these universal boundedness results do not resolve the
corresponding problem for an individual polynomial mapping.
For $k\ge3$, a complete characterization of the $L^p$ boundedness of
$\bm{\mathcal{H}}^{\vec P}_{(1,\ldots,1)}$
for an individual polynomial mapping
$\vec P\in\vec{\mathbb R}_{\Lambda}[t]$
remains unknown even in the graph case.

Motivated by the universal boundedness problem, one may ask for a
satisfactory criterion in arbitrary dimensions.
In~\cite{K}, the first author proved the following theorem.

\begin{theorem}[Multiple Hilbert transforms \cite{K}]\label{MCP}
	Fix $\vec b=(b_\nu)_{\nu=1}^d\in\{0,1\}^d$, and let 
	$\Lambda=(\{{\bf e}_1\},\cdots,\{{\bf e}_d\},\Lambda_{d+1})\footnote{In \cite{K}, the characterization of the non-graph case
$\Lambda=(\Lambda_1,\ldots,\Lambda_d)$ for \eqref{4232n}
necessarily involves dual faces; see \cite{F} for the relevant
polyhedral background.}$.
	Then  
	\begin{align}\label{4232n}
		\|\bm{\mathcal{H}}^{\vec P}_{\vec b}\|_{L^p(\mathbb R^{d+1})\rightarrow L^p(\mathbb R^{d+1})}
		\le C_{p,\vec P}\  \ (1<p<\infty) \ \text{for every
			$\vec P\in\vec{\mathbb R}_{\Lambda}[t]$ }
	\end{align}
	if and only if
	\[
	\mathbb F\cap\Lambda_{d+1}
	\text{ contains no odd set whenever}\ 
	\mathbb F\in\mathcal F({\bf N}_{\Lambda_{d+1}}(\vec b))\
	\text{has rank at most $d-1$}.
	\]
\end{theorem}

	The Newton polyhedron
	${\bf N}_{\Lambda_\nu}(\vec b)$
	is introduced to identify the dominant monomials of the coordinate
	polynomial
	$P_\nu\in\mathbb R_{\Lambda_\nu}[t]$
	appearing in the phase of the oscillatory integral~\eqref{os11}.
	In contrast, Main Theorems~1 and~2, as well as
	Theorems~\ref{cho4} and~\ref{thji}, concern only monomial mappings, namely
	\[
	\Lambda_\nu=\{\mathfrak m_\nu\},
	\qquad
	1\le\nu\le d,
	\]
	so that
	\[
	P_\nu(t)=t^{\mathfrak m_\nu}.
	\]
	Since each coordinate polynomial consists of a single monomial, there are
	no competing monomials, and hence the Newton polyhedron
	${\bf N}_{\Lambda_\nu}(\vec b)$
	plays no role in these results. 
	
However, we expect that, for the general polynomial mappings
appearing in Question~1.1, the Newton polyhedron will play a crucial
role, in contrast to the monomial setting considered in our main
theorems. Even the continuous analogue of Question~\ref{q11} remains open because of the difficulty arising from the dependence on the coefficients of the coordinate polynomials. These considerations indicate that Question~1.1 lies beyond the primary scope of the present paper, which is devoted mainly to the development of the multi-parameter circle method. We therefore leave its investigation for future work.

\subsection{A Related Problem: Multi-Parameter Pointwise Ergodic Theorems}We now  turn from discrete Hilbert transforms to pointwise ergodic theorems, where the circle method has played an important role.
We shall provide an overview of several subsequent developments in this area, with particular emphasis on multi-parameter pointwise ergodic theorems.\\\\
\noindent \textbf{Historical Developments in Pointwise Ergodic Theory.}  
Following foundational developments in ergodic theory, including Birkhoff's pointwise ergodic theorem \cite{Bi}, Furstenberg's ergodic-theoretic proof of Szemerédi's theorem \cite{F0,Sz} marked the beginning of modern ergodic Ramsey theory.
 In the early 1980s, Bellow \cite{Bel} and Furstenberg \cite{F1} posed the question of whether, for a polynomial $P \in \mathbb{Z}[t]$ satisfying $P(0)=0$ and an invertible measure-preserving transformation $T$ on a probability space $(X,\mathcal{B}(X),\mu)$, the limit \begin{align}\label{2}
	\lim_{N \to \infty}\frac{1}{N}\sum_{t=1}^{N} f(T^{P(t)}x)
\end{align}
exists for almost every $x \in X$ and every $f \in L^\infty(X)$.   
Bourgain resolved this problem affirmatively in a series of groundbreaking works \cite{B1,B2,B3}, establishing the almost everywhere convergence of polynomial ergodic averages for every $f\in L^p(X)$, $1<p<\infty$, on probability measure-preserving systems. A central ingredient in his approach was the establishment of the $\ell^p(\mathbb{Z})$-boundedness of the discrete polynomial maximal operator
$$x\in\mathbb{Z}\rightarrow\sup_{N\in \mathbb{N}}\left|\frac{1}{N}\sum_{t=1}^{N} f(x-P(t))\right|.$$

 A far-reaching generalization of the
pointwise convergence problem in \eqref{2}, arising from this circle of
ideas, is commonly known as the Furstenberg--Bergelson--Leibman conjecture.

\medskip
\noindent
\textbf{Conjecture 1.1} (Furstenberg--Bergelson--Leibman).  
Let $d,k,n \in \mathbb{N}$, and let $T_1,\dots,T_d : X \to X$ be invertible measure-preserving transformations on a probability space $(X,\mathcal{B}(X),\mu)$ generating a nilpotent group of step $\ell \in \mathbb{N}$. Suppose that
$P_{i,j} \in \mathbb{Z}[t_1,\dots,t_k]$ for $ 1 \le i \le d$ and  $1 \le j \le n.$
Then for any $f_1,\dots,f_n \in L^\infty(X)$, the nonconventional polynomial averages
\begin{align}\label{o1}
	\frac{1}{N_1\cdots N_k}
	\sum_{t_1=1}^{N_1}\cdots\sum_{t_k=1}^{N_k}
	\prod_{j=1}^n
	f_j\left(
	T_1^{P_{1,j}(t)} \cdots T_d^{P_{d,j}(t)} x
	\right)
\end{align}
converge for $\mu$-almost every $x \in X$ as $\min\{N_1,\dots,N_k\}\to\infty$.

\medskip

At the level of pointwise convergence, Birkhoff’s theorem covers the linear, single-parameter case $d=k=n=1$ with $P(t)=t$.  In the commuting case ($\ell=1$), Mirek and Trojan~\cite{MT} verified Conjecture~1.1 for $n=1$ under the diagonal condition. More precisely, in proving both pointwise and $L^{p}(X)$ convergence for all $1<p<\infty$, they showed that the averages
\begin{align}\label{80}
	\frac{1}{N^{k}}
	\sum_{t_1=1}^{N}\cdots\sum_{t_k=1}^{N}
	f\left(
	T_1^{P_1(t)}\cdots T_d^{P_d(t)}x
	\right)
\end{align}
converge for $\mu$-almost every $x \in X$ as $N \to \infty$.

Magyar, Stein, and Wainger \cite{MSW2}, along with Ionescu, Magyar, Stein, and Wainger \cite{IMSW}, considered specific instances of (\ref{o1}) involving polynomial mappings of degree at most two in step-two nilpotent group settings. More recently, Ionescu, Magyar,  Mirek and Szarek \cite{IMMS1} confirmed Conjecture 1.1 in cases where $d \in \mathbb{Z}_+$, $n= k = 1$, and $P_{1,1}, \ldots, P_{d,1} \in \mathbb{Z}[m]$, with $T_1, \ldots, T_d: X \to X$ being invertible, measure-preserving transformations generating a step-two nilpotent group on a $\sigma$-finite measure space.

For the multilinear scenario of Conjecture 1.1, Bourgain \cite{B44} demonstrated pointwise convergence when $P_{1,1}(t) = at$ and $P_{1,2}(t) = bt$ for integers $a, b \in \mathbb{Z}$. Subsequently, Krause, Mirek, and Tao \cite{KMT} extended these results by proving pointwise convergence  for $P_{1,1}(t) = t$ and $P_{1,2}(t) = P(t)$ with $P \in \mathbb{Z}[t]$ of degree at least two.
In a further advance, Kosz, Mirek, Peluse, Wan and Wright \cite{KMPW} established pointwise almost-everywhere convergence for the  multilinear averages
\[
\frac{1}{N}\sum_{t=1}^{N}
\prod_{j=1}^{n} f_j\bigl(T_j^{P_j(t)}x\bigr),
\]
where \(T_1,\dots,T_n\) are commuting invertible measure-preserving transformations and the polynomials \(P_1,\dots,P_n\in\mathbb Z[t]\) have pairwise distinct degrees. Their proof is based on the multilinear circle method. In the notation of \eqref{o1}, this corresponds to the case \(k=1\), \(d=n\),
$
P_{i,j}(t)=0 \ (i\neq j)$  and $ 
P_{j,j}(t)=P_j(t).
$\\\\
\noindent \textbf{Multi-Parameter Pointwise Ergodic Theorem.}
In connection with further generalizations and extensions of the
preceding ergodic theorems, it is natural to ask whether a
multi-parameter extension of \eqref{80} can be obtained without
imposing the diagonal condition; namely, whether the averages
\begin{align*}
	\frac{1}{N_1\cdots N_k}
\sum_{t_1=1}^{N_1}\cdots\sum_{t_k=1}^{N_k}
	f\bigl(T_1^{P_1(t_1, \cdots, t_k)} \cdots T_d^{P_d(t_1, \cdots, t_k)} x\bigr)
\end{align*}
converge for $\mu$-almost every $x \in X$ as $\min\{N_1,\dots,N_k\}\to\infty$.

Bourgain, Mirek, Stein, and Wright \cite{B4} and Mirek, Szarek, and Wright \cite{MSW1}, posed the following problem concerning general polynomial mappings. See Conjecture~1.22 in \cite{B4} and Conjecture~1.29 in \cite{MSW1}.

\medskip
\noindent
\textbf{Conjecture 1.2.}
Let $d,k\in\mathbb{N}$ be given, and let $(X,\mathcal{B}(X),\mu)$ be a probability measure space endowed with a family of invertible commuting measure-preserving transformations $T_1,\dots,T_d:X\to X$. Assume that $P_1,\dots,P_d\in\mathbb{Z}[t_1,\dots,t_k]$. Then for any $f\in L^\infty(X)$, the linear multi-parameter polynomial ergodic averages
\begin{align*}
	\frac{1}{N_1\cdots N_k}
\sum_{t_1=1}^{N_1}\cdots\sum_{t_k=1}^{N_k}
	f\bigl(T_1^{P_1(t_1,\dots,t_k)}\cdots T_d^{P_d(t_1,\dots,t_k)}x\bigr)
\end{align*}
converge for $\mu$-almost every $x\in X$ as $\min\{N_1,\dots,N_k\}\to\infty$.

\medskip

In 1951, Dunford  \cite{Dun} and Zygmund \cite{Zyg}  initiated the  multi-parameter problem by proving that the following limit
$$
\lim_{\min\{N_1,\dots,N_k\}\to\infty}	\frac{1}{N_1\cdots N_k}
\sum_{t_1=1}^{N_1}\cdots\sum_{t_k=1}^{N_k} f(T_1^{t_1} \cdots T_k^{t_k} x)
$$
exists for almost every $x \in X$ and for any $f \in L^p(X)$ with $p \in(1, \infty)$, where $T_i:X\to X$, $i=1,\ldots,k$, are invertible measure-preserving transformations  on a probability space $(X, \mathcal{B}(X), \mu)$.  In 1979, Arkhipov, Chubarikov and Karatsuba \cite{AC}  showed that for any nonconstant polynomial $P \in \mathbb{Z}[t_1, \cdots, t_{k}]$, any irrational $\theta \in \mathbb{R}$, and any interval $[a, b) \subseteq[0,1)$,
$$\lim _{\min \{N_1, \cdots, N_k\} \rightarrow \infty} \frac{\#\{(t_1, \cdots, t_k) \in \prod_{s=1}^{k}[N_s]:\{\theta P(t_1, \cdots, t_k)\} \in[a, b)\}}{N_1 \cdots N_k}=b-a.$$ 

More recently, in 2023, Bourgain, Mirek, Stein and Wright \cite{B4}  proved that for every
$f\in L^p(X)$ with $p\in(1,\infty)$, the limit
\begin{align*}
\lim_{\min\{N_1,\dots,N_k\}\to\infty}	\frac{1}{N_1\cdots N_k}
\sum_{t_1=1}^{N_1}\cdots\sum_{t_k=1}^{N_k}
	f\bigl(T_1^{P_1(t_1,\dots,t_k)}x\bigr)
\end{align*}
exists for $\mu$-almost every $x\in X$.
They also established $L^p(X)$ convergence for all $1<p<\infty$ in this case, together with
the corresponding oscillation seminorm estimates. In \cite{B4}, they established the $\ell^p$-boundedness of the following discrete maximal operator in the one-dimensional case $d=1$, defined by
\begin{align*}
	x_1\in\mathbb{Z}\rightarrow \sup_{N_1,\dots,N_k\in\mathbb{N}}
	\frac{1}{N_1\cdots N_k}\left|
\sum_{t_1=1}^{N_1}\cdots\sum_{t_k=1}^{N_k}
	f\bigl(x_1 - P_1(t_1,\dots,t_k)\bigr)\right|.
\end{align*}
However, when $d>1$, it remained an open problem until recently whether the discrete maximal operator
\begin{align*}
	&(x_1,\dots,x_d)\in\mathbb{Z}^{d}\longmapsto\\
	&\sup_{N_1,\dots,N_k\in\mathbb{N}}
	\frac{1}{N_1\cdots N_k}
	\left|\sum_{t_1=1}^{N_1}\cdots\sum_{t_k=1}^{N_k}
	f\bigl(x_1-P_1(t_1,\dots,t_k),\dots,
	x_d-P_d(t_1,\dots,t_k)\bigr)\right|
\end{align*}
is bounded on $\ell^p(\mathbb{Z}^d)$ and whether the corresponding multi-parameter ergodic theorems hold.

A key ingredient in the proof for $d=1$ is a multi-parameter circle method developed to analyze exponential sums of the form
\[
	\frac{1}{N_1\cdots N_k}
\sum_{t_1=1}^{N_1}\cdots\sum_{t_k=1}^{N_k}
e^{2\pi i \xi_1 P_1(t_1,\dots,t_k)},
\]
which involve a single frequency $\xi_1$ and a single polynomial $P_1$.
In contrast, for $d\ge 2$, one must deal with multiple frequencies. For instance, when $d=2$, the relevant multiplier takes the form
\begin{align}\label{ddff1}
	\frac{1}{N_1 N_2}
	\sum_{\substack{|t_1|\sim N_1\\ |t_2|\sim N_2}}
	e^{2\pi i\bigl(\xi_1 P_1(t_1,t_2)+\xi_2 P_2(t_1,t_2)\bigr)}.
\end{align}
In \cite{B4},  they  illustrate the difficulty of (\ref{ddff1}) by  pointing out the example :
\begin{align}\label{tt}
P_1(t)=t_1,\qquad P_2(t)=t_1 t_2^2	
\end{align}
which share the common factor $t_1$.
For this case, the phase can be written as
\[
\xi_1 P_1(t_1,t_2)+\xi_2 P_2(t_1,t_2)
= t_1\bigl(\xi_1 + \xi_2 t_2^2\bigr).
\]
In the regime $N_1\gg N_2$, even when  both $\xi_1$ and $\xi_2$ are approximated rationals with large denominators, the coefficient $\xi_1+\xi_2 t_2^2$ may be well-approximated by rationals with small denominators, thereby obstructing uniform minor arc estimates. This structure of the phase prevents one from obtaining satisfactory Weyl-type decay (or any suitable variant) even when both $\xi_1$ and $\xi_2$ lie in the minor arcs.  This interaction between different scales and frequencies constitutes a fundamental barrier to extending the multi-parameter circle method of  \cite{B4} to the multi-frequency setting. A detailed explanation of the difficulties arising in this case is given in Subsection~\ref{sec340}.

Very recently, however, Hejna--Łyżwa, Langowski, Mirek, Wright, and the authors of the present paper resolved Conjecture~1.2 in \cite{HKLMSW} by overcoming the above difficulties through new ideas related to the multi-parameter circle method, at the heart of which lies the observation of a major-arc rigidity phenomenon. Their principal contributions include the introduction of the
major-arc rigidity principle and the development of a multi-parameter
Ionescu--Wainger multiplier theory, which may be viewed as a culmination of
the Ionescu--Wainger framework in the multi-parameter setting.

\subsection{Organization}
In Section \ref{sec34}, we examine the main difficulties in resolving the problem and provide an overview of the key ideas used to overcome them. In Section~\ref{sec2}, we introduce preliminary lemmas. In Section~\ref{sec3}, we handle the major-arc contribution and divide the
summation ranges into balanced and unbalanced cases by comparing
the sizes of the parameters. In Section~\ref{sec4}, we treat the two-parameter case under a simplifying assumption on $\Lambda$.
In Section \ref{sec5}, we extend the arguments from the two-parameter case to the general multi-parameter setting, thereby establishing the sufficiency part of Main Theorem \ref{mt1}.
In Section \ref{sec10}, we prove the necessity part of Main Theorem~\ref{mt1}.
In Section~\ref{sec6}, we establish Main Theorem~\ref{mt2} by adopting the Ionescu–Wainger $\ell^{p}$ theory.
In  Section \ref{PPL}, we provide proofs of preliminary lemmas stated in Section~\ref{sec2}.

\subsection{Thanks}
Hoyoung Song would like to thank Professor Mariusz Mirek for his careful discussions of the central ideas of this paper at a stage when they were still preliminary and had not yet been fully refined.
Joonil Kim and Hoyoung Song were supported by the National Research Foundation of Korea (NRF) under Grant No. RS-2026-25482410.

\section{Overview of Multi-Parameter Circle Method and Main Idea}\label{sec34}In this section, we introduce our key method.
We begin by briefly outlining the framework for proving \eqref{11} in the single-parameter setting, based on the one-parameter circle method and the associated exponential-sum estimates developed in \cite{AO,B1,B2,B3,HR,SW2,V}.
  Consider the following dyadic multiplier with $j\in\mathbb{Z}_{+}$
$$ H^{\Lambda}_j(\xi)=\sum_{|t|\sim 2^j} \frac{e^{2\pi i \sum_{m\in \Lambda}\xi_mt^m}}{t} \ \text{for $\xi=(\xi_m)_{m\in\Lambda}\in \mathbb{R}^{|\Lambda|}$}.$$  
There are two key estimates:
\begin{itemize}
	\item[(1)] \textbf{Major arc estimate.} Let $\xi$ lie  in the major arc region:
	\[
	I^{\mathrm{major}}_j
	:=
	\bigcup_{\substack{a/q : (a,q)=1 \\ 1\le q \le 2^{j/10}}}
	\left\{
	\xi \in \mathbb{T}^{|\Lambda|} :
	|\xi_{\mathfrak{m}} - a_{\mathfrak{m}}/q| \le \frac{2^{j/10}}{2^{mj}}
	\ \text{for all } \mathfrak{m} \in \Lambda
	\right\}.
	\]
	Each dyadic piece $H_j^{\Lambda}(\xi)$ is approximated by the product  $S^{\Lambda}(a/q)\mathcal{H}^{\Lambda}_j(\xi-a/q)$ 
	of  Gauss-sum  and  integral Hilbert transform. Hence, the following properties  $$S^{\Lambda}(a/q)=O(q^{-c})\ \text{and}\  \sum_j|\mathcal{H}^{\Lambda}_j(\xi-a/q) |\lesssim 1$$ imply $\sum_{j\in\mathbb{Z}_+} |H_j^{\Lambda}(\xi)\chi_{I_j^{\rm{major}}}(\xi) |\lesssim 1$.
	\item[(2)] \textbf{Minor arc estimate.}
	Let $\xi$ lie  in the complementary minor arc region:
	\[
	I^{\mathrm{minor}}_j := \mathbb{T}^{|\Lambda|} \setminus I^{\mathrm{major}}_j.
	\]
	By Dirichlet approximation, the frequency vector $\xi$ can be approximated by $a/q$ with a large $q\ge 2^{j/10}$. Hence, the Weyl sum estimate gives  $$H_j^{\Lambda}(\xi)=O(2^{-cj}),$$ which implies
	$\sum_{j} |H_j^{\Lambda}(\xi)\chi_{I_j^{\rm{minor}}}(\xi) |\lesssim 1$.
\end{itemize}
\bigskip
Here, $\chi_D$ denotes the characteristic function of a set $D$.
 \subsection{The Multi-Parameter Circle Method in One Dimension}In this subsection, we recall the one-dimensional version of the multi-parameter circle method developed in \cite{KS}. We illustrate the underlying philosophy of this method through the model case $P(t_1,t_2)=t_1t_2^{2}+t_1$. 
 
  Consider  the  dyadic multiplier \begin{align*} H^{P}_{j_1,j_2}(\xi) := \sum_{|t_2|\sim 2^{j_2}} \sum_{|t_1|\sim 2^{j_1}} \frac{e^{2\pi i \xi P(t_1,t_2)}}{t_1t_2}, \qquad \xi\in\mathbb{T}, \end{align*} where $(j_1,j_2)\in\mathbb{Z}_+^2$ satisfies $j_1\geq j_2$. For $\Lambda=\{(1,2),(1,0)\},$ the Newton polyhedron ${\bf N}_{\Lambda}(-1,-1)$ has the vertex $(1,2)$.\\
  
 \noindent\textbf{Major Arc.} Let $\xi$ lie  in the major arc region:
	\begin{align}\label{rt}
			I^{\mathrm{major}}_{j_1,j_2}
		:=
		\bigcup_{\substack{a/q : (a,q)=1 \\ 1\le q \le 2^{j_2/10}}}
		\left\{
		\xi \in \mathbb{T} :
		|\xi - a/q| \le \frac{2^{j_2/10}}{2^{(1,2)\cdot(j_1,j_2)}}
		\right\}.
	\end{align}
	In this arc, each dyadic piece $H_{j_1,j_2}^{P}(\xi)$ is approximated by the product  $$S(a/q)\mathcal{H}^{P}_{j_1,j_2}(\xi-a/q)$$ 
	of two-parameter Gauss-sum  and  integral Hilbert transform. Note that $S(a/q)=O(q^{-c}).$ Moreover, since the vertex set of ${\bf N}_{\Lambda}(-1,-1)$ consists solely of $(1,2)$, which has an even component, we have $\sum_{j_1, j_2} \left| \mathcal{H}^{P}_{j_1,j_2}(\xi-a/q) \right| \lesssim 1.$ So, we have $$ \sum_{j_1\ge j_2} |H^{P}_{j_1,j_2}(\xi)\chi_{I_{j_1,j_2}^{\rm{major}}}(\xi) |\lesssim 1.$$
		
\noindent\textbf{Other Arcs.}
	Let $\xi$ lie  in the complementary arc region:
	$
 \mathbb{T} \setminus I^{\mathrm{major}}_{j_1,j_2}.
	$ To handle this arc, it suffices to consider the following three arcs:
\begin{align}
	I^{\mathrm{major^{\sharp}}}_{j_1,j_2}
	:=
	\bigcup_{\substack{a/q : (a,q)=1 \\ 1\le q \le 2^{j_2/10}}}
	\left\{
	\xi \in \mathbb{T} : \frac{2^{j_2/10}}{2^{(1,2)\cdot(j_1,j_2)}}\le
	|\xi - a/q| \le \frac{2^{j_1/10}}{2^{(1,2)\cdot(j_1,j_2)}}
	\right\},\label{y}\\
		I^{\mathrm{major-minor}}_{j_1,j_2}
	:=
	\bigcup_{\substack{a/q : (a,q)=1 \\ 2^{j_2/10}\le q \le 2^{j_1/10}}}
	\left\{
	\xi \in \mathbb{T} :
	|\xi - a/q| \le \frac{2^{j_1/10}}{2^{(1,2)\cdot(j_1,j_2)}}
	\right\},\label{y1}\\
		I^{\mathrm{minor}}_{j_1,j_2}
	:=\mathbb{T}\setminus
	\bigcup_{\substack{a/q : (a,q)=1 \\ 1\le q \le 2^{j_1/10}}}
	\left\{
	\xi \in \mathbb{T} :
	|\xi - a/q| \le \frac{2^{j_1/10}}{2^{(1,2)\cdot(j_1,j_2)}}
	\right\}.\label{y2}
\end{align}
In  $major^{\sharp}$ and $major-minor$ regimes,   $H_{j_1,j_2}^{P}(\xi)$ is approximated by  $$\sum_{|t_2|\sim 2^{j_2}}\frac{1}{t_2}S^{t_2}(a/q)\mathcal{H}^{t_2}_{j_1}(\xi-a/q),$$ where \begin{align*}
	\mathcal{H}_{j_1}^{t_2}(\xi-a/q)&:=\int_{|x_1|\sim 2^{j_1}} e^{-2\pi i[(\xi-a/q)(t_2^{2}+1) x_1]}\frac{dx_1}{x_1},\\
	S^{t_2}\left(\frac{a}{q}\right)&:=\frac{1}{q}\sum_{\ell_1=1}^q e^{-2\pi i \left(\frac{a}{q}(t_2^{2}+1)\ell_1\right)}.
\end{align*} In the  $major^{\sharp}$ arc regime, we apply oscillatory integral estimates to the Hilbert transforms $\mathcal{H}_{j_1}^{t_2}(\xi-a/q)$ and obtain a decay of the form $2^{-cj_2}$ from the high-frequency condition   $ |(\xi-a/q) 2^{(1,2)\cdot(j_1,j_2)}| \ge 2^{ j_2/10}$.	
 In  the  $major-minor$ arc regime, we instead use the averaged Gauss sum estimate: when  $2^{j_2/10}\le q$, there exists $c>0$ such that \begin{align}\label{u1}
	\frac{1}{2^{j_2}}\sum_{|t_2|\sim 2^{j_2}}\left|\frac{1}{q}\sum_{\ell_1=1}^q e^{2\pi i\left(\frac{a}{q}(t_2^{2}+1)\ell_1\right)}\right|\lesssim q^{-c}2^{-cj_2}.\end{align} Finally, in the minor-arc regime, we invoke the multi-parameter Weyl-sum estimate \begin{align}\label{u2}
	\left|\sum_{|t_2|\sim 2^{j_2}}\sum_{|t_1|\sim 2^{j_1}}e^{2\pi i \xi( t_1t_2^{2}+t_1)}\right|\lesssim 2^{j_1+j_2}/2^{cj_1}.
	\end{align} 
Therefore, combining the preceding estimates with the one-parameter continuous integral estimates, we obtain
\[
\sum_{j_1,j_2}
\left|
H_{j_1,j_2}^{(1,2)}(\xi)
\chi_{\mathbb{T} \setminus I^{\mathrm{major}}_{j_1,j_2}}(\xi)
\right|
\lesssim
\sum_{j_2}2^{-cj_2}
\lesssim1.
\]
The scheme for  the above four arcs is summarized in the following
table.  \begin{table}[H]
	\centering
	\caption{Layout of Two-Parameter Circle Method.}
	\label{t2}
	\begin{tabular}{c|c|c|c}
		\noalign{\smallskip}\noalign{\smallskip}
		\hline\hline
		Variation& Range of $q$ & Frequency $\beta:=\xi-\frac{a}{q} $ &Outcome \\
		\hline
		$\text{Major}$  & $ q<2^{ j_2/10} $ & Low $|\beta|2^{j\cdot \mathfrak{m}}<2^{ j_2/10} $ & $\int \frac{dt_1}{t_1}\frac{dt_2}{t_2} $\\
		\hline
		$\text{Major}^{\sharp}$&$q<2^{ j_2/10} $  &High $| \beta|2^{j\cdot \mathfrak{m}}\ge 2^{ j_2/10} $ & $\sum_{t_2}\frac{1}{t_2} \int \frac{1}{t_1}dt_1$ \\
		\hline
		Major-Minor & $2^{ j_2/10}\le q\le 2^{j_1/10} $ & $|\beta|2^{j\cdot \mathfrak{m}}<2^{j_1/10}$ &$\sum_{t_2}\frac{1}{t_2}\int \frac{1}{t_1}dt_1 $ \\
		\hline
		Minor &  $2^{j_1/10}  <q$ & $|\beta|2^{j\cdot \mathfrak{m}}<2^{j_1/10}$ &$\sum_{t_1,t_2} \frac{1}{t_1t_2} $\\
		\hline
		\hline
	\end{tabular}
\end{table}

\subsection{Difficulties in the Unbalanced Case}\label{sec340}
We now examine the difficulties arising in the case of (\ref{tt}). The following explanation will make clear why the preliminary method introduced in the previous subsection does not extend directly to general polynomial mappings $(P_1(t_1,\cdots,t_k),\ldots,P_d(t_1,\cdots,t_k))$  with $d\ge 2$.   The main difficulties in handling the exponential sum in
\eqref{ddff1} arise in the highly unbalanced regime
$j_1\gg j_2$.  When we use the above preliminary multi-parameter circle method, the following two ranges present the main difficulties:
\begin{itemize}
	\item[(i)] $q \in [2^{j_1/10},\infty)$:  
	The Weyl sum need not exhibit decay $2^{-c j_2}$.
	
	\item[(ii)] $q \in [2^{j_2/10}, 2^{j_1/10}]$:  
	The averaged Gauss sum need not exhibit decay  $2^{-c j_2}$.
\end{itemize}

\noindent\textbf{Obstacle (i). The Weyl sum estimate fails.}
Consider the phase
$$\xi_{1} t_1+\xi_{2} t_1t_2^2=t_1(\xi_{1}+\xi_{2}t_2^2).$$
For $\xi_2$ in the minor arc  with  \begin{align}\label{w1}
	q \in [2^{j_1/10},\infty), 
\end{align}in view of \eqref{u2}, one might expect the analogous Weyl sum estimate:
\begin{align}\label{ww11}
	\left|\sum_{t_2\sim 2^{j_2}}\sum_{t_1\sim 2^{j_1}} e^{2\pi  i\left(\xi_{1}t_1+\xi_{2}t_1t_2^2\right)}\right|\lesssim 2^{j_1+j_2}/2^{c j_1}.
\end{align}
However, we shall see that (\ref{ww11}) fails for the following case of $\xi=(\xi_1,\xi_2)$:
\begin{align*} 
	\text{ $(\xi_{1},\xi_2)=(-\frac{2^{2j_2}}{2^{j_1/2}},\frac{1}{2^{j_1/2}}) $ where  }  j_1\gg j_2.
\end{align*}
 Then, one has
\begin{align*}
	\left|\sum_{t_2\sim 2^{j_2}}\sum_{t_1\sim 2^{j_1}} e^{2\pi  i\left(\xi_{1}t_1+\xi_{2}t_1t_2^2\right)}\right|
	&=\left|\sum_{t_2=2^{j_2-1}}^{2^{j_2}}\sum_{t_1=2^{j_1-1}}^{2^{j_1}} e^{2\pi  i\left(\frac{(t_2^2-2^{2j_2})t_1}{2^{j_1/2}}\right)}\right|\\
	&\ge \left|\sum_{t_2=2^{j_2}} \sum_{t_1=2^{j_1-1}}^{2^{j_1}} 1\right|-\sum_{2^{j_2-1}\le t_2\le 2^{j_2}-1}\left|\sum_{t_1=2^{j_1-1}}^{2^{j_1}}  
	e^{2\pi  i\left(\frac{(t_2^2-2^{2j_2})t_1}{2^{j_1/2}}\right)}\right|\\
	&= 2^{j_1-1}-O(2^{j_1/2}2^{j_2})\approx 2^{j_1}\gg 2^{j_1} 2^{j_2}/2^{c j_1},
\end{align*}
which means (\ref{ww11}) fails. This failure  occurs due to the scale imbalance   $j_1\gg j_2$.\\

\noindent\textbf{Obstacle (ii). Averaged Gauss sum estimate  fails.} For $\xi_2$ in the major-minor arc  
 with \begin{align}\label{w2}
 	2^{Cj_2}\le q\le 2^{j_1/10}  \ \text{where $C\gg 1$},
 \end{align}  in view of \eqref{u1}, one might expect the analogous averaged Gauss sum estimate 
 \begin{align}\label{4g4}
 \frac{1}{2^{j_2}} \sum_{|t_2|\sim 2^{j_2}}\left|\frac{1}{q}\sum_{t_1=1}^q e^{2\pi i  (\frac{a_1}{q}+\frac{a_2}{q}t_2^2) t_1}\right| & \lesssim  q^{-c}\ \text{for some $c>0$}.
 \end{align}
However, there is  no    such estimate for some large $q\ge 2^{Cj_2}$  with $C\gg 1$. Indeed, if there exists   $t_2\sim 2^{j_2}$ such that   $a_1+a_2t_2^2\equiv 0\ \text{mod} \ q$, then one has
\begin{align*}
	\frac{1}{2^{j_2}} \sum_{|t_2|\sim 2^{j_2}}\left|\frac{1}{q}\sum_{t_1=1}^q e^{2\pi i  (\frac{a_1}{q}+\frac{a_2}{q}t_2^2) t_1}\right| &\ge \frac{1}{2^{j_2}}\left|\{t_2:  (a_1+a_2t_2^2)\equiv 0\  \text{mod}\ q\}\right| \ge\frac{1}{2^{j_2}}.
\end{align*}
This shows that a uniform bound of the form \eqref{4g4} cannot hold in general once $q$ is much larger than $2^{Cj_2}$ with $C\gg 1$.

Prior to the present work, the difficulties described above had precluded the treatment of higher-dimensional, multi-parameter polynomial mappings.
 However, due to  the new idea described below, we overcome these obstacles and ultimately establish our main result.

\subsection{The Central Idea: Layered Arcs and Major-Arc Rigidity} 
In a strongly
disparate regime $j_1\gg j_2$, estimates may fail to be summable with
respect to the smaller scale parameter.     The main method is designed precisely to overcome these difficulties.   Set $\Lambda=\{(1,0),(1,2)\}.$ Consider the dyadic multiplier with $\xi=(\xi_1,\xi_2)\in\mathbb{T}^{2}$ and $j=(j_1,j_2)\in \mathbb{Z}_{+}^{2}$ satisfying $j_1\gg j_2$:  $$H^{\Lambda}_{j}(\xi)=\sum_{|t_2|\sim 2^{j_2}}\sum_{|t_1|\sim 2^{j_1}} \frac{e^{2\pi i(\xi_1t_1+ \xi_2 t_1t_2^{2})}}{t_1t_2}.$$  
If $\xi\in I_j^{\mathrm{major}}$ where
\begin{align*}
	I_j^{\mathrm{major}}
	:=
	\bigcup_{\substack{a/q:\,(a,q)=1\\1\leq q\leq 2^{j_2/10}}}
	\left\{
	\xi\in\mathbb{T}^2:
	\left|\xi_1-\frac{a_1}{q}\right|
	\leq
	\frac{2^{j_2/10}}{2^{(1,0)\cdot(j_1,j_2)}} \ \text{and} \ \left|\xi_2-\frac{a_2}{q}\right|
	\leq
	\frac{2^{j_2/10}}{2^{(1,2)\cdot(j_1,j_2)}}
	\right\},
\end{align*}
then, as in \eqref{rt}, the dyadic multiplier $H^{\Lambda}_{j}(\xi)$  can be approximated by the corresponding two-parameter continuous Hilbert transform. Consequently,  by the continuous theory, we have  $$ \sum_{j_1\ge j_2} |H^{\Lambda}_{j}(\xi)\chi_{I_{j}^{\rm{major}}}(\xi) |\lesssim 1.$$\\

\noindent\textbf{Layered Arcs.}
Fix $\xi\in
\mathbb{T}^2 \setminus I^{\mathrm{major}}_{j}.$ To handle this arc, rewrite the phase
$\xi_{1} t_1+\xi_{2} t_1t_2^2=\xi(t_2)t_1,$ where $\xi(t_2)=\xi_{1}+\xi_{2}t_2^2.$ Then, rewrite $$|H^{\Lambda}_{j}(\xi)|=\left|\sum_{|t_2|\sim 2^{j_2}}\sum_{|t_1|\sim 2^{j_1}}\frac{e^{2\pi i\xi(t_2)t_1}}{t_1t_2}\right|\lesssim \frac{1}{2^{j_2}}\sum_{t_2\sim 2^{j_2}}\left|\sum_{|t_1|\sim 2^{j_1}}\frac{e^{2\pi i\xi(t_2)t_1}}{t_1}\right|.$$ As in \eqref{y}, \eqref{y1}, and \eqref{y2}, consider the following three arcs \begin{align*}
I^{\mathrm{major^{\sharp}}}_{j}:= 	\bigcup_{\substack{b/p : (b,p)=1 \\ 1\le p \le 2^{j_2/10}}}
	\left\{
	\tilde{\xi} \in \mathbb{R} : \frac{2^{j_2/100}}{2^{j_1}}\le
	|\tilde{\xi} - b/p| \le \frac{2^{j_1/10}}{2^{j_1}}
	\right\},\\
I^{\mathrm{major-minor}}_{j}:=	\bigcup_{\substack{b/p : (b,p)=1 \\ 2^{j_2/10}\le p \le 2^{j_1/10}}}
	\left\{
	\tilde{\xi} \in \mathbb{R} :
	|\tilde{\xi} - b/p| \le \frac{2^{j_1/10}}{2^{j_1}}
	\right\},\\
 I^{\mathrm{minor}}_{j}:=	\mathbb{R}\setminus
	\bigcup_{\substack{b/p : (b,p)=1 \\ 1\le p \le 2^{j_1/10}}}
	\left\{
	\tilde{\xi} \in \mathbb{R} :
	|\tilde{\xi} - b/p| \le \frac{2^{j_1/10}}{2^{j_1}}
	\right\}.
\end{align*}  For a fixed  $t_2$, suppose that $\xi(t_2)\in I^{\mathrm{major^{\sharp}}}_{j}\cup I^{\mathrm{major-minor}}_{j}\cup   I^{\mathrm{minor}}_{j}.$ Then, viewing $\xi(t_2)$ as a frequency variable $\widetilde{\xi}\in N+\mathbb{T}$ for some $N\in\mathbb{N}$, 
we can obtain
\begin{align}\label{y77}
\left|
\sum_{|t_1|\sim 2^{j_1}}
\frac{e^{2\pi i \xi(t_2)t_1}}{t_1}
\right|
\lesssim 2^{-c j_2}.
\end{align}
Indeed, this follows by applying the same arguments used to treat the $\mathrm{major}^{\sharp}$, major--minor, and minor arc regions in \eqref{y}, \eqref{y1}, and \eqref{y2}, respectively.

Consequently, in order to estimate 
\begin{align*} \frac{1}{2^{j_2}} \sum_{t_2\sim 2^{j_2}} \left| \sum_{|t_1|\sim 2^{j_1}} \frac{e^{2\pi i \xi(t_2)t_1}}{t_1} \right|, 
\end{align*} 
it remains only to consider those $t_2$ for which 
\begin{align*} \xi(t_2) &\in \mathbb{R}\setminus \Bigl(  I^{\mathrm{major^{\sharp}}}_{j}\cup I^{\mathrm{major-minor}}_{j}\cup   I^{\mathrm{minor}}_{j} \Bigr) \\ &= \bigcup_{\substack{b/p:\,(b,p)=1\\ 1\le p\le 2^{j_2/10}}} \left\{ \tilde{\xi}\in\mathbb{R}: \left|\tilde{\xi}-\frac{b}{p}\right| \le \frac{2^{j_2/100}}{2^{j_1}} \right\}. 
\end{align*}

\noindent\textbf{Major-Arc Rigidity.}
Fix $\xi\in \mathbb{T}^2\setminus I^{\mathrm{major}}_{j}$.  In the disparity condition $j_1\gg j_2$, it remains to count those $t_2 \sim 2^{j_2}$ for which  	\begin{align}\label{rt3} 
	\xi(t_2) \in \bigcup_{\substack{b/p:\,(b,p)=1\\ 1\le p\le 2^{j_2/10}}} \left\{ \widetilde{\xi}\in\mathbb{R}: \left| \widetilde{\xi}-\frac{b}{p} \right| \le \frac{2^{j_2/100}}{2^{j_1}} \ll 2^{-Cj_2} \right\}.
\end{align}
Accordingly, we define \begin{align*} 
	A(\xi,j_2) := \left\{ t_2\sim 2^{j_2}: t_2 \text{ satisfies \eqref{rt3}} \right\}. 
	\end{align*} Fix $\xi\in\mathbb{T}^2\setminus I_j^{\mathrm{major}}$ and $j_2\in\mathbb{Z}_+$, and suppose that $|A(\xi,j_2)|\leq 2^{j_2/2}.$ Using this assumption and  applying  \eqref{y77} on the complementary region in \eqref{rt3}, we obtain $$\frac{1}{2^{j_2}}\sum_{t_2\sim 2^{j_2}}\left|\sum_{|t_1|\sim 2^{j_1}}\frac{e^{2\pi i\xi(t_2)t_1}}{t_1}\right|\lesssim 2^{-cj_2}+ \frac{1}{2^{j_2}}\sum_{t_2\in A(\xi,j_2)}\lesssim  2^{-cj_2}.$$

We now consider the remaining case $|A(\xi,j_2)|>2^{j_2/2}.$
For each fixed $t_2$ and $b/p$, the set of frequencies
$\xi=(\xi_1,\xi_2)\in\mathbb{T}^2$ satisfying
\begin{align*}
	\left|
	\xi(t_2)-\frac{b}{p}
	\right|
	=
	\left|
	\xi_1+\xi_2t_2^2-\frac{b}{p}
	\right|
	\le 2^{-Cj_2}
\end{align*}
is contained in a thin strip, which may be viewed, up to a linear
change of coordinates, as a rectangle of dimensions
$2^{-Cj_2}\times 1$.
The assumption $|A(\xi,j_2)|>2^{j_2/2}$ therefore yields a large
collection of such thin rectangles passing through the fixed point
$
\xi\in\mathbb{T}^2\setminus I_j^{\mathrm{major}}.$
Their directions are determined by the distinct values of $t_2$,
whereas their locations are determined by the corresponding rational
numbers $b/p$. The existence of such a large collection of intersecting
rectangles gives rise to a major-arc rigidity phenomenon: although
$\xi$ lies outside the original major-arc region
$I_j^{\mathrm{major}}$, it must in fact belong to a larger and coarser
major-arc region.

 More precisely, in the case where $|A(\xi,j_2)|>2^{j_2/2}$, we may choose two distinct elements  $s_1,s_2\in A(\xi,j_2)$. By the definition of $A(\xi,j_2)$, there exist reduced rational numbers $b_1/p_1$ and $b_2/p_2$ such that
\begin{align*}
	\left(\begin{matrix}1&s_1^{2}\\
		1&s_2^{2}
	\end{matrix}\right) \left(\begin{matrix}\xi_{1}\\
		\xi_{2}
	\end{matrix}\right)=\left(\begin{matrix}\frac{b_1}{p_1}\\
		\frac{b_2}{p_2}
	\end{matrix}\right)+O(2^{-Cj_2}).
\end{align*}
Then, we have \begin{align*}
	 \left(\begin{matrix}\xi_{1}\\
		\xi_{2}
	\end{matrix}\right)=\left(\begin{matrix}1&s_1^{2}\\
	1&s_2^{2}
	\end{matrix}\right)^{-1}\left(\begin{matrix}\frac{b_1}{p_1}\\
		\frac{b_2}{p_2}
	\end{matrix}\right)+O(2^{-Cj_2/2}).
\end{align*}
Surprisingly, this means that, although
$\xi \in \mathbb{T}^2 \setminus I_j^{\mathrm{major}}$
appears to lie in the  minor-arc region, it actually belongs to the following coarser major-arc region:
\begin{align*}
	I_{j_2}^{\mathrm{coarse\text{-}major}}
	:=
	\bigcup_{\substack{a/q:\,(a,q)=1\\1\leq q\leq 2^{3j_2}}}
	\left\{
	\xi\in\mathbb{T}^2:
	\left|\xi-\frac{a}{q}\right|
	\lesssim 2^{-Cj_2/2}
	\right\}.
\end{align*}Here, we write $a=(a_1,a_2)$.

As we saw in the discussion of obstacles~(i) and~(ii) in Subsection~\ref{sec340}, in the unbalanced case $j_1\gg j_2$, when $\xi$ admits a rational approximation $a/q$ whose denominator satisfies \begin{align}\label{uu}
	 q\in[2^{j_1/10},\infty) \qquad\text{or}\qquad 2^{Cj_2}\leq q\leq 2^{j_1/10}, 
\end{align} as in \eqref{w1} and \eqref{w2}, neither the Weyl sum  nor the averaged Gauss sum estimate yields any decay in $j_2$. However, the major-arc rigidity phenomenon described above rules out the possibility that $\xi$ admits such a rational approximation with a denominator satisfying \eqref{uu}.  Consequently, the principal difficulties associated with these two denominator ranges do not arise, and the analysis reduces to the coarser major-arc region $I_{j_2}^{\mathrm{coarse\text{-}major}}$, which is considerably more tractable than either the minor arc region or the major--minor arc region.

The key observation is that a heavy  concentration of the major arcs of the sliced coefficient $\xi(t_2)$ necessarily forces the corresponding original frequency $\xi$ to lie in a coarse major arc which is a natural pathway to an original major arc:
\[
\text{Sliced Major arcs of $\xi(t_2)$}
\Longrightarrow
\text{Coarse Major Arc of $\xi$.}
\]
This mechanism lies at the heart of our method and constitutes its principal novelty and we call this phenomenon  major arc rigidity.

For the multi-parameter case, we regard the coefficient vector
\[
\xi(t_{s+1},\ldots,t_k)
=
\left(
\sum_{m_{s+1}}
\xi_{m_1,\ldots,m_{s+1}}(t_{s+2},\ldots,t_k)
t_{s+1}^{m_{s+1}}
\right)_{(m_1,\ldots,m_s)}
\]
as a vector-valued polynomial in $t_{s+1}$.
Our method consists in applying this principle consecutively, one variable
at a time. Starting from the major-arc localization of
\(\xi(t_{\ell+1},\ldots,t_k)\), we successively pass
\[
\xi(t_{\ell+1},\ldots,t_k)
\longrightarrow
\xi(t_{\ell+2},\ldots,t_k)
\longrightarrow \cdots \longrightarrow
\xi(t_k)
\longrightarrow
\xi.
\]
At each stage, the remaining outer variables are frozen, and the
major-arc information for many values of the current variable is converted
into a coarse, or connecting, major-arc localization for the coefficient
vector at the next level. In this way, the variables
\(t_{\ell+1},\ldots,t_{k-1}\) are peeled off successively, while the
major-arc structure is propagated from one layer of the phase to the next.

  \section{Preliminaries}\label{sec2}
  \noindent
{\bf Notation.} 
We denote by $\mathbb{Z}_+$ the set of non-negative integers, namely
$\mathbb{Z}_+ = \mathbb{N} \cup \{0\}.$
For $\mathfrak{m}=(m_1,\ldots,m_k)\in\mathbb{Z}_+^k$, set $ |\mathfrak{m}|:=m_1+\cdots+m_k.$ Given $q \in \mathbb{Z}_+ \setminus \{0\}$ and
$a = (a_{\mathfrak m})_{\mathfrak m \in \Omega} \in \mathbb{Z}^{|\Omega|}$,
where $\Omega \subset \mathbb{Z}_+^{k}$, we define
\[
(q,a) := \gcd\bigl(q, a_{\mathfrak m} : \mathfrak m \in \Omega\bigr).
\]
For $n \in \mathbb{N}$, we denote $$[n] := \{1,\dots,n\}.$$
Unless otherwise stated, constants $C,c>0$ may vary from line to line.
For any finite subset $\Omega \subset \mathbb{Z}_+^{k}$,
we write
\begin{equation*}
\xi_{\Omega} := (\xi_{\mathfrak m})_{\mathfrak m \in \Omega}
\in \mathbb{R}^{|\Omega|}.
\end{equation*}
For example, if $\Omega = \{(1,2),(2,5),(6,1)\}$, then
\[
\xi_{\Omega}
= (\xi_{(1,2)}, \xi_{(2,5)}, \xi_{(6,1)}) \in \mathbb{R}^{3}.
\]
 Let $j = (j_1,\dots,j_k) \in \mathbb{Z}_+^{k}$.
We  define the following non-negative cutoff functions:
\begin{itemize}
\item[(1)] $\psi$ is a smooth function supported in the unit ball
$\{u \in \mathbb{R} : |u| \le 1\}$ and satisfies
$\psi(u) \equiv 1$ for $|u| < 1/2$.
\item[(2)] For a finite set $\Omega \subset \mathbb{Z}_+^{k}$, define
\[
\psi^{\Omega}_{j}(\xi_\Omega)
:= \prod_{\mathfrak m \in \Omega}
\psi\!\left( \frac{\xi_{\mathfrak m}}{2^{-j \cdot \mathfrak m}} \right),
\]
which is supported in the region
\[
\left\{
\xi_{\Omega} : |\xi_{\mathfrak m}| \le 2^{-j \cdot \mathfrak m}
\ \text{for all } \mathfrak m \in \Omega
\right\}.
\]
\item[(3)] $\chi_D$ denotes the characteristic function of a set
$D$.
\end{itemize}
We employ the following dyadic decomposition.
Let $\mathfrak{t}=(t_1,\ldots,t_k)\in\mathbb {N}^k$ and
$j=(j_1,\ldots,j_k)\in\mathbb Z_+^k$.
We use the following dyadic size notation. Then we set the following size restriction.
\begin{itemize}
\item For $\nu \in [k]$, we write $t_\nu \sim 2^{j_\nu}$ if
$2^{j_\nu-1} < t_\nu \le 2^{j_\nu}$.
\item We write $\mathfrak t \sim 2^{j}$ if
$t_\nu \sim 2^{j_\nu}$ for all $\nu \in [k]$.
\item We write $|\mathfrak t| \sim 2^{j}$ if
$|t_\nu| \sim 2^{j_\nu}$ for all $\nu \in [k]$.
\end{itemize}

Fix $k\geq 1$ and a finite index set
$\Lambda\subset\mathbb{Z}_+^k$, which will remain fixed throughout the proofs of
Main Theorems~\ref{mt1} and~\ref{mt2}.
We may assume throughout that $\mathbf{0} \notin \Lambda$,
since $|e^{2\pi i \xi_{\mathbf{0}}}| = 1$. For two positive scalars $a,b>0$, we write $a \lesssim b$ if
$a \le C b$ for some constant $C>0$ depending only on $\Lambda$.
We write $a\approx b$ if both $a\lesssim b$ and $b\lesssim a$ hold.
Throughout the paper, the implicit constants in the notation $\lesssim$
are independent of $\xi \in \mathbb{R}^{|\Lambda|}$ with
$\Lambda \subset \mathbb{Z}_+^{k}$ and of $j \in \mathbb{Z}_+^{k}$.
Moreover, we write $a \ll b$ if 
$a \leq cb$ for some sufficiently small constant $1>c>0$ depending only on $\Lambda$. Define  the positive number $K$ depending on $\Lambda$ as
   \begin{align}\label{00gg}
   	K=N^{10k}\ \text{and}\ N=\lvert\Lambda\rvert!\sum_{\mathfrak{m}\in \Lambda} (|\mathfrak{m}|+k).
   \end{align}
  Given  $K$  in (\ref{00gg}) and two consecutive entries $j_\nu\ge j_{\nu+1}$,
\begin{align}\label{balance}
	\begin{cases}
		j_\nu>(20K)j_{\nu+1} &\text{: imbalanced, } j_\nu\gg_\Lambda j_{\nu+1},\\
		(20K)j_{\nu+1}\ge j_\nu &\text{: balanced, } j_\nu\approx_\Lambda j_{\nu+1}.
	\end{cases}
\end{align}
Throughout this paper, the notations $\gg_{\Lambda}$ and $\approx_{\Lambda}$ have meanings different from $\gg$ and $\approx$, respectively.

 \noindent
{\bf Basic Reduction.} For any $\ell\in \mathbb{Z}_{+}$, to  extend the domain $\mathbb{Z}$ of the characteristic function   $\chi_{(2^{\ell-1},2^{\ell}]}(|s|)$   to $\mathbb{R}$,  we
define a smooth even cutoff function   by taking a value for $s\in \mathbb{R}$ as
\begin{align}\label{aria}
	\chi_\ell(s):=\begin{cases} 1\ \text{if $2^{\ell-1}+1/2 \le |s|\le 2^{\ell}$}\\
		0 \ \text{if $|s|\le  2^{\ell-1}$ or $ |s|\ge 2^{\ell}+1/2$}
	\end{cases} 
\end{align}
Then one can show the following properties:
\begin{itemize}
	\item[(1)] $\text{supp}(\partial_s\chi_{\ell})\subset \{s\in \mathbb{R}:|s|\in (2^{\ell-1},2^{\ell-1}+1/2)\cup (2^{\ell},2^{\ell}+1/2)\}$ and $|\partial_s\chi_{\ell}|\lesssim 1$. 
	\item[(2)]  Let $s\in \mathbb{Z}$. Then
	$\chi_\ell(s)=1$  if $2^{\ell-1}< |s|\le 2^{\ell} $, and $\chi_\ell(s)=0$ otherwise.
\end{itemize}
Taking products, for
$j=(j_1,\ldots,j_k)\in\mathbb Z_+^k$ and
$t=(t_1,\ldots,t_k)\in\mathbb R^k$, we define
\[
\chi_j(t):=\prod_{\nu=1}^k\chi_{j_\nu}(t_\nu).
\]
Given $\Lambda \subset \mathbb{Z}_+^k$ and $j= (j_1, \ldots, j_k) \in \mathbb{Z}_+^k$, we consider the dyadic pieces of both the discrete and continuous versions of the Hilbert transforms, defined respectively by
\begin{align*}
	H^{\Lambda}_j(\xi):= \sum_{t\in \mathbb{Z}^k}  \frac{e^{2\pi i \sum_{\mathfrak{m}\in\Lambda} \xi_{\mathfrak{m}} t^{\mathfrak{m}}}\chi_j(t)}{t_1\cdots t_k}\  \text{and}\  \mathcal{H}^{\Lambda}_j(\xi):= \int_{\mathbb{R}^k} \frac{e^{2\pi i \sum_{\mathfrak{m}\in\Lambda} \xi_{\mathfrak{m}} t^{\mathfrak{m}}}\chi_j(t)}{t_1\cdots t_k}dt,
\end{align*}
where $\xi=\xi_{\Lambda}=(\xi_{\mathfrak{m}})_{\mathfrak{m}\in \Lambda}\in \mathbb{R}^{|\Lambda|}$.  In order to prove the sufficiency part of the Main Theorem 1, we shall prove that under the assumption (\ref{1009}), there exists a constant $C > 0$, independent of $\xi=\xi_{\Lambda}\in \mathbb{R}^{|\Lambda|}$, such that
\begin{align}\label{se11}
\sum_{j\in \mathbb{Z}_+^k} |H_j^{\Lambda}(\xi)| \le C.
\end{align}
To establish this bound, it suffices to restrict the summation over $\mathbb{Z}_+^k$ to the subset
 \begin{align}
Z(k)=\{j\in \mathbb{Z}_{+}^{k}: j_1\ge \cdots \ge j_{k-1}\ge j_k\},\label{joor}
\end{align} 
since there are only $k!$ permutations of the components $(j_1, \ldots, j_k)$.

We now  review the following four types of preliminary estimates that are necessary for proving \eqref{se11}.

\begin{itemize}
	\item[3.1.] Continuous analogues for Hilbert transforms,
	
	\item[3.2.] Multi-parameter Weyl sum estimates,
	
	\item[ 3.3.] Multi-parameter Gauss sum estimates and their averaged versions,
	
	\item[3.4.] Multi-parameter sublevel set estimates for lattice points.
\end{itemize}

\subsection{Continuous Multiple Hilbert Transform}
We introduce a fundamental estimate for the continuous $k$-parameter multiple Hilbert transform. It provides both the uniform summability of the associated oscillatory singular integrals and the $L^p$-boundedness of arbitrary partial sums. We defer the proof of Lemma~\ref{po3} to Subsection~\ref{MHT}.

\begin{lemma} [Multiple Hilbert Transform]\label{po3}
	Let $\Omega\subset \mathbb Z_+^k$ be a finite set. For    $\xi_{\Omega}=(\xi_{\mathfrak{m}})_{\mathfrak{m}\in\Omega}$ and  $J=(j_1,\cdots,j_k)\in \mathbb{Z}_+^k$,  denote the $k$-parameter oscillatory singular integrals as 
	\begin{align*} 
		\mathcal{H}^{\Omega}_{J}(\xi_{\Omega})&:= \int e^{2\pi i\sum_{\mathfrak{m}\in\Omega} \xi_{\mathfrak{m}} x^{\mathfrak{m}} }\prod_{\nu=1}^k\chi_{j_\nu}(x_\nu) \frac{dx_1\cdots dx_k }{x_1\cdots x_k}
	\end{align*}
	where $\chi_{j_\nu}$ is defined in (\ref{aria}). Assume that condition \eqref{v16} holds.
	  Then for any $\epsilon>0$, there is $C_{\Omega,\epsilon} >0$ depending  only on $\Omega$ and $\epsilon$ such that
	\begin{align}\label{po4}
		\sum_{J=(j_1,\cdots,j_k)\in \mathbb{Z}_+^k} \left( |\mathcal{H}^{\Omega}_{J}(\xi_{\Omega}) |+ |\mathcal{H}^{\Omega}_{J}(\xi_{\Omega}) |^{\epsilon}\right) \le C_{\Omega,\epsilon}. 
	\end{align}
	For any $1<p<\infty$ and any subset $E\subset\mathbb{Z}_+^k$, there is $C_{p,\Omega}>0$  such that
	\begin{align}\label{poo4}
		\| \sum_{J\in \mathbb{Z}_+^k\cap E}  [\mathcal{H}^{\Omega}_{J}]^{\vee}*f \|_{L^p(\mathbb{R}^{|\Omega|})}\le C_{p,\Omega}\| f \|_{L^p(\mathbb{R}^{|\Omega|})}.
	\end{align}The constant $C_{p,\Omega}$ may depend only on  $p$ and $\Omega$.
	
\end{lemma}

 \begin{remark}\label{re3.1}
Suppose that 	$\Omega\subset\mathbb{Z}_+^{k}$ satisfies the condition \eqref{1009}. Then, for every $\ell\in [k-1]$, its projection $\Omega_1$, defined by
	$$
	\left\{
	(m_1,\ldots,m_\ell):
	(m_1,\ldots,m_\ell,m_{\ell+1},\ldots,m_{k})
	\in\Omega \ \text{for some $(m_{\ell+1},\ldots,m_{k})\in \mathbb{Z}_{+}^{k-\ell}$}
	\right\},
	$$
	  satisfies the condition~\eqref{1009} with $k=\ell$, which directly implies that the condition~\eqref{v16} with $k=\ell$ also holds for $\Omega_1$. Consequently, for every $\ell\in[k-1]$, the projected set $\Omega_1$ satisfies the continuous estimates stated in Lemma~\ref{po3}.
\end{remark}

\subsection{Multi-Parameter Weyl Sums}
 Multi-parameter Weyl exponential sums have been studied extensively. See, for instance, Lemma~5.3 of~\cite{ACK} for a representative result, as well as Section~5 of~\cite{B4} and Sections~4--5 of~\cite{KS}.

\begin{definition}\label{df23}
	Let $\Omega$ be a finite set of indices in $\mathbb{Z}^{k}$.
	Then we  define  the set  $\mathbb{Q}^{\Omega}[2^{\alpha},2^{\beta}]$
	of  rational vectors  $b/p=(b_\mathfrak{m})_{\mathfrak{m}\in \Omega} /p$,  which is given by
	\begin{align*} 
		\ \ \mathbb{Q}^{\Omega}[2^{\alpha},2^{\beta}]=  \left\{\frac{b}{p}\in   \mathbb{Q}^{|\Omega|}: 2^{\alpha}\le p \le 2^{\beta}\ \text{and}\  \text{gcd}(p,b)=1\right\}.
	\end{align*}
\end{definition}

\begin{definition}
	Let $\Omega$ be a finite subset of $\mathbb{Z}_+^k \setminus \{\mathbf{0}\}$, and let 
$\xi_{\Omega} = (\xi_{\mathfrak{m}})_{\mathfrak{m}\in\Omega} \in \mathbb{R}^{|\Omega|}$. 
For any $j \in Z(k)$ and any  pair of parameters $0 < c_1, c_2 \le 1$, we define the major arc cutoff by
\begin{align}\label{t12}
\Psi_{j,(c_1,c_2)}^{\Omega,\mathrm{major}}(\xi_{\Omega})
:= \sum_{a/q=(a_{\mathfrak{m}})/q \in \mathbb{Q}^{\Omega}[1,2^{c_1 j_k/10}]}
\prod_{\mathfrak{m}\in \Omega} 
\psi\!\left(
\frac{\xi_{\mathfrak{m}} - a_{\mathfrak{m}}/q}
{2^{-\mathfrak{m}\cdot j}\, 2^{c_2 j_k/10}}
\right).
\end{align}
We introduce this family of cutoff functions to allow flexibility in the parameters $c_1$ and $c_2$, which will be chosen according to the needs of different estimates.
\end{definition}
\begin{definition}\label{defi33}
For a multi-index $j := (j_1, \dots, j_k) \in Z(k)$ and a vector $r = (r_1, \dots, r_k)$ with $r_i \in (0, 2^{j_i}] \cap \mathbb{Z}$ for each $i = 1, \dots, k$, we define the Weyl sum $W_{j,r}^{\Omega}(\xi_\Omega)$ by
\[
W_{j,r}^{\Omega}(\xi_\Omega) := \sum_{t_1=r_1}^{2^{j_1}} \cdots \sum_{t_k=r_k}^{2^{j_k}} e^{2\pi i \sum_{\mathfrak{m} \in \Omega} \xi_{\mathfrak{m}} t^{\mathfrak{m}}}.
\]
\end{definition}

The following lemma is an immediate consequence of Lemma~5.3 in~\cite{ACK}.
\begin{lemma}[Arkhipov, Chubarikov, Karatsuba \cite{ACK}]\label{lem:ACK_bound}
Let $j=(j_1,\cdots,j_k) \in Z(k) $.	Let $\Omega$ be a finite subset of $\mathbb{Z}_+^k \setminus \{\mathbf{0}\}$, and let 
$\xi_{\Omega} = (\xi_{\mathfrak{m}})_{\mathfrak{m}\in\Omega} \in \mathbb{R}^{|\Omega|}$. Set $d_{\Omega}:=\max\{|\mathfrak{m}|:\mathfrak{m}\in \Omega\}$.  Suppose that for each $\mathfrak{m} \in \Omega$, there exists a   rational number $c_{\mathfrak{m}}/q_{\mathfrak{m}}$ in lowest terms, i.e., $(q_{\mathfrak{m}}, c_{\mathfrak{m}}) = 1$, satisfying 
\[
1 \le q_{\mathfrak{m}} \le 2^{j \cdot \mathfrak{m}} 2^{-j_k/10}
\]
and
\[
\left| \xi_{\mathfrak{m}} - \frac{c_{\mathfrak{m}}}{q_{\mathfrak{m}}} \right| \le \frac{2^{j_k/10}}{q_{\mathfrak{m}} 2^{j \cdot \mathfrak{m}}}.
\]
 Assume, in addition, that \begin{align} \label{dmm1} \operatorname{lcm} \bigl(q_{\mathfrak m}:\mathfrak m\in\Omega,\ |\mathfrak m|\geq 2\bigr) \geq 2^{j_k/30}. \end{align} Then, for every $r=(r_1,\ldots,r_k)$ satisfying \[ r_i\in(0,2^{j_i}]\cap\mathbb{Z}, \qquad i=1,\ldots,k, \] there exist constants $C,c>0$, depending only on $k$ and $d_{\Omega}$, such that \begin{align} \label{wop} \left| W_{j,r}^{\Omega}(\xi_{\Omega}) \right| \leq C2^{j_1+\cdots+j_k}2^{-cj_k}. \end{align} \end{lemma}

 \begin{proposition}[Weyl Sum Estimate]\label{ws1}
Let $\Omega \subset \mathbb{Z}_+^k \setminus \{\mathbf{0}\}$ be a finite set of multi-indices, and let $\xi_{\Omega} := (\xi_{\mathfrak{m}})_{\mathfrak{m} \in \Omega} \in \mathbb{R}^{|\Omega|}$.   Then, for any $j \in  Z(k)$ and any $r = (r_1, \dots, r_k)$ with $r_i \in (0, 2^{j_i}] \cap \mathbb{Z}$ for each $i = 1, \dots, k$, there exist positive constants $C$ and $c$, depending only on $\Omega$, such that
\begin{equation}\label{wop22}
\left| W_{j,r}^{\Omega}(\xi_\Omega) \left( 1 - \Psi_{j,(1/2,1)}^{\Omega,\mathrm{major}}(\xi_{\Omega}) \right) \right| \le C 2^{j_1 + \dots + j_k} 2^{-c j_k}.
\end{equation}
\end{proposition}

\begin{proof}
For each $\mathfrak{m} \in \Omega$, Dirichlet's approximation theorem guarantees the existence of a rational number $c_{\mathfrak{m}}/q_{\mathfrak{m}}$ with $1 \le q_{\mathfrak{m}} \le 2^{j \cdot \mathfrak{m}} 2^{-j_k/10}$ and $(q_{\mathfrak{m}}, c_{\mathfrak{m}}) = 1$ such that
\[
\left| \xi_{\mathfrak{m}} - \frac{c_{\mathfrak{m}}}{q_{\mathfrak{m}}} \right| < \frac{2^{j_k/10}}{q_{\mathfrak{m}} 2^{j \cdot \mathfrak{m}}}.
\]
We divide the argument into two cases based on the size of the denominators $q_{\mathfrak{m}}$.

\medskip
\noindent\textbf{Case 1.} Suppose there exists a multi-index $\mathfrak{m} \in \Omega$ with $|\mathfrak{m}| = 1$ such that $q_{\mathfrak{m}} \ge 2^{j_k/(60k)} $. In this case, applying the standard one-dimensional Weyl sum estimate directly yields the desired bound in \eqref{wop22}.

\medskip
\noindent\textbf{Case 2.} Suppose instead that $q_{\mathfrak{m}} < 2^{j_k/(60k)} $ for all $\mathfrak{m} \in \Omega$ with $|\mathfrak{m}| = 1$. Since the expression is evaluated outside the major arcs—that is, $\Psi_{j,(1/2,1)}^{\Omega,\mathrm{major}}(\xi) \neq 1$—it follows that the least common multiple of the denominators satisfies
\[
\operatorname{lcm}(q_{\mathfrak{m}} : \mathfrak{m} \in \Omega) \ge 2^{j_k/20}.
\]
Together with the contribution from the range
$\prod_{|\mathfrak{m}|=1} q_{\mathfrak{m}}
\leq 2^{j_k/60},$
this recovers precisely the condition \eqref{dmm1} required to apply the
Weyl sum estimate \eqref{wop}, and hence establishes \eqref{wop22}.

\end{proof}

\subsection{Gauss Sums and Averaged Gauss Sums}
\begin{lemma}[Gauss Sum] \label{el22} 
Let $\Omega\subset \mathbb{Z}_+^k \setminus \{{\bf 0}\}$ be a finite set and set  $d_{\Omega}:=\max\{|\mathfrak{m}|:\mathfrak{m}\in \Omega\}$. Consider a rational vector: 
   $$
  \frac{a}{q}=\frac{(a_{\mathfrak{m}})_{\mathfrak{m}\in \Omega}}{q}  \ \text{where}  \ (a,q)=1.$$ 
   We denote the $k$-parameter  Gauss sums as 
 \begin{align*} 
 S^{\Omega}\left(\frac{a}{q}\right)& = \frac{1}{q^k}\sum_{t= (t_1,\cdots,t_k)\in [q]^k}e^{2\pi i \sum_{\mathfrak{m}\in\Omega } \frac{a_{\mathfrak{m}}}{q} t^{\mathfrak{m}}}.  
\end{align*}
 Then, there exist positive constants $C$ and $c$, depending only on $k$ and $d_{\Omega}$, such that
\begin{equation}\label{g2}
\left|S^{\Omega}\!\left(\frac{a}{q}\right)\right|
\le C q^{-c}.
\end{equation}
\end{lemma}

 \begin{definition}[Gauss Sum of Layered Polynomial]
 Let $\Omega\subset \mathbb{Z}_+^k \setminus \{{\bf 0}\}$ and fix $ \ell \in  [k-1]$. For each $\mathfrak{m}=(m_1,\ldots,m_k)\in\Omega$, write \[ m=(m_1,\ldots,m_\ell) \qquad\text{and}\qquad n=(m_{\ell+1},\ldots,m_k), \] so that \begin{align}\label{pf1} \mathfrak{m}=(m,n). \end{align}
Then, define $\Omega'\subset\Omega$ by
\begin{align*}
\Omega':=  
	\{ (m,n)\in\Omega: m\ne {\bf 0}\}, \ \text{where ${\bf 0}$ is the zero vector in $\mathbb{Z}^{\ell}$}\  
\end{align*}
with its projection
$\Omega'_{1}$ and fiber $\Omega'(m)$ defined as \begin{align*}
	\Omega_{1}' &:=\{m: (m,n)\in \Omega'\text{ for some } n\in\mathbb{Z}_{+}^{k-\ell}\} \ \text{and}\ \Omega'(m):= \{ n\in\mathbb Z_+^{k-\ell}: (m,n)\in \Omega' \}.
\end{align*}	Take a rational vector 
$
\frac{a}{q}=\frac{(a_{mn})_{(m,n)\in \Omega}}{q} 
$ where  $(q,a)=1$.  Write $\mathfrak{t}=(\mathfrak{t}_1,\mathfrak{t}_2)$ where $\mathfrak{t}_1=(t_{1},\dots,t_{\ell})\in \mathbb{Z}^{\ell}$ and $\mathfrak{t}_2=(t_{\ell+1},\dots,t_k)\in \mathbb{Z}^{k-\ell}$.
We consider a vector field, depending on $\mathfrak{t}_2=(t_{\ell+1},\dots,t_k)$,
\[
\left( \frac{a_{m}(\mathfrak{t}_2)}{q} \right)_{m\in \Omega_{1}'},
\]
where   for each $m=(m_1,\cdots,m_{\ell})\in \Omega_1'$,  the layered polynomial is defined by
\[
a_{m}(\mathfrak{t}_2):=\sum_{n \in \Omega'(m)} a_{mn}\,\mathfrak{t}_2^{n}, 
\qquad \text{with } \ \mathfrak{t}_2^n=t_{\ell+1}^{m_{\ell+1}}\cdots t_k^{m_k}.
\]
 Associated with the vector field  $\left( \frac{a_{m}(\mathfrak{t}_2)}{q} \right)_{m\in \Omega_{1}'}$,    define the averaged Gauss sum: $$S^{\Omega_{1}'}\left(\left( \frac{a_{m}(\mathfrak{t}_2)}{q} \right)_{m\in \Omega_{1}'}\right):= \frac{1}{q^\ell}\sum_{\mathfrak{t}_1= (t_1,\cdots,t_\ell)\in  [1,q]^{\ell}}e^{2\pi i \sum_{m\in\Omega_1' }  \frac{a_{m}(\mathfrak{t}_2)}{q}  \mathfrak{t}_1^{m}},$$  
where  $\mathfrak{t}_1^{m}=t_1^{m_1}\cdots t_\ell^{m_\ell}.$
\end{definition}

 We now introduce the averaged Gauss sum estimates in the following Proposition. We defer the proof of Proposition~\ref{lem22p}  to Subsection \ref{ag}.
 
\begin{proposition}[Averaged  Gauss Sum]\label{lem22p}
Let $\Omega\subset \mathbb{Z}_+^k \setminus \{{\bf 0}\}$ be a finite set and consider a rational vector: 
$
\frac{a}{q}=\frac{(a_{\mathfrak{m}})_{\mathfrak{m}\in \Omega}}{q}  \ \text{where}  \ (a,q)=1.$ 
Let $j=(j_1,\cdots,j_k) \in Z(k) $ and $\ell\in[k-1]$. Suppose that $\Omega_{1}'\neq \emptyset.$ Then, there exist positive constants $C$ and $c$, depending only on $\Omega$, such that
\begin{align}\label{po10}
 \frac{1}{ 2^{j_{\ell+1} +\cdots+ j_k}       } \sum_{\mathfrak{t}_2\in \prod_{\nu=\ell+1}^k [1,2^{j_\nu}] }  \left|S^{\Omega_{1}'}\left(\left( \frac{a_{m}(\mathfrak{t}_2)}{q} \right)_{m\in \Omega_{1}'}\right) \right|\le C d(q) \left(q^{-c}+2^{-cj_{k}}\right),
\end{align} 
where $d(q)$ is the number of divisors of $q$. 

 \end{proposition}

\noindent\textbf{Choice of the Decay Exponent $\delta_{\Lambda}$.}
	Let $\Lambda\subset \mathbb{Z}_+^k \setminus \{{\bf 0}\}$ be a finite set, and fix $ \ell \in  [k-1]$.    Let $j=(j_1,\cdots,j_k) \in Z(k) $ and  fix $r = (r_1, \dots, r_k)$ with $r_i \in (0, 2^{j_i}] \cap \mathbb{Z}$ for each $i = 1, \dots, k$. In view of the decay estimates \eqref{wop22}, \eqref{g2}, and \eqref{po10}, we choose a constant $\delta_{\Lambda}>0$, depending only on $\Lambda$, sufficiently small so that the following four estimates hold simultaneously.
	\begin{itemize}
		\item[(1)] Suppose that $j_1 \approx_{\Lambda} \cdots \approx_{\Lambda} j_k.$ Then, for every $\xi=(\xi_{\mathfrak{m}})_{\mathfrak{m}\in\Lambda} \in \mathbb{R}^{|\Lambda|},$ we have the Weyl sum estimate \begin{align}\label{weyl}
			\left| W_{j,r}^{\Lambda}(\xi) \left( 1 - \Psi_{j,(1/2,1)}^{\Lambda,\mathrm{major}}(\xi) \right) \right| \lesssim 2^{j_1 + \dots + j_k} 2^{-\delta_{\Lambda} j_1},
		\end{align}
		\item[(2)] Define $\Lambda_{2}:=\{n: (m,n)\in \Lambda\text{ for some } m\in\mathbb{Z}_{+}^{\ell}\}.$  Suppose that $\Lambda_2\neq \emptyset$ and $\textbf{0}\notin \Lambda_2.$ Assume  $j_{\ell+1} \approx_{\Lambda} \cdots \approx_{\Lambda} j_k.$  Then, for every $\xi_{\Lambda_2}=(\xi_{n})_{n\in\Lambda_2} \in \mathbb{R}^{|\Lambda_2|},$ we have the Weyl sum estimate \begin{align}\label{weyl1}
			\left| W_{(j_{\ell+1},\cdots,j_{k}),(r_{\ell+1},\cdots,r_{k})}^{\Lambda_2}(\xi_{\Lambda_2}) \left( 1 - \Psi_{(j_{\ell+1},\cdots,j_{k}),(1/2,1)}^{\Lambda_2,\mathrm{major}}(\xi_{\Lambda_2}) \right) \right| \lesssim 2^{j_{\ell+1} + \dots + j_k} 2^{-\delta_{\Lambda} j_{\ell+1}},
		\end{align}
		\item[(3)] For every $a/q\in\mathbb{Q}^{|\Lambda|}$ satisfying $(a,q)=1$, we have the Gauss sum estimate \begin{align}\label{gaus}
			\left|S^{\Lambda}\!\left(\frac{a}{q}\right)\right|
			\lesssim q^{-\delta_{\Lambda}},
		\end{align}
		\item[(4)] Suppose that $j_{\ell+1} \approx_{\Lambda} \cdots \approx_{\Lambda} j_k$ and $ 2^{j_k/20} \le q \le 2^{N^{4k}j_{\ell+1}},$ where $N$  is defined in \eqref{00gg}. Suppose that $\Lambda_1'\neq \emptyset$. Then, for every $a/q\in\mathbb{Q}^{|\Lambda|}$ satisfying $(a,q)=1$, we have the averaged Gauss sum estimate \begin{align}\label{aver}
			\frac{1}{ 2^{j_{\ell+1} +\cdots+ j_k}       } \sum_{\mathfrak{t}_2\in \prod_{\nu=\ell+1}^k [1,2^{j_\nu}] }  \left|S^{\Lambda_{1}'}\left(\left( \frac{a_{m}(\mathfrak{t}_2)}{q} \right)_{m\in \Lambda_{1}'}\right) \right|\lesssim q^{-\delta_{\Lambda}}2^{-\delta_{\Lambda} j_{\ell+1}}.
		\end{align}
	\end{itemize}

\subsection{Sublevel Set Estimates on $\mathbb{Z}^k$}
We introduce, without proof, the following sublevel set estimate on $\mathbb{Z}^k$ associated with the polynomial
$
\sum_{\mathfrak{m}\in\Lambda}\eta_{\mathfrak{m}}t^{\mathfrak{m}}\in \mathbb{R}_{\Lambda}[t_1,\cdots,t_k].
$

\begin{lemma}[Sublevel Set Estimates on $\mathbb{Z}^k$]\label{eep1}
	Let $\Lambda\subset \mathbb{Z}^k_+$ be a finite subset. Consider a polynomial $\sum_{\mathfrak{m}=(m_1,\cdots,m_k)\in \Lambda}\eta_\mathfrak{m}\mathfrak{t}^\mathfrak{m}\in \mathbb{R}_{\Lambda}[t_1,\cdots,t_k]$. Set $
	d:=\max\{|\mathfrak{m}|:\mathfrak{m}\in \Lambda\}$ where $|\mathfrak{m}|=m_1+\cdots+m_k$.
	Let   $j=(j_1,\cdots,j_k)\in Z(k)$.  Suppose    that $ j_k \ge r\ge 0$ and that $$A(j,\eta):=\sum_{\mathfrak{m}\in \Lambda}|\eta_\mathfrak{m}2^{j\cdot \mathfrak{m}} |\ge 2^r\epsilon\ \text{ for $\epsilon>0$. }$$ Then there exists $ C >0$  independent of $\eta=(\eta_{\mathfrak{m}}),j,r$ and $\epsilon$  such that
	\begin{align*}
		\left|\left\{\mathfrak{t}\in \mathbb{Z}^k:|t_\nu|\sim 2^{j_\nu}\ \text{for $\nu\in [k]$}\ \text{and}\  |\sum_{\mathfrak{m}\in \Lambda}\eta_{\mathfrak{m}}\mathfrak{t}^\mathfrak{m}|\le \epsilon\right\} \right|\le C  2^{j_1+\cdots+j_k} 2^{-(1/d)r}.
	\end{align*}
\end{lemma}

The continuous analogue of Lemma~\ref{eep1} is a  consequence of
the sublevel set inequality of Carbery and Wright~\cite{CW}.
Applying Theorem~2 of~\cite{CW} yields the desired continuous sublevel set
estimate with exponent $1/d$: \begin{align*}
	\left|\left\{\mathfrak{t}\in \mathbb{R}^k:|t_\nu|\sim 2^{j_\nu}\ \text{for $\nu\in [k]$}\ \text{and}\  |\sum_{\mathfrak{m}\in \Lambda}\eta_{\mathfrak{m}}\mathfrak{t}^\mathfrak{m}|\le \epsilon\right\} \right|&\lesssim 2^{j_1+\cdots+j_k}(\frac{\epsilon}{A(j,\eta)})^{1/d}\\
	&\le   2^{j_1+\cdots+j_k} 2^{-(1/d)r}.\notag
\end{align*}

  \section{Major Arc Estimates and Disparity Sectors}\label{sec3}
\subsection{Major Arc Estimates}To prove part (ii) of both Main Theorems \ref{mt1} and \ref{mt2}, we shall henceforth restrict our attention to sets $\Lambda\subset \mathbb{Z}_{+}^k$ satisfying the following condition: \begin{align}\label{wep2}
	\text{$\Lambda\subset \mathbb{Z}^k_+$ contains an odd subset  and $\Lambda$ satisfies the condition  (\ref{1009}).}
\end{align}  In this subsection, we complete the major arc estimate by applying Lemma~\ref{po3}. 
Recall
$$ \mathbb{Q}^{\Lambda}[1,2^{j_k/10}]:=  \left\{\frac{a}{q}=\frac{(a_{\mathfrak{m}})_{ \mathfrak{m}\in \Lambda }  }{q}\in  \mathbb{Q}^{|\Lambda|}: 1\le q \le 2^{j_{k}/10}\ \text{and}\  (q,a)=1\right\}.$$
For $j \in Z(k)$ as in \eqref{joor}, we define the major arc cutoff function
\begin{align*}
\Psi^{\Lambda,\mathrm{major}}_{j}(\xi)
&:= \sum_{a/q \in \mathbb{Q}^{\Lambda}[1,2^{j_k/10}]}
\psi_j^{\Lambda}\!\left( \frac{\xi-\frac{a}{q}}{2^{j_k/10}} \right),
\end{align*}
where
\[
\psi_j^{\Lambda}(\xi)
:= \prod_{\mathfrak{m}\in\Lambda}
\psi\!\left( \xi_{\mathfrak{m}}\, 2^{j\cdot \mathfrak{m}} \right).
\]
This coincides with $\Psi^{\Lambda,\mathrm{major}}_{j,(1,1)}(\xi)$ defined in \eqref{t12}; for notational simplicity, we henceforth omit the sub-index $(1,1)$.
Note that $\psi_j^{\Lambda}\left( \frac{\xi-\frac{a}{q}}{2^{j_{k}/10}} \right)$ are disjointly supported as $a/q$ ranges over $a/q\in \mathbb{Q}^{\Lambda}[1,2^{j_k/10}]$.   Moreover, the smallness of  $1\le q\le 2^{j_k/10}$ leads a  good integral approximation of $H_j^{\Lambda}(\xi) \Psi^{\Lambda,\rm{major}}_{j}(\xi)$ as in Proposition \ref{lemm29} below.
\begin{proposition}[Major Arc Estimate] \label{lemm29}
For $\Lambda\subset \mathbb{Z}_+^k $, let $\xi=(\xi_{\mathfrak{m}})_{\mathfrak{m}\in \Lambda}\in \mathbb{R}^{|\Lambda|}$. Under the condition \eqref{wep2},   we have (1) and (2) below. 
\begin{itemize}
\item[(1)]  Integeral Approximation.  For $j\in Z(k)$,
one has
\begin{align}\label{b40}
&H^{\Lambda}_j(\xi) \Psi^{\Lambda,\rm{major}}_{j}(\xi)\nonumber\\
&\qquad=  \sum_{a/q\in \mathbb{Q}^{\Lambda}[1, 2^{j_k/10}]} \left(S^{\Lambda}(a/q) \mathcal{H}^{\Lambda}_j(\xi-a/q)  \psi_j^{\Lambda}\left( \frac{\xi-\frac{a}{q}}{2^{j_{k}/10}} \right)+E_{j,a/q}(\xi)\right), 
\end{align}
where the error term is bounded by
\begin{align*}
&\sum_{(j_1,\cdots,j_{k-1}):j\in Z(k)}\sum_{a/q\in \mathbb{Q}^{\Lambda}[1, 2^{j_k/10}]}|E_{j,a/q}(\xi)|\lesssim 2^{-j_k/(40)}.
\end{align*}
\item[(2)]  Major Arc Estimate.
There exists a constant $C >0$ independent of $\xi$:
\begin{align}\label{23mm}
&\sum_{j\in Z(k) }\left|   \sum_{a/q\in \mathbb{Q}^{\Lambda}[1,2^{j_k/10}]}S^{\Lambda}(a/q)\mathcal{H}_j^{\Lambda}(\xi-a/q) \psi_j^{\Lambda}\left( \frac{\xi-\frac{a}{q}}{2^{j_{k}/10}} \right) \right|
  \le C,
\end{align}
  which  implies 
$
  \sum_{j\in Z(k)} |H^{\Lambda}_j(\xi) \Psi^{\Lambda,\rm{major}}_{j}(\xi)|\lesssim 1.
$
\end{itemize}

\end{proposition}

\begin{proof}[Proof of  (\ref{b40})] 
Since the supports 
\[
\Bigl\{\xi:\psi_j^{\Lambda}\!\left(\frac{\xi-\frac{a}{q}}{2^{j_k/10}}\right)\neq 0\Bigr\}
\]
are disjoint over all $\frac{a}{q}\in\mathbb{Q}^{\Lambda}[1,2^{j_k/10}]$, it suffices to work with a fixed such $a/q$.  
Set
\[
\beta=(\beta_{\mathfrak m}) := \xi-\frac{a}{q}
\]
and recall  $\chi_j(t)=\prod_{\nu=1}^k \chi_{j_{\nu}}(t_\nu)$ to rewrite 
 \begin{align*}
  H^{\Lambda}_j(\xi) &= \sum_{t\in\mathbb{Z}^k} e^{2\pi i\sum_{\mathfrak{m}\in \Lambda}  
 \frac{a_{\mathfrak{m}}}{q} t^{\mathfrak{m}}}F_{j,\beta}(t) \ \text{where}\
 F_{j,\beta}(t):=\frac{e^{2\pi i \sum_{\mathfrak{m}\in \Lambda}  \beta_{\mathfrak{m}}t^{\mathfrak{m}} }\chi_j(t) }{t_1\cdots t_k}.\  
 \end{align*}
 Using the change of variables
  $t=q\mu+\ell$ where $\mu= (\mu_\nu)_{\nu=1}^k\in \mathbb{Z}^k$ and $\ell=(\ell_\nu)_{\nu=1}^k  \in [q]^k$,  
 \begin{align}\label{4ggk}
 H^{\Lambda}_j(\xi)  = \frac{1}{q^k} \sum_{\ell\in [q]^k} e^{2\pi  i\sum_{\mathfrak{m}\in \Lambda}  \frac{a_{\mathfrak{m}}}{q}\ell^{\mathfrak{m}}}\sum_{\mu\in\mathbb{Z}^k} q^k F_{j,\beta}(\mu q+\ell).
\end{align}
For fixed $q,\ell$, the Fourier transform of $\mu\in \mathbb{R}^k \mapsto q^k F_{j,\beta}(q\mu+\ell)$  is
\[
e^{2\pi i w\cdot \ell/q}\,\widehat{F_{j,\beta}}(w/q),\qquad w\in\mathbb{Z}^k.
\] 
Applying Poisson summation,
\[
\sum_{\mu\in\mathbb{Z}^k} q^kF_{j,\beta}(q\mu+\ell)
=
\widehat{F_{j,\beta}}(0)+
\sum_{w\neq 0}
e^{2\pi i w\cdot\ell/q}\widehat{F_{j,\beta}}(w/q),
\]
and since 
\[
\widehat{F_{j,\beta}}(0)=\mathcal{H}_j^{\Lambda}(\beta),
\]
substitution into \eqref{4ggk} yields
\[
H_j^\Lambda(\xi)\,
\psi_j^\Lambda\!\left(\frac{\xi-\frac{a}{q}}{2^{j_k/10}}\right)
=
S^\Lambda(a/q)\,
\mathcal{H}_j^\Lambda(\xi-a/q)\,
\psi_j^\Lambda\!\left(\frac{\xi-\frac{a}{q}}{2^{j_k/10}}\right)
+
E_{j,a/q}(\xi),
\]
where
\[
E_{j,a/q}(\xi)
=
\frac{1}{q^k}
\sum_{\ell\in[q]^k}
e^{2\pi i\sum_{\mathfrak m\in\Lambda}\frac{a_{\mathfrak m}}{q}\ell^{\mathfrak m}}
\sum_{w\neq 0}
e^{2\pi i w\cdot\ell/q}\,
\widehat{F_{j,\beta}}(w/q)\,
\psi_j^\Lambda\!\left(\frac{\xi-a/q}{2^{j_k/10}}\right).
\]
Let
\[
E_j(\xi):=\sum_{a/q\in\mathbb{Q}^{\Lambda}[1,2^{j_k/10}]} |E_{j,a/q}(\xi)|.
\]
To prove \eqref{b40}, it suffices to show that
\begin{equation*}
\sum_{j'=(j_1,\cdots,j_{k-1}):\, j\in Z(k)} E_j(\xi) \lesssim 2^{- j_k/(40)}.
\end{equation*}
\\
\textbf{Step 1. Estimate of $\widehat{F_{j,\beta}}(w/q)$.}
We have
\begin{align}\label{10s}
\widehat{F_{j,\beta}}(w/q)
=
\int 
e^{2\pi i\left(\sum_{\nu=1}^k \frac{w_\nu}{q}s_\nu + \sum_{\mathfrak m\in\Lambda}\beta_{\mathfrak m}s^{\mathfrak m}\right)}
\frac{\chi_j(s)}{s_1\cdots s_k}\,ds.
\end{align}
Enlarging $\Lambda$ by adjoining ${\bf e}_\nu$ for each $\nu$ such that $w_\nu\neq0$, and denoting the resulting set by $\widetilde\Lambda$, we observe that every $\mathfrak n\in\widetilde\Lambda$ has at most one odd component. 
Consequently, if $\Omega\subset\widetilde\Lambda$ is an odd subset, then
$[\Omega]_2\supset\{ {\bf e}_1,\ldots,{\bf e}_k\},$ and hence $\operatorname{rank}(\Omega)=k$, as discussed in Remark~\ref{rk423}.
This implies that $\tilde{\Lambda}$ satisfies the condition \eqref{v16}. So we can apply \eqref{po4} to obtain that
\begin{equation}\label{kp2}
\sum_{j\in\mathbb{Z}_+^k}
|\widehat{F_{j,\beta}}(w/q)|^{1/2}
\lesssim 1.
\end{equation}
On the other hand, the phase function has the gradient 
\begin{align}\label{gt11}
&\nabla_s \left(\sum_{\nu}\frac{w_\nu}{q}  s_\nu +\sum_{\mathfrak{m}\in \Lambda}  \beta_{\mathfrak{m}} s^{\mathfrak{m}}\right)=\left(\frac{ w_\nu }{q}\right)_{\nu=1}^k+\left(\sum_{\mathfrak{m}\in \Lambda}  \beta_{\mathfrak{m}}m_\nu s^{\mathfrak{m}-{\bf e}_{\nu}}\right)_{\nu=1}^k.
\end{align}
Let $j\in Z(k)$.  By $ \psi_j^{\Lambda}\left( \frac{\beta}{2^{j_{k}/10}} \right)\ne 0$ with $|s_\nu|\sim 2^{j_\nu}$, one has in (\ref{gt11}),
\begin{align*}
|\beta_{\mathfrak{m}}   s^{\mathfrak{m}-{\bf e}_{\nu}}|\lesssim 2^{j_k /10-j_\nu}=O(2^{-(9/10)j_\nu})\ \text{   for all $\mathfrak{m}\in \Lambda$.}
\end{align*}
 This combined with $1\le q\le 2^{j_k/(10)}, $  implies that  
$$\text{the $\nu^{th}$ component in (\ref{gt11}),  for every  $\nu\in [k]$, has a lower bound $\frac{1}{4}|w_\nu 2^{- j_\nu/10}|$.
}$$
Moreover,  the derivatives with respect to $s_\nu$ of the amplitude in \eqref{10s} satisfy $$
\text{$|\partial^{\ell}_{s_\nu}(1/s_\nu)|\lesssim |1/s_\nu^{\ell+1}|$, $|\{s_\nu:\partial_{s_\nu}(  \chi_j(s) )\ne 0\}|\lesssim 1$, and $|\partial^\ell_{s_\nu}(  \chi_j(s) )|\lesssim 1$.}$$
Thus integrating by parts four times in each variable $s_\nu$ gives 
\begin{align*}
&\left|\widehat{F_{j,\beta}}\left(\frac{w}{q} \right)\right| \lesssim  \left(\prod_{\nu=1}^k (|2^{j_\nu/10}w_\nu|+1)^{-4} \right)
  2^{-j_k/10},
\end{align*}
under the condition that
$w \ne {\bf 0}$ and $j_1\ge \cdots \ge j_k$ in $Z(k)$. Hence
\begin{align*}
  \sum_{w \ne {\bf 0}}   \left|\widehat{F_{j,\beta}}\left(\frac{w}{q} \right)\right|^{1/2} \lesssim 2^{-j_k/20}.
\end{align*}
This with (\ref{kp2})  implies that
\begin{align*} 
  \sum_{w \ne {\bf 0}}\sum_{(j_1,\cdots,j_{k-1}):j\in Z(k)} 
 \left|\widehat{F_{j,\beta}}\left(\frac{w}{q} \right)\right| \lesssim 2^{-j_k/20} \le 2^{-j_k/(40)}q^{-1/4},
\end{align*}
for some $c>0$ under the assumption that $1\le q\le 2^{j_k/10}$. 
\\
\textbf{Step 2. Summing over $a/q$.}
Therefore, 
\begin{align*}
\sum_{j'  :j\in Z(k)} E_j(\xi)&\lesssim \sum_{a/q\in \mathbb{Q}^{\Lambda}[1, 2^{j_k/10}] } \sum_{w \ne {\bf 0}}  \sum_{j':j\in Z(k)} |\widehat{F_{j,\beta}}\left(\frac{w}{q}  \right)|   \psi_j^{\Lambda}\left( \frac{\xi-a/q}{2^{j_{k}/10}} \right)\\
&\lesssim 2^{-j_k/(40)} \sum_{1\le q\le 2^{j_k/10}:\rm{lacunary}}q^{-1/4}\sum_{a\in\mathbb{Z}^{|\Lambda|}}\psi\left( \frac{\xi-a/q}{q^{-2}}\right)  \lesssim 2^{-j_k/(40)},
\end{align*}
where $j'=(j_1,\cdots,j_{k-1})$. This proves (\ref{b40}).  
\end{proof}

\begin{proof}[Proof of (\ref{23mm})]
The evenness condition (\ref{1009}) implies the evenness condition of (\ref{v16}).
Hence  by applying Lemma \ref{po3},  we obtain the uniform boundedness of the multiple Hilbert transforms:
 $$\sum_{j\in Z(k)}\left|\mathcal{H}^{\Lambda}_{j}\left( \xi-a/q\right)\right|\lesssim 1,$$
where $\xi-a/q=(\xi_{\mathfrak{m}}-a_{\mathfrak{m}}/q)_{\mathfrak{m}\in \Lambda}$.
Combining this with the bound $S^{\Lambda}(a/q)\lesssim q^{-\delta_{\Lambda}}$ for $(q,a)=1$, we   have
 \begin{align*}
 &  \sum_{ j\in Z(k)}\sum_{\frac{a}{q} \in  \mathbb{Q}^{\Lambda}[1,2^{j_k/10}]}\prod_{\mathfrak{m}\in \Lambda } \psi\left(  \frac{  \xi_{\mathfrak{m}}-a_{\mathfrak{m}}/q}{ 2^{-\mathfrak{m}\cdot j}2^{j_k/10}} \right)
 |S^{\Lambda}\left(\frac{a}{q}\right)\mathcal{H}^{\Lambda}_{j}\left(\xi-a/q\right)| \\
&\lesssim  \sum_{q:\text{lacunary}} q^{-\delta_{\Lambda}}\sum_{a\in \mathbb{Z}^{|\Lambda|}}\psi\left(\frac{\xi-\frac{a}{q}}{1/(10q^2)}\right)  
  \sum_{ j\in Z(k)}\left|\mathcal{H}^{\Lambda}_{j}\left(\xi-a/q\right)\right| \lesssim   1. \end{align*}
This yields the desired estimate.
\end{proof}

\begin{remark}\label{yt}
	The above result remains valid if, in the sum, the range
$a/q \in \mathbb{Q}^{\Lambda}[1,2^{j_k/10}]$
	is replaced by the larger range
$a/q \in \mathbb{Q}^{\Lambda}[1,2^{j_k/5}].$ More precisely,
\begin{itemize}
	\item[(1)]  Integeral Approximation.  For $j\in Z(k)$,
	one has
	\begin{align*}
		&H^{\Lambda}_j(\xi) \Psi^{\Lambda,\rm{major}}_{j,(2,1)}(\xi)\nonumber\\
		&\qquad=  \sum_{a/q\in \mathbb{Q}^{\Lambda}[1, 2^{j_k/5}]} \left(S^{\Lambda}(a/q) \mathcal{H}^{\Lambda}_j(\xi-a/q)  \psi_j^{\Lambda}\left( \frac{\xi-\frac{a}{q}}{2^{j_{k}/10}} \right)+E_{j,a/q}(\xi)\right), 
	\end{align*}
	where there exists $c>0$ independent of $\xi$ such that 
	\begin{align*}
		&\sum_{(j_1,\cdots,j_{k-1}):j\in Z(k)}\sum_{a/q\in \mathbb{Q}^{\Lambda}[1, 2^{j_k/5}]}|E_{j,a/q}(\xi)|\lesssim 2^{-cj_k}.
	\end{align*}
	\item[(2)]  Major Arc Estimate.
	There exists a constant $C >0$ independent of $\xi$:
	\begin{align*}
		&\sum_{j\in Z(k) }\left|   \sum_{a/q\in \mathbb{Q}^{\Lambda}[1,2^{j_k/5}]}S^{\Lambda}(a/q)\mathcal{H}_j^{\Lambda}(\xi-a/q) \psi_j^{\Lambda}\left( \frac{\xi-\frac{a}{q}}{2^{j_{k}/10}} \right) \right|
		\le C,
	\end{align*}
	which  implies 
	$
	\sum_{j\in Z(k)} |H^{\Lambda}_j(\xi) \Psi^{\Lambda,\rm{major}}_{j}(\xi)|\lesssim 1.
	$
\end{itemize}
\end{remark}

\begin{remark}
	One might expect that the support of the major arcs can be enlarged by replacing 
	$j_k/10$ with $j_\nu/10$, for some  $\nu<k$, in the cutoff
	\[
	\psi\!\left(
	\frac{\left(\xi_{\mathfrak{m}} - \frac{a_{\mathfrak{m}}}{q}\right)
		2^{j\cdot \mathfrak{m}}}{2^{j_k/10}}
	\right).
	\]
	However, under such a modification, the full integral representation in Proposition~\ref{lemm29} generally fails whenever
	$
	j_{\ell} \gg_{\Lambda} j_{\ell+1}  
	$ for some $\ell$ satisfying  $\nu\le \ell\le k-1$. We shall discuss this disparity in the next subsection.	
\end{remark}

 \subsection{Disparity Sectors: Balanced and Unbalanced Cases}\label{sec30}
Recall the cutoff function
\begin{align} \label{100s}
	\Psi^{\Lambda,\rm{major}}_{j}(\xi)&=\sum_{a/q\in \mathbb{Q}^{\Lambda}[1,2^{j_k/10}]} 
	\prod_{\mathfrak{m}\in\Lambda} \psi\left( 
	\frac{ \left(\xi_{\mathfrak{m}}-\frac{a_{\mathfrak{m}} }{q}\right)      2^{j\cdot \mathfrak{m}}    }{  2^{j_k/10}  }   \right).
\end{align} 
As   we finish the  major arc case  Proposition \ref{lemm29},
to prove  (\ref{se11}), it suffices to show 
\begin{align}\label{rr18}
	\sum_{(j_1,\cdots,j_{k-1}):j=(j_1,\cdots,j_k)\in Z(k)}|H^{\Lambda}_j(\xi)(1-\Psi^{\Lambda,\rm{major}}_{j}(\xi))|\lesssim 2^{-cj_k}\ \text{for some $c>0$}
\end{align} 
under the  condition (\ref{wep2}). 
On the support of this cutoff function, unless $j_{\nu-1}\approx_{\Lambda} j_\nu$ for all $\nu$,  the Weyl sum decay $2^{-\rho j_k}$   in (\ref{wop})  is not enough to sum over all $j$.
This   suggests us to first divide the region $Z(k)$ of \eqref{joor} 
 into the regions
defined as
\begin{align*}
Z_0(k)&:=\{ j\in Z(k):   j_{1}\approx_{\Lambda}  \cdots\approx_{\Lambda} j_k\}, \\
	Z_{\ell}(k)&:=\{j\in Z(k):   j_\ell\gg_{\Lambda} j_{\ell+1}\approx_{\Lambda}\cdots\approx_{\Lambda} j_k\}. 
\end{align*}
One can observe that
\begin{align*}
	Z(k)=\bigcup_{\ell=0}^{k-1}Z_{\ell}(k) 
\end{align*}
since   
\begin{align*}
	Z(k)&=  \{j\in Z(k): j_{k-1}\gg_{\Lambda} j_k\} \cup \{j\in Z(k): j_{k-1}\approx_{\Lambda} j_k\}\\
	&=Z_{k-1}(k)\cup  \{j\in Z(k): j_{k-1}\approx_{\Lambda} j_k\}\\
	&= Z_{k-1}(k)\cup Z_{k-2}(k)\cup  \{j\in Z(k): j_{k-2}\approx_{\Lambda} j_{k-1}\approx_{\Lambda} j_k\}=\cdots=\bigcup_{\ell=0}^{k-1}Z_{\ell}(k).
\end{align*}
Hence, to show (\ref{rr18}), it suffies to prove  under the  condition (\ref{wep2}) that  
\begin{align}\label{se01}
	\sum_{ j \in Z_{\ell}(k) } |H_j^{\Lambda}(\xi)(1-\Psi^{\Lambda,\rm{major}}_{j}(\xi))| \lesssim 1,
\end{align}
 for each $\ell=0,\cdots,k-1$.\\\\
\textbf{Balanced Cases.}
We  shall  prove (\ref{se01}) for the case $\ell=0$ where $  Z_{\ell=0}(k)$.  Recall
\begin{align*}
	H^{\Lambda}_j(\xi) &=\sum_{|t_1|\sim 2^{j_1},\cdots,|t_k|\sim 2^{j_k}} \frac{e^{2\pi i \sum_{\mathfrak{m}\in \Lambda} \xi_{\mathfrak{m}}t^{\mathfrak{m}} }}{t_1\cdots t_k},\\
	\mathcal{H}^{\Lambda}_j(\xi) &=\int_{|t_1|\sim 2^{j_1},\cdots,|t_k|\sim 2^{j_k}} \frac{e^{2\pi i \sum_{\mathfrak{m}\in \Lambda} \xi_{\mathfrak{m}}t^{\mathfrak{m}} }}{t_1\cdots t_k}dt_1\cdots dt_k.
\end{align*}
\begin{lemma}\label{lemm31}[Comparable case]
Suppose that $j \in Z_0(k)$.
Then, there exists $c'>0$ such that
\begin{align}\label{mm50}
    \bigl| H^{\Lambda}_j(\xi)\bigl(1 - \Psi^{\Lambda,\mathrm{major}}_j(\xi)\bigr) \bigr|
    \lesssim 2^{-c' j_1}.
\end{align}
Consequently, there exists $c>0$ such that 
\[
\sum_{\substack{(j_1,\dots,j_{k-1}): \\ j \in Z_0(k)}}
\bigl| H^{\Lambda}_j(\xi)\bigl(1 - \Psi^{\Lambda,\mathrm{major}}_j(\xi)\bigr) \bigr|
\lesssim 2^{-c j_k}.
\]
\end{lemma}

\begin{proof}
We first note from \eqref{t12} that
\[
\bigl| H^{\Lambda}_j(\xi)\bigl(1 - \Psi^{\Lambda,\mathrm{major}}_j(\xi)\bigr) \bigr|
\le
\bigl| H^{\Lambda}_j(\xi)\bigl(1 - \Psi^{\Lambda,\mathrm{major}}_{j,(1/2,1)}(\xi)\bigr) \bigr|.
\]
We then observe the comparability condition $j \in Z_0(k)$:   all components of $j$ are comparable up to constants depending on $\Lambda$. So we can utilize the Weyl sum estimate \eqref{weyl} after applying summation by parts  to obtain the bound
\[
\bigl| H^{\Lambda}_j(\xi)\bigl(1 - \Psi^{\Lambda,\mathrm{major}}_{j,(1/2,1)}(\xi)\bigr) \bigr|
\lesssim 2^{-\delta_{\Lambda} j_1}.
\]
This proves \eqref{mm50}. The  summation  follows  from the monotonicity $j_1 \ge \cdots \ge j_k$.
\end{proof}
\noindent\textbf{Unbalanced Cases.} Hence, to establish \eqref{rr18}, we prove the following  theorem throughout Sections 5-6.

\begin{theorem}\label{th41}
	Suppose that
	$\Lambda\subset\mathbb{Z}_+^k$
	satisfies condition~\eqref{wep2}.
	Then there exist constants $C,c>0$, depending only on $\Lambda$,
	such that the following statements hold.
	For every $\ell\in[k-1]$, every fixed
	$(j_{\ell+1},\ldots,j_k)\in\mathbb{Z}_+^{k-\ell}$, and every
	$\xi\in\mathbb{R}^{|\Lambda|}$, we have
	\begin{align}\label{rr1}
		\sum_{\substack{(j_1,\ldots,j_\ell)\in\mathbb{Z}_+^\ell ;\\
				j\in Z_\ell(k)}}
		\left|
		H_j^\Lambda(\xi)
		\left(
		1-\Psi_j^{\Lambda,\mathrm{major}}(\xi)
		\right)
		\right|
		\leq C2^{-c j_{\ell+1}},
	\end{align}
	where
	$j=(j_1,\ldots,j_k)$.
	Moreover, for every $h\in\mathbb{Z}_+$ and every
	$\xi\in\mathbb{R}^{|\Lambda|}$,
	\begin{align}\label{bb330}
		\sum_{\substack{j\in Z(k)\\ j_k\geq h}}
		\left|
		H_j^\Lambda(\xi)
		\left(
		1-\Psi_j^{\Lambda,\mathrm{major}}(\xi)
		\right)
		\right|
		\leq C2^{-ch}.
	\end{align}
\end{theorem}

\section{Two-parameter case: Proof of Theorem \ref{th41} for $k=2$ and   $\textbf{0}\notin \Lambda_1$}\label{sec4}
In this section, rather than directly treating a general multi-parameter polynomial mapping, we establish Theorem~\ref{th41} in the case $k=2$ with $\textbf{0}\notin \Lambda_1$.
The proof in this simplified setting illustrates the core ideas of the multi-parameter circle method and clarifies how it overcomes the principal difficulties described in Subsection \ref{sec340}.

\subsection{Layered Major Arcs and the Connector}\label{sec35}
Throughout this section,  write
\[
\Lambda_{1} = \{ m \in \mathbb{Z}_+ : (m,n) \in \Lambda \text{ for some } n\in\mathbb{Z}_{+} \},
\qquad
\Lambda(m) = \{\, n : (m,n) \in \Lambda \,\}.
\]
We view the two-variable  polynomial as a layered family of one-variable polynomials:
\begin{align}\label{t00}
	\sum_{(m,n)\in \Lambda}\xi_{mn}t_1^m t_2^n
	= \sum_{m\in \Lambda_1}  \xi_{m}(t_2)t_1^m,
	\qquad 
	\text{where } 
	\xi_m(t_2) := \sum_{n\in \Lambda(m)} \xi_{mn}t_2^n.
\end{align}
For a fixed $t_2$, set $\xi(t_2) = (\xi_m(t_2))_{m\in \Lambda_1}$ and write
\[
H^{\Lambda_1}_{j_1}(\xi(t_2))
  = \sum_{|t_1|\sim 2^{j_1}}
    \frac{e^{2\pi i \sum_{m\in \Lambda_1}\xi_m(t_2)t_1^m}}{t_1}
\ \text{and} \ 
H^{\Lambda}_j(\xi)
  = \sum_{|t_2|\sim 2^{j_2}}\frac{H^{\Lambda_1}_{j_1}(\xi(t_2))}{t_2}.
\]

We now define the \emph{initial stage major arcs} (layered arc) at each level $t_2$ and the \emph{final stage major arcs} (original arc) as the supports of the following cutoff functions:
\begin{align}
\Psi_{j}^{\Lambda_1,\mathrm{major}}(\xi(t_2))
  &:= \sum_{b/p\in \mathbb{Q}^{\Lambda_1}[1,2^{j_2/10}]}
      \prod_{m\in \Lambda_1}
      \psi\!\left(
        \frac{\xi_m(t_2)-\frac{b_m}{p}}
             {2^{-j_1 m}\,2^{j_2/100}}
      \right), \label{2939}\\[4pt]
\Psi_j^{\Lambda,\mathrm{major}}(\xi)
  &:= \sum_{a/q\in \mathbb{Q}^{\Lambda}[1,2^{j_2/10}]}
      \prod_{(m,n)\in \Lambda}
      \psi\!\left(
        \frac{\xi_{mn}-\frac{a_{mn}}{q}}
             {2^{-j\cdot (m,n)}\,2^{j_2/10}}
      \right). \label{plm}
\end{align}

To transition from 
$\Psi_j^{\Lambda_1,\mathrm{major}}(\xi(t_2))$
to 
$\Psi_j^{\Lambda,\mathrm{major}}(\xi)$, 
we introduce  an indispensable auxiliary pair of cutoff functions depending only on $j_2$,
\begin{align}\label{ccnn}
L^{\Lambda_1,\mathrm{major}}_{j_2}(\xi(t_2)) &:=\sum_{b/p\in \mathbb{Q}^{\Lambda_1}[1,2^{j_2/10}]  }  \psi\left(\frac{\xi(t_2)-b/p}{2^{- 10 Kj_2}}\right),\\
L^{\Lambda,\mathrm{major}}_{j_2}(\xi)&
  := \sum_{a/q\in \mathbb{Q}^{\Lambda}[1,2^{N^4 j_2}]}
     \psi\!\left(
       \frac{\xi - a/q}{2^{- 10K j_2 + N^4 j_2}}
     \right),\nonumber
\end{align} 
whose supports are enlarged from those of (\ref{2939}) and (\ref{plm}) respectively, as it holds  
$$2^{-j_1m}2^{j_2/10}\ll  2^{- 10 Kj_2} \ \text{and}\  2^{-j\cdot (m,n)}2^{j_2/10}\ll  2^{- 10Kj_2+N^4j_2}$$
due to $j_1\gg_{\Lambda}j_2$ and $m\ne 0$.
 Here the constants $N$ and $K$ are defined in 
(\ref{00gg}).
Thus, we refer to $L^{\Lambda_1,\mathrm{major}}_{j_2}(\xi(t_2))$ and $L^{\Lambda,\mathrm{major}}_{j_2}(\xi), $   as the \emph{coarse major arc cutoffs}.  They provide the essential bridge from the layered major arcs to the original major arcs.

To show the case $k=2$ of (\ref{rr1}) and $0\notin \Lambda_1$  that is $\xi_0(t_2)\equiv 0$ in (\ref{t00}), 
 we shall show  the following three estimates: for some $c>0$,
\begin{align}
&\sum_{j_1 : j \in Z_1(2)} 
  \sum_{|t_2|\sim 2^{j_2}}
  \frac{\bigl|H^{\Lambda_1}_{j_1}(\xi(t_2))\bigr|}{2^{j_2}}
  \bigl(1 - \Psi_{j}^{\Lambda_1,\mathrm{major}}(\xi(t_2))\bigr)\lesssim 2^{-c j_2}, 
  \label{283a} \\
&\sum_{j_1 : j \in Z_1(2)} 
  \sum_{|t_2|\sim 2^{j_2}}
  \frac{\bigl|H^{\Lambda_1}_{j_1}(\xi(t_2))\bigr|}{2^{j_2}}
  \Psi_{j}^{\Lambda_1,\mathrm{major}}\!\left(\xi(t_2)\right)
  \bigl(1 - L^{\Lambda,\mathrm{major}}_{j_2}(\xi)\bigr)\lesssim 2^{-j_2},\label{283b}
\end{align}
\begin{equation}
	\sum_{j_1 : j \in Z_1(2)} 
	\sum_{|t_2|\sim 2^{j_2}}
	\frac{\bigl|H^{\Lambda_1}_{j_1}(\xi(t_2))\bigr|}{2^{j_2}} 
	\Bigl|\,
	\Psi_j^{\Lambda_1,\mathrm{major}}\!\left(\xi(t_2)\right)
	L^{\Lambda,\mathrm{major}}_{j_2}(\xi)
	- \Psi_j^{\Lambda,\mathrm{major}}(\xi)
	\Bigr|\lesssim 2^{-cj_2}.\label{283c}
\end{equation}

 We estimate these three terms according to the following observations:
\begin{itemize}
\item[\eqref{283a}] \textbf{One-Parameter Estimate}. The $t_1$-sum admits $j_2$-decay  for each fixed $t_2$.
\item[\eqref{283b}] \textbf{Intermediate-Arc Estimate}. The major arc rigidity.
\item[\eqref{283c}] \textbf{Merging-Major-Arc-Estimate}. Two major-arc layers merge into the final.
\end{itemize}

\subsection{One-Parameter Estimate: Proof of (\ref{283a})}\label{Sec42v} In this subsection, we treat the sum in   (\ref{283a}) which seems the one-parameter estimate.
Here we decompose the complement of the layered major arcs into the layered minor in the scale $j_1$ and intermediate regions by splitting
$$  1-\Psi^{\Lambda_1,\rm{major}}_{j}(\xi(t_2))= \Psi^{\Lambda_1,\rm{minor}}_{j}(\xi(t_2)) +
\Psi^{\Lambda_1,\rm{mediate}}_{j}(\xi(t_2))  $$ 
where we define
\begin{align*} 
&\Psi^{\Lambda_1,\rm{minor}}_{j}(\xi(t_2)) :=1 -\sum_{b/p\in  \mathbb{Q}^{\Lambda_1}[1,2^{j_1/10}]}\psi^{\Lambda_1}_{j_1}\left(\frac{\xi(t_2)-\frac{b}{p}}{2^{j_1/10}}\right), \nonumber \\
&\Psi^{\Lambda_1,\rm{mediate}}_{j}(\xi(t_2)):=\sum_{b/p\in  \mathbb{Q}^{\Lambda_1}[2^{j_2/10},2^{j_1/10}]}\psi^{\Lambda_1}_{j_1}\left(\frac{\xi(t_2)-\frac{b}{p}}{2^{j_1/10}}\right) \nonumber \\
&\qquad\quad\qquad\qquad\quad+
\sum_{b/p\in  \mathbb{Q}^{\Lambda_1}[1,2^{j_2/10}]}
 \left( \psi^{\Lambda_1}_{j_1}\left(\frac{\xi(t_2)-\frac{b}{p}}{2^{j_1/10}}\right)-\psi^{\Lambda_1}_{j_1}\left(\frac{\xi(t_2)-\frac{b}{p}}{2^{j_2/100}}\right) \right).
    \end{align*}
The minor arc part $\Psi^{\Lambda_1,\rm{minor}}_{j}(\xi(t_2))$, for each sliced coefficient $\xi(t_2)$, is controlled by a one-parameter Weyl sum estimate, while the intermediate arc part $\Psi^{\Lambda_1,\rm{mediate}}_{j}(\xi(t_2))$ is handled using an approximation by Gauss sums combined with oscillatory integral bounds with respect to the single variable $t_1$.

We shall obtain (\ref{283a}) by proving Lemmas \ref{prop130} and \ref{prop41} below.
\begin{lemma}[Estimate for $\Psi^{\Lambda_1,\rm{minor}}_{j_1}(\xi(t_2))$]\label{prop130}
Fix $|t_2|\sim 2^{j_2}$.
Then, there exists a constant   $c>0$  independent of $j$ and $t_2$ such that
\begin{align}\label{pk1}
|H^{\Lambda_1}_{j_1}( \xi(t_2)) | \Psi^{\Lambda_1,\rm{minor}}_{j}(\xi(t_2))  \lesssim 2^{-c j_1}. 
\end{align}
\end{lemma}
\begin{proof}
Fix $t_2$. Then we can obtain (\ref{pk1}) by the same argument as in the proof of Lemma~\ref{lemm31} for $k=1$, with $\xi$ replaced by $\xi(t_2)$. \end{proof}

\begin{lemma}[Estimate of $\Psi^{\Lambda_1,\rm{mediate}}_{j}(\xi(t_2))$]\label{prop41}
 Fix $|t_2|\sim 2^{j_2}.$ Then it holds that
\begin{align}\label{32}
&\sum_{j_1:(j_1,j_2)\in Z_1(2)}|H^{\Lambda_1}_{j_1}(\xi(t_2))|\Psi^{\Lambda_1,\rm{mediate}}_{j}(\xi(t_2)) \lesssim 2^{-j_2/(200d)}.
\end{align}
\end{lemma}

\begin{proof}
As we defined the Gauss sum in Lemma \ref{el22}, for $k=1$,  
\[
S^{\Lambda_1}\!\left(\frac{b}{p}\right)
   = \frac{1}{p}\sum_{1\le \ell_1\le p}
      \exp\!\left(2\pi i \sum_{m\in\Lambda_1}
      \frac{b_m}{p}\,\ell_1^m\right).
\]
Replacing $\xi$ in \eqref{b40} by the modified frequency $\xi(t_2)$, we obtain, for every 
\[
\frac{b}{p}=\frac{(b_m)}{p} \in \mathbb{Q}^{\Lambda_1}[1,2^{j_1/10}],
\]
the approximation
\[
\begin{aligned}
H_{j_1}^{\Lambda_1}(\xi(t_2))\,
\psi^{\Lambda_1}_{j_1}\!\left(\frac{\xi(t_2)-\frac{b}{p}}{2^{j_1/10}}\right)
&= 
S^{\Lambda_1}\!\left(\frac{b}{p}\right)\,
\mathcal{H}^{\Lambda_1}_{j_1}\!\bigl(\xi(t_2)-b/p\bigr)\,
\psi^{\Lambda_1}_{j_1}\!\left(\frac{\xi(t_2)-\frac{b}{p}}{2^{j_1/10}}\right)  \\
& 
+\, E_{j}(b/p,\xi).
\end{aligned}
\]
Here, the error term satisfies the uniform estimate
\[
\sum_{b/p\in \mathbb{Q}^{\Lambda_1}[1,2^{j_1/10}]}
   |E_j(b/p,\xi)|
   = O(2^{- j_1/(40)}).
\]
Applying the above approximation to the LHS of \eqref{32},  
we split
\[
\text{LHS of \eqref{32}}
   \ \le\ 
   \sup_{|t_2|\sim 2^{j_2}}
   \sum_{\substack{j_1; j_1\gg_{\Lambda} j_2}}
   \Bigl(
      A_j(\xi(t_2)) + B_j(\xi(t_2))
   \Bigr).
\]
Here, we   set
\[
\begin{aligned}
A_j(\xi(t_2))
&:=
\sum_{b/p\in \mathbb{Q}^{\Lambda_1}[\,2^{j_2/10},\,2^{j_1/10}\,]}
\psi^{\Lambda_1}_{j_1}\!\left(
   \frac{\xi(t_2)-\frac{b}{p}}{2^{j_1/10}}
\right)
\bigl|
   S^{\Lambda_1}(b/p)\,
   \mathcal{H}^{\Lambda_1}_{j_1}(\xi(t_2)-b/p)
\bigr|,
\\[1em]
B_j(\xi(t_2))
&:=
\sum_{b/p\in \mathbb{Q}^{\Lambda_1}[\,1,\,2^{j_2/10}\,]}
\psi^{\Lambda_1}_{j_1}\!\left(
   \frac{\xi(t_2)-\frac{b}{p}}{2^{j_1/10}}
\right)
\bigl|
   S^{\Lambda_1}(b/p)\,
   \mathcal{H}^{\Lambda_1}_{j_1}(\xi(t_2)-b/p)
\bigr|
\\
&\qquad\qquad \qquad\qquad\qquad\qquad\qquad\qquad\qquad\times
\left(
   1-
   \psi^{\Lambda_1}_{j_1}\!\left(
      \frac{\xi(t_2)-\frac{b}{p}}{2^{j_2/100}}
   \right)
\right).
\end{aligned}
\]
By applying the van der Corput lemma and using Lemma~\ref{el22} in the case $k=1$, there exists $c'>0$ such that  
\begin{align*}
\bigl|\mathcal{H}^{\Lambda_1}_{j_1}(\xi(t_2)-b/p)\bigr|
   &\lesssim 
   \min\!\left\{
      1,\ 
      \bigl|
         \bigl(
            (\xi_m(t_2)-b_m/p)\,2^{j_1 m}
         \bigr)_{m\in\Lambda_1}
      \bigr|^{-1/d}
   \right\},
\\[4pt]
S^{\Lambda_1}(b/p)
   &\lesssim p^{-c'}.
\end{align*}
 Observe that
\begin{enumerate}
\item[(i)] If $A_j(\xi(t_2))\neq 0$, then necessarily $p\ge 2^{j_2/10}$,
\item[(ii)] If $B_j(\xi(t_2))\neq 0$, then for some $m\in\Lambda_1$,
\[
\left|\xi_m(t_2)-\frac{b_m}{p}\right| 2^{j_1 m} \ge 2^{j_2/100}.
\]
\end{enumerate}
Therefore, in every summand appearing in $A_j(\xi(t_2))$ or $B_j(\xi(t_2))$, we obtain the uniform bound
\begin{equation}\label{sf2}
\left|
   S^{\Lambda_1}(b/p)\,
   \mathcal{H}^{\Lambda_1}_{j_1}(\xi(t_2)-b/p)
\right|
\ \lesssim\
\bigl|\mathcal{H}^{\Lambda_1}_{j_1}(\xi(t_2)-b/p)\bigr|^{1/2}\,
p^{-c}\,
2^{-cj_2}.
\end{equation}
In addition, the one-parameter continuous estimate implies that there exists a constant $C>0$, independent of $t_2\sim 2^{j_2}$, such that
\begin{equation}\label{sf1}
\sum_{j_1\ge 0}
\bigl|\mathcal{H}^{\Lambda_1}_{j_1}(\xi(t_2)-b/p)\bigr|^{1/2}
\ \le\ C.
\end{equation}

Combining \eqref{sf2} and \eqref{sf1}, we obtain that
\begin{align}\label{yu}
&\sum_{\substack{j_1:\\ j_1\ge j_2}}
\bigl(
   A_j(\xi(t_2)) + B_j(\xi(t_2))
\bigr) \\
&\lesssim\
\sum_{\substack{p:\\ \text{lacunary}}}
p^{-c}\,
2^{-cj_2}
\sum_{\substack{b:\ (b,p)=1}}
\psi\!\left(
   \frac{\xi(t_2)-b/p}{1/(10p^{2})}
\right)
\lesssim\
2^{-cj_2},\notag
\end{align}
where the implicit constants in $\lesssim$ are independent of $t_2$ and $\xi$.
Thus, we obtain the desired estimate in \eqref{32}. 
 \end{proof}

\subsection{Major Arc Rigidity: Proof of (\ref{283b})}\label{sec5.3}
Recall the cutoff functions
\begin{align*}
	\Psi_j^{\Lambda_1,\mathrm{major}}(\xi(t_2))
	&= \sum_{b/p\in \mathbb{Q}^{\Lambda_1}[1,2^{j_2/10}]}
	\prod_{m\in \Lambda_1}
	\psi\!\left(
	\frac{\xi_m(t_2)-\frac{b_m}{p}}
	{2^{-j_1 m}\,2^{j_2/100}}
	\right),\\
	L^{\Lambda_1,\mathrm{major}}_{j_2}(\xi(t_2)) &=\sum_{b/p\in \mathbb{Q}^{\Lambda_1}[1,2^{j_2/10}]  }  \psi\left(\frac{\xi(t_2)-b/p}{2^{-10Kj_2}}\right).
\end{align*}
The estimate on the support of
\[
\Psi_{j}^{\Lambda_1,\mathrm{major}}\!\bigl(\xi(t_2)\bigr)
\bigl(1-L^{\Lambda,\mathrm{major}}_{j_2}(\xi)\bigr)
\]
is the most challenging part of our proof. It corresponds to the situation in which the sliced frequency $\xi(t_2)$ lies in a layered major arc, while the original frequency $\xi$ appears to remain in the minor arc, where the averaged Gauss-sum estimate fails, as described in Obstacle~(ii) of Section~\ref{sec34}. To overcome this difficulty, we show that if the original frequency $\xi$ is contained in sufficiently many plates associated with the sliced frequencies $\xi(t_2)$, as enforced by the constraints of
$\Psi_{j}^{\Lambda_1,\mathrm{major}}\!\bigl(\xi(t_2)\bigr)$
for many distinct values of $t_2$, then $\xi$ must in fact belong to the coarse major arc determined by
$L^{\Lambda,\mathrm{major}}_{j_2}(\xi)$. We refer to this phenomenon as \emph{major-arc rigidity}.
This geometric observation provides the essential bridge from the layered major arcs to the original major arcs and is the key ingredient of our multi-parameter circle method.

 As we assumed that $m \ne 0$ for every $(m,n)\in \Lambda$ at the beginning of this Section, 
\[
2^{-j_1 m} 2^{j_2/10} \ll 2^{-10 K j_2}
\quad \text{for all}\  j \in Z_1(2)\ \text{where}\ j_1\gg_{\Lambda} j_2\ \text{in (\ref{00gg})}.
\]
Hence, one can deduce that
\begin{align}\label{s33}
\Psi_j^{\Lambda_1,\mathrm{major}}(\xi(t_2))
  \le
  L^{\Lambda_1,\mathrm{major}}_{j_2}(\xi(t_2)).
\end{align}
  On the other hand, by applying the one parameter result, one can assert that there exists a constant $C$   independent of $t_2$ and $\xi$ such that
\[
\sum_{j_1} \bigl|H^{\Lambda_1}_{j_1}(\xi(t_2))\bigr| \le C.
\]
Therefore, combining this estimate with \eqref{s33}, we obtain, for each fixed
$\xi$ and $j_2$,  
\begin{align*} 
&\sum_{j_1 : j \in Z_1(2)} 
  \sum_{|t_2|\sim 2^{j_2}}
  \frac{\bigl|H^{\Lambda_1}_{j_1}(\xi(t_2))\bigr|}{2^{j_2}}
  \Psi_j^{\Lambda_1,\mathrm{major}}\!\left(\xi(t_2)\right)
  \bigl(1 - L^{\Lambda,\mathrm{major}}_{j_2}(\xi)\bigr) \nonumber\\
&\qquad\quad  \;\lesssim\; 
  \frac{1}{2^{j_2}}\sum_{|t_2|\sim 2^{j_2}}
   L^{\Lambda_1,\mathrm{major}}_{j_2}(\xi(t_2))
(1-L^{\Lambda,\rm{major}}_{j_2}(\xi)). 
\end{align*}
Hence,  to prove  (\ref{283b}), we set $L^{\Lambda,\rm{minor}}_{j_2}(\xi):=1-L^{\Lambda,\rm{major}}_{j_2}(\xi)$.  Then, it remains to prove that
\begin{align} 
  \frac{1}{2^{j_2}}\sum_{|t_2|\sim 2^{j_2}}
   L^{\Lambda_1,\mathrm{major}}_{j_2}(\xi(t_2))
L^{\Lambda,\rm{minor}}_{j_2}(\xi)\le 2^{-j_2}.\label{4tb}
\end{align}
Before estimating \eqref{4tb}, we first show the following claim.
\begin{claim}\label{cl31}
Fix $\xi$ and  $|t_2|\sim 2^{j_2}.$  Then, the rational vector $\frac{b}{p} \in \mathbb{Q}^{\Lambda_1}[1,2^{j_2/10}]$
satisfying
\[
\psi\!\left(\frac{\xi(t_2)-b/p}{2^{-10K j_2}}\right) \neq 0
\]
in the summation over $b/p$  appearing in $L^{\Lambda_1,\mathrm{major}}_{j_2}(\xi(t_2))$ in (\ref{ccnn})  is unique.  
Moreover, this rational  vector can be written as $\frac{b(t_2)}{p(t_2)}$, and consequently
\begin{align}\label{i44}
L^{\Lambda_1,\mathrm{major}}_{j_2}(\xi(t_2))
  = 
  \psi\!\left(
    \frac{\xi(t_2)-\frac{b(t_2)}{p(t_2)}}{2^{-10K j_2}}
  \right),
  \qquad 
  \frac{b(t_2)}{p(t_2)} \in \mathbb{Q}^{\Lambda_1}[1,2^{j_2/10}].
\end{align}
\end{claim}

\begin{proof}[Proof of Claim \ref{cl31}]
  Assume that there exist two distinct rational vectors $b/p \neq b'/p'$ in $\mathbb{Q}^{\Lambda_1}[1,2^{j_2/10}] $ such that
\[
\psi\!\left(\frac{\xi(t_2)-b/p}{2^{-10K j_2}}\right)\neq 0
\quad \text{and} \quad 
\psi\!\left(\frac{\xi(t_2)-b'/p'}{2^{-10K j_2}}\right)\neq 0.
\]
Since the two rational vectors are distinct and  $p,p'\le 2^{j_2/10}$, we know that
\[
|b'/p' - b/p| \;\ge\; \frac{1}{pp'} \;\ge\; 2^{-j_2/5}.
\]
As both rational vectors lie within the support of the cut-off $\psi$, we  have
\[
|b'/p' - b/p|
\;\le\; |b'/p' - \xi(t_2)| + |\xi(t_2)-b/p|
\;<\; 2^{-10K j_2 + 2}.
\]
But for large $K$, the inequality $2^{-j_2/5} \gg 2^{-10K j_2 + 2}$ holds, which yields a contradiction. 
Therefore, such a rational $b/p$ is unique. We denote it by $b(\xi,t_2)/p(\xi,t_2)$, or simply by $b(t_2)/p(t_2)$ when $\xi$ is fixed. 
\end{proof} 
Hence,  by applying  (\ref{i44})  in (\ref{4tb}),
we shall show that  for each $\xi$ and $j_2$ 
\begin{align} \label{416ss}
  \frac{1}{2^{j_2}}
\sum_{|t_2|\sim 2^{j_2}}
\psi\!\left(
\frac{\xi(t_2)-\frac{b(t_2)}{p(t_2)}}{2^{-10Kj_2}}
\right)
L^{\Lambda,\mathrm{minor}}_{j_2}(\xi)
\lesssim 2^{-j_2}.
\end{align}

Observe that the desired estimate follows immediately whenever
\[
\sum_{|t_2|\sim 2^{j_2}}
\psi\!\left(
\frac{\xi(t_2)-\frac{b(t_2)}{p(t_2)}}{2^{-10Kj_2}}
\right)
<2|\Lambda|.
\]
   Therefore,  to prove \eqref{416ss}, it remains to consider the complementary case. In this case, the following proposition shows that the factor
$1-L^{\Lambda,\mathrm{major}}_{j_2}(\xi) =L^{\Lambda,\mathrm{minor}}_{j_2}(\xi) $ vanishes.

\begin{proposition}\label{prop6100}
Let $\xi=(\xi_{mn})_{(m,n)\in\Lambda}\in\mathbb R^{|\Lambda|}$ and
$j_2\in\mathbb Z_+$. Suppose that
\begin{equation}\label{420s}
\sum_{|t_2|\sim 2^{j_2}}
\psi\!\left(
\frac{\xi(t_2)-\frac{b(t_2)}{p(t_2)}}{2^{-10Kj_2}}
\right)
\ge 2|\Lambda|
\ \text{with} \
\frac{b(t_2)}{p(t_2)}
\in\mathbb Q^{\Lambda_1}[1,2^{j_2/10}].
\end{equation}
Then there exists
\[
\frac{a}{q}\in\mathbb Q^\Lambda[1,2^{N^4j_2}]
\]
such that
\begin{equation}\label{pkn}
\psi\!\left(
\frac{\xi-\frac{a}{q}}{2^{-10Kj_2+N^4j_2}}
\right)=1.
\end{equation}
Consequently,
\[
1-L^{\Lambda,\mathrm{major}}_{j_2}(\xi_\Lambda)=0,
\]
where
\[
L^{\Lambda,\mathrm{major}}_{j_2}(\xi_\Lambda)
=
\sum_{a/q\in\mathbb Q^\Lambda[1,2^{N^4j_2}]}
\psi\!\left(
\frac{\xi-\frac{a}{q}}{2^{-10Kj_2+N^4j_2}}
\right).
\]
\end{proposition}
 
\begin{proof}[Proof of Proposition \ref{prop6100}] 
Recall the definitions of the constants  $K$ and $N$, which depend only on $\Lambda$:
\begin{align*}
	K=N^{10k}\ \text{and}\ N=\lvert\Lambda\rvert!\sum_{\mathfrak{m}\in \Lambda} (|\mathfrak{m}|+k).
\end{align*} Fix $\xi$. Without loss of generality, we assume that 
there exist at least $|\Lambda|$ distinct  $t_2 \in (2^{j_2-1},2^{j_2}]\cap \mathbb{Z}$ such that
\begin{align} \label{390}
 \left|\xi_m(t_2)-\frac{b_m(t_2)}{p(t_2)}\right|\le 2^{-10Kj_2} \ \text{for all $m\in \Lambda_1$}.
\end{align}
In other words, the values $t_2=s_1,\cdots, s_{|\Lambda|} \in (2^{j_2-1},2^{j_2}]\cap \mathbb{Z}$ where 
   $s_1<\cdots<s_{|\Lambda|}$ satisfy the condition (\ref{390}).
 For each $m\in \Lambda_1$, let  $$\Lambda(m)=\{n: (m,n)\in \Lambda\}=
\{n_{1},\dots,n_{r}\} \quad \text{with } n_{1}<\cdots<n_{r}
.$$ Here $r=|\Lambda(m)|\le |\Lambda|$ and $n_i=n_i(m)$ may depend on $m$, but we omit the dependence  
  for brevity.  Then for each $m$, the  restrictions of (\ref{390}) can be written as 
\begin{align*}
s_1^{n_1}\xi_{mn_1}+\cdots+s_1^{n_{r}}\xi_{mn_{r}}&=\frac{b_m(s_1)}{p(s_1)}+O(2^{-10Kj_2})\\
&\vdots\\
s_{r}^{n_1}\xi_{mn_1}+\cdots+s_{r}^{n_{r}}\xi_{mn_{r}}&=\frac{b_m(s_{r})}{p(s_{r})}+O(2^{-10Kj_2}).
\end{align*}
We rewrite  this linear system  for each $m\in \Lambda_1$  in the matrix form as
\begin{align}\label{meg1}
\left(\begin{matrix}s_1^{n_{1}}&s_1^{n_{2}}& \cdots&s_1^{n_{r}}\\
s_2^{n_{1}}&s_2^{n_{2}}& \cdots&s_2^{n_{r}}\\
\vdots& &\vdots\\
s_{r}^{n_{1}}&s_{r}^{n_{2}}& \cdots&s_{r}^{n_{r}}
  \end{matrix}\right) \left(\begin{matrix}\xi_{mn_1}\\
\xi_{mn_2}\\
 \vdots\\
\xi_{mn_r}
  \end{matrix}\right)=\left(\begin{matrix}\frac{b_m(s_1)}{p(s_1)}\\
\frac{b_m(s_2)}{p(s_2)}\\
 \vdots\\
\frac{b_m(s_r)}{p(s_r)}
  \end{matrix}\right)+O(2^{-10Kj_2}).
\end{align}
Here,  $O(2^{-10Kj_2})$ means that the error in each component is  most $2^{-10Kj_2}$.

\begin{ob}\label{lem7100} 
Let $s_1,\cdots,s_r\in (2^{j_2-1},2^{j_2}]\cap \mathbb{Z}$ such that $s_1<\cdots<s_r$ and let $\Lambda(m)=\{n_i\}_{i=1}^r$ where $0\le n_1<\cdots<n_r\le N$ with $r=|\Lambda(m)|$  and $m\in \Lambda_1$ as above. Then
  the $r\times r$ generalized Vandermonde matrix defined as
  $$S_m=\left(\begin{matrix}s_1^{n_{1}}&s_1^{n_{2}}& \cdots&s_1^{n_{r}}\\
s_2^{n_{1}}&s_2^{n_{2}}& \cdots&s_2^{n_{r}}\\
\vdots& &\vdots\\
s_{r}^{n_{1}}&s_{r}^{n_{2}}& \cdots&s_{r}^{n_{r}}
  \end{matrix}\right)   $$  has its inverse  $S^{-1}_m=\left(\frac{c_{ij}(s)}{\det(S_m)}\right)$ where $ \left\|S_m^{-1}\right\|\le N^{2} 2^{j_2N^2}$ and $1\le   \det(S_m)\le N 2^{j_2N^2}.$
\end{ob}
\begin{proof}[Proof of Observation \ref{lem7100}] To claim its invertibility, 
assume the contrary so that $\det (S_m)=0$. Then $r$ columns $\left(\begin{matrix}s_1^{n_{\ell}} \\
s_{2}^{n_{\ell}} \\
\vdots \\ 
s_{r}^{n_{\ell}} 
  \end{matrix}\right) $ are linearly dependent.  Thus for some $(c_1,\cdots,c_r)\ne {\bf 0}$, the 1D polynomial function $f(x):= c_1x^{n_1}+\cdots+c_r x^{n_r}=x^{n_1}(c_1+c_2x^{n_2-n_1}+\cdots+c_rx^{n_r-n_1})$ with $0\le n_1<\cdots<n_r$ has $r$ distict zeros $2^{j_2-1}\le s_1<\cdots<s_r\le 2^{j_2}$.  This leads a contradiction via repeated applications of rolle's theorem. Thus $\det(S_m)\ne 0 $.  
Since the   entries are integers,  it follows that  $|\det(S_m)|\ge 1$.  The $(i,j)$-entry  of $S^{-1}_m$ is  given by
$ (S_m^{-1})_{ij}=\frac{c_{ij}(s)}{\det(S_m)}, $ where   $c_{ij}(s):=(-1)^{i+j}\det((S_m)^{*}_{ji}) $ and $(S_m)^{*}_{ij}$ denotes the minor matrix  of $S_m$ obtained by deleting the $i$-th row and $j$-column. Hence, $$|c_{ij}(s)|\le\max_{s_1,\cdots,s_r} r!|(s_1  s_2 \cdots s_r)^{n_{r}}|\le r!2^{j_2 rN}\le N2^{j_2N^2},$$ since $0\le  r!, n_r\le N$, and each $s_i$ satisfies $s_i \sim 2^{j_2}$. This   in turn implies 
 $\|S_m^{-1}\|\le \|S_m^{-1}\|_{HS}\le    N^2 2^{j_2N^2}$ where $\| \cdot \|_{HS}$ denotes the Hilbert-Schmidt norm.
\end{proof}

With Observation \ref{lem7100} in hand,
we now continue the proof of (\ref{pkn}).
Multiplying both sides of (\ref{meg1}) by $S_m^{-1}$ and using the bound 
$\|S_m^{-1}\| \le N^2 2^{j_2 N^2}$ from Observation~\ref{lem7100}, we obtain  that
\begin{align*} 
\begin{pmatrix}
\xi_{mn_1}\\
\xi_{mn_2}\\
\vdots\\
\xi_{mn_r}
\end{pmatrix}
&=
\begin{pmatrix}
s_1^{n_1} & s_1^{n_2} & \cdots & s_1^{n_r} \\
s_2^{n_1} & s_2^{n_2} & \cdots & s_2^{n_r} \\
\vdots & \vdots & & \vdots \\
s_r^{n_1} & s_r^{n_2} & \cdots & s_r^{n_r}
\end{pmatrix}^{-1}
\begin{pmatrix}
\dfrac{b_m(s_1)}{p(s_1)}\\[4pt]
\dfrac{b_m(s_2)}{p(s_2)}\\[4pt]
\vdots\\
\dfrac{b_m(s_r)}{p(s_r)}
\end{pmatrix}
+ O\!\left(N^2 2^{j_2 N^2} 2^{-10Kj_2}\right).
\end{align*}
Here the error term arises from applying $S_m^{-1}$ to the error vector in (\ref{meg1}), 
so that its magnitude is amplified by at most $\|S_m^{-1}\| \le N^2 2^{j_2 N^2}$.
  Using the entries of the inverse matrix 
$S_m^{-1}=\left(\dfrac{c_{ij}(s)}{\det(S_m)}\right)$, 
  compute the matrix multiplication on the right-hand side as
\begin{align*}
\begin{pmatrix}
\xi_{mn_1}\\
\xi_{mn_2}\\
\vdots\\
\xi_{mn_r}
\end{pmatrix}
&=
\left(
  \sum_{j=1}^r 
  \frac{c_{ij}(s)}{\det(S_m)} 
  \frac{b_m(s_j)}{p(s_j)}
\right)_{i=1}^r
+ O\!\left(N^2 2^{j_2 N^2} 2^{-10Kj_2}\right).
\end{align*}
Since  
$1\le |\det(S_m)| \le N 2^{j_2 N^2}$ 
and 
$1 \le p(s_j) \le 2^{j_2/10}$ 
for each $j$, 
the common multiple $q_m$ of the denominators 
$|\det(S_m)|\,p(s_j)$ for $j=1,\dots,r$ 
satisfies
\begin{align}\label{gh11}
1\le 
q_m \le 2^{j_2 N^3}.
\end{align}
Hence we obtain a rational vector 
$(h_{mn_1}, \dots, h_{mn_r})/q_m 
\in \mathbb{Q}^{\Lambda(m)}[1, 2^{j_2 N^3}]$ 
such that
\[
(\xi_{mn})_{n \in \Lambda(m)} =
\begin{pmatrix}
\xi_{mn_1}\\
\xi_{mn_2}\\
\vdots\\
\xi_{mn_r}
\end{pmatrix}
=
\left(\begin{matrix}\frac{h_{mn_1}}{q_m}\\
\frac{h_{mn_2}}{q_m}\\
 \vdots\\
\frac{h_{mn_{r}}}{q_m}
  \end{matrix}\right)
+ O\!\left(N^2 2^{j_2 N^2} 2^{-10Kj_2}\right).
\]
  Combining the above approximations for  all $m \in \Lambda_1$, we obtain that for $j_2\ge 1$
\begin{align}\label{42kp}
\left|
(\xi_{mn})_{(m,n)\in \Lambda}
- 
\left(\frac{h_{mn}}{q_m}\right)_{(m,n)\in \Lambda}
\right|
\le |\Lambda_1|
 N^3 2^{j_2 N^2} 2^{-10Kj_2}
\le 
2^{-2}2^{j_2 N^4} 2^{-10Kj_2}.
\end{align}
Next, assembling all vectors  with 
$1 \le q_m \le 2^{N^3 j_2}$ for every $m \in \Lambda_1$, 
we define
\begin{align}\label{yy}
\frac{a}{q}
=
\left( \frac{a_{mn}}{q} \right)_{(m,n)\in \Lambda}
:=
\left( \frac{h_{mn}}{q_m} \right)_{(m,n)\in \Lambda}.
\end{align} Then, we obtain that  \begin{align}\label{ep1}
\frac{a}{q} \in \mathbb{Q}^{\Lambda}\big([1, 2^{N^4 j_2}]\big),
\end{align} 
since the common multiple of $\{ q_m : m \in \Lambda_1 \}$ 
is at most $2^{j_2 N^3N} \le 2^{j_2 N^4}$.
This property and \eqref{42kp} yield \eqref{pkn}, completing the proof of Proposition~\ref{prop6100}.
\end{proof} 
For the multi-parameter setting $(k>2)$ in Section \ref{sec5}, we introduce the following slight variant of Proposition~\ref{prop6100}.
\begin{lemma}[Variant of Proposition~\ref{prop6100}]\label{vari61}
Let $\ell'\in  \{0,1,\cdots,k\}$.  Let $j'$ be an integer satisfying $0\leq j'\leq j_2$. Let $\xi=(\xi_{mn})_{(m,n)\in \Lambda}\in \mathbb{R}^{|\Lambda|}$ and $j_2\in \mathbb{Z}_{+}$. Assume that condition \emph{(\ref{420s})} is replaced by 
\begin{align*}
\sum_{|t_2| \sim 2^{j'}}
\psi\!\left(
  \frac{\xi(t_2) - \frac{b(t_2)}{p(t_2)}}{2^{-10K j_2 + N^{4\ell'} j_2}}
\right)
\ge 2|\Lambda|,
\qquad
\frac{b(t_2)}{p(t_2)}=\left( \frac{b_m(t_2)}{p_m(t_2)} \right)_{m\in \Lambda_1} \in \mathbb{Q}^{\Lambda_1}\!\left[1, 2^{N^{4\ell'} j_2}\right].
\end{align*}
Then there exists
\[
\frac{a}{q} \in \mathbb{Q}^{\Lambda}\!\left[1, 2^{N^{4(\ell'+1)} j_2}\right]
\quad \text{such that} \quad
\psi\!\left(
  \frac{\xi - \frac{a}{q}}{2^{-10K j_2 + N^{4(\ell'+1)} j_2}}
\right) = 1.
\]
\end{lemma}

\begin{proof}
We apply the same argument as in Proposition~\ref{prop6100}, with no essential modification. Compared with Proposition~\ref{prop6100}, the error term in \eqref{meg1} is now \[ O\!\left(2^{-10K j_2 + N^{4\ell'} j_2}\right) \]
rather than $O(2^{-10Kj_2})$. Therefore, as in \eqref{42kp} and \eqref{yy}, we have $\frac{a}{q}
=
\left( \frac{a_{mn}}{q} \right)_{(m,n)\in \Lambda}$ such that
\begin{align}\label{yy1}
	\left|\xi-\frac{a}{q}\right|
	\le \frac{1}{4}
	2^{N^{4(\ell'+1)}j_2}2^{-10Kj_2}.
\end{align}
Since
$
1\le p_m(t_2)\le 2^{N^{4\ell'}j_2},
$
it follows from \eqref{gh11} that
$
1\le q_m\le 2^{N^{4\ell'+3}j_2}.
$
Thus, arguing as in \eqref{ep1}, we can obtain $a/q$ satisfying 
$
1\le q\le 2^{N^{4(\ell'+1)}j_2}
$ and \eqref{yy1}.

This completes the proof.

\end{proof}
Hence, we have established \eqref{283b}. The preceding argument reveals a rigidity principle underlying the passage from the sliced frequencies back to the original frequency. Namely, when the major arcs associated with the sliced coefficient $\xi(t_2)$ occur with sufficiently high concentration, the frequency $\xi$ itself must admit a rational approximation at a coarser scale. Schematically, $$
\text{Concentration of sliced major arcs for $\xi(t_2)$}
\Longrightarrow
\text{Coarse major-arc localization of $\xi$.}
$$

\subsection{Merging Major Arcs: Proof of \eqref{283c}}\label{Sec5.4}
We refer to \eqref{283c} as the \emph{merging major arc estimate}. In this subsection, we establish this estimate.

Our analysis is based on the product cutoff
\[
\Psi^{\Lambda_1,\mathrm{major}}_{j}\bigl(\xi(t_2)\bigr)
L^{\Lambda,\mathrm{major}}_{j_2}(\xi),
\]
where
\begin{align}
\Psi^{\Lambda_1,\mathrm{major}}_{j}\bigl(\xi(t_2)\bigr)
&=
\sum_{b/p\in\mathbb{Q}^{\Lambda_1}[1,2^{j_2/10}]}
\prod_{m\in\Lambda_1}
\psi\!\left(
\frac{\xi_m(t_2)-\frac{b_m}{p}}
{2^{-mj_1}2^{j_2/100}}
\right),\label{put3}\\
L^{\Lambda,\mathrm{major}}_{j_2}(\xi)
&=
\sum_{a/q\in\mathbb{Q}^{\Lambda}[1,2^{N^4j_2}]}
\psi\!\left(
\frac{\xi-a/q}
{2^{-10Kj_2+N^4j_2}}
\right).\nonumber
\end{align} 
Here, $m\neq0$ for every $m\in\Lambda_1$ by hypothesis.

The cutoff $L^{\Lambda,\mathrm{major}}_{j_2}$ defines a coarse major arc, whose neighborhoods are wider than those of the original major arcs and whose rational approximations allow larger denominators. Nevertheless, when combined with the layered major arc cutoff
$\Psi^{\Lambda_1,\mathrm{major}}_{j}\bigl(\xi(t_2)\bigr)$,
which localizes the sliced frequencies at a much finer scale and with much smaller denominators, the combined information is sufficiently precise to recover the original major arc determined by
$\Psi^{\Lambda,\mathrm{major}}_{j}(\xi)$.

\medskip

\noindent
{\bf The coarse approximation determines the sliced major arc.}
We first show that the coarse rational approximation of $\xi$ uniquely determines the rational approximation appearing in the sliced major arc cutoff
$\Psi_j^{\Lambda_1,\mathrm{major}}(\xi(t_2))$.
More precisely,
\[
\left(\frac{b_m}{p}\right)_{m\in\Lambda_1}
=
\left(
\frac{\sum_{n\in\Lambda(m)}a_{mn}t_2^n}{q}
\right)_{m\in\Lambda_1},
\]
as stated in the following lemma.

\begin{lemma}[Removal of $b/p$]\label{lem31}
For
$a/q=(a_{mn})/q\in\mathbb{Q}^{\Lambda}[1,2^{N^4j_2}]$
and $|t_2|\sim2^{j_2}$, define
\begin{align*}
\frac{a(t_2)}{q}
:=
\left(
\frac{a_m(t_2)}{q}
\right)_{m\in\Lambda_1}
:=
\left(
\frac{\sum_n a_{mn}t_2^n}{q}
\right)_{m\in\Lambda_1}.
\end{align*}
Then the summation over $b/p$ in
$\Psi_j^{\Lambda_1,\mathrm{major}}(\xi(t_2))$
can be eliminated:
\begin{align}\label{345}
\Psi^{\Lambda_1,\mathrm{major}}_{j}\bigl(\xi(t_2)\bigr)
L^{\Lambda,\mathrm{major}}_{j_2}(\xi)
&\le
\sum_{a/q\in\mathbb Q^\Lambda[1,2^{N^4j_2}]}
\chi_{\mathbb Q^{\Lambda_1}[1,2^{j_2/10}]}
\!\left(\frac{a(t_2)}{q}\right)
\\
&\qquad\times
\psi\!\left(
\frac{\xi-a/q}{2^{-10Kj_2+N^4j_2}}
\right)
\psi^{\Lambda_1}_{j}\!\left(
\frac{\xi(t_2)-a(t_2)/q}{2^{j_2/100}}
\right).\nonumber
\end{align}

\end{lemma}

\begin{proof}
Recall the definitions in \eqref{00gg} and \eqref{balance}. If $\Psi^{\Lambda_1,\mathrm{major}}_{j}\bigl(\xi(t_2)\bigr)
L^{\Lambda,\mathrm{major}}_{j_2}(\xi)\neq 0$, then one has
\begin{align}
&\left|
\xi_m(t_2)-\frac{b_m}{p}
\right|
\le
2^{-mj_1}2^{j_2/100}
\le
2^{-19Kj_2},\label{77}
\\
&\left|
\xi_m(t_2)-\frac{a_m(t_2)}{q}
\right|
=
\left|
\sum_n
\left(
\xi_{mn}-\frac{a_{mn}}{q}
\right)
t_2^n
\right|
\lesssim
2^{Nj_2}
2^{-10Kj_2+N^4j_2}
\le
2^{-9Kj_2},\notag
\end{align}
where the second inequality of (\ref{77}) follows from the hypothesis $m>0$.
Hence,
\[
\left|
\frac{b_m}{p}
-
\frac{a_m(t_2)}{q}
\right|
\le
2^{-8Kj_2},
\qquad
m\in\Lambda_1.
\]

Since
$1\le p\le2^{j_2/10}$
and
$1\le q\le2^{N^4j_2}$,
the above estimate implies that, for each
$a_m(t_2)/q$, there exists a unique rational number
$b_m/p$ satisfying
\begin{equation}
\frac{b_m}{p}
=
\frac{a_m(t_2)}{q},
\qquad
m\in\Lambda_1.
\end{equation}
Therefore, the rational approximation $b/p$ is uniquely determined by
$a(t_2)/q$, and the summation over $b/p$ in
\eqref{put3} can be eliminated, yielding
\eqref{345}.
\end{proof}

\noindent{\bf To the Final Major Arc}. One can split the summation $\Psi^{\Lambda_1,\rm{major}}_{j}(\xi(t_2))L^{\Lambda,\rm{major}}_{j_2}(\xi)$  in \eqref{345} as
\begin{align*}
&\sum_{ a/q\in  \mathbb{Q}^{\Lambda}[2^{j_2/10},\,2^{N^4 j_2}]}\chi_{\mathbb Q^{\Lambda_1}[1,2^{j_2/10}]}
\!\left(\frac{a(t_2)}{q}\right)
\psi\!\left(
\frac{\xi-a/q}{2^{-10Kj_2+N^4j_2}}
\right)
\psi^{\Lambda_1}_{j}\!\left(
\frac{\xi(t_2)-a(t_2)/q}{2^{j_2/100}}
\right).\\
&\qquad+ \sum_{a/q\in  \mathbb{Q}^{\Lambda}[1,2^{j_2/10}]} 
\psi\!\left(
\frac{\xi-a/q}{2^{-10Kj_2+N^4j_2}}
\right)
\psi^{\Lambda_1}_{j}\!\left(
\frac{\xi(t_2)-a(t_2)/q}{2^{j_2/100}}
\right)
\end{align*}
where    $a/q \in \mathbb{Q}^{\Lambda}[1,2^{j_2/10}]$ implies $ \chi_{\mathbb Q^{\Lambda_1}[1,2^{j_2/10}]}
\!\left(\frac{a(t_2)}{q}\right)=1.$
Together with the identity $$\Psi^{\Lambda,\rm{major}}_{j}(\xi)= \sum_{a/q\in  \mathbb{Q}^{\Lambda}[1, 2^{j_2/10}] } \prod_{(m,n)\in \Lambda}\psi\left(  \frac{  \xi_{mn}-\frac{a_{mn}}{q}}{ 2^{-j\cdot (m,n)}2^{j_2/10}} \right) \psi\left(\frac{\xi-a/q}{2^{-(10K-N^4)j_2}}\right),$$ one can split  
\begin{align*}
	&\sum_{j_1:j\in Z_1(2)}  \sum_{|t_2|\sim 2^{j_2}}    \frac{|H^{\Lambda_1}_{j_1}(\xi(t_2))|}{|t_2|}  \left(\Psi^{\Lambda_1,\rm{major}}_{j}(\xi(t_2))L^{\Lambda,\rm{major}}_{j_2}(\xi)-\Psi^{\Lambda,\rm{major}}_{j}(\xi)\right) \\
	&\le  \sum_{j_1:j\in Z_1(2)} E_j^{\rm{gauss}}(\xi)+E_j^{\rm{sub}}(\xi) +E_j^{\rm{osc}}(\xi). 
\end{align*}
Here, we define
\begin{align}
	E^{\rm{gauss}}_j(\xi)&:= \sum_{ a/q\in \mathbb{Q}^{\Lambda}[2^{j_2/10},2^{N^4 j_2}]}  \sum_{|t_2|\sim 2^{j_2}}    \frac{|H^{\Lambda_1}_{j_1}(\xi(t_2))|}{|t_2|}    \prod_{m\in\Lambda_1}\psi\left( \frac{\sum_{n\in\Lambda(m)}\left( \xi_{mn}-\frac{a_{mn}}{q} \right)t_2^n }{2^{-mj_1} 2^{ j_2/100}} \right) \nonumber \\
	&\times \psi\left(\frac{\xi-\frac{a}{q}}{2^{-10Kj_2+N^4j_2}}\right), \nonumber\\
	E^{\rm{sub}}_j(\xi)&:=  \sum_{ a/q\in \mathbb{Q}^{\Lambda}[1,2^{j_2/10}]}  \sum_{|t_2|\sim 2^{j_2}}   \frac{|H^{\Lambda_1}_{j_1}(\xi(t_2))|}{|t_2|}   \prod_{m\in\Lambda_1}\psi\left( \frac{\sum_{n\in\Lambda(m)}\left( \xi_{mn}-\frac{a_{mn}}{q} \right)t_2^n }{2^{-mj_1} 2^{ j_2/100}} \right)\nonumber \\
	&\times \left(1-\prod_{(m,n)\in\Lambda}\psi\left(  \frac{  \xi_{mn}-\frac{a_{mn}}{q}}{ 2^{-j\cdot (m,n)}2^{j_2/10}} \right)\right)   \psi\left(\frac{\xi-\frac{a}{q}}{2^{-10Kj_2+N^4j_2}}\right),    \notag\\
	E^{\rm{osc}}_j(\xi)&:=   \sum_{ a/q\in \mathbb{Q}^{\Lambda}[1,2^{j_2/10}]}  \sum_{|t_2|\sim 2^{j_2}}   \frac{|H^{\Lambda_1}_{j_1}(\xi(t_2))|}{|t_2|}  \left(  \prod_{m\in\Lambda_1}\psi\left( \frac{\sum_{n\in\Lambda(m)}\left( \xi_{mn}-\frac{a_{mn}}{q} \right)t_2^n }{2^{-mj_1} 2^{ j_2/100}} \right) -1\right)\notag\\
	&\times  \prod_{(m,n)\in\Lambda}\psi\left(  \frac{  \xi_{mn}-\frac{a_{mn}}{q}}{ 2^{-j\cdot (m,n)}2^{j_2/10}} \right).\nonumber
\end{align}
In $E^{\mathrm{gauss}}_j(\xi)$, we drop the constraint  $\chi_{\mathbb Q^{\Lambda_1}[1,2^{j_2/10}]}$
since the summands are nonnegative.\\

\noindent{\bf Three  Estimates.}
We shall prove Lemmas~\ref{lem850}, \ref{lem57}, and \ref{lem574} below.
The proofs respectively rely on the averaged Gauss sum estimate for
\(E^{\mathrm{gauss}}_j(\xi)\), the sublevel set estimate for
\(E^{\mathrm{sub}}_j\), and the oscillatory integral estimate for
\(E^{\mathrm{osc}}_j\).

Before proving the three lemmas, we first observe that for every $m\in \Lambda_1$,
$$  \left|  \xi_m(t_2)-\frac{a_m(t_2)}{q} \right| \lesssim 2^{-mj_1} 2^{ j_2/10}$$ in the  three multipliers \(E^{\mathrm{gauss}}_j(\xi)\), \(E^{\mathrm{sub}}_j\) and  \(E^{\mathrm{osc}}_j\). Combining   this  with  $q\le 2^{N^4j_2}<2^{j_1/10}$, we can apply Proposition  \ref{lemm29}  in every   summand  of the three $q$-sums to have
\begin{align} \label{2300}
	H^{\Lambda_1}_{j_1}( \xi(t_2) )&=  S^{\Lambda_1}\left(\frac{a(t_2)}{q} \right)  \mathcal{H}^{\Lambda_1}_{j_1}\left( \xi(t_2)-\frac{a(t_2)}{q}\right)+O(2^{-cj_1}).
\end{align}
By using the one parameter oscillatory integral estimate and Lemma \ref{po3}, one has
\begin{align}
	&\left|\mathcal{H}^{\Lambda_1}_{j_1}\left(\xi(t_2)-\frac{a(t_2)}{q}\right)\right| \lesssim\min\left\{ \left|\left(\xi_{m_1}(t_2)-\frac{a_{m_1}(t_2)}{q}\right)2^{m_1 j_1}\right|^{-c'}\right\}_{m_1\in \Lambda_1}\ \text{and}\   \label{hees}\\
	&\sup_{t_2,\xi}\sum_{j_1\in \mathbb{Z}}\left( \left|\mathcal{H}^{\Lambda_1}_{j_1}\left(\xi(t_2)-\frac{a(t_2)}{q}\right)\right|   + \left|\mathcal{H}^{\Lambda_1}_{j_1}\left(\xi(t_2)-\frac{a(t_2)}{q}\right)\right|^{1/2} \right)  \lesssim 1. \label{2203}
\end{align}

 \begin{lemma}\label{lem850}[Estimate of $E_j^{\rm{gauss}}$]
  There exists $c>0$ independent of $\xi$ such that 
$$\sum_{j_1:j\in Z_1(2)} |E^{\rm{gauss}}_j(\xi)|\lesssim 2^{-cj_2}.$$
\end{lemma}
\begin{proof}
Using (\ref{2300}),  one has
  \begin{align*}
\sum_{j_1:j\in Z_1(2)} E^{\rm{gauss}}_j(\xi)
&\lesssim \sum_{a/q\in  \mathbb{Q}^{\Lambda}[2^{j_2/10},2^{N^4j_2}]} \frac{1}{2^{j_2}}\sum_{t_2\sim 2^{j_2}}\left|S^{\Lambda_1}\left( \frac{a(t_2)}{q} \right)\right|\psi\left(\frac{\xi-\frac{a}{q}}{2^{-(10K-N^4)j_2}}\right) \\
 &\times\sup_{t_2,\xi} \sum_{j_1:j\in Z_1(2)}\left|\mathcal{H}^{\Lambda_1}_{j_1}\left( \xi(t_2)-\frac{a(t_2)}{q}  \right)\right| +2^{-cj_2}
 \end{align*}
Applying the Hilbert-transform estimate \eqref{2203} and 
 the averaged Gauss-sum estimate in Proposition \ref{lem22p} with $k=2$ and $\ell=1$, in the range $  2^{j_2/10}\le q\le 2^{N^4j_2}$, we obtain
$$
 \sum_{j_1:j\in Z_1(2)} E^{\rm{gauss}}_j(\xi)\lesssim \sum_{a/q\in  \mathbb{Q}^{\Lambda}[2^{j_2/10},2^{N^4j_2}]}2^{-cj_2} \psi\left(\frac{\xi-\frac{a}{q}}{2^{-(10K-N^4)j_2}}\right)  \lesssim 2^{-cj_2}.
$$
Thus, we proved Lemma \ref{lem850}.
\end{proof}

\begin{lemma}\label{lem57}[Estimate of $E_j^{\rm{sub}}$]\label{lem36}  There is $c>0$  independent of $\xi$ such that 
$$\sum_{j_1:j\in Z_1(2)} |E^{\rm{sub}}_j(\xi)| \lesssim 2^{-cj_2}.$$
\end{lemma}

\begin{proof}
Fix the scale parameter \(j\) and the frequency offset \(\xi-a/q\).
For each \(m\in \Lambda_1\), write
$\eta_{mn}:=\xi_{mn}-\frac{a_{mn}}{q}.$
We then define the size functional
\begin{align}\label{4gh}
R_{m,j_2}(\eta):=\sum_{n\in\Lambda(m)} |\eta_{mn}|\, 2^{j_2 n}.
\end{align}
 Then we can control  the small support of the second multiplier $E^{\rm{sub}}_j(\xi)$ by the region independent of $j_1$ as
\begin{align}\label{43ut}
	&\prod_{m\in\Lambda_1}
	\psi\!\left(
	\frac{\sum_{n\in\Lambda(m)} \eta_{mn}\, t_2^n}
	{2^{-m j_1}\, 2^{j_2/100}}
	\right)
	\left(
	1-
	\prod_{(m,n)\in\Lambda}
	\psi\!\left(
	\frac{\eta_{mn}\, 2^{j_2 n}}
	{2^{-m j_1}\, 2^{j_2/10}}
	\right)
	\right)
	\\
	&\qquad\le
	\sum_{\substack{ m\in \Lambda_1\\ R_{m,j_2}(\eta)>0}}\psi\!\left(
	\frac{
		\frac{\sum_{n\in\Lambda(m)} \eta_{mn}\, t_2^n}
		{2R_{m,j_2}(\eta)}
	}
	{ 2^{j_2(1/100-1/10)}}
	\right). \nonumber
\end{align}
\begin{proof}[Proof of \eqref{43ut}]
By the second factor in \eqref{43ut}, there exists \(m\in\Lambda_1\) such that
\[
\prod_{n\in\Lambda(m)}
\psi\!\left(
\frac{\eta_{m n}\, 2^{j_2 n}}
{2^{-m j_1}\, 2^{j_2/10}}
\right)
\neq 1.
\]
Consequently, for this $m$, it holds that
\[R_{m,j_2}(\eta)=
\sum_{n\in\Lambda(m)}
\bigl|\eta_{m n}\bigr|\, 2^{j_2 n}
\;\ge\;
2^{-1}2^{-m j_1}2^{j_2/10}.
\]
This inequality on this  \(m\)-th factor in the first product of
\eqref{43ut} yields that
$$  \psi\!\left(
	\frac{\sum_{n\in\Lambda(m)} \eta_{mn}\, t_2^n}
	{2^{-m j_1}\, 2^{j_2/100}}
	\right) \le \psi\!\left(
	\frac{
		\frac{\sum_{n\in\Lambda(m)} \eta_{mn}\, t_2^n}
		{2R_{m,j_2}(\eta)}
	}
	{ 2^{j_2(1/100-1/10)}}
	\right)$$
which leads the desired inequality (\ref{43ut}). 
\end{proof}
 Note that the  right-hand side of \eqref{43ut} is independent of \( j_1 \).
Therefore, by using (\ref{43ut}), we can obtain that 
\begin{align}\label{subes1}
	\sum_{j_1:j\in Z_1(2)} |E_j^{\rm{sub}}(\xi)|
	& \lesssim \sum_{\substack{ m\in \Lambda_1\\ R_{m,j_2}(\eta)>0}}
	 \sum_{ a/q\in \mathbb{Q}^{\Lambda}[1,2^{j_2/10}]}     \sum_{|t_2|\sim 2^{j_2}}   
	\psi\!\left(
	\frac{
		\sum_{n\in\Lambda(m)} \left(\frac{\eta_{mn}}{R_{m,j_2}(\eta)} \right)t_2^n
	}
	{ 2\cdot 2^{j_2(1/100-1/10)}}
	\right)\\
	&\times
	\left(\sum_{j_1:j\in Z_1(2)} \frac{|H^{\Lambda_1}_{j_1}(\xi(t_2))|}{2^{j_2}}  \right)  
	\psi\left(\frac{\xi-\frac{a}{q}}{2^{-10Kj_2+N^4j_2}}\right). \nonumber
\end{align}
We first recall the 1D result for the discrete Hilbert transform:
$$\sup_{t_2,\xi} \sum_{j_1:j\in Z_1(2)} |H^{\Lambda_1}_{j_1}(\xi(t_2))| \le C\ \text{where $C$ is indendent of $t_2$}.$$
After finishing this $j_1$ sum estimate, we are able to  apply the lattice sublevel set estimate of Lemma \ref{eep1} associated with the 1D polynomial. By (\ref{4gh}), note that for every $m\in \Lambda_1$, $$ \sum_{n\in\Lambda(m)} \left|\frac{\eta_{mn}}{R_{m,j_2}(\eta)} 2^{nj_2}\right|=1.$$
Then, with  $\epsilon= 4\cdot2^{j_2(1/100-1/10)}$ and $2^r\epsilon =1$ where $j_2\ge r=(1/10-1/100) j_2-2$ in  Lemma \ref{eep1},
there exists $c>0$ such that
$$   \sum_{|t_2|\sim 2^{j_2}}   
\psi\!\left(
\frac{
	\sum_{n\in\Lambda(m)} \left(\frac{\eta_{mn}}{R_{m,j_2}(\eta)} \right)t_2^n
}
{ 2\cdot 2^{j_2(1/100-1/10)}}
\right)\lesssim 2^{j_2-cj_2}.  $$
Therefore,  we have
\begin{align*}
	\text{RHS of (\ref{subes1})}\lesssim |\Lambda_1| 2^{j_2-cj_2} \left( \frac{1}{2^{j_2}} \right) \sum_{ a/q\in \mathbb{Q}^{\Lambda}[1,2^{j_2/10}]}  \psi\left(\frac{\xi-\frac{a}{q}}{2^{-(10K-N^4)j_2}}\right)\lesssim 2^{-cj_2},
\end{align*}
which is the desired bound for  Lemma \ref{lem57}.\end{proof}

Next, we prove the following lemma.
\begin{lemma}\label{lem574}[Estimate of $E^{\rm{osc}}_j(\xi)$]  There is $c>0$ independent of $\xi$:
$$\sum_{j_1:j\in Z_1(2)} |E^{\rm{osc}}_j(\xi)|\lesssim 2^{-cj_2}.$$
\end{lemma}
\begin{proof}
By  the support condition in  $E^{\rm{osc}}_j(\xi) $ of
$$\left( 1- \prod_{m\in\Lambda_1}\psi\left( \frac{\sum_{n\in\Lambda(m)}\left( \xi_{mn}-\frac{a_{mn}}{q} \right)t_2^n }{2^{-mj_1} 2^{ j_2/100}} \right) \right),$$
one can observe that $  \left|\xi_m(t_2)-\frac{a_m(t_2)}{q}\right|2^{m j_1} \gtrsim 2^{j_2/100}  $ for some $m\in\Lambda_1$. This with (\ref{hees}) yields that
$\left|  \mathcal{H}^{\Lambda_1}_{j_1}\left(\xi(t_2)-\frac{a(t_2)}{q}\right)\right|\lesssim 2^{-cj_2} . $
By using this decay together with  (\ref{2300}) and (\ref{2203}),  we have with an error term $O(2^{-cj_2})$,
\begin{align*} 
\sum_{j_1:j\in Z_1(2)} |E^{\rm{osc}}_{j}(\xi)|
 &\lesssim \sum_{j_1:j\in Z_1(2)} \sum_{t_2\sim 2^{j_2}} \frac{1}{2^{j_2}} \sum_{ a/q\in \mathbb{Q}^{\Lambda}[1,2^{j_2/10}]}   \psi\left(\frac{\xi-\frac{a}{q}}{2^{-Kj_2}}\right)\\
&\times \left| \mathcal{H}^{\Lambda_1}_{j_1}\left(\xi(t_2)-\frac{a(t_2)}{q}\right)\right|\left( 1- \prod_{m\in\Lambda_1}\psi\left( \frac{\sum_{n\in\Lambda(m)}\left( \xi_{mn}-\frac{a_{mn}}{q} \right)t_2^n }{2^{-mj_1} 2^{ j_2/100}} \right) \right)\nonumber\\
&
\lesssim  \sum_{t_2\sim 2^{j_2}} \frac{1}{2^{j_2}} \sum_{ a/q\in \mathbb{Q}^{\Lambda}[1,2^{j_2/10}]}  2^{-cj_2/2}\psi\left(\frac{\xi-\frac{a}{q}}{2^{-Kj_2}}\right)  \le 2^{-cj_2/2}.
\end{align*}
Therefore, we proved Lemma \ref{lem574}.
\end{proof}

\section{Proof of Theorem \ref{th41} in the  Multi-Parameter Case}\label{sec5}
In this section, we extend the two-parameter arguments developed in the previous section to the general $k$-parameter setting. More precisely, we shall consider $\Lambda \subset \mathbb{Z}_{+}^k$ satisfying the condition of (\ref{wep2}) and   establish the corresponding result for
\begin{align*} 
	H^{\Lambda}_J(\xi)= \sum_{t\in \mathbb{Z}^k}  \frac{e^{2\pi i \sum_{\mathfrak{m}\in\Lambda} \xi_{\mathfrak{m}} t^{\mathfrak{m}}}\chi_j(t)}{t_1\cdots t_k} \ \text{ where $J=(j_1,\cdots,j_k)$}.
\end{align*}

   Fix $k\in \mathbb{N}$. To prove  Theorem \ref{th41},  we shall use the induction argument on (\ref{bb330}).
 Let $ k_0\in  [k]$. For $J_0=(j_1,\cdots,j_{k_0})\in \mathbb{Z}^{k_0}_+$ and $\Omega\subset \mathbb{Z}_+^{k_0}$ with ${\bf 0}\notin \Omega$, as in (\ref{100s}), we  define
\begin{align*}
\Psi^{\Omega,\mathrm{major}}_{J_0}(\xi) :=\sum_{a/q \in \mathbb{Q}^{\Omega}[1, 2^{j_{k_0}/10}]}
\psi^{\Omega}_{J_0} \left( \frac{\xi- \frac{a}{q}}{2^{j_{k_0}/10}} \right)\ \text{where}\   \psi^{\Omega}_{J_0} =\prod_{\mathfrak{m}\in \Omega}\psi\left(\xi_{\mathfrak{m}} 2^{J_0\cdot \mathfrak{m}} \right).
\end{align*}
We first consider the case $k_0=1$. For  every $\Omega\subset \mathbb{Z}_+^1$ and every $h\ge 0$, there is $c>0$ such that
$$ \sum_{j_1\in Z(1); j_1\ge h}| H^{\Omega}_{j_1}(\xi) \left( 1-  \Psi^{\Omega,\mathrm{major}}_{j_1}(\xi)\right)|\le 2^{-ch}. $$ So, (\ref{bb330})  holds when  $k_0=1$.
We now proceed with the induction step as follows:\\
{\bf Induction Hypothesis.}\  Assume  that   for any $ k_0\in  [k-1]$ and any $\Omega\subset \mathbb{Z}_+^{k_0}$  satisfying the evenness condition of (\ref{wep2}),  there is $c>0$ independent of $h\ge 0$:
\begin{align}\label{bb33}
 \sum_{J_0=(j_1,\cdots,j_{k_0})\in  Z(k_0)\ \text{and}\ j_{k_0}\ge h\ge 0}|H^{\Omega}_{J_0}(\xi)(1-\Psi^{\Omega,\rm{major}}_{J_0}(\xi))|\lesssim 2^{-ch}. 
\end{align}
\\
{\bf Goal.}\  Define the major-arc cutoff  function  by
$$  \Psi^{\Lambda,\mathrm{major}}_{J}(\xi) := \sum_{a/q \in \mathbb{Q}^{\Lambda}[1, 2^{j_{k}/10}]}
\psi^{\Lambda}_{J} \left( \frac{\xi- \frac{a}{q}}{2^{j_{k}/10}} \right)\ \text{where}\ J=(j_1,\cdots,j_{k}). $$ Then,
under  the  condition (\ref{wep2}),  we aim to show that for any  $\ell\in [k-1]$, there exists a constant \( c > 0 \), independent of \( J_2=(j_{\ell+1},\cdots,j_k) \) and \( \xi \), such that 
\begin{align}\label{rpkg9}
 \sum_{J_1=(j_1,\cdots,j_\ell):J=(j_1,\cdots,j_k)\in Z_{\ell}(k)}|H^{\Lambda}_J(\xi)(1-\Psi^{\Lambda,\rm{major}}_{J}(\xi))|\lesssim 2^{-cj_{\ell+1}}.
\end{align}
This directly implies the property (\ref{bb330}) for  $k$. Here, we recall   \begin{align*}
	Z(k)=\bigcup_{\ell=0}^{k-1}Z_{\ell}(k) \ \text{where}   \	Z_0(k)&:=\{ j\in Z(k):   j_{1}\approx_{\Lambda}  \cdots\approx_{\Lambda} j_k\}\  \text{and} \\
	Z_{\ell}(k)&:=\{j\in Z(k):   j_\ell\gg_{\Lambda} j_{\ell+1}\approx_{\Lambda}\cdots\approx_{\Lambda} j_k\}. 
\end{align*}

\begin{remark}\label{yu1}
 For any $ k_0\in  [k-1]$ and any $\Omega\subset \mathbb{Z}_+^{k_0}$  satisfying the evenness condition of (\ref{wep2}),	by the induction hypothesis (\ref{bb33}) combined with    the  major arc estimate in Proposition \ref{lemm29}, one has
	\begin{align*} 
		&\sum_{J_0=(j_1,\cdots,j_{k_0})\in  Z(k_0)} |H^{\Omega}_{J_0}(\xi)|\\
		&\le\sum_{J_0=(j_1,\cdots,j_{k_0})\in  Z(k_0)} \left(|H^{\Omega}_{J_0}(\xi)(1-\Psi^{\Omega,\rm{major}}_{J_0}(\xi))|+ |H^{\Omega}_{J_0}(\xi) \Psi^{\Omega,\rm{major}}_{J_0}(\xi)|\right) \lesssim 1. 
	\end{align*}

\end{remark}

\begin{definition}[Projection and Fiber]\label{diid}
Let $\Lambda\subset \mathbb{Z}_+^k$ and $\ell \in [k-1]$. For simplicity,   write
\begin{align*}
& \mathfrak{m}=(m,n)\in \Lambda \text{ where $m=(m_1,\cdots,m_\ell)$ and $n=(m_{\ell+1},\cdots,m_k)$.} 
\end{align*}
Next, express the variable $\mathfrak{t}=(t_1,\cdots,t_k)$ as
\begin{align*}
\mathfrak{t}=(\mathfrak{t}_1, \mathfrak{t}_2)\ \text{where $\mathfrak{t}_1=(t_1,\cdots,t_\ell)$ and $\mathfrak{t}_2=(t_{\ell+1},\cdots,t_k)$}.
\end{align*}
Similarly, define
$${\bf 1}:=(1,\cdots,1)=({\bf 1}_1,{\bf 1}_2)\ \text{where}\ {\bf 1}_1\in \mathbb{Z}_+^\ell,{\bf 1}_2\in \mathbb{Z}_{+}^{k-\ell}. $$
Set  the projection space of $m$ as the image of the projection map $\pi_{\le \ell}:\mathbb{Z}_{+}^k\rightarrow \mathbb{Z}_{+}^{\ell}$:
\begin{align}\label{11pp}
\Lambda_{1}:=\{m=(m_1,\cdots,m_\ell) \in \mathbb{Z}_+^{\ell}:(m,n)\in \Lambda\text{ for some } n\in\mathbb{Z}_{+}^{k-\ell}\}, 
\end{align}
and set the fiber over a fixed $m\in \Lambda_{1}$ as 
\begin{align*}
\Lambda(m):=\{n=(m_{\ell+1},\cdots,m_k):  (m,n)\in \Lambda\}.
\end{align*}
Next, we  define $\Lambda'\subset\Lambda$ as
\begin{align} \label{desi1}
\Lambda':=  
\{ (m,n)\in\Lambda: m\ne {\bf 0}\}, \  
\end{align}
with its projection 
  $\Lambda'_{1}$ and a fiber $\Lambda'(m)$ defined as \begin{align*}
	\Lambda_{1}' &:=\{m: (m,n)\in \Lambda' \text{ for some } n\in\mathbb{Z}_{+}^{k-\ell}\} \ \text{and}\ \Lambda'(m):= \{ n: (m,n)\in \Lambda' \}.
\end{align*} 
\end{definition}
\begin{definition}[Layer Polynomials]
For $\mathfrak{t}=(t_1,\cdots,t_k) = (\mathfrak{t}_1,\mathfrak{t}_2)\in \mathbb{Z}^\ell\times \mathbb{Z}^{k-\ell}$ in the above,  their powers of the exponents $\mathfrak{m}=(m,n)\in \Lambda$ are given by
\begin{align*}
\mathfrak{t}^{\mathfrak{m}}=t_1^{m_1}\cdots t_k^{m_k}\ \text{where}\ \mathfrak{t}_1^m=t_1^{m_1}\cdots t_\ell^{m_\ell}\ \text{and}\  \mathfrak{t}_2^n=t_{\ell+1}^{m_{\ell+1}}\cdots t_k^{m_k}.
\end{align*}
Then,  the polynomial $\sum_{\mathfrak{m}\in \Lambda}\xi_{\mathfrak{m}} t^{\mathfrak{m}}$  in the $k$ variables $\mathfrak{t}=(t_1,\cdots,t_k)$  may be viewed as a polynomial in the first $\ell$ variables $\mathfrak{t}_1=(t_1,\cdots,t_\ell)$ whose coefficients depend  on the remaining variables $\mathfrak{t}_2$. More precisely,
\begin{align*} 
 \sum_{(m,n)\in \Lambda} \xi_{mn} \mathfrak{t}_1^{m}\mathfrak{t}_2^n=\sum_{m\in \Lambda_{1}} \xi_{m}(\mathfrak{t}_2)  \mathfrak{t}_1^{m}, 
\end{align*}
where    the coefficient $ \xi_{m}(\mathfrak{t}_2)$ of \( \mathfrak{t}_1^{m}   \)   is itself a polynomial in $\mathfrak{t}_2$, given by 
\begin{align*}
 \xi_{m}(\mathfrak{t}_2) = \sum_{n\in \Lambda(m)} \xi_{mn}\mathfrak{t}_2^n.
\end{align*}
Set the vector polynomial
 as $\xi(\mathfrak{t}_2)=(\xi_{m}(\mathfrak{t}_2))_{m\in \Lambda_{1}}$, 
and for $J=(J_1,J_2)\in Z_\ell(k)$  where
$\text{$J_1=(j_1,\cdots,j_\ell)$ and $J_2=(j_{\ell+1},\cdots,j_k)$},$
we define the multiplier of the $\ell$-parameter discrete Hilbert transform by
\begin{align*}
H^{\Lambda_{1}}_{J_1}\left( \xi(\mathfrak{t}_2) \right) = \sum_{ |\mathfrak{t}_1| \sim 2^{J_1}} 
\frac{e^{2\pi i \sum_{m \in \Lambda_{1}} \xi_{m}(\mathfrak{t}_2) \, \mathfrak{t}_1^{m} }  }{t_1\cdots t_\ell}= \sum_{|\mathfrak{t}_1|\sim 2^{J_1} } 
\frac{e^{2\pi i \sum_{m \in \Lambda_{1}} \xi_{m}(\mathfrak{t}_2) \, \mathfrak{t}_1^{m} }  }{\mathfrak{t}_1^{{\bf 1}_1} }.
\end{align*}
Here, the notation  $|\mathfrak{t}_1|\sim 2^{J_1}$  abbreviates the conditions $ |t_1| \sim 2^{j_1},\cdots,|t_\ell|\sim 2^{j_{\ell}}$.
With this notation, we can express the full sum \( H^{\Lambda}_{J}(\xi) \) for \(J= (J_1,J_2) \in Z_\ell(k) \) as
\begin{align*}
H^{\Lambda}_J(\xi)=\sum_{| \mathfrak{t}_2 | \sim 2^{J_2}} 
\frac{  H^{\Lambda_{1}}_{J_1}\left( \xi(\mathfrak{t}_2) \right)}{t_{\ell+1}\cdots t_k}= \sum_{|\mathfrak{t}_2|\sim 2^{J_2}} 
\frac{  H^{\Lambda_{1}}_{J_1}\left( \xi(\mathfrak{t}_2) \right)}{\mathfrak{t}_{2}^{{\bf 1}_2}},\  
\end{align*}
where $|\mathfrak{t}_2|\sim 2^{J_2}:|t_{\ell+1}|\sim 2^{j_{\ell+1}},\cdots, |t_k|\sim 2^{j_k}$. Note also that
$$H^{\Lambda_{1}}_{J_1}\left( \xi(\mathfrak{t}_2) \right)=e^{2\pi i \xi_{0}(\mathfrak{t}_2)} \, H^{\Lambda'_{1}}_{J_1}\left( \xi(\mathfrak{t}_2) \right)\ \text{where}\ \xi_{0}(\mathfrak{t}_2) = \sum_{n \in \Lambda(0)} \xi_{0n} \mathfrak{t}_2^n.$$
\end{definition}

 \begin{definition}[Layered Major Arc] 
Consider $\Lambda'$ in (\ref{desi1}). We set
\[
\xi_{\Lambda'_{1}} := (\xi_{m})_{m\in \Lambda'_{1}} \in \mathbb{R}^{|\Lambda'_{1}|}, \quad
\xi_{\Lambda'} := (\xi_{mn})_{(m,n) \in \Lambda'} \in \mathbb{R}^{|\Lambda'|},
\]
and the corresponding smooth cutoffs:
\[
\psi^{\Lambda'_{1}}_{J_1}(\xi_{\Lambda'_{1}}) := \prod_{m \in \Lambda'_{1}} \psi\left( \xi_{m}   2^{m\cdot J_1} \right), \quad
\psi^{\Lambda'}_{J}(\xi_{\Lambda'}) := \prod_{(m,n)\in \Lambda'} \psi\left( \xi_{mn}   2^{(m,n)\cdot J} \right).
\]
With  $K$ in (\ref{00gg}), define  \begin{align}\label{rho3}
	\rho:=(400K)^{-k}\delta_{\Lambda}.
\end{align}The constant $\rho$ in the cutoff functions is required to address technical issues arising in Subsection~\ref{5.4}. Then,
we redefine the cutoff functions of two different layered major arcs:
\begin{align}\label{4022k9}
\Psi^{\Lambda'_{1},\mathrm{major}}_{J}(\xi(\mathfrak{t}_2)) &:= \sum_{b/p \in \mathbb{Q}^{\Lambda'_{1}}[1, 2^{j_k/10}]}
\psi^{\Lambda'_{1}}_{J_1} \left( \frac{\xi(\mathfrak{t}_2) - \frac{b}{p}}{2^{j_k(\rho/100)}} \right), \nonumber \\
\Psi^{\Lambda',\mathrm{major}}_{J}(\xi_{\Lambda'}) &:= \sum_{a/q \in \mathbb{Q}^{\Lambda'}[1, 2^{j_k/10}]}
\psi^{\Lambda'}_{J} \left( \frac{\xi_{\Lambda'} - \frac{a}{q}}{2^{j_k(\rho/10)}} \right).
\end{align}
The corresponding  cutoffs defined by
\begin{align}
L_{J_2}^{\Lambda'_1,\rm{major}}(\xi(\mathfrak{t}_2))&:=\sum_{b/p\in \mathbb{Q}^{\Lambda'_1}[1,2^{j_k/10}]} \psi\left(\frac{\xi(\mathfrak{t}_{2}) -\frac{b}{p}}{ 2^{-10Kj_{\ell+1}} }\right),\label{88332}\\
L^{\Lambda',\rm{major}}_{J_2}(\xi_{\Lambda'})&:= \sum_{a/q\in \mathbb{Q}^{\Lambda'}[1,2^{N^{4(k-\ell)}j_{\ell+1}}]} \psi\left(\frac{\xi_{\Lambda'}-a/q}{2^{-10Kj_{\ell+1}+N^{4(k-\ell)}j_{\ell+1}  }}\right).\label{56}
\end{align}
The above cutoffs $L_{J_2}^{\Lambda'_1,\rm{major}}(\xi(\mathfrak{t}_2))$ and $L^{\Lambda',\rm{major}}_{J_2}(\xi_{\Lambda'})$ serve as connectors from major arcs of the sliced frequency $\xi(\mathfrak{t}_2)$ to major arcs of the original frequency $\xi$. 
They are  coarse major arc cutoffs:
\[
\operatorname{supp} \left[ \Psi^{\Lambda',\mathrm{major}}_{J} \right]
\subset \operatorname{supp} \left[L^{\Lambda',\mathrm{major}}_{J_2} \right],\]
since for each $  (m,n) \in \Lambda'  $ with $m=(m_1,\cdots,m_\ell)\ne {\bf 0}$   it holds that
\begin{align}\label{pik3}
2^{-J \cdot (m,n)}   2^{j_k/10} \ll 2^{-10 K j_{\ell+1}}\ \text{ for $J\in Z_\ell(k)$  where $\ell\le k-1$}.
\end{align}
 \end{definition}
We then decompose the multiplier in the left hand side of (\ref{rpkg9}) as
\[  H^{\Lambda}_J(\xi) \left( 1 - \Psi^{\Lambda,\mathrm{major}}_{J}(\xi) \right)=
\sum_{|\mathfrak{t}_2|\sim 2^{J_2}} 
\frac{ H^{\Lambda_{1}}_{J_1}\left( \xi(\mathfrak{t}_2) \right)}{\mathfrak{t}_{2}^{{\bf 1}_2}} \left( 1 - \Psi^{\Lambda,\mathrm{major}}_{J}(\xi) \right),
\]
into the following four components:
\begin{align}\label{479aa9}
&\sum_{|\mathfrak{t}_2|\sim 2^{J_2}}
\frac{ H^{\Lambda_{1}}_{J_1}\!\left( \xi(\mathfrak{t}_2) \right)}{\mathfrak{t}_{2}^{{\bf 1}_2}}
\begin{cases}
1 - \Psi^{\Lambda'_{1},\mathrm{major}}_{J}\bigl(\xi(\mathfrak{t}_2)\bigr), \\[3pt]
\Psi^{\Lambda'_{1},\mathrm{major}}_{J}\bigl(\xi(\mathfrak{t}_2)\bigr)\,
\bigl( 1 - L^{\Lambda',\mathrm{major}}_{J_2}(\xi_{\Lambda'}) \bigr), \\[3pt]
\Psi^{\Lambda'_{1},\mathrm{major}}_{J}\bigl(\xi(\mathfrak{t}_2)\bigr)\,
L^{\Lambda',\mathrm{major}}_{J_2}(\xi_{\Lambda'})
- \Psi^{\Lambda',\mathrm{major}}_{J}(\xi_{\Lambda'}),
\end{cases} \\[5pt]
&\quad +
\sum_{|\mathfrak{t}_2|\sim 2^{J_2}}
\frac{ H^{\Lambda_{1}}_{J_1}\!\left( \xi(\mathfrak{t}_2) \right)}{\mathfrak{t}_{2}^{{\bf 1}_2}}
\Bigl(
\Psi^{\Lambda',\mathrm{major}}_{J}(\xi_{\Lambda'})
- \Psi^{\Lambda,\mathrm{major}}_{J}(\xi)
\Bigr). \nonumber
\end{align}
To prove \eqref{rpkg9}, it suffices to show that, for each of the four terms above, the sum over $J_1$ admits a decay factor of $2^{-c j_{\ell+1}}$.
The first three terms correspond  to (\ref{283a})--(\ref{283c}) for the case $k=2$. The fourth term is new and appears only when $\Lambda\ne \Lambda'$. In this case, the property (\ref{pik3}) may fail for terms $(m,n)\in\Lambda\setminus\Lambda'$ with $m={\bf 0}$.
\\
\\
\textbf{Overview of the Proof of \eqref{479aa9}.}\\
\noindent\textbf{1. The first term of \eqref{479aa9}.}
For the first component of \eqref{479aa9}, we combine the argument underlying 
\eqref{283a} with the induction hypothesis \eqref{bb33} to recover the same minor–arc gain as in the case $k=2$. \\
\noindent\textbf{2. The second term of \eqref{479aa9}.}
We now turn to the second line of \eqref{479aa9}.
Refining the argument used in \eqref{283b},  we may pass
\[\ \text{from}\ 
\Psi_{J}^{\Lambda'_1,\mathrm{major}}\!\left(\xi(\mathfrak{t}_2)\right)
\ \text{to}\ 
\Psi_{J}^{\Lambda'_1,\mathrm{major}}\!\left(\xi(\mathfrak{t}_2)\right)
\,L^{\Lambda',\mathrm{major}}_{J_2}(\xi_{\Lambda'}).
\]
To carry out this refinement, we transfer the localization
\[
L_{J_2}^{\Lambda'_1,\mathrm{major}}\!\left(\xi(\mathfrak{t}_2)\right)
\ \text{to}\ 
L_{J_2}^{\Lambda',\mathrm{major}}(\xi_{\Lambda'})
\]
of connecting major arcs (coarse major arcs) appearing in  \eqref{88332} and \eqref{56}.

This transfer is achieved through a sequence of consecutive
connecting coarse major arc  estimates, which allow us to pass successively  
\[\ \text{from}\
\xi(t_{\ell+1},\ldots,t_k)
\ \text{to}\
\xi(t_{\ell+2},\ldots,t_k),
\quad \ldots, \quad
\text{and finally from $\xi(t_k)$ to}\ 
\xi.
\]
Here, for \(s=\ell+1,\ldots,k-1\), the layered polynomials are defined by
\[
\xi(t_{s+1},\ldots,t_k)
=
\Bigl(
\sum_{m_{s+1}}
\xi_{m_1\cdots m_s m_{s+1}}(t_{s+2},\ldots,t_k)\,
t_{s+1}^{m_{s+1}}
\Bigr)_{(m_1,\ldots,m_s)}.
\]
The crucial estimate for the passage from
\(\xi(t_{s+1},\ldots,t_k)\) to \(\xi(t_{s+2},\ldots,t_k)\)
follows from Proposition~\ref{prop6100} applied with
\(t_{s+2},\ldots,t_k\) frozen. This estimate reveals the following major-arc rigidity phenomenon:
\begin{align*}
	&\text{if }\xi(t_{s+1},t_{s+2},\ldots,t_k)\text{ lies in the major-arc region for many distinct values of }t_{s+1},\\
	&\qquad\text{then }\xi(t_{s+2},\ldots,t_k)\text{ must lie in a slight enlargement of the major-arc region.}
\end{align*}This rigidity phenomenon is a key ingredient in the proof of our main theorem.\\
\noindent\textbf{3. The third term of \eqref{479aa9}.}
For the third line of \eqref{479aa9}, we  pass
\[
\text{from}\ 
\Psi_{J}^{\Lambda'_1,\mathrm{major}}\bigl(\xi(\mathfrak{t}_2)\bigr)\,
L^{\Lambda',\mathrm{major}}_{J_2}(\xi_{\Lambda'})
\ \text{to}\ 
\Psi^{\Lambda',\mathrm{major}}_{J}(\xi_{\Lambda'}),
\]
by combining the estimate \eqref{283c} with the induction hypothesis
\eqref{bb33}. This step is entirely analogous to the case \(k=2\).

\medskip
\noindent\textbf{4. The last term of \eqref{479aa9}.}
For the final component, we  pass
\[
\text{from}\ 
\Psi^{\Lambda',\mathrm{major}}_{J}(\xi_{\Lambda'})
\ \text{to}\ 
\Psi^{\Lambda,\mathrm{major}}_{J}(\xi).
\]
This requires estimating the Weyl sum over \(\mathfrak{t}_2\) with phase
\[
\xi_0(\mathfrak{t}_2)
=
\sum_{n\in \Lambda(0)} \xi_{0n}\,\mathfrak{t}_2^n,
\]
where
\[
H^{\Lambda_{1}}_{J_1}\!\left( \xi(\mathfrak{t}_2)\right)
=
e^{2\pi i \xi_0(\mathfrak{t}_2)}
\,H^{\Lambda'_{1}}_{J_1}\!\left( \xi(\mathfrak{t}_2)\right).
\]
The oscillatory factor \(e^{2\pi i \xi_0(\mathfrak{t}_2)}\) corrects the
discrepancy between \(\Lambda\) and \(\Lambda'\), thereby allowing the
major--arc structure to be transferred from \(\xi_{\Lambda'}\) to \(\xi\).

In the successive transitions described above,
\[
\text{from } \xi(t_{\ell+1},\ldots,t_k)
\text{ to } \xi(t_{\ell+2},\ldots,t_k),
\quad \ldots, \quad
\text{and finally from } \xi(t_k) \text{ to } \xi,
\]
one can observe that the layers indexed by \(t_{\ell'}\) \((\ell+1 \le \ell' \le k-1)\)
are peeled off one by one.

\subsection{Minor Arc Estimate for $\xi(\mathfrak{t}_2)$}\label{sec5.1}
In this subsection, we estimate the first term in \eqref{479aa9}. The proof proceeds by further decomposing the relevant frequency region into minor-arc and intermediate regimes. The resulting bounds are obtained by combining the induction hypothesis with suitable extensions of the arguments developed in  Subsection \ref{Sec42v}.

\begin{proposition}\label{prop419b9}
Suppose that $\Lambda$ satisfies   the  condition (\ref{wep2}). Then
there is $c>0$  independent of $J_2$ and $\xi$ such that
\begin{align*}
&\sup_{|\mathfrak{t}_2|\sim 2^{J_2}}\sum_{J_1: J=(J_1,J_2)\in Z_\ell(k)}   |H^{\Lambda_{1}}_{J_1}( \xi(\mathfrak{t}_2))|  \left(1-\Psi_{J}^{\Lambda'_1,\mathrm{major}}(\xi(\mathfrak{t}_2))  \right) \lesssim 2^{-cj_{\ell+1}}.
\end{align*}
\end{proposition}

To prove Proposition~\ref{prop419b9}, we begin with the following decomposition: 
$$1-\Psi_{J}^{\Lambda'_1,\mathrm{major}}(\xi(\mathfrak{t}_2))=\Psi^{\Lambda'_1,\rm{minor}}_{J_1}(\xi(\mathfrak{t}_2)) +
\Psi^{\Lambda'_1,\rm{mediate}}_{J_1}(\xi(\mathfrak{t}_2)),$$
where  from (\ref{4022k9}), we set
\begin{align*} 
&\Psi^{\Lambda'_1,\rm{minor}}_{J_1}(\xi(\mathfrak{t}_2)) :=1 -\sum_{b/p\in  \mathbb{Q}^{\Lambda'_1}[1,2^{j_{\ell}/10}]}\psi^{\Lambda'_1}_{J_1}\left(\frac{\xi(\mathfrak{t}_2)-\frac{b}{p}}{2^{j_{\ell}/10}}\right),\nonumber \\
&\Psi^{\Lambda'_1,\rm{mediate}}_{J_1}(\xi(\mathfrak{t}_2)):=\sum_{b/p\in  \mathbb{Q}^{\Lambda'_1}[2^{j_{k}/10},2^{j_{\ell}/10}]}\psi^{\Lambda_1'}_{J_1}\left(\frac{\xi(\mathfrak{t}_2)-\frac{b}{p}}{2^{j_{\ell}/10}}\right)  \\
&\qquad\quad\qquad\qquad\quad+
\sum_{b/p\in  \mathbb{Q}^{\Lambda'_1}[1,2^{j_k/10}]}
 \left( \psi^{\Lambda'_1}_{J_1}\left(\frac{\xi(\mathfrak{t}_2)-\frac{b}{p}}{2^{j_{\ell}/10}}\right)-\psi^{\Lambda'_1}_{J_1}\left(\frac{\xi(\mathfrak{t}_2)-\frac{b}{p}}{2^{j_{k}(\rho/100)}}\right) \right).\nonumber
    \end{align*}
 Hence, to complete the proof of Proposition~\ref{prop419b9}, it suffices to
establish Lemmas~\ref{prop130y} and~\ref{prop41y} below.  
The proofs of these lemmas are obtained by extending
the arguments used in Lemmas~\ref{prop130} and~\ref{prop41} of Section \ref{Sec42v}.

\begin{lemma}\label{prop130y}
Suppose that $\Lambda$ satisfies   the  condition (\ref{wep2}). Then there exists a constant $c>0$, independent of $J_2$ and $\xi$, such that
\begin{equation}\label{pk1y}
\sup_{|\mathfrak{t}_2|\sim 2^{J_2}}
\sum_{J_1:\, J=(J_1,J_2)\in Z_\ell(k)}
\bigl| H^{\Lambda_{1}}_{J_1}\bigl( \xi(\mathfrak{t}_2) \bigr) \bigr|\,
\Psi^{\Lambda'_1,\mathrm{minor}}_{J_1}\bigl(\xi(\mathfrak{t}_2)\bigr)
\;\lesssim\;
2^{-c j_{\ell+1}}, 
\end{equation}under the  induction hypothesis~\eqref{bb33}.
\end{lemma}

\begin{proof}
Recall $\Lambda'_{1}=\{m\in\Lambda_{1}:m\neq {\bf 0}\}$.  We also note that
\[
H^{\Lambda_{1}}_{J_1}\bigl(\xi(\mathfrak{t}_2)\bigr)
=
e^{2\pi i\, \xi_{0}(\mathfrak{t}_2)}\,
H^{\Lambda'_{1}}_{J_1}\bigl(\xi(\mathfrak{t}_2)\bigr).
\]
Then, for fixed $\mathfrak{t}_2$, we regard $\xi(\mathfrak{t}_2)$ as the variable~$\xi$ in the induction hypothesis~\eqref{bb33}.  
In this identification, the present index $\ell$ in $J_1=(j_1,\dots,j_\ell)$ corresponds to the index $k_0$ in $J_0=(j_1,\cdots,j_{k_0})$ of~\eqref{bb33}, with $\Lambda'_1 \subset \mathbb{Z}_+^{k_0}$. Moreover, set $h = j_{\ell+1}$  in \eqref{bb33}.

With the above setup, we conclude that 
the induction hypothesis~\eqref{bb33}  directly implies \eqref{pk1y}.
\end{proof}

\begin{lemma}\label{prop41y}
Suppose that $\Lambda$ satisfies   the  condition (\ref{wep2}).  There exists  $c>0$  independent of $J_2$ and $\xi$ such that
\begin{align}\label{32y}
\sup_{|\mathfrak{t}_2|\sim 2^{J_2}}\sum_{J_1: J=(J_1,J_2)\in Z_\ell(k)}   |H^{\Lambda_{1}}_{J_1}( \xi(\mathfrak{t}_2))|  \Psi^{\Lambda'_1,\rm{mediate}}_{J_1}(\xi(\mathfrak{t}_2))  \lesssim 2^{-cj_{\ell+1}}. 
\end{align}
\end{lemma}

\begin{proof}
For any  $ b/p =(b_m)/p\in  \mathbb{Q}^{\Lambda_1'}[1,2^{j_\ell/10}]$, define $$S^{\Lambda'_1}(\frac{b}{p}):= \frac{1}{p^{\ell} }\sum_{1\le t_1,\cdots,t_\ell\le p}e^{2\pi i \sum_{m\in\Lambda'_{1}} \frac{b_{m}}{p} \mathfrak{t}_1^m}.$$ 
Fix $\mathfrak{t}_2$ and substitute $\xi(\mathfrak{t}_2)$ for $\xi$, with the index $k$ replaced by $\ell$ in \eqref{b40}.  Then as in  Proposition \ref{lemm29}, we obtain the following integral approximation  
\begin{align*}
&H^{\Lambda'_1}_{J_1}(\xi(\mathfrak{t}_2)) \psi^{\Lambda'_1}_{J_1}\left(\frac{\xi(\mathfrak{t}_2)-\frac{b}{p}}{2^{j_{\ell}/10}}\right)\\
&\quad=  S^{\Lambda'_1}\left(\frac{b}{p}\right)  \mathcal{H}^{\Lambda'_1}_{J_1}\Bigl(\xi(\mathfrak{t}_2)-b/p\Bigr)\psi^{\Lambda'_1}_{J_1}\left(\frac{\xi(\mathfrak{t}_2)-\frac{b}{p}}{2^{j_{\ell}/10}}\right)+ E_{J_1}(b/p,\xi),
\end{align*}
where   the error term satisfies the uniform estimate
\[
\sum_{J_1: J=(J_1,J_2)\in Z_\ell(k)}\sum_{b/p\in \mathbb{Q}^{\Lambda'_1}[1,2^{j_\ell/10}]}
   |E_{J_1}(b/p,\xi)|
   = O(2^{-c j_{\ell+1}}).
\]
Define
\begin{align*}
	A_J(\xi(\mathfrak{t}_2))&:=  \sum_{b/p\in  \mathbb{Q}^{\Lambda'_1}[2^{j_{k}/10},2^{j_{\ell}/10}]} \psi^{\Lambda'_1}_{J_1}\left(\frac{\xi(\mathfrak{t}_2)-\frac{b}{p}}{2^{j_{\ell}/10}}\right)\left|S^{\Lambda'_1}\left(\frac{b}{p}\right)\mathcal{H}^{\Lambda'_1}_{J_1}(\xi(\mathfrak{t}_2)-b/p)\right|,\\
 B_J(\xi(\mathfrak{t}_2))&:=  \sum_{b/p\in  \mathbb{Q}^{\Lambda'_1}[1,2^{j_k/10}]} \psi^{\Lambda'_1}_{J_1}\left(\frac{\xi(\mathfrak{t}_2)-\frac{b}{p}}{2^{j_{\ell}/10}}\right)\left|S^{\Lambda'_1}\left(\frac{b}{p}\right)\mathcal{H}^{\Lambda'_1}_{J_1}(\xi(\mathfrak{t}_2)-b/p)\right| \\
 &\times  \left(1- \psi^{\Lambda'_1}_{J_1}\left(\frac{\xi(\mathfrak{t}_2)-\frac{b}{p}}{2^{j_k(\rho/100)}}\right) \right).\end{align*} 
  By applying Lemma \ref{po3}, there is  $C>0$ independent of   $\mathfrak{t}_2 $ such that
 \begin{align*}
 	&\sum_{J_1 }|\mathcal{H}^{\Lambda'_1}_{J_1}(\xi(\mathfrak{t}_2)-b/p)|^{1/2}  \le C.
 \end{align*}
Therefore,   by the same argument used to establish \eqref{yu}, there exist constants $c',c>0$ such that
 \begin{align*}
 \sum_{J_1; J\in Z_\ell(k)}  A_J(\xi(t_2)) +  B_J(\xi(t_2)) 
 \lesssim 2^{-c' j_{k}}   \lesssim  2^{-c j_{\ell+1}}.
 \end{align*}
For the last inequality, we used the comparability
$
j_{\ell+1}\approx_\Lambda\cdots\approx_\Lambda j_k
$
for $J\in Z_\ell(k)$. Thus, we   get the desired bound of (\ref{32y}).
 \end{proof}

 \subsection{Connecting Layered Arcs}\label{sec5.2}
In this subsection, we shall estimate the second term of (\ref{479aa9}). 
The proof extends the two--parameter argument by iteratively propagating major arc information through successive layers. One can observe the major arc rigidity phenomenon  throughout this subsection. 

\begin{proposition}\label{prop5167b9}
For  $J_2=(j_{\ell+1},\cdots,j_k)$, recall $ L^{\Lambda',\rm{major}}_{J_2}(\xi_{\Lambda'})$ in (\ref{56}) and set
 $$ L^{\Lambda',\rm{minor}}_{J_2}(\xi_{\Lambda'}):= 1- L^{\Lambda',\rm{major}}_{J_2}(\xi_{\Lambda'}).$$  Suppose that $\Lambda$ satisfies   the  condition (\ref{wep2}).  Then
there is $c>0$  independent of $J_2$ and $\xi$ such that
\begin{align} 
\label{44gg9} \sum_{J_1:J\in Z_\ell(k)}\sum_{|\mathfrak{t}_{2}|\sim 2^{J_2}} \left| \frac{ H^{\Lambda_{1}}_{J_1}( \xi(\mathfrak{t}_2)) }{2^{j_{\ell+1}+\cdots+j_k}}
\Psi_{J}^{\Lambda'_1,\mathrm{major}}( \xi(\mathfrak{t}_2)) \right|L^{\Lambda',\rm{minor}}_{J_2}(\xi_{\Lambda'})\lesssim
 2^{-cj_{\ell+1}}.
\end{align} 
\end{proposition}
\begin{proof}[Proof of Proposition \ref{prop5167b9}] 
We shall extend the two parameter proof   (\ref{283b})  in Section \ref{sec5.3} to that of $k$-parameters.
As in (\ref{pik3}),  
observe that  $2^{-J\cdot (m,0)} 2^{j_k(\rho/100) } \ll 2^{-10Kj_{\ell+1} }$ because $m=(m_1,\cdots,m_\ell) \in \Lambda'_1$ is nonzero and $j_{\ell}\gg_{\Lambda} j_{\ell+1}$. Thus, one has
\begin{align}\label{b390}
\Psi_{J}^{\Lambda'_1,\mathrm{major}}( \xi(\mathfrak{t}_2)) \le L^{\Lambda_{1}',\rm{major}}_{J_2}(\xi(\mathfrak{t}_2) ).
\end{align}
On the other hand, in view of Remark \ref{yu1}, one has
\begin{align*} 
   \left(\sup_{|\mathfrak{t}_2| \sim 2^{J_2} } \sum_{J_1\in \mathbb{Z}_+^{\ell} }|H^{\Lambda_{1}}_{J_1}(\xi(\mathfrak{t}_2)) |   \right)\lesssim 1.
\end{align*} 
Combining this with (\ref{b390}), for a fixed $J_2=(j_{\ell+1},\cdots,j_k)$, we can obtain that
\begin{align} \label{v011}
\text{ LHS of (\ref{44gg9})} \lesssim \frac{1}{2^{j_{\ell+1}+\cdots+j_k}} \sum_{|\mathfrak{t}_{2}|\sim 2^{J_2}}  
L^{\Lambda_{1}',\rm{major}}_{J_2}(\xi(\mathfrak{t}_2) )L^{\Lambda',\rm{minor}}_{J_2}(\xi_{\Lambda'}).
\end{align} 

Fix $1\le \ell \le  k-1$. For $s\ge \ell$, we set $m_{\le s}=(m_1,\cdots,m_s)$ and define
\begin{align*}
	 \Lambda_{\le s}:=\Lambda_{1\cdots s}:=\{m_{\le s}:m=(m_{\le s}, m_{s+1},\cdots,m_k)\in \Lambda \text{ for some } (m_{s+1},\cdots,m_k)\in\mathbb{Z}_{+}^{k-s}\},\\
	 \Lambda'_{\le s}:=\Lambda'_{1\cdots s}:=\{m_{\le s}:m=(m_{\le s}, m_{s+1},\cdots,m_k)\in \Lambda' \text{ for some } (m_{s+1},\cdots,m_k)\in\mathbb{Z}_{+}^{k-s}\}.
\end{align*}
From this notation,  the projection $$\Lambda_{1}=\{(m_1,\cdots,m_\ell):(m_1,\cdots,m_k)\in \Lambda \text{ for some } (m_{\ell+1},\cdots,m_k)\in\mathbb{Z}_{+}^{k-\ell}\}$$ in (\ref{11pp}) can  be denoted as $\Lambda_{\le \ell}$. 
Moreover,  we rewrite $\xi_{m_1\cdots m_s}(t_{s+1},\cdots, t_k)$ as
\begin{align}
\xi_{m_{\le s}}(t_{s+1},\cdots,t_k)&=\sum_{(m_{s+1},\cdots,m_k)\in \Lambda(m_{\le s})} \xi_{m_{\le s} m_{s+1}\cdots m_k} t_{s+1}^{m_{s+1}}\cdots t_k^{m_k}\ \text{and}\nonumber\\
  \xi(t_{s+1},\cdots,t_k)&=(\xi_{m_{\le s} }(t_{s+1},\cdots, t_k))_{m_{\le s}\in \Lambda_{\le s}}.\label{sk41}
\end{align}
Accordingly, we can rewrite $L^{\Lambda_{1}',\rm{major}}_{J_2}(\xi(\mathfrak{t}_2) )$ and $L^{\Lambda',\rm{major}}_{J_2}(\xi_{\Lambda'})$ in (\ref{88332}) as
\begin{align*}
 L^{\Lambda_{\le \ell}',\rm{major}}_{J_2}(\xi(t_{\ell+1},\cdots,t_k) )&:=\sum_{b/p\in \mathbb{Q}^{\Lambda'_{\le \ell}}[1,2^{j_{k}/10)} ]}\psi\left(\frac{\xi(t_{\ell+1},\cdots,t_k)-\frac{b}{p}}{ 2^{-10Kj_{\ell+1} } }\right),\nonumber \\
 L^{\Lambda_{\le k}',\rm{major}}_{J_2}(\xi_{\Lambda'} )&:=\sum_{b/p\in \mathbb{Q}^{\Lambda'_{\le k}}[1,2^{N^{4(k-\ell)}j_{\ell+1}} ]}\psi\left(\frac{\xi_{\Lambda'}-\frac{b}{p}}{ 2^{-10Kj_{\ell+1} +N^{4 (k-\ell)}j_{\ell+1}} }\right),
\end{align*}
respectively.
Then, to prove (\ref{44gg9}),  it suffices to show that  in (\ref{v011}), for  each $J_2$ with $J=(J_1,J_2)\in Z_{\ell}(k)$,
\begin{align} \label{rara}
  \frac{1}{2^{j_{\ell+1}+\cdots+j_k}}  \sum_{|\mathfrak{t}_{2}|\sim 2^{J_2}}  
L^{\Lambda_{\le \ell}',\rm{major}}_{J_2}(\xi(t_{\ell+1},\cdots,t_k) )L^{\Lambda',\rm{minor}}_{J_2}(\xi_{\Lambda'})\lesssim 2^{- j_{k}}.
\end{align} 
Define the cutoff functions  for the frequency  $\xi(t_{s+1},\cdots,t_k)$ where $\ell+1\le  s\le k-1$ in (\ref{sk41}),
\begin{align*} 
 L^{\Lambda_{\le s}',\rm{major}}_{J_2}(\xi(t_{s+1},\cdots,t_k) )&:=\sum_{b/p\in \mathbb{Q}^{\Lambda'_{\le s}}[1,2^{N^{4(s-\ell)}j_{\ell+1}} ]}\psi\left(\frac{\xi(t_{s+1},\cdots,t_k)-\frac{b}{p}}{ 2^{-10Kj_{\ell+1} +N^{4 (s-\ell)}j_{\ell+1}} }\right),\\
 L^{\Lambda_{\le s}',\rm{minor}}_{J_2}(\xi(t_{s+1},\cdots,t_k) )&:= 1-L^{\Lambda_{\le s}',\rm{major}}_{J_2}(\xi(t_{s+1},\cdots,t_k) ).\nonumber
\end{align*}
Then, we shall obtain (\ref{rara}) by proving Claims \ref{cl61} and \ref{cl62} below.
\begin{claim}\label{cl61} Let $ \ell \in [k-1]$. Then the estimate can be reduced to a sequence of consecutive layered cases:
\begin{align} \label{718m}
& \sum_{|\mathfrak{t}_{2}|\sim 2^{J_2}}  
L^{\Lambda_{\le \ell}',\rm{major}}_{J_2}(\xi(t_{\ell+1},\cdots,t_k) )L^{\Lambda',\rm{minor}}_{J_2}(\xi_{\Lambda'}) \\
&\qquad\lesssim  \sum_{s=\ell}^{k-1}  \sum_{|\mathfrak{t}_2|\sim 2^{J_2}} L^{\Lambda_{\le s}',\rm{major}}_{J_2}(\xi(t_{s+1},\cdots,t_k) )L^{\Lambda_{\le s+1}',\rm{minor}}_{J_2}(\xi(t_{s+2},\cdots,t_k) ).\nonumber
\end{align} 
Here, in the case $s' > k$, we identify $\xi(t_{s'}, \ldots, t_k)$ with $\xi_{\Lambda'}$.  
\end{claim}
\begin{proof}[Proof of (\ref{718m})]
Note that  $L^{\Lambda_{\le s}',\rm{major}}_{J_2}(\xi(t_{>s}) )\le 1$ for every $\ell\le s\le k-1$. Then, by using the notation $t_{>s}=(t_{s+1},\cdots,t_k)$,  we have
\begin{align*}
1&=\prod_{s=\ell+1}^{k-1}\left(L^{\Lambda_{\le s}',\rm{minor}}_{J_2}(\xi(t_{>s}) +L^{\Lambda_{\le s}',\rm{major}}_{J_2}(\xi(t_{>s}) )\right)  \le  L^{\Lambda_{\le \ell+1}',\rm{minor}}_{J_2}(\xi(t_{>\ell+1}) )\\
&+L^{\Lambda_{\le \ell+1}',\rm{major}}_{J_2}(\xi(t_{>\ell+1}) )\prod_{s=\ell+2}^{k-1}\left(L^{\Lambda_{\le s}',\rm{minor}}_{J_2}(\xi(t_{>s}) )+ L^{\Lambda_{\le s}',\rm{major}}_{J_2}(\xi(t_{>s}) )\right)\\
&\le L^{\Lambda_{\le \ell+1}',\rm{minor}}_{J_2}(\xi(t_{>\ell+1}) )+ L^{\Lambda_{\le \ell+1}',\rm{major}}_{J_2}(\xi(t_{>\ell+1}) )L^{\Lambda_{\le \ell+2}',\rm{minor}}_{J_2}(\xi(t_{>\ell+2}) )\\
&+\cdots+ L^{\Lambda_{\le k-2}',\rm{major}}_{J_2}(\xi(t_{>k-2}) )L^{\Lambda_{\le k-1}',\rm{minor}}_{J_2}(\xi(t_{>k-1}) )+ L^{\Lambda_{\le k-1}',\rm{major}}_{J_2}(\xi(t_{>k-1}) ).
\end{align*} 
We insert the last term of the above inequality into the  LHS of (\ref{718m}). Using  the bounds $ L^{\Lambda_{\le \ell}',\rm{major}}_{J_2}(\xi(t_{\ell+1},\cdots,t_k) )\le 1$ and  $L^{\Lambda_{\le k}',\rm{minor}}_{J_2}(\xi_{\Lambda'})=L^{\Lambda',\rm{minor}}_{J_2}(\xi_{\Lambda'})\le 1$, we obtain the desired bound.
\end{proof}
In view of \eqref{718m}, in order to prove \eqref{rara}, it suffices to establish the estimate for each $s=\ell,\dots,k-1$ and for every fixed choice of $t_{s+2},\dots,t_k$,
\begin{align}\label{44kk9}
\sum_{|t_{s+1}|\sim 2^{j_{s+1}}}
L^{\Lambda_{\le s}',\mathrm{major}}_{J_2}\bigl(\xi(t_{s+1},\dots,t_k)\bigr)
L^{\Lambda_{\le s+1}',\mathrm{minor}}_{J_2}\bigl(\xi(t_{s+2},\dots,t_k)\bigr)
\le C_{\Lambda},
\end{align}
  where $C_{\Lambda}$ depends only on $\Lambda$. 
 Before proving (\ref{44kk9}), for $\ell \le s\le k-1$ and
a given $m_{\le s}$, 
we  define the 1 dimensional fiber $$\Lambda_{s+1}(m_{\le s}):=\{m_{s+1}: m_{\le s+1}=(m_1,\cdots,m_s,m_{s+1})\in \Lambda_{\le s+1}\}.$$
Consider the vector polynomial  in \eqref{sk41},
\[
\xi(t_{s+1},\dots,t_k)
= \bigl(\xi_{m_{\le s}}(t_{s+1},\dots,t_k)\bigr)_{m_{\le s}\in \Lambda'_{\le s}}.
\]
Then for each fixed $(t_{s+2},\dots,t_k)$, we view this as a vector-valued polynomial in the variable $t_{s+1}$ whose exponents belong to $\Lambda_{s+1}(m_{\le s})$. More precisely,
\begin{align*}
\xi_{m_{\le s}}(t_{s+1},\dots,t_k)
&= \sum_{m_{s+1}\in \Lambda_{s+1}(m_{\le s})}
\xi_{m_{\le s}m_{s+1}}(t_{s+2},\dots,t_k)\, t_{s+1}^{m_{s+1}},
\end{align*}
where $m_{\le s}m_{s+1} = m_{\le s+1}$.
Then    denote the corresponding coefficient vector by
\[
\tilde{\xi} := (\tilde{\xi}_{m_{\le s}m_{s+1}}),
\quad \text{where } \tilde{\xi}_{m_{\le s}m_{s+1}}
= \xi_{m_{\le s}m_{s+1}}(t_{s+2},\dots,t_k),
\]
suppressing the dependence on the fixed variables $t_{s+2},\dots,t_k$. With this notation, we may write $
\xi(t_{s+1},\dots,t_k)$ in (\ref{sk41}) as the  vector polynomial $t_{s+1}$:
\begin{align*}
\tilde{\xi}(t_{s+1})
&:= \left(
\sum_{m_{s+1}\in \Lambda_{s+1}(m_{\le s})}
\tilde{\xi}_{m_{\le s}m_{s+1}}\, t_{s+1}^{m_{s+1}}
\right)_{m_{\le s}\in \Lambda'_{\le s}}.
\end{align*}
 
  Then to show (\ref{44kk9}), we claim that
\begin{align}\label{44kk99}
&\qquad \sum_{|t_{s+1}|\sim 2^{j_{s+1}}}  L^{\Lambda_{\le s}',\rm{major}}_{J_2}(\tilde{\xi}(t_{s+1}) )L^{\Lambda_{\le s+1}',\rm{minor}}_{J_2}(\tilde{\xi})\lesssim |\Lambda|,
\end{align}
where as in (\ref{i44}), there exists $\frac{b(t_{s+1})}{p(t_{s+1})}\in \mathbb{Q}^{\Lambda_{\le s}'}[1,2^{N^{4(s-\ell)}j_{\ell+1}}]$ such that
\begin{align*}
&L^{\Lambda_{\le s}',\rm{major}}_{J_2}(\tilde{\xi}(t_{s+1}) )
  = 
 \psi\!\left(
    \frac{\tilde{\xi}(t_{s+1})-\frac{b(t_{s+1})}{p(t_{s+1})}}{2^{-10K j_{\ell+1}+N^{4(s-\ell)}j_{\ell+1}}}
  \right).
\end{align*}
Suppose that $$\sum_{|t_{s+1}|\sim 2^{j_{s+1}}}  
   \psi\!\left(
    \frac{\tilde{\xi}(t_{s+1})-\frac{b(t_{s+1})}{p(t_{s+1})}}{2^{-10K j_{\ell+1}+N^{4(s-\ell)}j_{\ell+1}}}
  \right)< 2|\Lambda|,$$  then   (\ref{44kk99}) holds.  Hence to complete the proof of (\ref{44kk99}), there remains to show the following claim.

\begin{claim}\label{cl62} 
Let $\ell \le s\le k-1$. Suppose that
\begin{align*} 
&\sum_{|t_{s+1}|\sim 2^{j_{s+1}}}  
   \psi\!\left(
    \frac{\tilde{\xi}(t_{s+1})-\frac{b(t_{s+1})}{p(t_{s+1})}}{2^{-10K j_{\ell+1}+N^{4(s-\ell)}j_{\ell+1}}}
  \right)\ge 2|\Lambda|\ \text{where}\  \frac{b(t_{s+1})}{p(t_{s+1})}\in \mathbb{Q}^{\Lambda_{\le s}'}[1,2^{N^{4(s-\ell)}j_{\ell+1}}].
\end{align*}
Then one can find a rational vector $a/q$ such that
\begin{align*}
 \psi\left(\frac{\tilde{\xi}-\tilde{a}/\tilde{q}}{2^{-10Kj_{\ell+1}+ N^{4(s-\ell+1)}j_{\ell+1} }  }\right)= 1\ \text{where}\  \tilde{a}/\tilde{q}\in \mathbb{Q}^{\Lambda_{\le s+1}'}[1,2^{N^{4(s-\ell+1)}j_{\ell+1}}].
 \end{align*}
This implies $L^{\Lambda_{\le s+1}',\rm{minor}}_{J_2}(\tilde{\xi})=1-L^{\Lambda_{\le s+1}',\rm{major}}_{J_2}(\tilde\xi )=0$ in the left hand side of (\ref{44kk99}).
\end{claim}
\begin{proof}[Proof of  Claim \ref{cl62}]
 In Lemma  \ref{vari61}, let
 $\tilde{\xi}(t_{s+1})$ play the same role as $\xi(t_2)$ and replace the parameters 
 $$\ell',j_2 \ \text{and} \ p(t_2)\in [1,2^{N^{4\ell'}j_2}] \ \text{ and }\  q\in  [1,2^{N^{4(\ell'+1)}j_2}] \ \text{in Lemma \ref{vari61}}$$
respectively by
 $$(s-\ell), j_{\ell+1}  \ \text{and} \ p(t_{s+1})\in[1,2^{N^{4(s-\ell)}j_{\ell+1}}] \  \text{and}\   \tilde{q}\in [1, 2^{N^{4(s-\ell+1)} j_{\ell+1} }] \ \text{  in Claim \ref{cl62}.}$$    Then the desired conclusion follows.
\end{proof}
Therefore,   the   proof of Proposition \ref{prop5167b9} is complete.
 \end{proof}

\subsection{Averaged Gauss Sum, Sublevel Set, Oscillation Estimates}\label{5.3}
We now estimate the third term in~(\ref{479aa9}). To this end, we follow the argument employed in the proof of \eqref{283c} in Section~\ref{Sec5.4}.

 For $J\in Z_{\ell}(k)$ and $\xi(\mathfrak{t}_2)=(\xi_{m}(\mathfrak{t}_2))_{m\in \Lambda_1'}$,  set
	\begin{align*} 
		\mathcal{H}_{J_1}^{\Lambda'_1}(\xi(\mathfrak{t}_2) ) &:=\int_{|\mathfrak{t}_1|\sim 2^{J_1}} \frac{e^{2\pi i \sum_{m\in \Lambda'_1} \xi_{m}(\mathfrak{t}_2)\mathfrak{t}_1^{m}} }{\mathfrak{t}_1^{{\bf 1}_1}}.
	\end{align*} 
	Consider a rational vector: 
	$
	\frac{a}{q}=\frac{(a_{\mathfrak{m}})_{\mathfrak{m}\in \Lambda}}{q}  \ \text{where}  \ (a,q)=1.$ We then
	define a vector field, depending on $\mathfrak{t}_2=(t_{\ell+1},\dots,t_k)$ by
	\[
	\frac{a(\mathfrak{t}_2)}{q}:=\left( \frac{a_{m}(\mathfrak{t}_2)}{q} \right)_{m\in \Lambda_{1}'},
	\]
	where   for each $m=(m_1,\cdots,m_{\ell})\in \Lambda_1'$,  the layered polynomial is defined by
	\[
	a_{m}(\mathfrak{t}_2):=\sum_{n \in \Lambda'(m)} a_{mn}\,\mathfrak{t}_2^{n}, 
	\qquad \text{with } \ \mathfrak{t}_2^n=t_{\ell+1}^{m_{\ell+1}}\cdots t_k^{m_k}.
	\]
	
	\begin{lemma}\label{kw999}
Suppose that $\Lambda$ satisfies   the  condition (\ref{wep2}).  Fix $J_2=(j_{\ell+1},\ldots,j_k)$ and
$\mathfrak{t}_2\in\mathbb{Z}^{k-\ell}$ satisfying
$|\mathfrak{t}_2|\sim2^{J_2}$.
Let $\xi$ be fixed and let $a/q \in \mathbb{Q}^{\Lambda'}\!\left[1,2^{N^{4(k-\ell)} j_{\ell+1}}\right]$. 
For $J\in Z_{\ell}(k)$, assume the support condition
\begin{align}\label{um199}
\sum_{m\in \Lambda_1'}\left| 
\sum_{n\in\Lambda(m)}
\left( \xi_{mn}-\frac{a_{mn}}{q} \right)\mathfrak{t}_2^n \, 2^{m\cdot J_1} 
\right|
\lesssim 2^{j_{\ell}/10}.
\end{align}
Then one has the following approximation by a continuous multiplier:
\begin{align}\label{um1990}
H^{\Lambda'_{1}}_{J_1}( \xi(\mathfrak{t}_2) )
&=  S^{\Lambda'_{1}}\!\left( \left( \frac{a_{m}(\mathfrak{t}_2)}{q}\right)_{m\in \Lambda'_{1}}\right)  
\mathcal{H}^{\Lambda'_{1}}_{J_1}\!\left(\left(\xi_{m}(\mathfrak{t}_2)-\frac{a_{m}(\mathfrak{t}_2)}{q}\right)_{m\in \Lambda'_{1}}\right) \notag \\
&\quad + E_{J_1}(a/q,\xi,\mathfrak{t}_2),
\end{align}
where the error term satisfies
\[
\sum_{J_1} \ \sum_{a/q \in \mathbb{Q}^{\Lambda'}[1,2^{N^{4(k-\ell)} j_{\ell+1}}]}  
|E_{J_1}(a/q,\xi,\mathfrak{t}_2)|
= O\bigl(2^{-c j_{\ell+1}}\bigr)
\]
for some $c>0$.
\end{lemma}
	 \begin{proof}
The  result follows from the same reasoning as in the proof of part~(1) of Proposition~\ref{lemm29}.
\end{proof}  

We shall estimate the third term of (\ref{479aa9}) by showing the following proposition. 
\begin{proposition}\label{prop8139}
Suppose that $\Lambda$ satisfies   the  condition (\ref{wep2}). There  is $c>0$  independent of $J_2$ and $\xi$ such that
\begin{align*}
 \sum_{J_1: J\in Z_\ell(k)} \sum_{\mathfrak{t}_2\sim 2^{J_2}} \left|\frac{H^{\Lambda_{1}}_{J_1}( \xi(\mathfrak{t}_2)) }{2^{J_2\cdot {\bf 1}_2} }\right|  \left(\Psi_{J}^{\Lambda'_1,\mathrm{major}}(\xi(\mathfrak{t}_2)) L^{\Lambda',\rm{major}}_{J_2}(\xi_{\Lambda'})-\Psi^{\Lambda',\rm{major}}_{J}(\xi_{\Lambda'})  \right) \lesssim 2^{-cj_{\ell+1}}.
\end{align*}
\end{proposition}
 
\begin{proof}[Proof of Proposition \ref{prop8139}.]
Let $m=(m_1,\dots,m_{\ell})$ and $n=(m_{\ell+1},\dots,m_k)$.  
On the product of the following two summations:
\begin{align*}
\Psi_{J}^{\Lambda'_1,\mathrm{major}}\!\left(\xi(\mathfrak{t}_2)\right)
&= 
\sum_{b/p \in \mathbb{Q}^{\Lambda'_1}[1, 2^{j_k/10}]}  
\prod_{m\in \Lambda_{1}'}
\psi\!\left(
\frac{\xi_{m}(\mathfrak{t}_2)-b_{m}/p}{\,2^{-m\cdot J_1} 2^{j_k\rho/100}}
\right), \\
L^{\Lambda',\mathrm{major}}_{J_2}(\xi_{\Lambda'})
&=
\sum_{a/q \in \mathbb{Q}^{\Lambda'}[1,2^{N^{4(k-\ell)}\,j_{\ell+1}}]}
\psi\!\left(
\frac{\xi_{\Lambda'}-a/q}{\,2^{-10K j_{\ell+1} + N^{4(k-\ell)} j_{\ell+1}}}
\right),
\end{align*}
one can observe  that
 $1 \le p,q \le 2^{K j_{\ell+1}}$  and
\begin{align*}
\frac{b_{m}}{p}
- \frac{\sum_{n\in \Lambda(m)} a_{mn}\,\mathfrak{t}_2^{n}}{q} 
&=
\left(
\frac{b_{m}}{p}
-
\sum_{n\in \Lambda(m)} \xi_{mn}\mathfrak{t}_2^{n}
\right)
+
\sum_{n\in \Lambda(m)}
\left(\xi_{mn}-\frac{a_{mn}}{q}\right)\mathfrak{t}_2^{n}  \\
&= O(2^{-9K j_{\ell+1}}).
\end{align*}
Hence for every \(m \in \Lambda_1'\), the following identity holds
$
\frac{b_m}{p}
=
\frac{\sum_{n\in \Lambda(m)} a_{mn}\,\mathfrak{t}_2^{n}}{q}.
$
So,  
\begin{align*}
\Psi_{J}^{\Lambda'_1,\mathrm{major}}\!\left(\xi(\mathfrak{t}_2)\right)
\, &L^{\Lambda',\mathrm{major}}_{J_2}(\xi_{\Lambda'})=\sum_{a/q \in \mathbb{Q}^{\Lambda'}[1,2^{N^{4(k-\ell)}j_{\ell+1}}]}    \chi_{ \mathbb{Q}^{\Lambda'_1}[1, 2^{j_k/10}]  }(a(\mathfrak{t}_2)/q)\\
&\times \psi\!\left(
\frac{\xi_{\Lambda'}-a/q}{\,2^{-10K j_{\ell+1}+ N^{4(k-\ell)}j_{\ell+1}}}
\right)
\prod_{m\in \Lambda'_1}
\psi\!\left(
\frac{
\sum_{n\in\Lambda(m)}
\left( \xi_{mn}-\frac{a_{mn}}{q}\right)
\mathfrak{t}_2^{n}
}{
2^{-m\cdot J_1}\,2^{j_k\rho/100}
}
\right).
\end{align*}
As in Section \ref{Sec5.4},  by combining
$$
\Psi^{\Lambda',\rm{major}}_{J}(\xi_{\Lambda'})=\sum_{ a/q\in \mathbb{Q}^{\Lambda'}[1,2^{j_k/10}]} \prod_{(m,n)\in \Lambda'} \psi\left(  \frac{  \xi_{mn}-\frac{a_{mn}}{q}}{2^{-J\cdot (m,n)}2^{j_k\rho/10}} \right)\psi\left(\frac{\xi_{\Lambda'}-a/q}{2^{-10Kj_{\ell+1}+N^{4(k-\ell)}j_{\ell+1} }}\right) ,$$    majorize the LHS of the inequality in Proposition \ref{prop8139} into the following three terms:
\begin{align*}
	&\sum_{J_1: J\in Z_{\ell}(k)}  \sum_{|\mathfrak{t}_2|\sim 2^{J_2}}    \frac{|H^{\Lambda_1}_{J_1}(\xi(\mathfrak{t}_2))|}{|\mathfrak{t}_2^{{\bf 1}_2}|}  \left(\Psi^{\Lambda_1',\rm{major}}_{J}(\xi(\mathfrak{t}_2))L^{\Lambda',\rm{major}}_{J_2}(\xi_{\Lambda'})-\Psi^{\Lambda',\rm{major}}_{J}(\xi_{\Lambda'})\right) \\
	&\le \sum_{J_1: J\in Z_{\ell}(k)} E_J^{\rm{gauss}}(\xi)+E_J^{\rm{sub}}(\xi) +E_J^{\rm{osc}}(\xi) 
\end{align*}
where     we define $E_J^{\rm{gauss}}(\xi)$, $E_J^{\rm{sub}}(\xi)$ and $E_J^{\rm{osc}}(\xi)$ by 
\begin{align*}
	E^{\rm{gauss}}_J(\xi)&:= \sum_{ a/q\in \mathbb{Q}^{\Lambda'}[2^{j_k/10},2^{N^{4(k-\ell)} j_{\ell+1} }]}  \sum_{\mathfrak{t}_2\sim  2^{J_2}}    \frac{|H^{\Lambda_{1}}_{J_1}(\xi(\mathfrak{t}_2))|}{2^{J_2\cdot {\bf 1}_2}} \nonumber  \\
	&  \times \prod_{m\in\Lambda'_{1}}\psi\left( \frac{\sum_{n\in\Lambda(m)}\left( \xi_{mn}-\frac{a_{mn}}{q} \right)\mathfrak{t}_2^n }{2^{-m\cdot J_1} 2^{ j_k\rho/100}} \right)\psi\left(\frac{\xi_{\Lambda'}-a/q}{2^{-10Kj_{\ell+1}+N^{4(k-\ell)}j_{\ell+1}}}\right),\nonumber\\
	E^{\rm{sub}}_J(\xi)&:=  \sum_{ a/q\in \mathbb{Q}^{\Lambda'}[1,2^{j_k/10}]}   \sum_{\mathfrak{t}_2\sim  2^{J_2}}    \frac{|H^{\Lambda_{1}}_{J_1}(\xi(\mathfrak{t}_2))|}{2^{J_2\cdot {\bf 1}_2}}  \prod_{m\in\Lambda'_{1}}\psi\left( \frac{\sum_{n\in\Lambda(m)}\left( \xi_{mn}-\frac{a_{mn}}{q} \right)\mathfrak{t}_2^n }{2^{-m\cdot J_1} 2^{ j_k\rho/100}} \right)
	\nonumber \\
	&\times   \left(1-\prod_{(m,n)\in\Lambda'}\psi\left(  \frac{  \xi_{mn}-\frac{a_{mn}}{q}}{ 2^{-J\cdot (m,n)}2^{j_k\rho/10}} \right)\right)\psi\left(\frac{\xi_{\Lambda'}-a/q}{2^{-10Kj_{\ell+1}+N^{4(k-\ell)}j_{\ell+1} }}\right),\\
		E^{\rm{osc}}_J(\xi)&:=   \sum_{ a/q\in \mathbb{Q}^{\Lambda'}[1,2^{j_k/10}]}  \sum_{\mathfrak{t}_2\sim  2^{J_2}}    \frac{|H^{\Lambda_{1}}_{J_1}(\xi(\mathfrak{t}_2))|}{2^{J_2\cdot {\bf 1}_2}} \left(  \prod_{m\in\Lambda'_{1}}\psi\left( \frac{\sum_{n\in\Lambda(m)}\left( \xi_{mn}-\frac{a_{mn}}{q} \right)\mathfrak{t}_2^n }{2^{-m\cdot J_1} 2^{ j_k\rho/100}} \right) -1\right)\nonumber\\
	&\times \prod_{(m,n)\in \Lambda'} \psi\left(  \frac{  \xi_{mn}-\frac{a_{mn}}{q}}{ 2^{-J\cdot (m,n)}2^{j_k\rho/10}} \right)\psi\left(\frac{\xi_{\Lambda'}-a/q}{2^{-10Kj_{\ell+1}+N^4(k-\ell)j_{\ell+1} }}\right).
\end{align*}

Hence, Proposition \ref{prop8139} follows from  lemmas \ref{lem859}, \ref{lem57b9} and \ref{lem574b9} below, where we shall apply Gauss-sum estimate to $E^{\rm{gauss}}_J(\xi)$,  the sub-level set estimate to $E^{\rm{sub}}_J$ and the oscillatory integral estimate to $E^{\rm{osc}}_J$,   respectively.

 \begin{lemma}\label{lem859}[Estimate of $E^{\rm{gauss}}_J(\xi)$]
 Suppose that $\Lambda$ satisfies   the  condition (\ref{wep2}). There exists $c>0$ independent of $J_2$ and $\xi$ such that 
$$\sum_{J_1:J\in Z_{\ell}(k)} |E^{\rm{gauss}}_J(\xi)|\lesssim 2^{-cj_{\ell+1}}.$$
\end{lemma}
\begin{proof} 
Fix $\xi$. By using Lemma \ref{po3} and (\ref{um1990}),  one has up to $O(2^{-cj_{\ell+1}})$,
  \begin{align*}
\sum_{J_1:J\in Z_{\ell}(k)} E^{\rm{gauss}}_J(\xi)
&\lesssim \sum_{a/q\in  \mathbb{Q}^{\Lambda'}[2^{j_k/10}, 2^{N^{4(k-\ell)}j_{\ell+1} }] } \frac{1}{2^{J_2\cdot {\bf 1}_2}}\sum_{|\mathfrak{t}_2|\sim 2^{J_2}}\left|S^{\Lambda_{1}'}\left( \frac{a(\mathfrak{t}_2)}{q} \right)\right| \\
 &\times\sup_{\mathfrak{t}_2,\xi} \sum_{J_1:J\in Z_\ell(k)}\left|\mathcal{H}^{\Lambda_{1}'}_{J_1}\left( \xi(\mathfrak{t}_2)-\frac{a(\mathfrak{t}_2)}{q}  \right)\right|\psi\left(\frac{\xi_{\Lambda'}-\frac{a}{q}}{2^{-9Kj_{\ell+1}}}\right)\\
&\lesssim \sum_{a/q\in  \mathbb{Q}^{\Lambda'}[2^{j_k/10},2^{N^{4(k-\ell)}j_{\ell+1}} ]}  q^{-\delta_{\Lambda}}2^{-\delta_{\Lambda}j_{\ell+1}}\psi\left(\frac{\xi_{\Lambda'}-\frac{a}{q}}{2^{-9 Kj_{\ell+1}}}\right)  \lesssim   2^{-cj_{\ell+1}}.
 \end{align*}  
where the second inequality follows from the averaged Gauss sum estimate in \eqref{aver}. 
\end{proof}

\begin{lemma}\label{lem57b9}[Estimate of $E_J^{\rm{sub}}$] Suppose that $\Lambda$ satisfies   the  condition (\ref{wep2}). There is $c>0$ independent of $J_2$ and $\xi$ such that 
\begin{align*}
	\sum_{J_1:J\in Z_{\ell}(k)} |E^{\rm{sub}}_J(\xi)| \lesssim 2^{-cj_{\ell+1}}.
\end{align*}
\end{lemma}
\begin{proof}
Fix $\xi$. For each $m\in \Lambda_1'$,  we denote
	$$\eta_{mn}:=\xi_{mn}-\frac{a_{mn}}{q}\ \text{and}\ R_{m,J_2}(\eta):=\sum_{n\in\Lambda(m)} |\eta_{mn}|\, 2^{J_2 \cdot n}. $$ Then as in (\ref{43ut}), majorize the cutoff for $\mathfrak{t}_2$ depending on $J_1$ by  
	\begin{align*} 
		&\prod_{m\in\Lambda_1'}
		\psi\!\left(
		\frac{\sum_{n\in\Lambda(m)} \eta_{mn}\, \mathfrak{t}_2^n}
		{2^{-m\cdot J_1}\, 2^{j_{k}\rho /100}}
		\right)
		\left(
		1-
		\prod_{(m,n)\in\Lambda'}
		\psi\!\left(
		\frac{\eta_{mn}\, 2^{J_2\cdot n}}
		{2^{-m\cdot  J_1}\, 2^{j_k\rho /10}}
		\right)
		\right) \\
&\qquad \le
		\sum_{m\in\Lambda_1'}
		\psi\!\left(
		\frac{
			\displaystyle
			\frac{\sum_{n\in\Lambda(m)} \eta_{mn}\, \mathfrak{t}_2^n}
			{2R_{m,J_2}(\eta)}
		}
		{ 2^{j_k\rho  (1/100-1/10)}}
		\right),\nonumber
	\end{align*}
	which is independent of $J_1$. 
We now apply the lattice sublevel set estimate of Lemma \ref{eep1} to the $(k-\ell)$ parameter polynomial for  each  $m\in \Lambda_1'$. Note that $$ \sum_{n\in\Lambda(m)} \left|\frac{\eta_{mn}}{R_{m,J_2}(\eta)}\right| \, 2^{J_2\cdot n}=1.$$ Then,
	with $\epsilon=4\cdot 2^{j_k\rho  (1/100-1/10)} $  and $2^r\epsilon =1$ where $j_k\ge r=(1/10-1/100)\rho  j_k-2$ in Lemma \ref{eep1},  there exists $c'>0$ such that
	$$ \sum_{|\mathfrak{t}_2|\sim 2^{J_2}}   
	\psi\!\left(
	\frac{
		\sum_{n\in\Lambda(m)} \left(\frac{\eta_{mn}}{R_{m,J_2}(\eta)} \right)\mathfrak{t}_2^n
	}
	{ 2\cdot 2^{j_k \rho  (1/100-1/10)}}
	\right)\lesssim 2^{J_2\cdot {\bf 1}_2- c' j_k}. $$ Moreover, by Remark \ref{yu1}, one has $$ \sup_{|\mathfrak{t}_2|\sim 2^{J_2}} \sum_{J_1:J\in Z_\ell(k)}\left| H^{\Lambda_{1}'}_{J_{1}}\left(\xi(\mathfrak{t}_2) \right)\right|\lesssim 1.$$ Therefore,    we can obtain that
	\begin{align*}
	\sum_{J_1:J\in Z_{\ell}(k)} |E^{\rm{sub}}_J(\xi)| &\lesssim  \frac{1}{2^{J_2\cdot {\bf 1}_2}}
	 \sum_{ a/q\in \mathbb{Q}^{\Lambda'}[1,2^{j_k/10}]}  	\sum_{m\in\Lambda_1'} \sum_{|\mathfrak{t}_2|\sim  2^{J_2}}    
		\psi\!\left(
		\frac{
			\sum_{n\in\Lambda(m)} \left(\frac{\eta_{mn}}{R_{m,J_2}(\eta)} \right)\mathfrak{t}_2^n
		}
		{ 2\cdot 2^{j_k\rho  (1/100-1/10)}}
		\right)\\
&\times \psi\left(\frac{\xi_{\Lambda'}-\frac{a}{q}}{2^{-9Kj_{\ell+1}}}\right)
		 \sup_{|\mathfrak{t}_2|\sim 2^{J_2}} \sum_{J_1:J\in Z_\ell(k)}\left| H^{\Lambda_{1}'}_{J_{1}}\left(\xi(\mathfrak{t}_2) \right)\right|
		 \\
		& \lesssim 2^{-c'j_k}\sum_{ a/q\in \mathbb{Q}^{\Lambda'}[1,2^{j_k/10}]}\psi\left(\frac{\xi_{\Lambda'}-\frac{a}{q}}{2^{-9Kj_{\ell+1}}}\right)\lesssim 2^{-cj_{\ell+1}}. \nonumber
	\end{align*}
For the last inequality, we used the disjointness of the rationals
\( a/q\in \mathbb{Q}^{\Lambda'}[1,2^{j_k/10}] \)
from the support of
\( \psi\!\left(\frac{\xi_{\Lambda'}-\frac{a}{q}}{2^{-9Kj_{\ell+1}}}\right) \) and the comparability condition
\( j_{\ell+1}\approx_{\Lambda}\cdots \approx_{\Lambda} j_{k} \).
	This completes a proof of Lemma \ref{lem57b9}.\end{proof}

\begin{lemma}\label{lem574b9}[Estimate of $E^{\rm{osc}}_J(\xi)$] Suppose that $\Lambda$ satisfies   the  condition (\ref{wep2}). There is $c>0$ independent of $J_2$ and $\xi$ such that
$$\sum_{J_1:J\in Z_\ell(k)} |E^{\rm{osc}}_J(\xi)|\lesssim 2^{-cj_{\ell+1}}.$$
\end{lemma}
\begin{proof}  
 Note that for fixed $\xi$ and $a/q$ and $J\in Z_\ell(k)$, the identity $$\prod_{(m,n)\in \Lambda'} \psi\left(  \frac{  \xi_{mn}-\frac{a_{mn}}{q}}{ 2^{-J\cdot (m,n)}2^{j_k\rho/10}} \right)=1$$ implies that the support condition above  (\ref{um199}) is satisfied.
Hence, by applying  the integral approximation in Lemma \ref{kw999}, we have
\begin{align*} 
\sum_{J_1:J\in Z_{\ell}(k)} |E^{\rm{osc}}_{J}(\xi)| &\lesssim \sup_{|\mathfrak{t}_2|\sim 2^{J_2}} \sum_{ a/q\in \mathbb{Q}^{\Lambda'}[1,2^{j_k/10}]} \sum_{J_1:J\in Z_{\ell}(k)}       \prod_{(m,n)\in \Lambda'} \psi\left(  \frac{  \xi_{mn}-\frac{a_{mn}}{q}}{ 2^{-J\cdot (m,n)}2^{j_k\rho/10}} \right)  \nonumber \\
&\times \left(  \prod_{m\in\Lambda'_{1}}\psi\left( \frac{\sum_{n\in\Lambda(m)}\left( \xi_{mn}-\frac{a_{mn}}{q} \right)\mathfrak{t}_2^n }{2^{-m\cdot J_1} 2^{ j_k\rho/100}} \right) -1\right)\\
&\times \left|S^{\Lambda_{1}'}\left(\frac{a(\mathfrak{t}_2)}{q}\right)  \mathcal{H}^{\Lambda_{1}'}_{J_{1}}\left(\xi(\mathfrak{t}_2)-\frac{a(\mathfrak{t}_2)}{q}\right)\right| +O(2^{-cj_{\ell+1}}).\nonumber
\end{align*}
 By  using Lemma \ref{po3}  and the    oscillatory integral estimate together with   the support condition 
 $$ \sum_{m\in \Lambda_1'} \left|  \sum_{n\in\Lambda(m)}\left( \xi_{mn}-\frac{a_{mn}}{q} \right)\mathfrak{t}_2^n 2^{m\cdot J_1}    \right|\gtrsim  2^{ j_k\rho/100},$$ we have $c'>0$ such that $$ \sum_{J_1:j\in Z_{\ell}(k)} |E^{\rm{osc}}_{J}(\xi)|\lesssim 2^{-c'j_{k}}  \sum_{ a/q\in \mathbb{Q}^{\Lambda'}[1,2^{j_k/10}]}  \psi\left(\frac{\xi-\frac{a}{q}}{2^{-9Kj_{\ell+1}}}\right).$$    Combined with the bound $$ \sum_{ a/q\in \mathbb{Q}^{\Lambda'}[1,2^{j_k/10}]}  \psi\left(\frac{\xi-\frac{a}{q}}{2^{-9Kj_{\ell+1}}}\right)\lesssim 1,$$ we conclude that there exists $c>0$ such that
\begin{align*} 
\sum_{J_1:J\in Z_{\ell}(k)} |E^{\rm{osc}}_{J}(\xi)| \le 2^{-cj_{\ell+1}},
\end{align*}
under the condition $j_{\ell+1}\approx_{\Lambda}\cdots\approx_{\Lambda} j_k$.
This completes the proof of Lemma \ref{lem574b9}.
\end{proof}
Therefore, we have proved Proposition \ref{prop8139}.
\end{proof}

\subsection{Minor Arc Estimate for $\Lambda\setminus \Lambda'$}\label{5.4}
We  shall    treat the minor arc contribution associated with $\xi_0(\mathfrak{t}_2)=\sum_{(0,n)\in \Lambda\setminus \Lambda'} \xi_{0n}\mathfrak{t}_2^n$ to show   the fourth   term of (\ref{479aa9}). 
\begin{proposition}\label{prop7449}  
Suppose that $\Lambda$ satisfies   the  condition (\ref{wep2}). There exists $c>0$ independent of $J_2$ and  $\xi$ such that
$$\sum_{J_1:J\in Z_\ell(k)}\left|\sum_{|\mathfrak{t}_2|\sim 2^{J_2} }   \frac{H^{\Lambda_{1}}_{J_1}(\xi(\mathfrak{t}_2))}{\mathfrak{t}_2^{{\bf 1}_2} }  \left(\Psi^{\Lambda',\rm{major}}_{J}(\xi_{\Lambda'})-  \Psi^{\Lambda,\rm{major}}_{J}(\xi)  \right)\right|\lesssim 2^{-cj_{\ell+1}}.$$
\end{proposition}
To prove Proposition \ref{prop7449}, we define
\begin{align}\label{dom1}
\Psi_J^{\Lambda,\mathrm{major}}(\xi)
&:= \sum_{\frac{a}{q} \in \mathbb{Q}^{\Lambda}[1,2^{j_k/10}]}
\prod_{(m,n)\in \Lambda} 
\psi\left(
\frac{\xi_{mn}-\frac{a_{mn}}{q}}{2^{-J\cdot (m,n)}\,2^{j_k/10}}
\right),\\
\Psi_J^{\Lambda\setminus\Lambda',\mathrm{major}}(\xi_{\Lambda\setminus\Lambda'})
&:= \sum_{\frac{a}{q} \in \mathbb{Q}^{\Lambda\setminus\Lambda'}[1,2^{j_k/10}]}
\prod_{(m,n)\in \Lambda\setminus\Lambda'} 
\psi\left(
\frac{\xi_{mn}-\frac{a_{mn}}{q}}{2^{-J\cdot (m,n)}\,2^{j_k/10}}
\right).\nonumber
\end{align}
We emphasize that the major arcs associated with $\Psi^{\Lambda',\rm{major}}_{J}(\xi_{\Lambda'})$, defined in (\ref{4022k9}), are more tightly localized than those associated with $\Psi_J^{\Lambda,\mathrm{major}}(\xi)$, owing to the smaller scaling factor $2^{j_k\rho/10}$,  where  $\rho$ chosen  as in (\ref{rho3}). 

Decompose
\begin{align*}
\Psi_J^{\Lambda',\mathrm{major}}(\xi_{\Lambda'})
- \Psi_J^{\Lambda,\mathrm{major}}(\xi)
&=
\Psi_{J}^{\Lambda',\mathrm{major}}(\xi_{\Lambda'})
\Bigl(1 - \Psi_J^{\Lambda\setminus\Lambda',\mathrm{major}}(\xi_{\Lambda\setminus\Lambda'})\Bigr) \\
&\quad
+ \Psi_J^{\Lambda\setminus\Lambda',\mathrm{major}}(\xi_{\Lambda\setminus\Lambda'})\,
\Psi_{J}^{\Lambda',\mathrm{major}}(\xi_{\Lambda'})
- \Psi_J^{\Lambda,\mathrm{major}}(\xi).
\end{align*}
Proposition~\ref{prop7449} then follows from Lemmas~\ref{lee44b9} and~\ref{e7559} below, which estimate the two terms above, respectively.

\begin{lemma}\label{lee44b9}
[Minor arc for $\Lambda\setminus \Lambda'$] Suppose that $\Lambda$ satisfies   the  condition (\ref{wep2}). There exists $c>0$ independent of $J_2$ and  $\xi$ such that
\begin{equation*}
  \sum_{J_1:J\in Z_{\ell}(k)}   \left| \sum_{|\mathfrak{t}_2|\sim 2^{J_2} }   \frac{e^{2\pi i \xi_0(\mathfrak{t}_2)}H^{\Lambda'_{1}}_{J_{1}}(\xi(\mathfrak{t}_2))}{\mathfrak{t}_2^{{\bf 1}_2} }\right|  \Psi_{J}^{\Lambda',\rm{major}}(\xi_{\Lambda'})
\left(1-\Psi_J^{\Lambda\setminus \Lambda',\rm{major}}(\xi_{\Lambda\setminus\Lambda'})\right) \lesssim 2^{-cj_{\ell+1}}.
\end{equation*}
\end{lemma}
\begin{proof}[Proof of Lemma \ref{lee44b9}] 
Note that
\[
\Lambda\setminus \Lambda'
=
\{(m,n)\in \Lambda : m=0\}.
\]
Accordingly, we may write
\begin{align*}
\sum_{(m,n)\in\Lambda}\xi_{mn}\mathfrak{t}_1^m\mathfrak{t}_2^n
=
\sum_{m\in \Lambda_1'} \xi_m(\mathfrak{t}_2)\mathfrak{t}_1^m
+\xi_0(\mathfrak{t}_2),
\end{align*}
where $\Lambda\setminus \Lambda'$ is the set of the indices $(0,n)$ of the coefficients $\xi_{0n}$ in $\xi_0(\mathfrak{t}_2)$:
\begin{align}\label{0011}
\xi_m(\mathfrak{t}_2)
=
\sum_{n\in \Lambda(m)} \xi_{mn}\mathfrak{t}_2^n,
\qquad
\xi_0(\mathfrak{t}_2)
=
\sum_{n\in \Lambda(0)} \xi_{0n}\mathfrak{t}_2^n.
\end{align}

By the major arc condition on $(\xi_{mn})_{(m,n)\in\Lambda'}$, encoded in
$\Psi_{J}^{\Lambda',\mathrm{major}}(\xi_{\Lambda'})$, Lemma~\ref{kw999} gives an approximation of the discrete multiplier in the $\mathfrak{t}_1$-variables by its continuous counterpart. Therefore, it suffices to prove
\begin{align}\label{t8}
\sum_{\frac{a}{q}\in \mathbb{Q}^{\Lambda'}[1,2^{j_k/10}]}
\sum_{J_1:J\in Z_{\ell}(k)}
\Biggl|
\sum_{|\mathfrak{t}_2|\sim 2^{J_2}}
\frac{
e^{2\pi i \xi_0(\mathfrak{t}_2)}
S^{\Lambda_1'}\!\left(\frac{a(\mathfrak{t}_2)}{q}\right)
\mathcal{H}^{\Lambda_1'}_{J_1}\!\left(\xi(\mathfrak{t}_2)-\frac{a(\mathfrak{t}_2)}{q}\right)
}{
\mathfrak{t}_2^{{\bf 1}_2}
}
\Biggr|
\\
\qquad\qquad\times
\psi_J^{\Lambda'}\!\left(\frac{\xi_{\Lambda'}-a/q}{2^{j_k\rho/10}}\right)
\left(1-\Psi_J^{\Lambda\setminus \Lambda',\mathrm{major}}(\xi_{\Lambda\setminus\Lambda'})\right)
\lesssim 2^{-c j_{\ell+1}}.\notag
\end{align}

The key point is that the factor
\[
1-\Psi_J^{\Lambda\setminus \Lambda',\mathrm{major}}(\xi_{\Lambda\setminus\Lambda'})
\]
places $(\xi_{0n})_{(0,n)\in \Lambda\setminus\Lambda'}$ in the minor arcs. This allows us to exploit Weyl-type decay for variants of the exponential sum
\[
\sum_{|\mathfrak{t}_2|\sim 2^{J_2}} e^{2\pi i \xi_0(\mathfrak{t}_2)}.
\]
Combining this minor arc decay with the continuous approximation from Lemma~\ref{kw999}, we shall use Claims~\ref{ap90}, \ref{cl46}, and estimate~\eqref{t9} below to establish \eqref{t8}.

\begin{claim}\label{ap90} [Gauss-Weighted-Weyl Sum] 
Let  $J\in Z_{\ell}(k)$. For    $\frac{a}{q} \in  \mathbb{Q}^{\Lambda'}[1,2^{j_k/10}]$, and $\mathfrak{t}_2\sim 2^{J_2}$,  consider the following Gauss sum
$$ S^{\Lambda_1'}(a(\mathfrak{s})/q)=(1/q^{\ell})\sum_{t_1,\cdots,t_{\ell}=1}^q e^{2\pi i \sum_{m\in \Lambda_1'}\left(\sum_{n\in \Lambda(m)} \frac{a_{mn}}{q}\mathfrak{s}^n\right)\mathfrak{t}_1^{m}},$$  whose phase function depends on $\mathfrak{s}=(s_{\ell+1},\cdots,s_k)\sim 2^{J_2}$. Next, associated with the phase function  $ \xi_0(\mathfrak{s})=\sum_{(0,n)\in \Lambda} \xi_{0n}\mathfrak{s}^{n} $  in (\ref{0011}), we set the Gauss-weighted Weyl sum:
\begin{align}\label{kanj}
 W^S_{J_2,a/q,\xi_0}(\mathfrak{t}_2)& :=\sum_{s_k=2^{j_k-1}+1}^{t_k} \cdots \sum_{s_{\ell+1}=2^{j_{\ell+1}-1}+1}^{t_{\ell+1}} e^{2\pi i \xi_0(\mathfrak{s})}S^{\Lambda_{1}'}\left(\frac{a(\mathfrak{s})}{q}\right).
\end{align}
  Then we  claim that 
\begin{align}\label{sstt19}
\sup_{ \mathfrak{t}_2 \sim 2^{J_2}}\left| W^S_{J_2,a/q,\xi_0}(\mathfrak{t}_2)\left(1-\Psi_J^{\Lambda\setminus \Lambda',\rm{major}}(\xi_{\Lambda\setminus\Lambda'})\right)\right|\lesssim 2^{J_2\cdot {\bf 1}_2}2^{-k\rho j_{k}}.
\end{align}
\end{claim}

\begin{proof}[Proof of Claim \ref{ap90}]
   Express the phase function of (\ref{kanj}),  to focus on the coefficients of $s^n$ with $n\in \Lambda(0)$,
	\begin{align}\label{4vv}
		& \sum_{n\in \Lambda(0)} \xi_{0n}\mathfrak{s}^n+ \sum_{m\in \Lambda_1'}\left(\sum_{n\in \Lambda(m)} \frac{a_{mn}}{q}\mathfrak{s}^n\right)\mathfrak{t}_1^{m}  \nonumber \\
		&=\sum_{n\in \Lambda(0)} \left( \xi_{0n}+ \frac{\sum_{m\in \Lambda_1'} a_{mn}\mathfrak{t}_1^{m}}{q} \right)\mathfrak{s}^{n}+ \sum_{n\notin  \Lambda(0)} \left(\sum_{m\in \Lambda_1'}\frac{a_{mn}}{q}\mathfrak{t}_1^{m}  \right)\mathfrak{s}^n. 
	\end{align}
	Here  $n\in \Lambda(0)$ corresponds to the element of $\Lambda\setminus \Lambda'=\{(0,n)\in\Lambda\}$.\\
\textbf{	Case 1.} Let  $q\le 2^{ j_k/20}$.  Under the  condition  that  $1-\Psi_J^{\Lambda\setminus \Lambda',\rm{major}}\left( (\xi_{0n})_{(0,n)\in\Lambda\setminus\Lambda'} \right)\neq 0$, we first claim  that
\begin{align}\label{440}
1- \Psi_{J,(1/2,1)}^{\Lambda\setminus \Lambda',\rm{major}}\left(\left(\xi_{0n}+ \frac{\sum_{m\in \Lambda_1'} a_{mn}\mathfrak{t}_1^{m}}{q}\right)_{(0,n)\in \Lambda\setminus \Lambda'}\right)\neq 0,
\end{align}
where 
$$\Psi_{J,(1/2,1)}^{\Lambda\setminus \Lambda',\rm{major}}\left(  \eta_{\Lambda\setminus\Lambda'}  \right)=\sum_{\frac{a'}{q'} \in  \mathbb{Q}^{\Lambda\setminus\Lambda'}[1,2^{j_k/20}]}\prod_{(0,n)\in\Lambda\setminus \Lambda'} \psi\left(  \frac{ \eta_{0n}-a'_{0n}/q'}{  2^{-J\cdot (0,n)}2^{j_k /10}} \right),$$ as in (\ref{t12}).
To show (\ref{440}), we 
assume the contrary. Then there exists $$ \frac{(a_{0n}')_{(0,n)\in \Lambda\setminus\Lambda'}}{q'}\in  \mathbb{Q}^{\Lambda\setminus\Lambda'}[1,2^{j_k/20}] $$
such that
$$\left| \left(\xi_{0n}+ \frac{\sum_{m\in \Lambda_1'} a_{mn}\mathfrak{t}_1^{m}}{q}\right)- \frac{a'_{0n}}{q'}\right|< 2^{-J\cdot (0,n)} 2^{j_k/10}\ \text{for all $(0,n)\in \Lambda\setminus \Lambda'$}.$$
Here by the conditions $q\le 2^{j_k/20}$ and $q'\le 2^{j_k/20}$,
$$ \left( \frac{a'_{0n}}{q'}- \frac{\sum_{m\in \Lambda_1'} a_{mn}\mathfrak{t}_1^{m}}{q}\right)_{(0,n)\in \Lambda\setminus\Lambda'}\in \mathbb{Q}^{\Lambda\setminus\Lambda'}[1,2^{j_k/10}], $$
which implies $\Psi_J^{\Lambda\setminus \Lambda',\rm{major}}\left( (\xi_{0n})_{(0,n)\in\Lambda\setminus\Lambda'} \right)=1.$ This is a contradiction.  So we proved (\ref{440}).
Therefore, we are able to use the minor arc condition (\ref{440}) of the first term in (\ref{4vv})   to apply the Weyl sum estimate  in \eqref{weyl1}  to conclude that  
	\begin{align*}
		\left|\sum_{s_k=2^{j_k-1}+1}^{t_k} \cdots \sum_{s_{\ell+1}=2^{j_{\ell+1}-1}+1}^{t_{\ell+1}}  e^{2\pi i\left( \sum_{n\in \Lambda(0)} \xi_{0n}\mathfrak{s}^n+ \sum_{m\in \Lambda_1'}\left(\sum_{n\in \Lambda(m)} \frac{a_{mn}}{q}\mathfrak{s}^n\right)\mathfrak{t}_1^{m}  \right)}\right|
		\lesssim 2^{J_2\cdot {\bf 1}_2}2^{-\delta_{\Lambda} j_{\ell+1}}.
	\end{align*} under the condition $j_{\ell+1}\approx_{\Lambda} \cdots \approx_{\Lambda} j_{k}$.  Recall the definition  $\rho$ in  (\ref{rho3}). Then, the above Weyl sum estimate yields (\ref{sstt19}). 
	\\
	\textbf{Case 2.}  $2^{j_k/10}\ge q\ge 2^{ j_k/20}$. By (\ref{aver}),  we have
	$$\frac{1}{2^{J_2\cdot {\bf 1}_2}}\sum_{s_k=2^{j_k-1}+1}^{t_k} \cdots \sum_{s_{\ell+1}=2^{j_{\ell+1}-1}+1}^{t_{\ell+1}} |e^{2\pi i \xi_0(\mathfrak{s})}|\cdot |S^{\Lambda'_1}(a(\mathfrak{s})/q)|\lesssim 2^{-\delta_{\Lambda}j_{\ell+1}},$$  under the condition $j_{\ell+1}\approx_{\Lambda} \cdots \approx_{\Lambda} j_{k}$. 
	Therefore, by the definition  $\rho$, we  obtain the estimate (\ref{sstt19}).  
\end{proof}

For each $\nu\in \{\ell+1,\dots,k\}$, define the forward difference operator $D_\nu$ and its adjoint $\bar{D}_\nu$ by
\begin{align*}
\begin{cases}
D_\nu F(\mathfrak{t}_2)
:= F(\mathfrak{t}_2)-F(\mathfrak{t}_2-{\bf e}_\nu),\\
\bar{D}_\nu F(\mathfrak{t}_2)
:= F(\mathfrak{t}_2)-F(\mathfrak{t}_2+{\bf e}_\nu),
\end{cases}
\end{align*}
where ${\bf e}_\nu$ denotes the unit vector in the $\nu$-th coordinate direction.  
For a subset $V=\{\nu_1,\dots,\nu_r\}\subset \{\ell+1,\dots,k\}$, we write
\[
\prod_{\nu\in V}\bar{D}_\nu := \bar{D}_{\nu_1}\cdots \bar{D}_{\nu_r}.
\]

\begin{claim}[Summation by parts]\label{cl46}
Fix $\frac{a}{q}\in \mathbb{Q}^{\Lambda'}[1,2^{j_k/10}]$ and $J\in Z_\ell(k)$. Define
\begin{align*}
\mathcal{H}_{J_1,a/q,\xi}(\mathfrak{t}_2)
:=
\frac{\mathcal{H}^{\Lambda'_1}_{J_1}\!\left(\xi(\mathfrak{t}_2)-\frac{a(\mathfrak{t}_2)}{q}\right)}{\mathfrak{t}_2^{{\bf 1}_2}}.
\end{align*}
Then the main term in~\eqref{t8} can be written as
\begin{align}\label{t9}
&\sum_{J_1:J\in Z_\ell(k)}\sum_{\mathfrak{t}_2\sim 2^{J_2}}
e^{2\pi i \xi_0(\mathfrak{t}_2)}
S^{\Lambda_1'}\!\left(\frac{a(\mathfrak{t}_2)}{q}\right)
\mathcal{H}_{J_1,a/q,\xi}(\mathfrak{t}_2)\nonumber\\
&\quad=
\sum_{J_1:J\in Z_\ell(k)}
\sum_{\substack{U\cup V=\{\ell+1,\dots,k\}\\ U\cap V=\emptyset}}
\ \sum_{(t_\nu)_{\nu\in V}\in \prod_{\nu\in V}(2^{j_\nu-1},2^{j_\nu})}
W^S_{J_2,a/q,\xi_0}\!\big((2^{j_i})_{i\in U},(t_\nu)_{\nu\in V}\big) \\
&\quad\qquad\quad\qquad\quad\qquad\quad\qquad\quad\qquad \times
\left(\prod_{\nu\in V}\bar{D}_\nu\right)
\mathcal{H}_{J_1,a/q,\xi}\!\big((2^{j_i})_{i\in U},(t_\nu)_{\nu\in V}\big).\nonumber
\end{align}
\end{claim}

\begin{proof}[Proof of Claim \ref{cl46}]
Recall that $t_\nu\sim 2^{j_\nu}$ means
\[
t_\nu\in (2^{j_\nu-1},2^{j_\nu}]
:=\{2^{j_\nu-1}+1,\dots,2^{j_\nu}\}.
\]
Fix $\mathfrak{t}_2\in \prod_{\nu=\ell+1}^k (2^{j_\nu-1},2^{j_\nu}]$.
By applying the composition of difference operators $D_{\ell+1}\cdots D_k$ to $W^S_{J_2,a/q,\xi_0}$ (see \eqref{kanj}), we obtain
\begin{align*}
e^{2\pi i \xi_0(\mathfrak{t}_2)}
S^{\Lambda_1'}\!\left(\frac{a(\mathfrak{t}_2)}{q}\right)
= D_{\ell+1}\cdots D_k \, W^S_{J_2,a/q,\xi_0}(\mathfrak{t}_2),
\end{align*}
since
\[
W^S_{J_2,a/q,\xi_0}(\mathfrak{t}_2)\Big|_{t_\nu=2^{j_\nu-1}}=0
\quad \text{for all } \nu=\ell+1,\dots,k.
\]
Fix $\nu$ and freeze all variables except $t_\nu$. Write
\[
W(t_\nu):=W^S_{J_2,a/q,\xi_0}(\mathfrak{t}_2),
\qquad
\mathcal{H}(t_\nu):=\mathcal{H}_{J_1,a/q,\xi}(\mathfrak{t}_2).
\]
Then $W(2^{j_\nu-1})=0$, and a single-variable summation by parts yields
\begin{align*}
\sum_{t_\nu=2^{j_\nu-1}+1}^{2^{j_\nu}}
D_\nu W(t_\nu)\,\mathcal{H}(t_\nu)
&=
\sum_{t_\nu=2^{j_\nu-1}+1}^{2^{j_\nu}-1}
W(t_\nu)\,\bar{D}_\nu \mathcal{H}(t_\nu)
+ W(2^{j_\nu})\,\mathcal{H}(2^{j_\nu}).
\end{align*}
Applying this identity successively for $\nu=\ell+1,\dots,k$,
and keeping track of whether the boundary term is taken or not,
we obtain a decomposition indexed by subsets $U$ and $V$ with
$U\cup V=\{\ell+1,\dots,k\}$ and $U\cap V=\emptyset$.
This yields the desired formula.
\end{proof}

\begin{proof}[Proof of (\ref{t8})]
In \eqref{t9}, for any  $V\subset \{\ell+1,\dots,k\}$ and
$(t_{\nu})_{\nu\in V}\in \prod_{\nu\in V}(2^{j_\nu-1},2^{j_\nu})$,  
\begin{align}
\left|
\left(\prod_{\nu\in V}\bar{D}_\nu\right)
\mathcal{H}_{J_1,a/q,\xi}\!\big((2^{j_i})_{i\in U},(t_\nu)_{\nu\in V}\big)
\psi_J^{\Lambda'}\!\left(\frac{\xi_{\Lambda'}-a/q}{2^{j_k\rho/10}}\right)
\right|\label{ff44} \\
\lesssim \frac{1}{2^{J_2\cdot {\bf 1}_2}}
2^{|V| j_k\rho/10 - J_V\cdot {\bf 1}_V}. \nonumber
\end{align}

\begin{proof}[Proof of \eqref{ff44}]
Fix $V$ as above and freeze all variables except $(t_\nu)_{\nu\in V}$. Writing $\big((2^{j_i})_{i\in U},(t_\nu)_{\nu\in V}\big)=\tilde{\mathfrak{t}}_2$,
we bound the left-hand side of \eqref{ff44} by
\begin{align*}
\Bigg|
\int \left(\prod_{\nu\in V}\bar{D}_\nu\right)
\left(
\frac{
e^{2\pi i \sum_{m}\left(\sum_{n\in \Lambda(m)}
\left(\xi_{mn}-\frac{a_{mn}}{q}\right)\tilde{\mathfrak{t}}_2^n\right)
s_1^{m_1}\cdots s_\ell^{m_\ell}}
}{
\mathfrak{t}_2^{{\bf 1}_2}
}
\right)
\prod_{i=1}^\ell \frac{\chi_{j_i}(s_i)}{s_i}
\, ds_1\cdots ds_\ell
\Bigg|.
\end{align*}
Applying the mean value theorem repeatedly to the integrand yields
\begin{align*}
\Big| \left(\prod_{\nu\in V}\bar{D}_\nu\right)\left(\frac{e^{2\pi i (\cdot)} }{\mathfrak{t}_2^{{\bf 1}_2}}\right)\Big|
  &\lesssim
\frac{1}{2^{J_2\cdot {\bf 1}_2}}
\left(
\sum_{(m,n)}
\left|\xi_{mn}-\frac{a_{mn}}{q}\right|
2^{J\cdot(m,n)} + 1
\right)^{|V|}
2^{-J_V\cdot {\bf 1}_V}.
\end{align*}
On the support of
$\psi_J^{\Lambda'}\!\left(\frac{\xi_{\Lambda'}-a/q}{2^{j_k\rho/10}}\right)$,
we have
\[
\left|\xi_{mn}-\frac{a_{mn}}{q}\right|
\lesssim 2^{-J\cdot(m,n)} 2^{j_k\rho/10},
\]
which implies
\begin{align*}
\left(
\sum_{(m,n)}
\left|\xi_{mn}-\frac{a_{mn}}{q}\right|
2^{J\cdot(m,n)} + 1
\right)^{|V|}
\lesssim 2^{|V| j_k\rho/10}.
\end{align*}
Thus,
\[ \int_{\prod_{i=1}^{\ell} (2^{j_i-1} ,2^{j_{i}}] } \left(\prod_{\nu\in V}\bar{D}_\nu\right)
\left(\frac{e^{2\pi i (\cdot)} }{\mathfrak{t}_2^{{\bf 1}_2}}\right)
 \frac{ds_1\cdots ds_{\ell} }{s_1\cdots s_\ell}
\lesssim
\frac{2^{|V| j_k\rho/10 - J_V\cdot {\bf 1}_V}}{2^{J_2\cdot {\bf 1}_2}}.
\]
\end{proof}

On the support of
\begin{align*}
	\psi_J^{\Lambda'}\left(
	\frac{\xi_{\Lambda'}-a/q}{2^{j_k\rho/10}}
	\right)
	\left(
	1-\Psi_J^{\Lambda\setminus\Lambda',\mathrm{major}}
	(\xi_{\Lambda\setminus\Lambda'})
	\right),
\end{align*}
we first apply the geometric-mean estimate
$|\cdot|^{\rho^2}|\cdot|^{1-\rho^2}$ to
\begin{align*}
	\left|\left(\prod_{\nu\in V}\bar{D}_\nu\right)
	\mathcal{H}_{J_1,a/q,\xi}
	\left(
	(2^{j_i})_{i\in U},
	(t_\nu)_{\nu\in V}
	\right)\right|,
\end{align*}
using \eqref{ff44} to control the first term of the two factors. We then apply the Weyl-sum estimate \eqref{sstt19}, which gives the bound $2^{J_2\cdot{\bf 1}_2}2^{-\rho k j_k}.$
Consequently, the sum in \eqref{t9} is bounded by
\begin{align}\label{yu4}
 \lesssim  \frac{ 2^{J_2\cdot {\bf 1}_2} 2^{-\rho k j_k}}{2^{J_2\cdot {\bf 1}_2}}   \sum_{V\subset \{\ell+1,\cdots,k\}}
 \sum_{(t_\nu)_{\nu\in V}\in \prod_{\nu\in V}(2^{j_\nu-1},2^{j_\nu})}\left(2^{|V| \rho j_k /10-J_V\cdot {\bf 1}_{V}} \right)^{1-\rho^2}\notag\\
 \times
\sup_{\mathfrak{t}_2}\sum_{J_1:J\in Z} \left|\mathcal{H}^{\Lambda'_{1}}_{J_{1}}\left(\xi(\mathfrak{t}_2)-\frac{a(\mathfrak{t}_2)}{q}\right) \right|^{\rho^2}\notag\\
 \lesssim 2^{-\rho k j_k}  \sum_{V\subset \{\ell+1,\cdots,k\}} 2^{|V| \rho j_k(1-\rho^2) /10}  2^{\rho^2 J_V\cdot {\bf 1}_V}\lesssim 2^{-\rho k j_k} 2^{  \rho k j_k(1-\rho^2) /10}  2^{\rho^2 k j_{\ell+1}},
\end{align}
where the last  two inequalities in \eqref{yu4} follow   from $|V|\le k$ and the two bounds
\begin{align*}
&\sum_{(t_{j_\nu})_{\nu\in V}\in \prod_{\nu\in V}(2^{j_\nu-1},2^{j_\nu})} 1 \lesssim 2^{J_V\cdot {\bf 1}_V}\\
&\sup_{\mathfrak{t}_2}\sum_{J_1:J\in Z_{\ell}(k)}  \left|\mathcal{H}^{\Lambda'_{1}}_{J_{1}}\left(\xi(\mathfrak{t}_2)-\frac{a(\mathfrak{t}_2)}{q}\right) \right|^{\rho^{2}} \lesssim 1 \ \text{ in (\ref{po4}).}
\end{align*}
By the property that      $\rho   j_k\ge 10  \rho^2   j_{\ell+1} $ due to the smallness of $\rho$ defined  in   (\ref{rho3}) and the condition $j_{\ell+1}\approx_{\Lambda}\cdots \approx_{\Lambda} j_{k}$,  the last term in (\ref{yu4}) is bounded by  \begin{align}\label{y7}
	2^{-\rho k j_k} 2^{  \rho k j_k(1-\rho^2) /10}  2^{\rho^2 k j_{\ell+1}}\le	2^{-\rho k j_k} 2^{  \rho k j_k(1-\rho^2) /10}  2^{\rho k j_{k}/10}\le  2^{-\rho k j_k/2}.
\end{align}
Therefore, by (\ref{t9}), \eqref{yu4} and \eqref{y7}, there exists $c>0$ such that 
\begin{align*}
 \text{LHS of} \  (\ref{t8})\lesssim 2^{-\rho k j_k/2}  \sum_{\frac{a}{q}\in \mathbb{Q}^{\Lambda'}[1,2^{j_k/10}]} 
 \psi\left(\frac{\xi_{\Lambda'}-a/q}{1/q^2}\right) \lesssim 2^{-cj_{\ell+1}}
\end{align*}
under the condition that $j_{\ell+1}\approx_{\Lambda} \cdots \approx_{\Lambda} j_{k}$.
This completes the proof of (\ref{t8}).
 \end{proof}
Thus, we have proved   Lemma \ref{lee44b9}.
\end{proof}

\begin{lemma}\label{e7559}
Suppose that $\Lambda$ satisfies   the  condition (\ref{wep2}). Then there exists a constant $c>0$, independent of $J_2$ and $\xi$, such that
\begin{align*}
\sum_{J_1:\, J\in Z_\ell(k)}
\left|
\sum_{|\mathfrak{t}_2|\sim 2^{J_2}}
\frac{ H^{\Lambda_{1}}_{J_1}(\xi(\mathfrak{t}_2))}{\mathfrak{t}_2^{{\bf 1}_2}}
\left(
\Psi_J^{\Lambda\setminus\Lambda',\mathrm{major}}(\xi_{\Lambda\setminus\Lambda'})
\Psi_{J}^{\Lambda',\mathrm{major}}(\xi_{\Lambda'})
-
\Psi_J^{\Lambda,\mathrm{major}}(\xi)
\right)
\right|\\
 \lesssim 2^{-c j_{\ell+1}}.
\end{align*}
\end{lemma}

\begin{proof}[Proof of Lemma \ref{e7559}]
Observe that
\begin{align*}
\sum_{|\mathfrak{t}_2|\sim 2^{J_2}}
\frac{ H^{\Lambda_{1}}_{J_1}(\xi(\mathfrak{t}_2))}{\mathfrak{t}_2^{{\bf 1}_2}}
= H^{\Lambda}_{J}(\xi).
\end{align*}
Moreover, recall the following definition in (\ref{t12})
\begin{align*}
		\Psi_{j,(c_1,c_2)}^{\Lambda,\mathrm{major}}(\xi)
		= \sum_{a/q=(a_{\mathfrak{m}})/q \in \mathbb{Q}^{\Lambda}[1,2^{c_1 j_k/10}]}
		\prod_{\mathfrak{m}\in \Lambda} 
		\psi\!\left(
		\frac{\xi_{\mathfrak{m}} - a_{\mathfrak{m}}/q}
		{2^{-\mathfrak{m}\cdot j}\, 2^{c_2 j_k/10}}
		\right),
	\end{align*}
and decompose
$
\Psi_J^{\Lambda\setminus\Lambda',\mathrm{major}}(\xi_{\Lambda\setminus\Lambda'})
\Psi_{J}^{\Lambda',\mathrm{major}}(\xi_{\Lambda'})
-
\Psi_J^{\Lambda,\mathrm{major}}(\xi)
$
as
\begin{align}\label{543p}
&\Psi_J^{\Lambda\setminus\Lambda',\mathrm{major}}(\xi_{\Lambda\setminus\Lambda'})
\Bigl(
\Psi_{J,(1,\rho)}^{\Lambda',\mathrm{major}}(\xi_{\Lambda'})
-
\Psi_{J,(1,1)}^{\Lambda',\mathrm{major}}(\xi_{\Lambda'})
\Bigr)\\
&\qquad
+
\Bigl(
\Psi_J^{\Lambda\setminus\Lambda',\mathrm{major}}(\xi_{\Lambda\setminus\Lambda'})
\Psi_{J,(1,1)}^{\Lambda',\mathrm{major}}(\xi_{\Lambda'})
-
\Psi_J^{\Lambda,\mathrm{major}}(\xi)
\Bigr).\nonumber
\end{align}
The contributions of the two terms in \eqref{543p} are estimated in Claims~\ref{t7} and~\ref{to77} below, respectively.

\begin{claim}\label{t7}
For each fixed $J_2$ and $\xi$, there exists $c>0$ such that
\begin{align*}
Y_{J_2}(\xi)
:=
\sum_{J_1:\, J\in Z_\ell(k)}
\left|
H^{\Lambda}_{J}(\xi)\,
\Psi_J^{\Lambda\setminus\Lambda',\mathrm{major}}(\xi_{\Lambda\setminus\Lambda'})
\left(
\Psi_{J,(1,\rho)}^{\Lambda',\mathrm{major}}(\xi_{\Lambda'})
-
\Psi_{J,(1,1)}^{\Lambda',\mathrm{major}}(\xi_{\Lambda'})
\right)
\right|
\lesssim 2^{-c j_{\ell+1}}.
\end{align*}
\end{claim}

\begin{proof}[Proof of Claim \ref{t7}]
In \eqref{dom1}, we relabel the rational approximations as follows:
\begin{itemize}
\item write $a'/q' \in \mathbb{Q}^{\Lambda'}[1,2^{j_k/10}]$ 
for the parameter appearing in 
$\Psi_{J,(1,1)}^{\Lambda',\mathrm{major}}(\xi_{\Lambda'})$,
\item write $a_0/q_0 \in \mathbb{Q}^{\Lambda\setminus\Lambda'}[1,2^{j_k/10}]$ 
for the parameter appearing in 
$\Psi_J^{\Lambda\setminus\Lambda',\mathrm{major}}(\xi_{\Lambda\setminus\Lambda'})$,
\end{itemize}
We combine these into a single rational vector
\[
\frac{a}{q}
:= \left(\frac{a'}{q'},\frac{a_0}{q_0}\right)
\in 
\mathbb{Q}^{\Lambda'}[1,2^{j_k/10}]\oplus\mathbb{Q}^{\Lambda\setminus\Lambda'}[1,2^{j_k/10}],
\]
where $(a,q)=1$ and $1\le q\le q'q_0\le 2^{j_k/5}$.

With this choice, $\xi$ lies in a slightly enlarged major arc, and hence
the integral approximation \eqref{b40} from Proposition \ref{lemm29}
applies. Consequently, up to an error of size $O(2^{-c j_k})$, we have
\begin{align*}
Y_{J_2}(\xi)
&\lesssim 
\sum_{J_1:\,J\in Z_\ell(k)}
\sum_{a/q}
|S^{\Lambda}(a/q)|
\,|\mathcal{H}^{\Lambda}_J(\xi-a/q)| \\
&\quad\times
\prod_{(m,n)\in \Lambda}
\psi\!\left(
\frac{\xi_{mn}-a_{mn}/q}{2^{-J\cdot(m,n)}\,2^{j_k/10}}
\right)
\left(1-
\prod_{(m,n)\in \Lambda'}
\psi\!\left(
\frac{\xi_{mn}-a_{mn}/q}{2^{-J\cdot(m,n)}\,2^{j_k\rho/10}}
\right)
\right),
\end{align*}
where the summation over $a/q$ runs over
\[
\mathbb{Q}^{\Lambda'}[1,2^{j_k/10}]\oplus\mathbb{Q}^{\Lambda\setminus\Lambda'}[1,2^{j_k/10}].
\]

On the support of the last factor, we have
\[
\sum_{(m,n)\in \Lambda'}
\left|
\left(\xi_{mn}-\frac{a_{mn}}{q}\right)
2^{J\cdot(m,n)}
\right|
\gtrsim 2^{j_k\rho/10}.
\]
Therefore, by the    oscillatory integral estimate, there exists $c'>0$ such that
\[
|\mathcal{H}^{\Lambda}_J(\xi-a/q)|^{1/2}
\lesssim 2^{-c'j_k}.
\]
We now sum   $|\mathcal{H}^{\Lambda}_J(\xi-a/q)|^{1/2}$ over $J$ using \eqref{po4}, and combine this with the Gauss sum estimate in (\ref{g2}) to obtain that
\[
Y_{J_2}(\xi)\lesssim 2^{-c' j_k}.
\]
Finally, under the assumption $j_{\ell+1}\approx_{\Lambda}\cdots\approx_{\Lambda} j_k$,
we conclude
\[
Y_{J_2}(\xi)\lesssim 2^{-c j_{\ell+1}}.
\]
\end{proof}

To complete the proof of Lemma \ref{e7559}, it remains to estimate the second term in \eqref{543p}.

\begin{claim}\label{to77}
For each fixed $J_2$ and $\xi$, there exists $c>0$ such that
\begin{align*}
\sum_{J_1:\, J\in Z_\ell(k)}
\left|
H^{\Lambda}_{J}(\xi)\,
\Big(
\Psi_J^{\Lambda\setminus\Lambda',\mathrm{major}}(\xi_{\Lambda\setminus\Lambda'})
\Psi_{J,(1,1)}^{\Lambda',\mathrm{major}}(\xi_{\Lambda'})
-
\Psi_J^{\Lambda,\mathrm{major}}(\xi)
\Big)
\right|
\lesssim 2^{-c j_{\ell+1}}.
\end{align*}
\end{claim}

\begin{proof}[Proof of Claim \ref{to77}]
Set
\[
A:=\mathbb{Q}^{\Lambda'}[1,2^{j_k/10}]\oplus\mathbb{Q}^{\Lambda\setminus \Lambda'}[1,2^{j_k/10}],
\qquad
B:=\mathbb{Q}^{\Lambda}[1,2^{j_k/10}].
\]
By the definition,
\[
A
=
\left\{
\left(\frac{a'}{q'},\frac{a_0}{q_0}\right):
\gcd(a_0,q_0)=\gcd(a',q')=1,\ 
1\le q_0,q'\le 2^{j_k/10}
\right\}.
\]
If $a/q\in B=\mathbb{Q}^{\Lambda}[1,2^{j_k/10}]$, then
\[
\frac{a}{q}
=
\left(\frac{a'}{q},\frac{a_0}{q}\right),
\]
and reducing each component to lowest terms yields an element of $A$.
Thus $B\subset A$. Hence, one sees that
\begin{align}
\Psi_{J}^{\Lambda\setminus \Lambda',\mathrm{major}}(\xi_{\Lambda\setminus\Lambda'})
\Psi_{J,(1,1)}^{\Lambda',\mathrm{major}}(\xi_{\Lambda'})
-
\Psi_{J}^{\Lambda,\mathrm{major}}(\xi)
&=
\left(\sum_{a/q\in A}-\sum_{a/q\in B}\right)
\prod_{(m,n)\in\Lambda}
\psi\left(
\frac{\xi_{mn}-a_{mn}/q}{2^{-J\cdot(m,n)}2^{j_k/10}}
\right)\notag \\
&=
\sum_{a/q\in A\setminus B}
\prod_{(m,n)\in\Lambda}
\psi\left(
\frac{\xi_{mn}-a_{mn}/q}{2^{-J\cdot(m,n)}2^{j_k/10}}
\right).\label{t6}
\end{align}
Moreover, if $(a,q)=1$ and $a/q\in A$, then
\[
\frac{a}{q}
=
\left(\frac{a'}{q'},\frac{a_0}{q_0}\right),
\qquad
1\le q_0,q'\le 2^{j_k/10},
\]
which implies
$q \le \mathrm{lcm}(q_0,q') \le 2^{j_k/5}.$
Hence \begin{align}\label{ii}
	A\subset \mathbb{Q}^{\Lambda}[1,2^{j_k/5}].
\end{align}

Therefore, by \eqref{t6} and \eqref{ii}, it suffices, in order to prove Claim~\ref{to77}, to estimate the following multiplier:
\begin{align*} 
	\sum_{J_1:\,J\in Z_\ell(k)} |H^{\Lambda}_J(\xi)| \sum_{a/q \in \mathbb{Q}^{\Lambda}[2^{j_k/10},\,2^{j_k/5}]} \prod_{(m,n)\in\Lambda} \psi\!\left( \frac{\xi_{mn}-a_{mn}/q}{2^{-J\cdot (m,n)}2^{j_k/10}} \right). 
\end{align*}
 In view of part~(1) of both Proposition~\ref{lemm29} and Remark~\ref{yt}, we  reduce the estimate to showing that there exists $c>0$ such that
\begin{align}\label{b2}
\sum_{J_1:\,J\in Z_\ell(k)}
\sum_{a/q \in \mathbb{Q}^{\Lambda}[2^{j_k/10},\,2^{j_k/5}]}
|S^{\Lambda}(a/q)|\,
|\mathcal{H}^{\Lambda}_J(\xi-a/q)|
\psi_J^{\Lambda}\!\left(
\frac{\xi-a/q}{2^{j_k/10}}
\right)
\lesssim 2^{-c j_k}.
\end{align}
Applying Lemma~\ref{po3}, together with the Gauss sum estimate~\eqref{gaus} and the lower bound $q\ge 2^{j_k/10}$, we  obtain \eqref{b2}. Under the assumption
$j_{\ell+1}\approx_{\Lambda}\cdots\approx_{\Lambda} j_k,$
estimate~\eqref{b2} yields the desired bound in Claim~\ref{to77}.

\end{proof}
 Therefore, we have proved Lemma \ref{e7559}.
\end{proof}

\section{Necessity Proof}\label{sec10}
In this section, we prove the necessity part of Main Theorem~\ref{mt1}. The proof proceeds through a decomposition into residue classes modulo $3$, Poisson summation, and delicate parity considerations that force cancellation in all but one dominant contribution.
\begin{lemma}\label{propnec}
Fix an integer $s\geq2$. Let  $0<\epsilon\le (s+1)^{-1}$. Let $(m_1,\cdots,m_s)\in \mathbb{N}^{s}$. For any sufficiently large number $N\gg1$, set $\xi_1=\frac{N^{\epsilon}}{N^{m_1+\cdots+m_s}}$. Then, there exists $c>0$ independent of $N$ such that
$$ \int_{1}^{N^{m_s}}  \cdots \int_{1}^{N^{m_1}}  \sin\left(2\pi \xi_1 \prod_{\nu=1}^s t_\nu \right) \frac{dt_1\cdots dt_s}{t_1\cdots t_s} \ge c(\log N)^{s-1}.$$
\end{lemma}
\begin{proof}
 By the change of variable $t_1'=t_1 \xi_1 \prod_{\nu=2}^s t_\nu $, one has
\begin{align}\label{t4}
&\int_{1}^{N^{m_s}}  \cdots \int_{1}^{N^{m_1}}  \sin\left(2\pi \xi_1t_1 \prod_{\nu=2}^s t_\nu \right) \frac{dt_1\cdots dt_s}{t_1\cdots t_s}\notag\\
&\qquad=\int_{1}^{N^{m_s}}  \cdots \int_{1}^{N^{m_2}}\left(  \int_{\xi_1 \prod_{\nu=2}^s t_\nu  }^{N^{m_1}\xi_1 \prod_{\nu=2}^s t_\nu }\sin(2\pi t_1' ) \frac{dt_1'}{t_1'}\right) \frac{dt_2\cdots dt_s}{t_2\cdots t_s}.
\end{align} Note that  if $N\gg 1 $, then there exist  $$0<a\ll 1 \ \text{and} \  N^{\epsilon^{3}}\ge K\gg1$$ such  that  for any $a'<a$ and  any $b>K$,  $ \int_{a'}^{b }\sin(2\pi t_1 ) \frac{dt_1}{t_1}=\pi/2+o(1)>0$.
Then, we split  the integral in (\ref{t4}) as
 \begin{align*}
\sum_{i=1}^3 \int_{(t_2,\cdots,t_s)\in A_i}   \left( \int_{\xi_1 \prod_{\nu=2}^s t_\nu  }^{N^{m_1}\xi_1 \prod_{\nu=2}^s t_\nu }\sin(2\pi t_1 ) \frac{dt_1}{t_1}\right) \frac{dt_2\cdots dt_s}{t_2\cdots t_s},
\end{align*}
where we set
\begin{align*}
A_1&:=\left\{(t_2,\cdots, t_s)\in \prod_{\nu=2}^s[1,N^{m_\nu}]:   KN^{-\epsilon} N^{m_2+\cdots+m_s} \le t_2 \cdots t_s\le  N^{m_2+\cdots+m_s}\right\},\\
A_2&:=\left\{(t_2,\cdots, t_s)\in \prod_{\nu=2}^s[1,N^{m_\nu}]:  K^{-1}N^{-\epsilon}   N^{m_2+\cdots+m_s} \le t_2 \cdots t_s\le K N^{-\epsilon} N^{m_2+\cdots+m_s}\right\},\\
A_3&:=\left\{(t_2,\cdots, t_s)\in \prod_{\nu=2}^s[1,N^{m_\nu}]:   1\le t_2 \cdots t_s\le  K^{-1}N^{-\epsilon} N^{m_2+\cdots+m_s}\right\}.
\end{align*}
For any $(t_2,\cdots,t_s)\in A_1$, one can observe that 
$$  \xi_1 \prod_{\nu=2}^s t_\nu\le \frac{N^{\epsilon}  N^{m_2+\cdots+m_s} }{N^{m_1+\cdots+m_s}}=\frac{N^{\epsilon}}{N^{m_1}} \ll 1 \ll K= \frac{N^{\epsilon} K N^{-\epsilon} N^{m_1+m_2+\cdots+m_s}}{N^{m_1+\cdots+m_s}}\le 
N^{m_1} \xi_1 \prod_{\nu=2}^s t_\nu,   $$ which implies that $$\int_{\xi_1 \prod_{\nu=2}^s t_\nu  }^{N^{m_1}\xi_1 \prod_{\nu=2}^s t_\nu }\sin(2\pi t_1 ) \frac{dt_1}{t_1}\gtrsim 1 .$$
Suppose that $N^{m_i-\epsilon^2}\le t_i\le N^{m_i}$ for each $i=2,\dots,s.$ Then,
$$N^{m_2+\cdots+m_s-(s-1)\epsilon^2}
\le t_2\cdots t_s
\le N^{m_2+\cdots+m_s}.$$ Since $0<\epsilon\le (s+1)^{-1}$, we have
$\epsilon\ge (s-1)\epsilon^2+ \epsilon^3.$ Consequently, using $N^{\epsilon^3}\ge K$, we obtain
$$t_2\cdots t_s
\ge N^{m_2+\cdots+m_s-(s-1)\epsilon^2}
\ge K N^{-\epsilon}N^{m_2+\cdots+m_s}.$$ This means that  the rectangular region
$
\prod_{i=2}^{s}[N^{m_i-\epsilon^2},N^{m_i}]
$
is contained in $A_1$ which implies that  $$\int_{A_1}\frac{d t_2\cdots dt_s}{t_2\cdots t_s}> \int_{N^{m_2}N^{-\epsilon^2}}^{N^{m_2}}\cdots\int_{N^{m_s}N^{-\epsilon^2}}^{N^{m_s}}  \frac{dt_2\cdots dt_s}{t_2\cdots t_s}  \gtrsim (\log N)^{s-1}.$$ Therefore, we obtain that
\begin{align}\label{hj}
	 \int_{(t_2,\cdots,t_s)\in A_1}   \left( \int_{\xi_1 \prod_{\nu=2}^s t_\nu  }^{N^{m_1}\xi_1 \prod_{\nu=2}^s t_\nu }\sin(2\pi t_1 ) \frac{dt_1}{t_1}\right) \frac{dt_2\cdots dt_s}{t_2\cdots t_s}\gtrsim (\log N)^{s-1}.
\end{align}
On the other hand, by the support condition of $A_2$ and $A_3$, we have
\begin{align}
&\left| \int_{(t_2,\cdots,t_s)\in A_2}   \left( \int_{\xi_1 \prod_{\nu=2}^s t_\nu  }^{N^{m_1}\xi_1 \prod_{\nu=2}^s t_\nu }\sin(2\pi t_1 ) \frac{dt_1}{t_1}\right) \frac{dt_2\cdots dt_s}{t_2\cdots t_s}\right|\lesssim (\log N)^{s-2},\label{hj2}\\
& \int_{(t_2,\cdots,t_s)\in A_3}   \left( \int_{\xi_1 \prod_{\nu=2}^s t_\nu  }^{N^{m_1}\xi_1 \prod_{\nu=2}^s t_\nu }\sin(2\pi t_1 ) \frac{dt_1}{t_1}\right) \frac{dt_2\cdots dt_s}{t_2\cdots t_s}\ge 0.\label{hj3}
\end{align}
We briefly verify the two inequalities above. For $(t_2,\dots,t_s)\in A_2$, we have
$$
K^{-1}\le N^{m_1}\xi_1\prod_{\nu=2}^s t_\nu\le K,
$$
and hence
$$
\left|
\int_{\xi_1\prod_{\nu=2}^s t_\nu}^{N^{m_1}\xi_1\prod_{\nu=2}^s t_\nu}
\frac{\sin(2\pi t_1)}{t_1},dt_1
\right|
\lesssim 1.
$$
Moreover, for fixed $t_2,\dots,t_{s-1}$, setting $Q=t_2\cdots t_{s-1}$, the defining condition of $A_2$ yields
$$
\frac{K^{-1}N^{-\epsilon}N^{m_2+\cdots+m_s}}{Q}
\le t_s\le
\frac{KN^{-\epsilon}N^{m_2+\cdots+m_s}}{Q}.
$$
Consequently,
$$
\int_{{t_s:(t_2,\dots,t_s)\in A_2}}\frac{dt_s}{t_s}
\le 2\log K.
$$ Therefore,
$$
\int_{A_2}\frac{dt_2\cdots dt_s}{t_2\cdots t_s}
\lesssim (\log N)^{s-2},
$$
which yields
$$
\left|
\int_{A_2}
\left(
\int_{\xi_1\prod_{\nu=2}^s t_\nu}^{N^{m_1}\xi_1\prod_{\nu=2}^s t_\nu}
\frac{\sin(2\pi t_1)}{t_1},dt_1
\right)
\frac{dt_2\cdots dt_s}{t_2\cdots t_s}
\right|
\lesssim (\log N)^{s-2}.
$$

On the other hand, for $(t_2,\dots,t_s)\in A_3$,
$$
0<
\xi_1\prod_{\nu=2}^s t_\nu
\le
N^{m_1}\xi_1\prod_{\nu=2}^s t_\nu
\le K^{-1}.
$$
By $K\gg 1$, we have 
$$
\int_{A_3}
\left(
\int_{\xi_1\prod_{\nu=2}^s t_\nu}^{N^{m_1}\xi_1\prod_{\nu=2}^s t_\nu}
\frac{\sin(2\pi t_1)}{t_1},dt_1
\right)
\frac{dt_2\cdots dt_s}{t_2\cdots t_s}
\ge0.
$$

Therefore, the above three integrals (\ref{hj}), (\ref{hj2}) and (\ref{hj3}) yield the desired lower bound.
\end{proof}

We now prove the necessity part of (ii) in Main Theorem \ref{mt1}.  
To this end, we assume that 
$\Lambda$ contains an odd subset, but fails the condition \eqref{1009}.
The assumption  means that there exists a subset $\Omega\subset \Lambda$ such that
\begin{align}\label{101}
\text{$\Omega$ is odd and contains a vector $\mathfrak{m}$ with at least two odd components.}
\end{align}
Among all such subsets, choose $\Omega$ with minimal cardinality, say $|\Omega|=n$. Then, by minimality, for every non-empty proper subset $\Omega'\subsetneq \Omega$, we have
\begin{align}\label{102}
\sum_{\mathfrak{m}\in \Omega'} \mathfrak{m} \neq (\text{odd},\cdots,\text{odd}).
\end{align}
\begin{proof}[Proof of \eqref{102}]
Suppose, to the contrary, that
	$\Omega'$ is odd. By the minimality of $\Omega$, the set $\Omega'$ cannot contain a vector
	having at least two odd components. Hence every
	$\mathfrak m\in\Omega'$ has at most one odd component. Choose
	$\mathfrak m^{*}
	=
	(m_1^{*},\ldots,m_k^{*})
	\in\Omega$
	having at least two odd components. Necessarily,
	$\mathfrak m^{*}\notin\Omega'$, since otherwise $\Omega'$ would be an
	odd subset of smaller cardinality satisfying \eqref{101}.

	Let
	\[
	I:=\{i\in[k]:m_i^{*}\ \text{is even}\}.
	\]
	Since $\mathfrak m^{*}$ has at least two odd components, we have
	$|I|\leq k-2.$
	Because $\Omega'$ is odd, for each $i\in I$ there exists
	$\mathfrak m^{(i)}\in\Omega'$ whose $i$-th component is odd.
	Moreover, since every vector in $\Omega'$ has at most one odd component,
	all the remaining components of $\mathfrak m^{(i)}$ are even.

	Now set
	\[
	\widetilde{\Omega}
	:=
	\{\mathfrak m^{*}\}
	\cup
	\{\mathfrak m^{(i)}:i\in I\}.
	\]
	Then
	\[
	\sum_{\mathfrak m\in\widetilde{\Omega}}\mathfrak m
	=
	(\text{odd},\cdots,\text{odd}).
	\]
	Thus, $\widetilde{\Omega}$ is odd and contains the vector
	$\mathfrak m^{*}$, which has at least two odd components. On the other
	hand,
	\[
	|\widetilde{\Omega}|
	=
	1+|I|
	\leq k-1.
	\]
	Since $\Omega'$ is odd and every vector in $\Omega'$ has at most one
	odd component, $\Omega'$ must contain at least one distinct vector
	corresponding to each of the $k$ coordinates. Therefore,
	\[
	|\Omega'|\geq k,
	\]
	and hence
	\[
	|\widetilde{\Omega}|
	\leq k-1
	<
	|\Omega'|
	<
	|\Omega|=n.
	\]
	This contradicts the minimality of $\Omega$. Thus, we have \eqref{102}.
\end{proof}
Write
\begin{align}\label{t0}
\Omega=\{\mathfrak{m}_1,\cdots,\mathfrak{m}_n\}\subset \Lambda,
\end{align}
satisfying \eqref{101}--\eqref{102}. By \eqref{101}, there exists an element of  $ \Omega$ with at least two odd components. Without loss of generality, we may assume that for some $s\ge 2$,
\begin{align}\label{103}
\mathfrak{m}_1=(m_{11},\cdots,m_{1s},m_{1(s+1)},\cdots,m_{1k})
=(\text{odd},\cdots,\text{odd},\text{even},\cdots,\text{even}).
\end{align}
Fix such an $s$.
For each $\mathfrak{m}_j=(m_{j1},\cdots,m_{jk})\in \Omega$, define the projections
\[
[\mathfrak{m}_j]_{\le s}=(m_{j1},\cdots,m_{js})\in \mathbb{Z}_+^s,
\qquad
[\mathfrak{m}_j]_{>s}=(m_{j(s+1)},\cdots,m_{jk})\in \mathbb{Z}_+^{k-s}.
\]
Then, by the first condition of \eqref{101} and \eqref{103}, we have
\begin{align}
\sum_{j=2}^n[\mathfrak{m}_j]_{\le s}&=(\text{even},\cdots,\text{even}),\label{104} \\
\sum_{j=2}^n[\mathfrak{m}_j]_{>s}&=(\text{odd},\cdots,\text{odd}).\label{1044}
\end{align}
Moreover, by   \eqref{102} and (\ref{103}),  for any non-empty subset $V\subsetneq \{2,\cdots,n\}$, we obtain
\begin{align}\label{105}
\sum_{j\in V } [\mathfrak{m}_j]_{\le s} \neq (\text{even},\cdots,\text{even})
\quad \text{or} \quad
\sum_{j\in V } [\mathfrak{m}_j]_{> s} \neq (\text{odd},\cdots,\text{odd}).
\end{align}

Using these properties, we establish the divergence
$$
\sup_{\vec N,\xi}\bigl|H_{\vec N}^{\Lambda}(\xi)\bigr|=\infty
$$
through the following proposition.

\begin{proposition}\label{h450}
 Fix  $0<\epsilon\le (k+1)^{-1}$. Let $\Omega=\{\mathfrak{m}_1,\cdots,\mathfrak{m}_n\}\subset \Lambda$ be as in \eqref{t0}, with $\mathfrak{m}_1$ satisfying \eqref{103}.  Then, for all sufficiently large integers $N\gg 1$, define
\[
\xi_1=\frac{N^{\epsilon}}{N^{m_{11}+\cdots+m_{1s}}},
\qquad
\xi_2=\cdots=\xi_n=\frac{1}{3}.
\]
There exists a constant $c>0$, independent of $N$, such that
\begin{align}\label{0487}
\left|\sum_{\substack{|t_{s+1}|,\cdots,|t_k|=1 \\ 1\le |t_1|,\cdots,|t_s|\le N}}
e^{2\pi i \sum_{j=1}^n \xi_j t^{\mathfrak{m}_j}}
\frac{1}{t_1\cdots t_k}\right|
\ge c\log N.
\end{align}
\end{proposition}

\begin{proof}
For each $N\in\mathbb{N}$, let
\[
\psi_N:\mathbb{R}\to[0,1]
\]
be a smooth even function satisfying the following properties:
\begin{itemize}
	\item[(1)]
	$\operatorname{supp}(\psi_N)
	\subset
	\left\{
	u\in\mathbb{R}:
	\frac{1}{2}\leq |u|\leq N+\frac{1}{2}
	\right\}.$
	\item[(2)]
$\psi_N(u)=1$ whenever
	$1\leq |u|\leq N.$
	\item[(3)]
	For every integer $r$ with $1\leq r\leq s+1$, there exists
	a constant $C_r>0$, independent of $N$, such that
	$\left\|
	\psi_N^{(r)}
	\right\|_{L^\infty(\mathbb{R})}
	\leq C_r.$
	In particular, we may take
$\left\|
	\psi_N'
	\right\|_{L^\infty(\mathbb{R})}
	\leq 2.$
	\item[(4)]
	For every integer $r$ with $1\leq r\leq s+1$,
$\operatorname{supp}\!\left(\psi_N^{(r)}\right)
	\subset
	\left\{
	u\in\mathbb{R}:
	\frac{1}{2}<|u|<1
	\right\}
	\cup
	\left\{
	u\in\mathbb{R}:
	N<|u|<N+\frac{1}{2}
	\right\}.$
\end{itemize}
Then the sum appearing inside the absolute value on the left-hand side of \eqref{0487} can be rewritten as
\begin{align}\label{04870}
& \sum_{|t_{s+1}|,\cdots,|t_k|=1}\sum_{1\le |t_1|,\cdots,|t_s|\le N} e^{2\pi i \xi_1 t^{\mathfrak{m}_1}  } e^{\frac{2}{3}\pi i 
 \sum_{j=2}^n t^{\mathfrak{m}_j} } \frac{1}{t_1\cdots t_k}\nonumber \\
&= \sum_{|t_{s+1}|,\cdots,|t_k|=1}\sum_{(t_1,\cdots,t_s)\in \mathbb{Z}^s} e^{2\pi i \xi_1 t^{\mathfrak{m}_1}  } e^{\frac{2}{3}\pi i 
 \sum_{j=2}^n t^{\mathfrak{m}_j} }\left(\prod_{\nu=s+1}^k t_{\nu}\right) \prod_{\nu=1}^s\frac{\psi_N(t_\nu)}{t_\nu}. 
\end{align}
By the change of variables $t_\nu=3w_\nu+\ell_\nu$  with $\ell_\nu=0,-1,1$ and $w_\nu\in\mathbb{Z}$ for $\nu=1,\cdots,s$,   we can express the term in (\ref{04870}) as
\begin{align*}
& \sum_{|t_{s+1}|,\cdots,|t_k|=1}  \sum_{(\ell_1,\cdots,\ell_s)\in\{0,1,-1\}^s} \sum_{(w_1,\cdots,w_s)\in \mathbb{Z}^s} e^{2\pi i \xi_1  \prod_{\nu=1}^s(3w_\nu+\ell_\nu)^{m_{1\nu}} \prod_{\nu=s+1}^k t_\nu^{m_{1\nu}}  }\\
&\times e^{\frac{2}{3}\pi i 
 \sum_{j=2}^n \prod_{\nu=1}^s(3w_\nu+\ell_\nu)^{m_{j\nu}}
 \prod_{\nu=s+1}^k t_\nu^{m_{j\nu}} }\left(\prod_{\nu=s+1}^k t_{\nu} \right)\prod_{\nu=1}^s\frac{\psi_N(3w_\nu+\ell_\nu)}{3w_\nu+\ell_\nu}.
\end{align*}
Here, one can observe that
\begin{itemize}
\item for each $\nu=s+1,\cdots,k$, $m_{1\nu}=even$ from (\ref{103}), which implies that  $\prod_{\nu=s+1}^k t_\nu^{m_{1\nu}}=1$ whenever $|t_{s+1}|,\cdots,|t_k|=1$  on the first line,
\item $(3w_\nu+\ell_\nu)^{m_{j\nu}}\equiv \ell_\nu^{m_{j\nu}} \pmod{3}$ on the second line.
\end{itemize}
Hence, one can rewrite the above as
\begin{align*}
&\sum_{|t_{s+1}|,\cdots,|t_k|=1}  \sum_{(\ell_1,\cdots,\ell_s)\in\{0,1,-1\}^s} e^{\frac{2}{3}\pi i 
 \sum_{j=2}^n \prod_{\nu=1}^s \ell_\nu^{m_{j\nu}} \prod_{\nu=s+1}^k t_\nu^{m_{j\nu}}  }\left(\prod_{\nu=s+1}^k t_{\nu}\right)\\
&\qquad\qquad\qquad\qquad\qquad\qquad\qquad\qquad\qquad \times\sum_{(w_1,\cdots,w_s)\in \mathbb{Z}^s}
 G_{\ell,N}(w_1,\cdots,w_s),
\end{align*}
where we set
\begin{align}
G_{\ell,N}(w_1,\cdots,w_s)&:= e^{2\pi i \xi_1 \prod_{\nu=1}^s(3w_\nu+\ell_\nu)^{m_{1\nu}} } \prod_{\nu=1}^s\frac{\psi_{N}(3w_\nu+\ell_\nu)}{3w_\nu+\ell_\nu}.\nonumber
\end{align}
		By the Poisson-summation formula, it holds that
\begin{align*}
&\sum_{w\in \mathbb{Z}^s} G_{\ell,N}(w_1,\cdots,w_s)=\sum_{w\in \mathbb{Z}^s} \widehat{G_{\ell,N}}(w_1,\cdots,w_s),\end{align*}
where for $w\in \mathbb{Z}^s$ we define
\begin{align*}
\widehat{G_{\ell,N}}(w)&:=\int e^{2\pi i w\cdot (t_1,\cdots,t_s)} G_{\ell,N}(t_1,\cdots,t_s) dt_1\cdots dt_s\\
& = \int e^{2\pi i w\cdot (t_1,\cdots,t_s)}  e^{2\pi i \xi_1 \prod_{\nu=1}^s(3t_\nu+\ell_\nu)^{m_{1\nu}} } \prod_{\nu=1}^s\frac{\psi_{N}(3t_\nu+\ell_\nu)}{3t_\nu+\ell_\nu}dt_1\cdots dt_s.
		\end{align*}
	We split
	$$\sum_{w\in \mathbb{Z}^s} \widehat{G_{\ell,N}}(w)=\widehat{G_{\ell,N}}(0,\cdots,0)+\sum_{w\ne (0,\cdots,0)} \widehat{G_{\ell,N}}(w_1,\cdots,w_s)$$
	and observe that 
	$$\sum_{w\ne (0,\cdots, 0)} \widehat{G_{\ell,N}}(w_1,\cdots,w_s)=O(1),$$
since $ \widehat{G_{\ell,N}}(w_1,\cdots,w_s)=O(1/|w|^{s+1})$. This follows from the derivative conditions 
$$\left|\nabla_{t} \left((w_1,\cdots,w_s)\cdot (t_1,\cdots,t_s)+ \xi_1 \prod_{\nu=1}^s(3t_\nu+\ell_\nu)^{m_{1\nu}}\right)\right|\approx |w|\ \text{if $w\ne \textbf{0}$},$$
because $\xi_1=\frac{N^{\epsilon}}{N^{m_{11}+\cdots+m_{1s}}}$, $m_{1\nu}=odd\ne 0$ for all $\nu=1,\cdots,s$, and the support condition $|3t_\nu+\ell_\nu|\le N+1/2$. So, we have
\begin{align*}
	&\sum_{w\in \mathbb{Z}^s} \widehat{G_{\ell,N}}(w_1,\cdots,w_s)=\widehat{G_{\ell,N}}({\bf 0})+O(1)\\
	&=\frac{1}{3^s}\int   e^{2\pi i \xi_1 \prod_{\nu=1}^s t_\nu^{m_{1\nu}} } \prod_{\nu=1}^s \frac{\psi_{N}(t_\nu)}{t_\nu}   dt_1\cdots dt_s+O(1).
\end{align*}
Hence,
the LHS of (\ref{0487}) is $|AB|$ where
\begin{align*}
A&:=\sum_{|t_{s+1}|,\cdots,|t_k|=1}  \sum_{(\ell_1,\cdots,\ell_s)\in\{0,1,-1\}^s} e^{\frac{2}{3}\pi i 
 \sum_{j=2}^n \prod_{\nu=1}^s \ell_\nu^{m_{j\nu}} \prod_{\nu=s+1}^k t_\nu^{m_{j\nu}}  }\left(\prod_{\nu=s+1}^k t_{\nu}\right),\\
B&:= \frac{1}{3^s}\int   e^{2\pi i \xi_1 \prod_{\nu=1}^s t_\nu^{m_{1\nu}} } \prod_{\nu=1}^s \frac{\psi_{N}(t_\nu)}{t_\nu}   dt_1\cdots dt_s+O(1).
\end{align*}
By Lemma \ref{propnec} combined  with coordinate changes $t_\nu^{m_{1\nu}}\rightarrow t_\nu$ where $m_{1\nu}=odd$, one can obtain that
\begin{align*}
  &\left|  \frac{1}{3^s}\int   e^{2\pi i \xi_1 \prod_{\nu=1}^s t_\nu^{m_{1\nu}} } \prod_{\nu=1}^s \frac{\psi_{N}(t_\nu)}{t_\nu}   dt_1\cdots dt_s\right| \\
 &=\left| \frac{1}{3^s}  \int   \sin \left(2\pi  \xi_1 \prod_{\nu=1}^s t_\nu^{m_{1\nu}}\right)   \prod_{\nu=1}^s \frac{\psi_{N}(t_\nu)}{t_\nu}   dt_1\cdots dt_s\right|\gtrsim  \log N.
\end{align*}

Using $\prod_{\nu=s+1}^k t_{\nu}=0$ if $t_\nu=0$ in a factor, and $t_{\nu}=\pm 1 $ for $\nu=s+1,\cdots,k$, one can observe that
\begin{align*}
A& =  \sum_{(\ell_1,\cdots,\ell_k)\in\{0,1,-1\}^k} e^{\frac{2}{3}\pi i 
 \sum_{j=2}^n \prod_{\nu=1}^s \ell_\nu^{m_{j\nu}} \prod_{\nu=s+1}^n \ell_\nu^{m_{j\nu}}  }\left(\prod_{\nu=s+1}^k \ell_{\nu}\right) \\
& =  \sum_{(\ell_1,\cdots,\ell_k)\in\{0,1,-1\}^k} \prod_{j=2}^n e^{\frac{2}{3}\pi i 
  \ell^{\mathfrak{m}_j}   }\left(\prod_{\nu=s+1}^k \ell_{\nu}\right)\\
& =  \sum_{(\ell_1,\cdots,\ell_k)\in\{0,1,-1\}^k} \prod_{j=2}^n\left( \cos \left(\frac{2}{3}\pi 
  \ell^{\mathfrak{m}_j}\right)+i  \sin \left(\frac{2}{3}\pi 
  \ell^{\mathfrak{m}_j}\right)\right)\left(\prod_{\nu=s+1}^k \ell_{\nu}\right)\\
&=\sum_{U,V: U\ne\emptyset, \ U\cup V=\{2,\cdots,n\}}\sum_{(\ell_1,\cdots,\ell_k)\in\{0,1,-1\}^k}  \prod_{j\in U} \cos \left(\frac{2}{3}\pi 
  \ell^{\mathfrak{m}_j}\right)\prod_{j\in V}i \sin \left(\frac{2}{3}\pi 
  \ell^{\mathfrak{m}_j}\right)\left(\prod_{\nu=s+1}^k \ell_{\nu}\right)\\
  &+\sum_{(\ell_1,\cdots,\ell_n)\in\{0,1,-1\}^k}  \prod_{j=2}^ni \sin \left(\frac{2}{3}\pi 
  \ell^{\mathfrak{m}_j}\right)\left(\prod_{\nu=s+1}^k\ell_{\nu}\right).
\end{align*}
Given $V\subset\{2,\cdots,n\}$ and $\nu\in\{1,\cdots,k\}$,  we consider the mapping defined by
\begin{align}
\ell_\nu\rightarrow F_\nu(\ell_\nu):=\prod_{j\in V} \sin \left(\frac{2}{3}\pi 
  \ell^{\mathfrak{m}_j}\right)\left(\prod_{\nu'=s+1}^k \ell_{\nu'}\right).
\end{align}
Then one can observe the following symmetry properties of the functions $F_\nu$:
\begin{itemize}
\item For any  $\nu=1,\cdots,s$,
$\ell_\nu\rightarrow F_\nu(\ell_\nu)$ is  odd  if and only if $\sum_{j\in V} m_{j\nu}$ is odd
\item For any  $\nu=s+1,\cdots,k$,
$\ell_\nu\rightarrow F_\nu(\ell_\nu)$ is odd if and only  if $\sum_{j\in V}m_{j\nu}$ is even.
\end{itemize}
For every $\nu=1,\cdots,k$,
$$\ell_\nu\rightarrow  \prod_{j=2}^n \sin \left(\frac{2}{3}\pi 
  \ell^{\mathfrak{m}_j}\right)\left(\prod_{\nu=s+1}^k \ell_{\nu}\right)$$ is an even function  because of (\ref{104}) and (\ref{1044}). So,  we have
$$\left|\sum_{(\ell_1,\cdots,\ell_k)\in\{0,1,-1\}^k} \prod_{j=2}^n \sin \left(\frac{2}{3}\pi 
  \ell^{\mathfrak{m}_j}\right)\left(\prod_{\nu=s+1}^k \ell_{\nu}\right)\right|\gtrsim 1,$$
where the lower bound in RHS is computed by taking the sum over only $(\ell_1,\cdots,\ell_k)\in \{1,-1\}^k$.

However, in view of (\ref{105}), if  $U\ne \emptyset$ and $U\cup V=\{2,\cdots,n\}$ which means that $V\subsetneq \{2,\cdots,n\}$, then there exists  $\nu=1,\cdots,k$ such that
$\ell_\nu\rightarrow \prod_{j\in V} i \sin \left(\frac{2}{3}\pi 
  \ell^{\mathfrak{m}_j}\right)\left(\prod_{\nu=s+1}^k \ell_{\nu}\right)$ is an odd function. So, for any pair of subsets $(U,V)$ satisfying $U\ne\emptyset$ and  $U\cup V=\{2,\cdots,n\}$ both, one has
\begin{align}
\sum_{(\ell_1,\cdots,\ell_k)\in\{0,1,-1\}^k}  \left(\prod_{j\in U} \cos \left(\frac{2}{3}\pi 
  \ell^{\mathfrak{m}_j}\right)\right)\prod_{j\in V}i \sin \left(\frac{2}{3}\pi 
  \ell^{\mathfrak{m}_j}\right)\left(\prod_{\nu=s+1}^k \ell_{\nu}\right)=0.\end{align}
 Therefore, there exists a constant $c'>0$, independent of $N$, such that $|A|\ge c'$, which implies that $$LHS\  of\  (\ref{0487}) \gtrsim \log N.$$
 
Thus, we have proved Proposition \ref{h450}.
\end{proof}

\section{$\ell^p$ Boundedness}\label{sec6}
In this section, for $1<p<\infty$, we prove the $\ell^p$ boundedness of  the discrete multiple Hilbert transform
  ${\bf H}^{\Lambda}$ defined in (\ref{s2}), under the condition  (\ref{wep2}).
This establishes  (ii) of Main Theorem~\ref{mt2}. The proof proceeds by induction on the number of parameters $k$, using the Ionescu--Wainger $\ell^p$ theory developed in \cite{IW}.

Recall the following definition in (\ref{s2}) 	$${\bf H}^{\Lambda}(f)(x) := \lim_{\min\{N_1, \dots, N_k\} \to \infty} {\bf H}^{\Lambda}_{(N_1, \dots, N_k)}(f)(x) \ \text{for $f\in c_{00}(\mathbb{Z}^{|\Lambda|})$}$$
where $ c_{00}(\mathbb{Z}^{|\Lambda|})$ denotes the space of finitely supported  functions on  $\mathbb{Z}^{|\Lambda|}$.  
For $f\in c_{00}(\mathbb{Z}^{|\Lambda|}) \subset \ell^2(\mathbb{Z}^{|\Lambda|})$ and $J\in \mathbb{Z}_+^k$, in view of  Fourier inversion,  we define
$$  (H_J^{\Lambda})^{\vee}*f(x):=\int_{\mathbb{T}^{|\Lambda|}} e^{2\pi i x\cdot \xi}H^{\Lambda}_J(\xi)\widehat{f}(\xi)d\xi\ \text{with}\ H^{\Lambda}_J(\xi):=\sum_{|\mathfrak{t}|\sim 2^{J}} \frac{e^{2\pi i \sum_{\mathfrak{m}\in \Lambda}\xi_{\mathfrak{m}} t^{\mathfrak{m}}}}{t_1\cdots t_k}. $$
In Section 4-6, we have shown that  there exists $C>0$ independent of $\xi=(\xi_{\mathfrak{m}})_{\mathfrak{m}\in \Lambda}\in \mathbb{R}^{|\Lambda|}$ such that
$$\sum_{J\in \mathbb{Z}_+^{k}} \left|\sum_{|t|\sim 2^J}\frac{  e^{2\pi i \sum_{\mathfrak{m}\in\Lambda} \xi_{\mathfrak{m}}t^\mathfrak{m}} }{t_1\cdots t_k}\right|\le C,$$
under the condition  (\ref{wep2}).  Consequently, for every \(f\in c_{00}(\mathbb{Z}^{|\Lambda|})\), the dominated convergence theorem gives  $${\bf H}^{\Lambda}f(x)=\sum_{J\in \mathbb{Z}_+^{k}} (H_J^{\Lambda})^{\vee}*f(x).$$ 

Therefore, after usual reduction to the sector \(J\in Z(k)\) in the summation, in order to prove (ii) of Main Theorem~\ref{mt2},
we shall establish the following truncated estimate  that for any $1<p<\infty$ and any $\sigma\in\mathbb{N}$ and any  $f\in c_{00}(\mathbb{Z}^{|\Lambda|})$,  
\begin{align}\label{wep}
	\|\mathbf{H}_{\ge \sigma}^{\Lambda} f\|_{\ell^p(\mathbb{Z}^{|\Lambda|})}
	\leq
	C_{\Lambda,p}
	\|f\|_{\ell^p(\mathbb{Z}^{|\Lambda|})},
\end{align} 
under the assumption (\ref{wep2}).  Here,  for any $\sigma\in\mathbb{N}$, we set \[
\mathbf{H}_{\ge \sigma}^{\Lambda} f
\,:=\,
\sum_{\substack{J\in Z(k); j_k\ge \sigma}}
\bigl[ H_J^{\Lambda} \bigr]^{\vee} * f.
\]
The truncation parameter $\sigma$  is introduced to carry out the induction argument on $k$ under this restriction $J\in Z(k)\subset \mathbb Z_+^k$.

For the base case $k=1$, the estimate \eqref{wep} is known to hold for every $1<p<\infty$ by the result of Ionescu and Wainger~\cite{IW}.
 So, we assume the following hypothesis.
\\

\noindent  {\bf Induction Hypothesis.}\    Let $k_0\in[k-1].$ If  $\Omega \subset \mathbb{Z}_+^{k_0}$ contains an odd subset and 
 every $\mathfrak{m} \in \Omega$ has at most one odd component as in (\ref{1009}),  then for every $p \in (1,\infty)$ and for any $\sigma\in\mathbb{N}$, 
there exists a constant $C_{\Omega,p} > 0$ such that
\begin{align}\label{ind}
	\| {\bf H}_{\ge \sigma}^{\Omega} \|_{\ell^p(\mathbb{Z}^{|\Omega|}) \to \ell^p(\mathbb{Z}^{|\Omega|})}
	\le C_{\Omega,p}.
\end{align}
Here, $C_{\Omega,p}$ may depend only on $\Omega$ and $p$. Hence to prove (ii) of Main Theorem 2 in this section, we shall prove (\ref{wep}) under the assumptions of   (\ref{wep2}) and  (\ref{ind}).

\subsection{Application of Induction Hypothesis with Norm Invariances}
Let $\Omega \subset \mathbb{Z}_+^{k_0}$ be a finite set satisfying condition~\eqref{wep2}. 
For $J_0=(j_1,\ldots,j_{k_0})$, 
$\xi=(\xi_{\mathfrak{m}})_{\mathfrak{m}\in\Omega}\in\mathbb{R}^{|\Omega|}$,
$\mathfrak{t}=(t_1,\ldots,t_{k_0})$, and
$\alpha=(\alpha_{\mathfrak{m}})_{\mathfrak{m}\in\Omega}
\in(\mathbb{Z}\setminus\{0\})^{|\Omega|}$, define
\begin{align*}
	H_{J_0}^{\Omega,\alpha}(\xi)
	:=
	\sum_{|\mathfrak{t}|\sim 2^{J_0}}
	\frac{
		e^{2\pi i\sum_{\mathfrak{m}\in\Omega}
			\xi_{\mathfrak{m}}\alpha_{\mathfrak{m}}\mathfrak{t}^{\mathfrak{m}}}
	}{
		t_1\cdots t_{k_0}
	}.
\end{align*}
The associated convolution operator is given by
\begin{align*}
	[H_{J_0}^{\Omega,\alpha}]^\vee*f(x)
	:=
	\sum_{|\mathfrak{t}|\sim 2^{J_0}}
	\frac{
		f\!\left(
		(x_{\mathfrak{m}}
		-\alpha_{\mathfrak{m}}\mathfrak{t}^{\mathfrak{m}})
		_{\mathfrak{m}\in\Omega}
		\right)
	}{
		t_1\cdots t_{k_0}
	},
	\qquad x\in\mathbb{Z}^{|\Omega|}.
\end{align*}
 For $f\in c_{00}(\mathbb{Z}^{|\Omega|})$ and $\sigma\in\mathbb{N}$, define
\begin{align}\label{42cc}
	\mathbf{H}_{\geq \sigma}^{\Omega,\alpha}f
	:=
	\sum_{\substack{J_0\in Z(k_0)\\ j_{k_0}\geq \sigma}}
	[H_{J_0}^{\Omega,\alpha}]^\vee*f.
\end{align}
When $\alpha_{\mathfrak{m}}=1$ for every $\mathfrak{m}\in\Omega$, we write
$\mathbf{H}_{\geq \sigma}^{\Omega}$ in place of
$\mathbf{H}_{\geq \sigma}^{\Omega,\alpha}$.

We begin with introducing norm invariance under dilation and norm invariance regarding quasi-translation in Lemma \ref{c1} and \ref{lem:shear}. We omit the proof since it is well-known.

 \begin{lemma}\label{c1}
Suppose that for $p \in (1,\infty)$,
\[
\| {\bf H}_{\ge \sigma}^{\Omega} \|_{\ell^p(\mathbb{Z}^{|\Omega|}) \to \ell^p(\mathbb{Z}^{|\Omega|})}
\le C_{\Omega,p}.
\]
Then for any $\alpha=(\alpha_{\mathfrak m})_{\mathfrak m\in\Omega}\in (\mathbb{Z}\setminus\{0\})^{|\Omega|}$, 
\begin{align}\label{ind1}
\| {\bf H}_{\ge \sigma}^{\Omega,\alpha} \|_{\ell^p(\mathbb{Z}^{|\Omega|}) \to \ell^p(\mathbb{Z}^{|\Omega|})}
\le C_{\Omega,p}.
\end{align}
Here, $C_{\Omega,p}$ is the same constant as above.
\end{lemma}

\begin{lemma}\label{lem:shear} Let $d\ge 1$ and let $1\le s\le d$. Suppose that $ \ell_1+\cdots+\ell_s=d$ and, for each $1\le i\le s$, let $\mathfrak{c}_i = (c_{i,1},\ldots,c_{i,\ell_i}) \in\mathbb{Z}^{\ell_i}$ with $ c_{i,1}=0.$ Set  $ \mathfrak{c} = (\mathfrak{c}_i)_{i=1}^s\in\mathbb{Z}^d,$ and let $M^{\mathfrak{c}}$ be the $d\times d$ block-diagonal matrix whose $i$-th block is the $\ell_i\times\ell_i$ shear matrix \[ M_i^{\mathfrak{c}_i} = \bigl( {\bf e}_1+\mathfrak{c}_i \mid {\bf e}_2 \mid\cdots\mid {\bf e}_{\ell_i} \bigr), \] where $\{{\bf e}_j\}_{j=1}^{\ell_i}$ denotes the standard basis of $\mathbb{R}^{\ell_i}$. Then $M^{\mathfrak{c}}\in GL_d(\mathbb{Z})$. In particular, for every function $f:\mathbb{Z}^d\to\mathbb{C}$ and every $1\le p<\infty$, \[ \sum_{x\in\mathbb{Z}^d} \left|f(M^{\mathfrak{c}}x)\right|^p = \sum_{x\in\mathbb{Z}^d} |f(x)|^p. \] 
\end{lemma}

Using the above two norm invariances together with the induction hypothesis \eqref{ind}, we aim to establish the $\ell^p$ estimate for the sum over $(j_1,\cdots,j_{k-1})$ for each fixed $j_k$. For any $f\in c_{00}(\mathbb{Z}^{|\Lambda|})$, note that
\[
\sum_{(j_1,\cdots,j_{k-1}):\, J\in Z(k)}[H_J^{\Lambda}]^{\vee}*f(x)
=
\sum_{(j_1,\cdots,j_{k-1}):\, J\in Z(k)}
\sum_{|\mathfrak{t}|\sim 2^{J}}
\frac{ f\big((x_{\mathfrak{m}}-\mathfrak{t}^{\mathfrak{m}})_{\mathfrak{m}\in \Lambda} \big)}{t_1\cdots t_k}.
\]
Decompose the above expression as
\begin{align}\label{9nn}
\sum_{|t_k|\sim 2^{j_k}} \frac{1}{t_k}
\left(
\sum_{(j_1,\cdots,j_{k-1}):\, J\in Z(k)}
{\bf H}_{t_k,(j_1,\cdots,j_{k-1})}f(x)
\right),
\end{align}
where for each fixed $|t_k|\sim 2^{j_k}$,
\[
{\bf H}_{t_k,(j_1,\cdots,j_{k-1})}f(x)
:=
\sum_{|(t_1,\cdots,t_{k-1})|\sim 2^{(j_1,\cdots,j_{k-1})}}
\frac{
f\big((x_{\mathfrak{m}}-t_k^{m_k} t_1^{m_1}\cdots t_{k-1}^{m_{k-1}})_{\mathfrak{m}\in \Lambda}\big)
}{t_1\cdots t_{k-1}}.
\]
 The desired estimate will be established in the following lemma.

\begin{lemma}[$k\!-\!1$ Parameter Discrete Hilbert Transforms]\label{ii11}
Suppose that $\Lambda$ satisfies the condition   \eqref{wep2}. For any $|t_k|\sim 2^{j_k}$ and any $f\in c_{00}(\mathbb{Z}^{|\Lambda|})$, define
\begin{align}\label{34m}
[{\bf H}_{t_k}]^{\Lambda}_{\ge \sigma}f(x)
:=
\sum_{\substack{(j_1,\cdots,j_{k-1})\in Z(k-1):\\ j_{k-1}\ge \sigma}}
{\bf H}_{t_k,(j_1,\cdots,j_{k-1})}f(x).
\end{align}
Under the inductive hypothesis \eqref{ind}, for each fixed  $|t_k|\sim 2^{j_k}$, there exists a constant $C_{\Lambda,p}$ such that
\begin{align}\label{dji}
\left\| [{\bf H}_{t_k}]^{\Lambda}_{\ge \sigma}f \right\|_{\ell^p(\mathbb{Z}^{|\Lambda|})}
\le C_{\Lambda,p}\|f\|_{\ell^p(\mathbb{Z}^{|\Lambda|})}.
\end{align}
Consequently, taking $\sigma=j_k$ in (\ref{dji}), we obtain
\begin{align}\label{9nn2}
\left\|   
\sum_{(j_1,\cdots,j_{k-1}):\, J\in Z(k)}
[H_J^{\Lambda}]^{\vee}*f
\right\|_{\ell^p(\mathbb{Z}^{|\Lambda|})}
\le C_{\Lambda,p} \|f\|_{\ell^p(\mathbb{Z}^{|\Lambda|})}.
\end{align}
\end{lemma}

\begin{proof}
We first consider the case $t_k=1$ in (\ref{dji}).
Denote $\underline{\mathfrak{m}}=(m_1,\cdots,m_{k-1})$ for $\mathfrak{m}=(\underline{\mathfrak{m}},m_k)\in \Lambda$. 
Decompose $\Lambda$ as
\[
\Lambda=\bigcup_{i=1}^s\Lambda_i,
\qquad
\Lambda_i=\{(\underline{\mathfrak{m}}^i,m_{k,\nu}^i)\in \Lambda\}_{\nu=1}^{\ell(i)},
\]
so that each $\Lambda_i$ shares the same $(k-1)$-dimensional component $\underline{\mathfrak{m}}^i=(m_1^{i},\cdots,m_{k-1}^{i})$.

Let $M^{\mathfrak{c}}$ be the shear matrix associated with $\mathfrak{c}=(\mathfrak{c}_i)_{i=1}^s$, where $\mathfrak{c}_i=(0,1,\cdots,1)\in \mathbb{Z}^{\ell(i)}$. Then,  writing $\underline{t}=(t_1,\cdots,t_{k-1})$ with $t_k=1$ below  (\ref{9nn}), we have
\begin{align*}
f\big((x_{\mathfrak{m}}- \underline{t}^{\underline{\mathfrak{m}}})_{\mathfrak{m}\in \Lambda}\big)
=
f\!\left(
M^{\mathfrak{c}}\!\left(
[M^{\mathfrak{c}}]^{-1}x
-
\big((\underline{t}^{\underline{\mathfrak{m}}^i},0,\cdots,0)\big)_{i=1}^s
\right)
\right).
\end{align*}
Therefore,  by combining the $\ell^p$ norm invariance   under this transformation, and  the induction hypothesis \eqref{ind}, we obtain \eqref{dji} in the case $t_k=1$.

For general $t_k$, define $\alpha=(\alpha_{\mathfrak{m}})$ by
\[
\alpha_{\mathfrak{m}}=t_k^{m_k}, \qquad \mathfrak{m}=(m_1,\cdots,m_{k-1},m_k).
\]
 Then in \eqref{34m}, we can write  $[{\bf H}_{t_k}]^{\Lambda}_{\ge \sigma}$ as the dilated version of the case $t_k=1$:
\[
[{\bf H}_{t_k}]^{\Lambda}_{\ge \sigma}
=
[{\bf H}_{1}]^{\Lambda,\alpha}_{\ge \sigma}
\]
as those appeared in   (\ref{42cc}).
Hence, the desired estimate \eqref{dji} follows from \eqref{ind1}.
Therefore, applying the triangle inequality in \eqref{9nn}, we bound the left-hand side of \eqref{9nn2} by
\[
\frac{1}{2^{j_k}} \sum_{|t_k|\sim 2^{j_k}}
\left\| [{\bf H}_{t_k}]^{\Lambda}_{\ge j_k}f\right\|_{\ell^p(\mathbb{Z}^{|\Lambda|})}
\le
C_{\Lambda,p}\left\| f\right\|_{\ell^p(\mathbb{Z}^{|\Lambda|})},
\]
which proves \eqref{9nn2}.
\end{proof}

 The proof of the inequality \eqref{wep} proceeds in two steps. We first establish a refined asymptotic approximation (Theorem~\ref{IW3}) and then combine it with the operator decomposition in Lemma~\ref{IW2}. 
The proof of Lemma~\ref{IW2} relies crucially on the Ionescu--Wainger theorem stated in Theorem~\ref{lem82}.

\subsection{Asymptotic Approximation}
Define
\[
\psi_{h{\bf 1}}(\xi):= \prod_{\mathfrak{m}\in \Lambda}
\psi\!\left(  \xi_{\mathfrak{m}}\, 2^{h {\bf 1}\cdot \mathfrak{m}}   \right).
\] In this section, we redefine  	\begin{align*} 
	\ \ \mathbb{Q}^{\Lambda}[2^{\alpha},2^{\beta}]:=  \left\{\frac{a}{q}\in   (\mathbb{Q}\cap \mathbb{T})^{|\Lambda|}: 2^{\alpha}\le q \le 2^{\beta}\ \text{and}\  \text{gcd}(q,a)=1\right\}.
\end{align*}
By exploiting (\ref{bb330}), we further refine the approximation in (\ref{b40}).

\begin{theorem}\label{IW3}  
Suppose that $\Lambda$ satisfies the condition   \eqref{wep2}. Set $s=\delta_{\Lambda}$ where $\delta_{\Lambda}$ is the constant appearing in \eqref{gaus}.
Given any integer $h,\sigma,D \ge 1$,  there exists a constant $C_D>0$ such that
\begin{align}\label{IW5}
\sum_{J\in Z(k):\, j_k\ge \max\{h,\sigma\}} H^{\Lambda}_{J}(\xi)
= \sum_{a/q\in \mathbb{Q}^{\Lambda}[1,h^{D/s}]} S^{\Lambda}\!\left(\frac{a}{q}\right) m_{h}( \xi -a/q)  +E_{h}(\xi),
\end{align}
where
\begin{align*}
\left|E_{h}(\xi)\right| \leq C_{D} h^{-D},
\qquad 
m_{h}( \xi-a/q )
:=  \sum_{J \in Z(k):\, j_k\ge \max\{h,\sigma\}}  
\psi_{h{\bf 1}}\!\left(\frac{\xi-a/q}{2^{h/10}}\right)\mathcal{H}^\Lambda_{J}(\xi-a/q).
\end{align*}
\end{theorem}

\begin{proof}
 In this proof, we only consider the case where $h\gg 1$. We decompose
\[
\sum_{J\in Z(k):\, j_k\ge \max\{h,\sigma\}}  H^{\Lambda}_{J}(\xi)
=E^1_h(\xi)+E^2_h(\xi)+E^3_h(\xi)+E^4_h(\xi),
\]
where
\begin{align*}
E^1_h(\xi)& = \sum_{J:\, j_k\ge \max\{h,\sigma\}} 
H^{\Lambda}_J(\xi)\bigl(1- \Psi^{\Lambda,\mathrm{major}}_{J}(\xi)\bigr),\\
E^2_h(\xi)&=\sum_{J:\, j_k\ge \max\{h,\sigma\}} 
\sum_{a/q\in \mathbb{Q}^{\Lambda}[2^{h/10}, 2^{j_k/10} ] }  
H^{\Lambda}_J(\xi)\prod_{\mathfrak{m}\in \Lambda}
\psi\!\left(\frac{(\xi_{\mathfrak{m}}-a_{\mathfrak{m}}/q)\, 2^{J\cdot \mathfrak{m}}}{2^{j_k/10}}\right),\\
E^3_h(\xi)&=\sum_{J:\, j_k\ge \max\{h,\sigma\}}  
\sum_{a/q\in \mathbb{Q}^{\Lambda}[1,2^{h/10}]}  
H^{\Lambda}_J(\xi) 
\left(\prod_{\mathfrak{m}\in \Lambda}
\psi\!\left(\frac{(\xi_{\mathfrak{m}}-a_{\mathfrak{m}}/q)\, 2^{J\cdot \mathfrak{m}}}{2^{j_k/10}}\right)
-\psi_{h{\bf 1}}\!\left(\frac{\xi-a/q}{2^{h/10}}\right)\right),\\
E^4_h(\xi)&=\sum_{J:\, j_k\ge \max\{h,\sigma\}}    
\sum_{a/q\in \mathbb{Q}^{\Lambda}[1,2^{h/10}]}  
H^{\Lambda}_J(\xi)\,
\psi_{h{\bf 1}}\!\left(\frac{\xi-a/q}{2^{h/10}}\right).
\end{align*}

Due to (\ref{bb330}), there exists $c>0$ such that $E^1_h(\xi)=O(2^{-ch})$. Next, observe that the supports of $E^i_h(\xi)$ for $i=2,3,4$ are contained in the major arc. 
By Proposition~\ref{lemm29}, on each $E^i_h(\xi)$, we have the following approximation
\begin{align*}
H_J^{\Lambda}(\xi)
= S^{\Lambda}(a/q)\, \mathcal{H}^{\Lambda}_J(\xi-a/q)
+ E_J(\xi,a,q),
\end{align*}
where
\[
\sum_{J:\, j_k\ge \max\{h,\sigma\}} 
\sum_{a/q\in \mathbb{Q}^{\Lambda}[1,2^{j_k/10}]}
|E_J(\xi,a,q)|
=O(2^{-ch}).
\] Moreover, for
\(
J
\)
in the summation range
\(
j_k\ge \max\{h,\sigma\},
\)
the condition
\begin{align*}
	\prod_{\mathfrak{n}\in\Lambda}
	\psi\!\left(
	\frac{(\xi_{\mathfrak{n}}-a_{\mathfrak{n}}/q)\,
		2^{J\cdot\mathfrak{n}}}
	{2^{j_k/10}}
	\right)
	-
	\psi_{h\mathbf{1}}\!\left(
	\frac{\xi-a/q}{2^{h/10}}
	\right)
	\neq 0
\end{align*}
implies that, for some
\(
\mathfrak{m}\in\Lambda,
\)
\begin{align*}
	\bigl|2^{J\cdot\mathfrak{m}}
	(\xi_{\mathfrak{m}}-a_{\mathfrak{m}}/q)\bigr|\gtrsim \bigl|2^{h\mathbf{1}\cdot\mathfrak{m}}
	(\xi_{\mathfrak{m}}-a_{\mathfrak{m}}/q)\bigr|
	\gtrsim 2^{h/10}.
\end{align*}
Therefore, we obtain that
\begin{align*}
E^2_h(\xi)&=O(2^{-ch}) 
\qquad \text{since } S^{\Lambda}(a/q)=O(q^{-\delta_{\Lambda}})=O(2^{-ch/10}) \ \text{when} \  q\ge 2^{h/10},\\
E^3_h(\xi)&=O(2^{-ch})
\qquad \text{since } 
|\mathcal{H}_{J}^{\Lambda}(\xi-a/q)|
\lesssim |2^{J\cdot \mathfrak{m}} (\xi_{\mathfrak{m}}-a_{\mathfrak{m}}/q)|^{-c}
\le 2^{-ch}.
\end{align*}

Finally, up to an error of $O(2^{-ch})$, we further decompose
\begin{align*}
E^4_h(\xi)
&=
\sum_{J:\, j_k\ge \max\{h,\sigma\}} 
\sum_{a/q\in \mathbb{Q}^{\Lambda}[1,h^{D/s}]}  
S^{\Lambda}(a/q)\, \mathcal{H}^{\Lambda}_J(\xi-a/q)
\psi_{h{\bf 1}}\!\left(\frac{\xi-a/q}{2^{h/10}}\right)\\
&\quad+
\sum_{J:\, j_k\ge \max\{h,\sigma\}} 
\sum_{a/q\in \mathbb{Q}^{\Lambda}[h^{D/s},2^{h/10}]}  
S^{\Lambda}(a/q)\, \mathcal{H}^{\Lambda}_J(\xi-a/q)
\psi_{h{\bf 1}}\!\left(\frac{\xi-a/q}{2^{h/10}}\right).
\end{align*}
The second term is bounded by $O(h^{-\delta_{\Lambda} D/s})=O(h^{-D})$ by choosing $s=\delta_{\Lambda}$, together with the estimate $S^{\Lambda}(a/q)=O(q^{-\delta_{\Lambda}})$. 
Taking the first term as the main contribution $ 
 \sum_{a/q\in \mathbb{Q}^{\Lambda}[1,h^{D/s}]} S^{\Lambda}\!\left(\frac{a}{q}\right) m_{h}( \xi-a/q )$, we obtain the desired asymptotic formula \eqref{IW5}.
\end{proof}

\subsection{Proof of $\ell^{p}$ Boundedness}
We shall prove (\ref{wep}) under the assumptions of   (\ref{wep2}) and  (\ref{ind}).
To establish the $\ell^p$ boundedness, we shall employ the arguments of \cite{IW}. 
Set $d=|\Lambda|$. Let $m:\mathbb{R}^{d}\to\mathbb{C}$ be a bounded function supported in the cube $[-1/2,1/2]^d$. 
Assume that for every $p\in(1,\infty)$,
\begin{align}\label{IW1}
\left\|(m \cdot \widehat{g})^{\vee}\right\|_{L^{p}(\mathbb{R}^{d})}
\leq B_{p}\|g\|_{L^{p}(\mathbb{R}^{d})}
\end{align}
holds for all Schwartz functions $g:\mathbb{R}^{d}\to\mathbb{C}$.

Let $Y\subset \mathbb{N}$. Define the set of reduced rational vectors with denominators in $Y$ by
\[
\mathbb{Q}^{[d]}(Y)
:=\left\{\frac{a}{q}:\ q\in Y,\ a=(a_1,\dots,a_d)\in\mathbb{Z}^{d},\ \gcd(a,q)=1\right\}.
\]
Given such a set $Y$ and $\eta\in(0,1]$, define
\begin{align}\label{91k}
m_{\eta,Y}(\xi)
:=\sum_{a/q\in \mathbb{Q}^{[d]}(Y)} m\!\left(\frac{\xi-a/q}{\eta}\right).
\end{align}
Then $m_{\eta,Y}$ is $\mathbb{Z}^d$-periodic, i.e.
\[
m_{\eta,Y}(\xi+n)=m_{\eta,Y}(\xi)\quad \text{for all } n\in\mathbb{Z}^d.
\]
Denote $Z_N:=\mathbb{Z}\cap[1,N]$.

\begin{theorem}\label{lem82}
$(\text{Ionescu and Wainger Discrete Multiplier Theory \cite{IW}.})$ 
For every $\delta>0$ and $p\in(1,\infty)$, there exist constants
$A_{\delta}>0$ and $C_{p,\delta}>0$ with the following property.
For every $N\ge A_{\delta}$, there exists a set
$Y_N=Y_{N,\delta}$ of integers satisfying
\begin{align*}
	\mathbb{Z}_N
	\subseteq Y_N
	\subseteq \mathbb{Z}_{e^{N^\delta}}.
\end{align*}
Moreover, for every $\eta\le e^{-N^{2\delta}}$, the operator
$T_N=T_{N,\eta}$ defined by the Fourier multiplier
$m_{\eta,Y_N}$ in~\eqref{91k} extends to a bounded operator on
$\ell^p(\mathbb{Z}^d)$ and satisfies
\begin{align*}
	\|T_N f\|_{\ell^p(\mathbb{Z}^d)}
	\le
	C_{p,\delta}(\ln N)^{2/\delta}
	\|f\|_{\ell^p(\mathbb{Z}^d)}.
\end{align*}
The constant $A_{\delta}$ may depend only on $\delta$ and $d$,
whereas $C_{p,\delta}$ may depend only on $\delta$, $d$, $p$,
and the constant $B_p$ appearing in~\eqref{IW1}.
\end{theorem}

\begin{remark}
Theorem~\ref{lem82} was later extended in a more general setting in \cite{KMPW} and \cite{Tao}.
For our purposes, however, the original formulation in \cite{IW} suffices.
\end{remark}

By the argument on p.~380 of \cite{IW}, once Lemma~\ref{IW2} below is established, it immediately implies \eqref{wep} for every $1<p<\infty$.  Therefore, it remains only to prove Lemma~\ref{IW2}.

\begin{lemma}\label{IW2}
Suppose that $\Lambda$ satisfies   the  condition (\ref{wep2}). Set $s=\delta_{\Lambda}$.		Given   $\epsilon \in(0,1]$ and $\lambda \in(0, \infty)$, we can choose the linear operators $A_{\lambda, \epsilon}$ and $B_{\lambda, \epsilon}$ such that ${\bf H}_{\ge\sigma}^{\Lambda} =A_{\lambda, \epsilon} +B_{\lambda, \epsilon} $ satisfying that for some constants $C_{\Lambda,\epsilon}$ and $C_{\Lambda,r,\epsilon}>0$,
\begin{align*}
	\left\|A_{\lambda, \epsilon}\right\|_{\ell^{2}(\mathbb{Z}^d) \rightarrow \ell^{2}(\mathbb{Z}^d) } \leq \frac{C_{\Lambda,\epsilon}}{\lambda}\ \text{and}\ 
	\left\|B_{\lambda, \epsilon}\right\|_{\ell^{r}(\mathbb{Z}^d)  \rightarrow \ell^{r}(\mathbb{Z}^d) } \leq C_{\Lambda,r,\epsilon} \lambda^{\epsilon} 
\end{align*}
for any $r \in[2, \infty)$. 
\end{lemma}

\begin{proof}[Proof of Lemma \ref{IW2}]
	Fix a sufficiently large constant $C>0$. Suppose first that
	$0<\lambda<C$. In this case, we simply set 
	$A_{\lambda,\epsilon}:={\bf H}_{\ge \sigma}^{\Lambda},
	\qquad
	B_{\lambda,\epsilon}:={\bf 0}.$

We now assume that $\lambda\ge C$. We shall construct the decomposition
	\[
	A_{\lambda,\epsilon}
	=
	A_{\lambda,\epsilon}^{1}
	+
	A_{\lambda,\epsilon}^{2}
	+
	A_{\lambda,\epsilon}^{3},
	\qquad
	B_{\lambda,\epsilon}
	=
	B_{\lambda,\epsilon}^{1}
	+
	B_{\lambda,\epsilon}^{2},
	\]
	proceeding in the order
	\(A_{\lambda,\epsilon}^{1}\),
\(B_{\lambda,\epsilon}^{1}\),
	\(A_{\lambda,\epsilon}^{2}\),
	\(B_{\lambda,\epsilon}^{2}\),
	and finally
	\(A_{\lambda,\epsilon}^{3}\).
	Let $h$ be the largest integer less than or equal to $\lambda^{\epsilon}$, i.e.,
\[
h := [\lambda^{\epsilon}]\ \text{and}\  N:=[h^{D/s}]\ \text{where $D=3/(2\epsilon)$}
\]
where  $[c]$ denotes the greatest integer less than or equal to $c$. This notation is used only here and should not be confused with the earlier set notation $[c]=\{1,\dots,c\}$.  	Choose $$\delta=\frac{s}{4D}\ \text{so that}\ N^{2\delta}\sim (h^{\frac{D}{s}})^{2\delta}=h^{1/2},$$  
which implies \begin{align}\label{hh}
  2^{-h/10}\ll e^{-N^{2 \delta}} \text{since $h\gg1$}.
\end{align}

By (\ref{IW5}), the multiplier $m(\xi)$ of the operator $ {\bf H}_{\ge\sigma}^{\Lambda}$ is
\begin{align*} 
\sum_{J\in Z(k):\, j_k\ge  \sigma} H^{\Lambda}_{J}(\xi)&=\sum_{J\in Z(k):\, \sigma\le j_k< \max\{h,\sigma\} } H^{\Lambda}_{J}(\xi)+\sum_{J\in Z(k):\, j_k\ge \max\{h,\sigma\}} H^{\Lambda}_{J}(\xi)\\
&= \sum_{J\in Z(k):\, \sigma\le j_k< \max\{h,\sigma\} } H^{\Lambda}_{J}(\xi)+\sum_{a/q\in \mathbb{Q}^{\Lambda}[1,N]} S^{\Lambda}\!\left(\frac{a}{q}\right) m_{h}( \xi -a/q)  +E_{h}(\xi).
\end{align*}
 Choose  the operator $A_{\lambda, \epsilon}^{1}$ whose multiplier is $$m_{A^1_{\lambda,\epsilon}}(\xi):=E_{h}(\xi)\ \text{satisfying}\ \left\|A_{\lambda, \epsilon}^{1}\right\|_{\ell^{2} \rightarrow \ell^{2}}\lesssim  \frac{1}{h^{D}}\lesssim\lambda^{-1} \  \text{from Theorem \ref{IW3}}.$$
Choose
	$B_{\lambda, \epsilon}^{1}$ as an operator whose multiplier is
	$$m_{B_{\lambda,\epsilon}^1}(\xi):=\sum_{J\in Z(k):\, \sigma\le j_k< \max\{h,\sigma\} }  H^{\Lambda}_{J}(\xi) \ \text{where
		$ B_{\lambda,\epsilon}^1f=\sum_{J\in Z(k):\, \sigma\le j_k< \max\{h,\sigma\} }  [H^{\Lambda}_{J}]^{\vee}*f$.}
	$$  
Fix $ \sigma\le j_k< \max\{h,\sigma\}  $. By   
applying (\ref{9nn2}) in Lemma \ref{ii11}, we have $$ \left\|\sum_{(j_1,\cdots,j_{k-1}): (j_1,\cdots,j_{k-1},j_k)\in Z(k)} [H^{\Lambda}_J]^{\vee}*f \right\|_{\ell^r}\lesssim \|f\|_{\ell^r}.$$
 By this with the  triangular inequality,  one can obtain that $$\left\|B_{\lambda, \epsilon}^{1}\right\|_{\ell^{r} \rightarrow \ell^{r}} \lesssim h\le\lambda^{\epsilon}.$$
	Hence, there remains to control
	\begin{align}\label{aa}
		\sum_{a/q\in \mathbb{Q}^{[d]}(Z_N)} S^{\Lambda} (a/ q) m_{h}( \xi -a/q),
	\end{align}
	where  $m_h$ is defined below (\ref{IW5}).
	
	  Consider the set $Y_{N}=Y_{N, \delta}$ in Theorem \ref{lem82}.
Then, we split (\ref{aa}) as
	\begin{align}\label{aa1}
		 \sum_{a/q \in \mathbb{Q}^{[d]}(Y_N)}S^{\Lambda} (a/ q) m_{h}( \xi -a/q)-\sum_{a/q \in \mathbb{Q}^{[d]}(Y_N\setminus Z_N)}  S^{\Lambda} (a/ q) m_{h}( \xi -a/q)
	\end{align}
	To treat   the second term,
	define the operator $A_{\lambda, \epsilon}^{2}$ whose  multiplier is
	\begin{align*}
		m_{A_{\lambda, \epsilon}^{2}}(\xi)=	-\sum_{a/q \in \mathbb{Q}^{[d]}(Y_N\setminus Z_N)} S^{\Lambda}(a/ q) m_{h}   (\xi-a/q).
	\end{align*}
	Then from $q\ge N\gtrsim h^{\frac{D}{s}}$, we have  $|S^{\Lambda}(a / q)|=O(q^{-s})=O(q^{-s/3}h^{-2D/3})=O(q^{-s/3}\lambda^{-1})$ since $[\lambda^{\epsilon}]=h$ and $  D=3/(2\epsilon)$. 
	Combining this with \eqref{po4} in Lemma~\ref{po3} and the disjointness of the  arcs indexed by
	$a/q\in\mathbb{Q}^{[d]}(Y_N)$, which follows from \eqref{hh} and the support properties of the cutoff functions
$\psi_{h{\bf 1}}\!\left(\frac{\xi-a/q}{2^{h/10}}\right),$
	we obtain
	\[
	\left\|A_{\lambda,\epsilon}^{2}\right\|_{\ell^{2}\to\ell^{2}}
	\lesssim
	\lambda^{-1}.
	\]

	Next, to treat the first term of (\ref{aa1}), take
	$$v_{h}(\xi):=\frac{1}{[2^{h/2}]^k} \sum_{t\in \{1,\cdots,[2^{h/2}]\}^k} e^{2 \pi i \sum_{\mathfrak{m}\in\Lambda} \xi_{\mathfrak{m}} \mathfrak{t}^{\mathfrak{m}}},$$ and split it as the sum of the following two  multipliers:
	\begin{align}
		&m_{B_{\lambda, \epsilon}^{2} }(\xi):=\sum_{a/q \in \mathbb{Q}^{[d]}(Y_N)}  v_{h}(\xi)m_{h} (\xi -a/q), \label{IW11}
		\\
		&m_{A_{\lambda, \epsilon}^{3} }(\xi):=\sum_{a/q \in \mathbb{Q}^{[d]}(Y_N)}( S^{\Lambda}(a / q) -v_h(\xi))m_{h}(\xi-a/q).\label{pol2}
	\end{align}
	\begin{proof}[Proof of $\ell^r$ boundedness for $B_{\lambda, \epsilon}^{2}$]
		We can write   $B_{\lambda, \epsilon}^{2}=U_{N} \circ V_{h}$ in (\ref{IW11}) where the operator $V_h$ has the bounded multiplier $v_h(\xi)$  such that $\left\|V_{h}\right\|_{\ell^{r} \rightarrow \ell^{r}} \lesssim_{r} 1$ and
		the operator $U_{N}$ has  the multiplier:
		\begin{align*}
			m_{U_N}(\xi):=\sum_{a/q \in \mathbb{Q}^{[d]}(Y_N)}m_{h}  (\xi-a/q).
		\end{align*}
		Then, we can express
		\begin{align*}
			m_{U_N}(\xi)= \sum_{a/q \in \mathbb{Q}^{[d]}(Y_N)} m \left(\frac{\xi-a/q}{\eta}\right),
		\end{align*}
		where  $\eta=2^{-(4/5)h}$ and $m$ is given by
		\begin{align*}
			\sum_{J\in Z(k); j_k\ge \max\{h,\sigma\}} \psi_{h{\bf 1}}\left(\frac{2^{-(4/5)h}\xi}{2^{h/10}}\right)\mathcal{H}^{\Lambda}_{J}(2^{-(4/5)h}\xi ).
		\end{align*}
Since Lemma~\ref{po3} yields the $L^p$-boundedness of the operator associated with the Fourier multiplier $m$, and since $\eta=2^{-4h/5}\leq e^{-N^{2\delta}}$ by \eqref{hh},  we can apply Theorem \ref{lem82} to the multiplier $m_{U_N}(\xi)$. Therefore, we  obtain
		\begin{align}\label{IW10}
			\left\|U_{N}\right\|_{\ell^{r} \rightarrow \ell^{r}}\lesssim_{r,\epsilon} (\ln  N)^{2/\delta}\lesssim_{r,\epsilon}(\ln \lambda)^{2 / \delta} .
		\end{align}
		Due to (\ref{IW10}) and $\left\|V_{h}\right\|_{\ell^{r} \rightarrow \ell^{r}} \lesssim_{r} 1$, we have
		$ 
		\ \|B_{\lambda, \epsilon}^{2} \|_{\ell^{r} \rightarrow \ell^{r}} \lesssim_{r,\epsilon}  \lambda^{\epsilon}.
		$ 
	\end{proof}
	
	\begin{proof}[Proof of $\ell^2$ boundedness for $A_{\lambda, \epsilon}^{3}$]
	For $\xi$ and $q$ satisfying the conditions in \eqref{pol2},  observe that \begin{itemize} \item For every $\mathbf{m}\in\Lambda$ and every
		$\mathbf{t}=(t_1,\ldots,t_k)\in\{1,\ldots,2^{h/2}\}^k$
		appearing in the exponential factor
		\[
		e^{2\pi i\sum_{\mathbf{m}\in\Lambda}
			\xi_{\mathbf{m}}\mathbf{t}^{\mathbf{m}}}
		\]
		defining $v_h(\xi)$, one has
		\[
		\left|
		\xi_{\mathbf{m}}-\frac{a_{\mathbf{m}}}{q}
		\right|
		\left|\mathbf{t}^{\mathbf{m}}\right|
		\leq
		C\,2^{-h\mathbf{1}\cdot\mathbf{m}}
		2^{h/10}
		2^{(h/2)\mathbf{1}\cdot\mathbf{m}}
		\lesssim
		2^{-2h/5}.
		\], 
		 \item  $ 1\leq q\leq e^{N^\delta}\ll e^{h^{1/2}} $  since $q\in Y_N\subset Z_{e^{N^\delta}}.$ \end{itemize}	
		Using the above size   after dividing $\{1,\cdots,2^{h/2}\}^k$ into smaller $q\times\cdots\times q$ sized boxes,  \begin{align*} 
			v_{h}(\xi)-S^{\Lambda}(a / q)=O(2^{-h/10})\ \text{leading} \  \left\|A_{\lambda, \epsilon}^{3}\right\|_{\ell^{2} \rightarrow \ell^{2}} \lesssim_{\epsilon}1/\lambda 
		\end{align*}  
in (\ref{pol2})
		from Plancherel's theorem. 
	\end{proof}
	Therefore, we proved that
	$$
	\ \sum_{i=1}^3\|A_{\lambda, \epsilon}^{i}\|_{\ell^2\rightarrow \ell^2}\le  C_{\epsilon} / \lambda\ \text{and}\ \sum_{i=1}^2\| B_{\lambda, \epsilon}^{i} \|_{\ell^r\rightarrow \ell^r}\le   C_{r,\epsilon}\lambda^{\epsilon}. $$
	So, we  finish the proof of Lemma \ref{IW2}.
\end{proof}  

This completes the proof of \eqref{wep} under the assumptions in \eqref{wep2}, and hence establishes (ii) of Main Theorem~\ref{mt2}.

\section{Proofs of Preliminary Lemmas}\label{PPL}
\subsection{Multiple Hilbert Transform}\label{MHT}

\begin{proof}[Proof of Lemma \ref{po3}]
	Let $\Omega'\subset \Omega$ and $J\in \mathbb{Z}_+^k$.  We set the operator $ \mathcal{H}^{\Omega'}_J$ by defining $\widehat{\mathcal{H}^{\Omega'}_J(f)}(\xi_{\Omega})=\mathcal{H}^{\Omega'}_{J}(\xi_{\Omega}) \widehat{f}(\xi_{\Omega})$ where $\xi_{\Omega}\in \mathbb{R}^{|\Omega|}$.
	Let $\{ {\bf e}_{\mathfrak{n}}\}_{\mathfrak{n}\in \Omega}$ be standard basis in $\mathbb{R}^{|\Omega|}$. Then for each $\mathfrak{n}\in \Omega'$, let
	$$R^{\mathfrak{n}} \mathcal{H}^{\Omega'}_J = \mathcal{H}_J^{\Omega'\setminus \{\mathfrak{n}\}}\ \text{and}\ D^{\mathfrak{n}}\mathcal{H}^{\Omega'}_J  = \mathcal{H}^{\Omega'}_J- R^{\mathfrak{n}} \mathcal{H}^{\Omega'}_J$$
	so that $ ( D^{\mathfrak{n}}+ R^{\mathfrak{n}}) \mathcal{H}^{\Omega'}_J=\mathcal{H}^{\Omega'}_J.$ 
	On the multiplier side,  we can also denote 
	\begin{align}\label{0bb}
		D^{\mathfrak{n}}\mathcal{H}^{\Omega'}_J (\xi)=\mathcal{H}^{\Omega'}_J (\xi)-\mathcal{H}^{\Omega'\setminus \{\mathfrak{n}\}}_J (\xi)\ \text{so that}\ 
		D^{\mathfrak{n}}\mathcal{H}^{\Omega'}_Jf= \left([D^{\mathfrak{n}}\mathcal{H}^{\Omega'}_J](\cdot)\widehat{f}\right)^{\vee}.\end{align}
	Moreover, for   $J $ and $\mathfrak{n}\in \Omega$, set the two projection operators:
	$$P_J^{\mathfrak{n}}f(x)=\int e^{2\pi i x\cdot \xi} \psi\left(2^{J\cdot \mathfrak{n}} \xi_{\mathfrak{n}} \right) \widehat{f}\left((\xi_{\mathfrak{m}} )_{\mathfrak{m}\in \Omega} \right) d\xi \ \text{and}\    Q_J^{\mathfrak{n}}=I-P_J^{\mathfrak{n}}.$$
	For finitely many operators $T_1,\cdots,T_r$, denote their composition $T_1\circ \cdots \circ T_r= \prod_{i=1}^r T_i $.
	
	 By applying $I=\prod_{\mathfrak{n}\in \Omega} (P_J^{\mathfrak{n}} +  Q_J^{\mathfrak{n}}) $ where $I$ is the identity operator, express 
	\begin{align}
		\mathcal{H}^{\Omega}_J=\mathcal{H}^{\Omega}_J \sum_{(U,V): U\cap V=\emptyset \  \text{and}\ U\cup V=\Omega} \left( \prod_{\mathfrak{m}\in U}P_J^{\mathfrak{m}}  \prod_{\mathfrak{n}\in V}Q_J^{\mathfrak{n}} \right).
	\end{align}
	Combined with $\mathcal{H}^{\Omega}_J= \prod_{\mathfrak{n}\in U}( D^{\mathfrak{n}} + R^{\mathfrak{n}})\mathcal{H}^{\Omega}_J $, decompose  
	\begin{align}\label{ss33}
		\mathcal{H}^{\Omega}_J=  \sum_{(U,V): U\cap V=\emptyset \  \text{and}\ U\cup V=\Omega} \ \sum_{A\subset U} \left(\prod_{\mathfrak{n}\in   A} D^{\mathfrak{n}}\prod_{\mathfrak{n}\in U\setminus A} R^{\mathfrak{n}}\right)\mathcal{H}^{\Omega}_J \left(\prod_{\mathfrak{m}\in U}  P_J^{\mathfrak{m}}  \prod_{\mathfrak{n}\in V}Q_J^{\mathfrak{n}}\right). 
	\end{align}
	Fix $A\subset U$ and $V$ in the above.  Then using (\ref{0bb}), we have the Fourier multiplier of $$\left(\prod_{\mathfrak{n}\in   A} D^{\mathfrak{n}}\prod_{\mathfrak{n}\in U\setminus A} R^{\mathfrak{n}}\right)\mathcal{H}^{\Omega}_J \left(\prod_{\mathfrak{m}\in U}  P_J^{\mathfrak{m}}  \prod_{\mathfrak{n}\in V}Q_J^{\mathfrak{n}}\right),$$ given by
	\begin{align*}
		\left( \prod_{\mathfrak{n}\in   A} D^{\mathfrak{n}} \right)\mathcal{H}_J^{A\cup V} \left((\xi_{\mathfrak{m}} )_{\mathfrak{m}\in \Omega} \right) \prod_{\mathfrak{m}\in U}  \psi \left(2^{J\cdot \mathfrak{m}} \xi_{\mathfrak{m}} \right)    \prod_{\mathfrak{m}\in V}\psi^c\left(2^{J\cdot \mathfrak{m}} \xi_{\mathfrak{m}} \right).  
	\end{align*}
	
	If $\text{rank}(A\cup V)<k$, then $A\cup V$ is an even set, which implies $\mathcal{H}_J^{A\cup V} \left((\xi_{\mathfrak{m}} )_{\mathfrak{m}\in \Omega} \right) \equiv 0$.
	Let $\text{rank}(A\cup V)=k$. Then by applying the mean value properties and the van der Corput lemma, there exists  $c>0$ such that
	\begin{align}\label{abb1}
		\left| \left( \prod_{\mathfrak{n}\in   A} D^{\mathfrak{n}} \right)\mathcal{H}_J^{A\cup V} \left((\xi_{\mathfrak{m}} )_{\mathfrak{m}\in \Omega} \right)\right|\lesssim \min\{ |\xi_{\mathfrak{m}}2^{J\cdot \mathfrak{m}}|, |\xi_{\mathfrak{n}}2^{J\cdot \mathfrak{n}}|^{-c}: \mathfrak{m}\in A, \mathfrak{n}\in V\}.
	\end{align}
This, together with the rank condition $\text{rank}(A\cup V)=k$ and the support condition yields that
	\begin{align*}
		&\sum_{J}\left| \left( \prod_{\mathfrak{n}\in   A} D^{\mathfrak{n}} \right)\mathcal{H}_J^{A\cup V} \left((\xi_{\mathfrak{m}} )_{\mathfrak{m}\in \Omega} \right) \prod_{\mathfrak{m}\in U}  \psi \left(2^{J\cdot \mathfrak{m}} \xi_{\mathfrak{m}} \right)    \prod_{\mathfrak{m}\in V}\psi^c\left(2^{J\cdot \mathfrak{m}} \xi_{\mathfrak{m}} \right)  \right|\\
		& \lesssim  \sum_J   \min\{ |\xi_{\mathfrak{m}}2^{J\cdot \mathfrak{m}}|, |\xi_{\mathfrak{n}}2^{J\cdot \mathfrak{n}}|^{-c}: \mathfrak{m}\in A, \mathfrak{n}\in V\}\prod_{\mathfrak{m}\in U}  \psi \left(2^{J\cdot \mathfrak{m}} \xi_{\mathfrak{m}} \right)    \prod_{\mathfrak{m}\in V}\psi^c\left(2^{J\cdot \mathfrak{m}} \xi_{\mathfrak{m}} \right)\le C.
	\end{align*}
	This holds for all $A\subset U$ and $V$ in (\ref{ss33}).  Therefore, we have $\sum_{J \in \mathbb{Z}_+^k}  |\mathcal{H}^{\Omega}_{J}(\xi_{\Omega}) |  \le C. $ Here we can include the case $A=\emptyset$ where we only utilize the decay property of (\ref{abb1}) since $\text{rank}(V)= k$ for the case. Similarly, we can obtain 	\begin{align*}
		\sum_{J\in \mathbb{Z}_+^k}   |\mathcal{H}^{\Omega}_{J}(\xi_{\Omega}) |^{\epsilon} \le C. 
	\end{align*} 
	
	 Set $\varphi(\cdot):=\psi(\cdot/2)-\psi(\cdot)$.  Applying the Littlewood-Paley inequality associated with $\prod_{\mathfrak{m}\in A}  \varphi \left(2^{J\cdot \mathfrak{m}} \xi_{\mathfrak{m}} /2^{\ell_{\mathfrak{m}}}\right)    \prod_{\mathfrak{m}\in V}\varphi \left(2^{J\cdot \mathfrak{m}} \xi_{\mathfrak{m}}/2^{\ell_{\mathfrak{m}}} \right)$
	and the decay estimate $2^{-c\sum_{\mathfrak{m}\in A\cup V} |\ell_{\mathfrak{m}}|}$ of  (\ref{abb1})    on its  support, one has  the $L^p$ boundedness for
	\begin{align}
		\left\|\sum_{J} \left(\prod_{\mathfrak{n}\in   A} D^{\mathfrak{n}}\prod_{\mathfrak{n}\in U\setminus A} R^{\mathfrak{n}}\right)\mathcal{H}^{\Omega}_J \left(\prod_{\mathfrak{m}\in U}  P_J^{\mathfrak{m}}  \prod_{\mathfrak{n}\in V}Q_J^{\mathfrak{n}}\right) \right\|_{L^p(\mathbb{R}^{|\Omega|})\rightarrow L^p(\mathbb{R}^{|\Omega|})}\le C,
	\end{align}
	where $1<p<\infty$. This   yields (\ref{poo4}).
\end{proof}

\subsection{Averaged Gauss Sum Estimates} \label{ag}
In this subsection, we shall prove Proposition \ref{lem22p}.

\noindent\textbf{Anisotropic Solution Counting.}
We state the following multi-parameter extension of   Theorem 1 of Konyagin \cite{Kon} for polynomial congruences over
anisotropic multi-scale boxes
$
\prod_{\nu=1}^k [1,2^{j_\nu}] \subset \mathbb{Z}^k
$,
whose proof proceeds via standard induction on the parameter $k$ and valuation-level partitioning and is therefore omitted.  
\begin{lemma}[Konyagin-type Theorem]\label{mkl1}
Fix $q\in \mathbb{N}$.	Let $\Omega \subset \mathbb{Z}_+^k \setminus \{\mathbf{0}\}$ be a finite set  and consider $P(\mathfrak{t})=\sum_{\mathfrak{m}\in \Omega} a_{\mathfrak{m}}t^{\mathfrak{m}}$ in $k$ parameters $\mathfrak{t} = (t_1, \dots, t_k)$ having total degree $$d_{\Omega} = \max\{|\mathfrak{m}| : \mathfrak{m} \in \Omega\}\ge 2\  \text{and}\ \gcd((a_{\mathfrak{m}})_{\mathfrak{m} \in \Omega}, q) = 1.$$
	Let $j=(j_1,\cdots,j_{k})\in Z(k)$.  Then, there exist  constants $C, c>0$, depending only on $k$ and $d_{\Omega}$, such that
	\begin{equation}\label{kya3}
		|\{ \mathfrak{t} \in  \prod_{\nu=1}^k [1, 2^{j_\nu}]: P(\mathfrak{t}) \equiv 0 \pmod q \}| \le C \, 2^{j_1 + \dots + j_k} (2^{-cj_k}+q^{-c}).
	\end{equation}
\end{lemma}

\noindent \textbf{Proof of the Averaged  Gauss Sum Estimate.} 
We combine the Gauss sum estimate (\ref{g2}) with the counting estimate (\ref{kya3}) to prove Proposition \ref{lem22p}.

 \begin{proof}[Proof of Proposition \ref{lem22p}] To keep the notation familiar, we set $\Omega=\Lambda$ throughout this proof.
	We split the averaging domain into the two subsets
	\begin{align*}
		A_{\rm gauss}
		&:=
		\left\{
		\mathfrak{t}_2\in
		\prod_{\nu=\ell+1}^k[1,2^{j_\nu}]
		:
		\bigl((a_m(\mathfrak{t}_2))_{m\in\Lambda_1'},q\bigr)
		<
		q^{2/3}
		\right\},\\
		A_{\rm sub}
		&:=
		\left\{
		\mathfrak{t}_2\in
		\prod_{\nu=\ell+1}^k[1,2^{j_\nu}]
		:
		\bigl((a_m(\mathfrak{t}_2))_{m\in\Lambda_1'},q\bigr)
		\ge
		q^{2/3}
		\right\}.
	\end{align*}
	
	We first consider the contribution from $A_{\rm gauss}$.
	For every $\mathfrak t_2\in A_{\rm gauss}$, we may write
	\[
	\frac{a(\mathfrak t_2)}{q}
	=
	\frac{\widetilde a}{\widetilde q},
	\qquad
	(\widetilde a,\widetilde q)=1,
	\]
	where
	\[
	q^{1/3}< \widetilde q\le q.
	\]
By	 Lemma~\ref{el22}, for every $\mathfrak t_2\in A_{\rm gauss}$, there exists $c>0$ such that
	\begin{align*}
		\left|
		S^{\Lambda_1}
		\left(
		\frac{a(\mathfrak t_2)}{q}
		\right)
		\right|
		=
		\frac1{q^\ell}
		\frac{q^\ell}{\widetilde q^\ell}
		\left|
		\sum_{\mathfrak t_1\in[1,\widetilde q]^\ell}
		e\!\left(
		\sum_{m\in\Lambda_1}
		\frac{\widetilde a_m}{\widetilde q}
		\mathfrak t_1^m
		\right)
		\right|   
		\lesssim q^{-c}.
	\end{align*}
	Therefore,
	\begin{align}\label{4040}
		\frac1{2^{j_{\ell+1}+\cdots+j_k}}
		\sum_{\mathfrak t_2\in A_{\rm gauss}}
		\left|
		S^{\Lambda_1}
		\left(
		\frac{a(\mathfrak t_2)}{q}
		\right)
		\right|
		\lesssim
	q^{-c}.
	\end{align}
	
	Next we estimate the contribution from $A_{\rm sub}$.
	For every divisor $\widetilde q\mid q$, let
	\[
	A_{\widetilde q}
	:=
	\left\{
	\mathfrak t_2:
	\widetilde q
	\mid
	a_m(\mathfrak t_2)
	\text{ for all }
	m\in\Lambda_1'
	\right\}.
	\]
	Since the lower bound of common divisors $\widetilde q$  in $ A_{\rm sub}$ satisfies:
	$q^{2/3}\le\widetilde q\le q$, 
	\[
	|A_{\rm sub}|
	\le
	\sum_{\substack{\widetilde q\mid q\\
			q^{2/3}\le\widetilde q\le q}}
	|A_{\widetilde q}|.
	\]
	
	\begin{claim}[Primitive Projection]\label{hc1}
		Let $\Lambda\subset \mathbb Z^k$ be finite. Write
		\begin{align*}
			\Lambda_1
			&=
			\bigl\{
			m\in\mathbb{Z}^{\ell}
			:
			(m,n)\in\Lambda
			\text{ for some } n\in\mathbb{Z}^{k-\ell}
			\bigr\},\\
			\Lambda_2
			&=
			\bigl\{
			n\in\mathbb{Z}^{k-\ell}
			:
			(m,n)\in\Lambda
			\text{ for some } m\in\mathbb{Z}^{\ell}
			\bigr\}.
		\end{align*}
		For $(m,n)\notin\Lambda$, set $a_{mn}=0$. Assume that for an integer $q\ge 2$, 
		\[
		\bigl(q,(a_{mn})_{(m,n)\in\Lambda}\bigr)=1.
		\]
		Then there exists an  integer vector $(c_m)_{m\in \Lambda_1}$ such that
		\[
		\!\left(
		q,
		\left(\sum_{m\in \Lambda_1} c_m a_{mn}\right)_{n\in\Lambda_2}
		\right)=1.
		\]
	\end{claim}
	Since $(\widetilde q,(a_{mn})_{(m,n)\in \Lambda})=1$ when $(q,a)=1$ and $\widetilde q\mid q$,
	Claim~\ref{hc1} yields an integer vector
	$(c_m)_{m\in\Lambda_1}$ such that
\begin{align}\label{yr}
		(\widetilde q,
	(\mathfrak a_{n})_{n\in\Lambda_2})
	=1,\ \text{where
		$
		\mathfrak a_{n}
		=
		\sum_{m\in\Lambda_1}
		c_m a_{mn}.
		$}
\end{align}
	Define
	\[
	\widetilde a(\mathfrak t_2)
	:=
	\sum_{n\in\Lambda_2}
	\mathfrak a_{n}
	\mathfrak t_2^{\mathfrak n}.
	\]
	Then, observe that 
	\begin{align}\label{yr1}
	\widetilde q
	\mid
	a_m(\mathfrak t_2)
	\quad
	\text{for every }
	m\in\Lambda_1 \Rightarrow \widetilde q
	\mid
	\widetilde a(\mathfrak t_2).	
	\end{align}
	Hence, by (\ref{kya3}) in Lemma~\ref{mkl1} combined with (\ref{yr}) and (\ref{yr1}), there exists $c'>0$ such that
	\begin{align*}
		|A_{\widetilde q}|
		\le
		\left|
		\left\{
		\mathfrak t_2
		:
		\widetilde a(\mathfrak t_2)
		\equiv0
		\pmod{\widetilde q}
		\right\}
		\right|  
		\lesssim 
		2^{j_{\ell+1}+\cdots+j_k}
		(2^{-c'j_k}
		+\widetilde q^{-c'}).
	\end{align*}
Therefore,	since $\widetilde q\ge q^{2/3}$, there exists $c>0$ such that
	\begin{align*}
		\frac1{2^{j_{\ell+1}+\cdots+j_k}}
		\sum_{\mathfrak t_2\in A_{\rm sub}}
		\left|
		S^{\Lambda_1}
		\left(
		\frac{a(\mathfrak t_2)}{q}
		\right)
		\right|   
		&\le
		\sum_{\substack{\widetilde q\mid q\\
				q^{2/3}\le\widetilde q\le q}}
		\frac{|A_{\widetilde q}|}
		{2^{j_{\ell+1}+\cdots+j_k}}    \\
		&\lesssim d(q)
		\left(
		2^{-cj_k}
		+
		q^{-c}
		\right).
	\end{align*}
	This combined with (\ref{4040}) yields the desired bound of (\ref{po10}). 
\end{proof}

\begin{proof}[Proof of Claim~\ref{hc1}]
	Write the prime factorization of $q$ as
	\begin{align*}
		q=p_1^{\alpha_1}\cdots p_M^{\alpha_M},
	\end{align*}
	where $p_1,\ldots,p_M$ are distinct primes. Since
	\begin{align*}
		\gcd\bigl(q,(a_{mn})_{(m,n)\in\Lambda}\bigr)=1,
	\end{align*}
	for each $1\le i\le M$, there exists
	$(m,n)\in\Lambda$ such that
	\begin{align*}
		a_{mn}\not\equiv 0\pmod{p_i}.
	\end{align*}
	
	Consider the linear map
	\begin{align*}
		T_i:\mathbb{F}_{p_i}^{\,|\Lambda_1|}
		&\longrightarrow
		\mathbb{F}_{p_i}^{\,|\Lambda_2|},\\
		(c_m)_{m\in\Lambda_1}
		&\longmapsto
		\left(
		\sum_{m\in\Lambda_1}c_ma_{mn}
		\right)_{n\in\Lambda_2}.
	\end{align*}
	The preceding observation shows that $T_i$ is not the zero map.
	Therefore, one may choose
	\begin{align*}
		c^{(i)}
		=
		(c_m^{(i)})_{m\in\Lambda_1}
		\in\mathbb{F}_{p_i}^{\,|\Lambda_1|}
	\end{align*}
	such that
	\begin{align*}
		T_i(c^{(i)})\neq 0.
	\end{align*}
	
	By applying the Chinese remainder theorem separately to each
	$m\in\Lambda_1$, we can choose integers $c_m$ such that
	\begin{align*}
		c_m\equiv c_m^{(i)}\pmod{p_i}
		\qquad
		(1\le i\le M).
	\end{align*}
	For this choice, we have
	\begin{align*}
		\left(
		\sum_{m\in\Lambda_1}c_ma_{mn}
		\right)_{n\in\Lambda_2}
		\equiv
		T_i(c^{(i)})
		\not\equiv 0
		\pmod{p_i}
	\end{align*}
	for every $1\le i\le M$. Here, the vector being nonzero modulo
	$p_i$ means that at least one of its coordinates is nonzero modulo
	$p_i$.
	
	Consequently, no prime divisor $p_i$ of $q$ divides all the integers
	\begin{align*}
		\sum_{m\in\Lambda_1}c_ma_{mn},
		\qquad n\in\Lambda_2.
	\end{align*}
	It follows that
	\begin{align*}
		\gcd\left(
		q,
		\left(
		\sum_{m\in\Lambda_1}c_ma_{mn}
		\right)_{n\in\Lambda_2}
		\right)=1.
	\end{align*}
	This completes the proof.
\end{proof}

Thus, we have proved   Proposition \ref{lem22p}.

\end{document}